\documentclass[reqno,12pt,letterpaper]{amsart}
\usepackage{sdmacros}

\let\smallsection=\subsubsection

\let\epsilon=\varepsilon % sorry Knuth
\def\Mext{M_{\mathrm{ext}}}

\newtheorem*{conj}{Conjecture}
\newtheorem{example}{Example}
\newtheorem*{ADRegularDefn}{Definition \ref{d:ad-regular}}
\newtheorem{defn}[prop]{Definition}

\newcommand{\betaXBound}{\delta\exp\big[-\mathbf{K}(1-\delta)^{-14}(1+\log^{14}C)\big]}
\newcommand{\betaXBoundBFC}{\delta\exp\big[-\mathbf{K}(1-\delta)^{-28}(1+\log^{14}\mathbf{C})\big]}

\newcommand{\RR}{{\mathbb R}}

\newcommand{\setdim}{\delta}
\newcommand{\h}{\alpha}
\newcommand{\metSpace}{\mathcal{M}}
\newcommand{\metricChar}{d}
\newcommand{\X}{\mathcal{X}}
\newcommand{\ZZ}{\mathbb{Z}}
\title[Spectral gaps, additive energy, and a fractal uncertainty principle]%
{Spectral gaps, additive energy,\\
and a fractal uncertainty principle}
\author{Semyon Dyatlov}
\email{dyatlov@math.mit.edu}
\author{Joshua Zahl}
\email{jzahl@mit.edu}
\address{Department of Mathematics, Massachusetts Institute of Technology,
77 Massachusetts Ave, Cambridge, MA 02139}
\begin{document}

\begin{abstract}
We obtain an essential spectral gap for $n$-dimensional convex co-compact hyperbolic manifolds
with the dimension $\delta$ of the limit set close
to ${n-1\over 2}$. The size of the gap is expressed using
the additive energy of stereographic projections of the limit set.
This additive energy can in turn be estimated in terms of
the constants in Ahlfors-David regularity of the limit set.
Our proofs use new microlocal methods, in particular a notion of a fractal uncertainty principle.
\end{abstract}

\maketitle

%%%%%%%%%%%%%%%%%%%%%%%%%%%%%%%%%%%%%%%%%%%%%%%%%%%%%%%%%%%%%%%%%%%%%%%%%%%%%%%%
%                                 INTRODUCTION                                 %
%%%%%%%%%%%%%%%%%%%%%%%%%%%%%%%%%%%%%%%%%%%%%%%%%%%%%%%%%%%%%%%%%%%%%%%%%%%%%%%%
\addtocounter{section}{1}
\addcontentsline{toc}{section}{1. Introduction}

In this paper we study essential spectral gaps for convex co-compact
hyperbolic quotients $M=\Gamma\backslash \mathbb H^n$. To formulate our result in the simplest
setting, consider $n=2$ and take the \emph{Selberg zeta function}~\cite[(10.1)]{Borthwick}
$$
Z_M(\lambda)=\prod_{\ell\in \mathcal L_M}\prod_{k=0}^\infty \big(1-e^{-(s+k)\ell}\big),\quad
s={1\over 2}-i\lambda,
$$
where $\mathcal L_M$ consists of all lengths of primitive closed geodesics on $M$ (with multiplicity).
The spectral representation
of $Z_M$ implies that it has only finitely many singularities (that is, zeros and poles) in $\{\Im\lambda>0\}$.
The work of Patterson~\cite{Patterson} and Sullivan~\cite{Sullivan}
shows that $Z_M(\lambda)$ has no singularities in $\{\Im\lambda>\delta-{1\over 2}\}$
and a simple zero at $i(\delta-{1\over 2})$, where $\delta\in [0,1]$ is the dimension of the limit
set of the group (see~\eqref{e:delta}). Therefore $Z_M$ has finitely many singularities
in $\{\Im\lambda>-\max(0,{1\over 2}-\delta)\}$.

For $\delta\in \big(0,{1\over 2}\big]$, Naud~\cite{Naud} obtained the stronger
statement that $Z_M$
has finitely many singularities in $\{\Im\lambda>-\beta\}$ for some $\beta$
strictly greater than ${1\over 2}-\delta$.
Naud's result, generalized to
higher dimensional quotients by Stoyanov~\cite{Stoyanov1}, is based
on the method of Dolgopyat~\cite{Dolgopyat} and does not specify the size of the improvement.
Our first result in particular gives explicit estimates on the value of~$\beta$ when $\delta={1\over 2}$:
%%%%%%%%%%%%%%%%%%%%%%%%%%%%%%%%%%%%%%%%%%%%%%%%%%%%%%%%%%%%%%%%%%%%%%%%%%%%%%%%
\begin{theo}
  \label{t:marketing}
Let $M=\Gamma\backslash\mathbb H^2$ be a convex co-compact hyperbolic surface.
Then for each $\varepsilon>0$, the function $Z_M$ has finitely many singularities in $\{\Im\lambda>-\beta+\varepsilon\}$,
where
\begin{equation}
  \label{e:our-gap}
\beta={3\over 8}\Big({1\over 2}-\delta\Big)+{\beta_E\over 16},\quad
\beta_E:=\betaXBoundBFC.
\end{equation}
Here $\mathbf K>0$ is a global constant and $\mathbf C\geq 1$ is the constant
in the Ahlfors-David regularity for the limit set $\Lambda_\Gamma$,
see~\eqref{e:ad-regular-limit}.
\end{theo}
%%%%%%%%%%%%%%%%%%%%%%%%%%%%%%%%%%%%%%%%%%%%%%%%%%%%%%%%%%%%%%%%%%%%%%%%%%%%%%%%
See Figure~\ref{f:gaps}. We say that $M$ has an \emph{essential spectral gap}
of size $\beta$. Theorem~\ref{t:main} improves over the standard gap
$\beta_{\mathrm{std}}=\max(0,{1\over 2}-\delta)$
for all surfaces with $\delta={1\over 2}$.
In fact, we show on the example of three-funnel surfaces
that the regularity constant~$\mathbf C$ is bounded when the surface $M$
varies in a compact set in the moduli space, and thus~\eqref{e:our-gap}
improves over $\beta_{\mathrm{std}}$ for surfaces close to those with $\delta={1\over 2}$~--
see Theorem~\ref{t:3funny}
in~\S\ref{s:3fun}.
This includes some surfaces with $\delta>{1\over 2}$, which to our knowledge provide the first
examples of fractal chaotic systems where the pressure $P({1\over 2})$ is
positive but there is an essential spectral gap below the real line.
On the other hand our methods cannot improve over $\beta_{\mathrm{std}}$
even in the most favorable situation unless
$\delta\in ({5\over 11},{3\over 5})$~-- see the remark following Theorem~\ref{t:ae-reduction}.

A numerical investigation of the gap was done by Borthwick~\cite[\S7]{BorthwickNum}
and Borthwick--Weich~\cite{Borthwick-Weich} (see Figure~\ref{f:gaps}(b)).
Both~\cite[Figure~14]{Borthwick-Weich} and the experimental results of~\cite[Figure~4]{prl}
suggest that the improvement $\beta-\beta_{\mathrm{std}}$ should indeed be larger when $\delta\approx {1\over 2}$
than for other values of $\delta$.
%%%%%%%%%%%%%%%%%%%%%%%%%%%%%%%%%%%%%%%%%%%%%%%%%%%%%%%%%%%%%%%%%%%%%%%%%%%%%%%%
\begin{figure}
\includegraphics[scale=0.975]{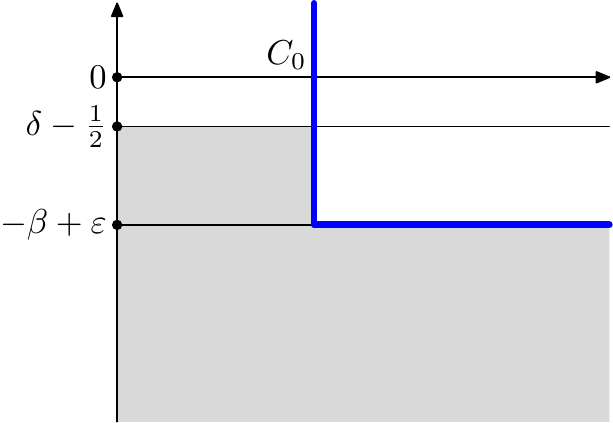}
\qquad
\includegraphics[scale=0.975]{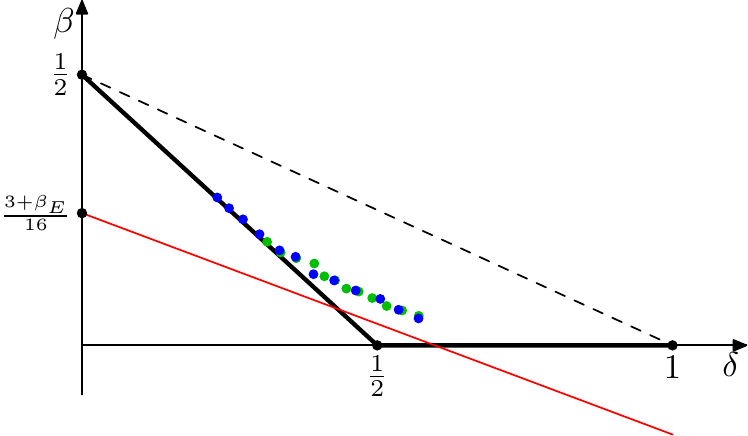}
\hbox to\hsize{\hss(a)\hss\hss (b)\hss}
\caption{(a) Resonance free regions (in white) in dimension 2 for $\delta<{1\over 2}$ given by the Patterson--Sullivan gap~\eqref{e:standard-gap}
and an essential spectral gap of Theorem~\ref{t:main} (outlined in blue).
(b) Essential spectral gaps~$\beta$ in dimension~2 depending on $\delta$: the standard gap (in bold) and
the Jakobson--Naud conjecture (dashed).
The circles correspond to numerically computed
spectral gaps for symmetric 3-funneled surfaces (blue)
and 4-funneled surfaces (green) from~\cite[Figure~14]{Borthwick-Weich}
(specifically, $G^{I_1}_{100}$ in the notation of~\cite{Borthwick-Weich}).
The red line is the gap given by Theorem~\ref{t:marketing} for
a \emph{sample choice} of $\beta_E:={1\over 2}$.}
\label{f:gaps}
\end{figure}
%%%%%%%%%%%%%%%%%%%%%%%%%%%%%%%%%%%%%%%%%%%%%%%%%%%%%%%%%%%%%%%%%%%%%%%%%%%%%%%%

In fact, our results apply to convex co-compact
quotients $M$ of any dimension $n$ and give bounds on the
\emph{scattering resolvent}
$$
R(\lambda)=\Big(-\Delta-{(n-1)^2\over 4}-\lambda^2\Big)^{-1}:L^2_{\comp}(M)\to H^2_{\loc}(M),\quad
\lambda\in\mathbb C,
$$
which is the meromorphic continuation of the $L^2$ resolvent from the upper half-plane~-- see~\S\ref{s:scattering-resolvent}.
The standard Patterson--Sullivan gap in this setting is~\cite[\S14.4]{Borthwick}
\begin{equation}
  \label{e:standard-gap}
R(\lambda)\text{ has only finitely many poles in }
\Big\{\Im\lambda \geq -\max\Big(0,{n-1\over 2}-\delta\Big)\Big\}.
\end{equation}

The poles of $R(\lambda)$, called \emph{resonances}, are related to the scattering poles and to the zeroes of $Z_M$ as proved
by Bunke--Olbrich~\cite{Bunke-Olbrich} and Patterson--Perry~\cite{Patterson-Perry};
see Borthwick~\cite[Chapter~10]{Borthwick} for an expository proof and the history of the subject.
Therefore Theorem~\ref{t:marketing} is a direct corollary of the following
stronger statement:
%%%%%%%%%%%%%%%%%%%%%%%%%%%%%%%%%%%%%%%%%%%%%%%%%%%%%%%%%%%%%%%%%%%%%%%%%%%%%%%%
\begin{theo}
  \label{t:main}
Under the assumptions of Theorem~\ref{t:marketing}, we have the resolvent estimate
\begin{equation}
  \label{e:essential-gap2}
\|\chi R(\lambda)\chi\|_{L^2\to L^2}\leq C|\lambda|^{-1-2\min(0,\Im\lambda)+\varepsilon},\quad
|\lambda|>C_0,\ 
\Im\lambda\in [-\beta+\varepsilon,1],
\end{equation}
where $\beta$ is given by~\eqref{e:our-gap}, $\varepsilon>0$ and $\chi\in C_0^\infty(M)$ are arbitrary,
the constant $C_0$ depends on $\varepsilon$, and the constant $C$ depends on $\varepsilon,\chi$.
\end{theo}
%%%%%%%%%%%%%%%%%%%%%%%%%%%%%%%%%%%%%%%%%%%%%%%%%%%%%%%%%%%%%%%%%%%%%%%%%%%%%%%%
Spectral gaps for the special case of arithmetic quotients
have recently found important applications to diophantine problems, see Bourgain--Gamburd--Sarnak~\cite{BGS}
and Magee--Oh--Winter~\cite{MOW}. For these applications one needs a uniform resonance free region for
congruence subgroups of $\Gamma$; a uniform logarithmic region was obtained in~\cite{BGS}
and a uniform gap by Oh--Winter~\cite{Oh-Winter}.

Our results have
the following features compared to previous works on spectral gaps:
\begin{itemize}
\item The size of the gap is more explicit, expressed in terms of additive
energy of the limit set (Theorem~\ref{t:ae-reduction}) or the constants
in Ahlfors-David regularity of this set. Compared
to Dolgopyat's method, we decouple analytical
aspects of the problem from combinatorial ones. This
makes it more feasible to compute the size of the gap for specific hyperbolic
manifolds.
\item We obtain a polynomial resolvent bound~\eqref{e:essential-gap2}, rather
than just a resonance free strip. This can be used to obtain polynomial
bounds on other objects such as Eisenstein functions.
\item We rely on $C^\infty$ microlocal analysis (for instance, nowhere using explicitly
that $M$ is analytic); this gives hope that our strategy may apply
to more general classes of manifolds.
\end{itemize}
Regarding the last item on the above list, we make the following conjecture
which would improve over the \emph{pressure gap}
studied in more general cases by Ikawa~\cite{Ikawa}, Gaspard--Rice~\cite{GaspardRice}, and Nonnenmacher--Zworski~\cite{NoZwActa}. 
%%%%%%%%%%%%%%%%%%%%%%%%%%%%%%%%%%%%%%%%%%%%%%%%%%%%%%%%%%%%%%%%%%%%%%%%%%%%%%%%
\begin{conj}
Let $(M,g)$ be a convex co-compact hyperbolic surface with $\delta={1\over 2}$.
Then all sufficiently small smooth compactly supported metric perturbations
of $(M,g)$ satisfy~\eqref{e:essential-gap2} for some $\beta>0$.
\end{conj}
%%%%%%%%%%%%%%%%%%%%%%%%%%%%%%%%%%%%%%%%%%%%%%%%%%%%%%%%%%%%%%%%%%%%%%%%%%%%%%%%
This conjecture is related to the improved gaps obtained for scattering
by several convex obstacles
by Petkov--Stoyanov~\cite{PetkovStoyanov} and Stoyanov~\cite{Stoyanov2}
using the methods of~\cite{Dolgopyat}.

We now describe the scheme of proof of Theorem~\ref{t:main}:

%%%%%%%%%%%%%%%%%%%%%%%%%%%%%%%%%%%%%%%%%%%%%%%%%%%%%%%%%%%%%%%%%%%%%%%%%%%%%%%%
\subsection{Spectral gaps via a fractal uncertainty principle}
  \label{s:spfup}

We first reduce the estimate~\eqref{e:essential-gap2} to a fractal uncertainty principle.
To state it, let $\Lambda_\Gamma\subset\mathbb S^{n-1}$ be the limit set of
the group $\Gamma$, $M=\Gamma\backslash\mathbb H^n$ (see~\eqref{e:Lambda-Gamma}) and
 denote by $\indic_{\Lambda_\Gamma(\alpha)}$ the indicator function
of the set
\begin{equation}
  \label{e:limit-nbhd}
\Lambda_\Gamma(\alpha)=\{y\in \mathbb S^{n-1}\mid d(y,\Lambda_\Gamma)\leq\alpha\}
\end{equation}
where $d(y,y')=|y-y'|$ denotes the Euclidean distance function
on $\mathbb S^{n-1}\subset\mathbb R^n$.
Note that the Minkowski dimension of $\Lambda_\Gamma$ is equal to $\delta$, therefore (see~\eqref{e:AD-estimate-Lebesgue})
\begin{equation}
  \label{e:LG-upper-1}
\alpha^{n-1-\delta}/C\,\leq\,\mu_L(\Lambda_\Gamma(\alpha))\,\leq\, C \alpha^{n-1-\delta},\quad
\alpha\in (0,1)
\end{equation}
where $\mu_L$ denotes the Lebesgue measure on $\mathbb S^{n-1}$.

Define the operator $\mathcal B_\chi=\mathcal B_\chi(h):L^2(\mathbb S^{n-1})\to L^2(\mathbb S^{n-1})$
by
\begin{equation}
  \label{e:B-chi}
\mathcal B_\chi v(y)=(2\pi h)^{1-n\over 2}\int_{\mathbb S^{n-1}} |y-y'|^{2i/h} \chi(y,y') v(y')\,dy'
\end{equation}
where $dy'$ is the standard volume form on $\mathbb S^{n-1}$ and $\chi\in C_0^\infty(\mathbb S^{n-1}_\Delta)$, where
\begin{equation}
  \label{e:s-diag}
\mathbb S^{n-1}_\Delta=\{(y,y')\in\mathbb S^{n-1}\times\mathbb S^{n-1}\mid
y\neq y'\}.
\end{equation}
%%%%%%%%%%%%%%%%%%%%%%%%%%%%%%%%%%%%%%%%%%%%%%%%%%%%%%%%%%%%%%%%%%%%%%%%%%%%%%%%
\begin{defi}
  \label{d:fup}
We say that $\Lambda_\Gamma$ satisfies the \textbf{fractal uncertainty principle}
with exponent $\beta>0$, if for each $\varepsilon>0$ there exists $\rho\in (0,1)$
such that
\begin{equation}
  \label{e:fup-standard}
\|\indic_{\Lambda_\Gamma(C_1h^\rho)}\mathcal B_\chi(h) \indic_{\Lambda_\Gamma(C_1h^\rho)}\|_{L^2(\mathbb S^{n-1})\to L^2(\mathbb S^{n-1})}
\leq C h^{\beta-\varepsilon},\quad
h\in (0,1)
\end{equation}
for every $h$-independent constant $C_1$ and function $\chi\in C_0^\infty(\mathbb S^{n-1}_\Delta)$,
and some $C$ depending on $C_1,\chi$.
\end{defi}
%%%%%%%%%%%%%%%%%%%%%%%%%%%%%%%%%%%%%%%%%%%%%%%%%%%%%%%%%%%%%%%%%%%%%%%%%%%%%%%%

%%%%%%%%%%%%%%%%%%%%%%%%%%%%%%%%%%%%%%%%%%%%%%%%%%%%%%%%%%%%%%%%%%%%%%%%%%%%%%%%
\begin{figure}
\includegraphics{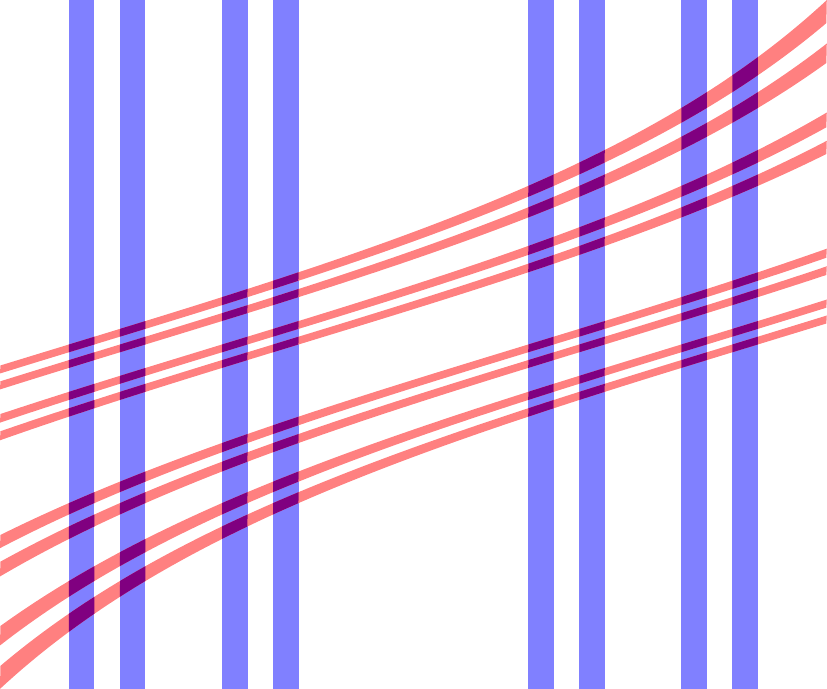}
\caption{The horizontal leaves~\eqref{e:hor-lag}, in red, and the vertical leaves~\eqref{e:ver-lag}, in blue, for $n=2$.
The horizontal variable
is $y\in\mathbb S^1$ and the vertical variable is $\eta$. For the fractal uncertainty principle, the width
of the distorted rectangles is slightly larger than $h$.}
\label{f:foliations}
\end{figure}
%%%%%%%%%%%%%%%%%%%%%%%%%%%%%%%%%%%%%%%%%%%%%%%%%%%%%%%%%%%%%%%%%%%%%%%%%%%%%%%%
\noindent\textbf{Remark}. The fractal uncertainty principle always holds with $\beta=\max(0,{n-1\over 2}-\delta)$, see~\eqref{e:tb-1} and~\eqref{e:tb-2}.
On the other hand, by~\eqref{e:JN} the maximal $\beta$
for which~\eqref{e:fup-standard} can be true is $\beta={n-1\over 2}-{\delta\over 2}$,
which (in dimension 2) is exactly the value of the essential spectral gap conjectured by
Jakobson--Naud~\cite{Jakobson-Naud2}.
%%%%%%%%%%%%%%%%%%%%%%%%%%%%%%%%%%%%%%%%%%%%%%%%%%%%%%%%%%%%%%%%%%%%%%%%%%%%%%%%

To explain how the estimate~\eqref{e:fup-standard} represents an uncertainty principle
associated to the set $\Lambda_\Gamma$, we consider the extremal case $\rho=1$,
put $C_1=1$,
and cover $\Lambda_\Gamma(h)$ by a collection of balls of radius $h$ centered at some
points $y_1,\dots,y_N\in\Lambda_\Gamma$, where $N\sim h^{-\delta}$ by~\eqref{e:LG-upper-1}. Then for each
$v\in L^2(\mathbb S^{n-1})$, the function $\mathcal B_\chi(h)\indic_{\Lambda_\Gamma(h)}v$ microlocally
concentrates (see~\eqref{e:oppa} below) in an $h$-neighborhood of the union of `horizontal' Lagrangian leaves
\begin{equation}
  \label{e:hor-lag}
\bigcup_{j=1}^N \big\{\big(y,\partial_y\log (|y-y_j|^2)\big)\,\big|\, (y,y_j)\in\supp\chi\big\}\ \subset\ T^*\mathbb S^{n-1},
\end{equation}
while the operator $\indic_{\Lambda_\Gamma(h)}$ microlocalizes to an $h$-neighborhood of the union of `vertical' Lagrangian
leaves
\begin{equation}
  \label{e:ver-lag}
\bigcup_{j=1}^N \{(y_j,\eta)\mid \eta\in T^*_{y_j}\mathbb S^{n-1}\}\ \subset\ T^*\mathbb S^{n-1}.
\end{equation}
The estimate~\eqref{e:fup-standard} with $\beta>0$ then says that no function can be perfectly localized to $h$-neighorhoods
of both~\eqref{e:hor-lag} and~\eqref{e:ver-lag}~-- see Figure~\ref{f:foliations}. Note that $h$-neighborhoods here
cannot be replaced by, say, $h^{1/2}$ neighborhoods since Gaussians provide examples of functions that concentrate
$h^{1/2}$ close to any fixed leaf of~\eqref{e:hor-lag} and to any fixed leaf of~\eqref{e:ver-lag}.
A related statement in the context of normally hyperbolic trapping was proved by  Nonnenmacher--Zworski~\cite[Lemma~5.12]{NoZwInv}.

If $\Lambda_\Gamma$ satisfies the fractal uncertainty principle, then an essential spectral
gap is given by the following
%%%%%%%%%%%%%%%%%%%%%%%%%%%%%%%%%%%%%%%%%%%%%%%%%%%%%%%%%%%%%%%%%%%%%%%%%%%%%%%%
\begin{theo}
  \label{t:fup-reduction}
Assume that $\Lambda_\Gamma$ satisfies the fractal uncertainty principle with
exponent $\beta>0$.
Then~\eqref{e:essential-gap2} holds, in particular
$R(\lambda)$ has finitely many poles in $\{\Im\lambda>-\beta+\varepsilon\}$
for each $\varepsilon>0$.
\end{theo}
%%%%%%%%%%%%%%%%%%%%%%%%%%%%%%%%%%%%%%%%%%%%%%%%%%%%%%%%%%%%%%%%%%%%%%%%%%%%%%%%
We outline the proof of the resonance free region of Theorem~\ref{t:fup-reduction}
(the resolvent bound follows directly from the argument). It suffices to show
that for $\Re\lambda\gg 1$ and $\Im\lambda\geq -\beta+\varepsilon$, there
are no nontrivial \emph{resonant states}, that is solutions to the equation
\begin{equation}
  \label{e:laggie}
\Big(-\Delta-{(n-1)^2\over 4}-\lambda^2\Big)u=0
\end{equation}
which satisfy certain \emph{outgoing} conditions asymptotically at the infinite ends of $M$.

Put $h:=(\Re\lambda)^{-1}$ and assume that $u$ is $L^2$-normalized on a sufficiently large fixed compact
subset of $M$. We study concentration of $u$ in the \emph{phase space} $T^*M$ using semiclassical quantization
\begin{equation}
  \label{e:oppa}
a\in C^\infty(T^*M)\ \mapsto\ \Op_h(a):C^\infty(M)\to C^\infty(M)
\end{equation}
where $a$ satisfies certain growth conditions~-- see~\S\ref{s:semiclassical}.

Let $\Gamma_+\subset T^*M\setminus 0$ be the \emph{outgoing tail},
consisting of geodesics which are trapped backwards in time;
define also the \emph{incoming tail} $\Gamma_-\subset T^*M\setminus 0$
(see~\eqref{e:GpmDef}).
The work of Vasy~\cite{vasy1,vasy2} near the infinite ends together with propagation
of semiclassical singularities shows that $u$ is microlocalized on $\Gamma_+$
(see~\cite{BonyMichel,NoZwActa} for related results in Euclidean scattering).
More precisely, for an $h$-independent symbol $a_+$,
\begin{equation}
  \label{e:g+1}
\supp (1-a_+)\cap\Gamma_+=\emptyset\quad \Longrightarrow\quad (1-\Op_h(a_+))u=\mathcal O(h^\infty)_{C^\infty(M)}.
\end{equation}
Moreover, $u$ has positive mass near $\Gamma_-$; more precisely, for $h$-independent $a_-$
\begin{equation}
  \label{e:g+2}
\supp(1-a_-)\cap \Gamma_-=\emptyset\quad \Longrightarrow\quad \|\Op_h(a_-) u\|_{L^2}\geq C^{-1}>0.
\end{equation}
(The statement~\eqref{e:g+2} is not quite correct since $\Gamma_-$ extends to the infinite ends of $M$
and thus $a_-$ cannot be compactly supported; however, we may argue in a fixed neighborhood
of the trapped set $K=\Gamma_+\cap \Gamma_-$. See Lemma~\ref{l:outgoing} for precise statements.)
The main idea of the proof is to replace $h$-independent symbols in~\eqref{e:g+1} and~\eqref{e:g+2} with
symbols that concentrate $h^\rho$ close to $\Gamma_\pm$:
\begin{align}
  \label{e:g++1}
d(\supp(1-a_+),\Gamma_+)>h^\rho\quad &\Longrightarrow\quad (1-\Op_h(a_+))u=\mathcal O(h^\infty)_{C^\infty(M)},\\
  \label{e:g++2}
d(\supp(1-a_-),\Gamma_-)>h^\rho\quad &\Longrightarrow\quad \|\Op_h(a_-)u\|_{L^2}\geq C^{-1}h^{(-\Im\lambda)\rho}.
\end{align}
The constant $\rho\in (0,1)$ is taken very close to $1$. See Lemma~\ref{l:second} for precise statements.

The proofs of~\eqref{e:g+1} and~\eqref{e:g+2} use propagation estimates for some $h$-independent time.
The proofs of~\eqref{e:g++1} and~\eqref{e:g++2} use similar estimates for
time $t=\rho\log(1/h)$, and the factor $h^{(-\Im\lambda)\rho}=e^{(\Im\lambda)t}$
results from the imaginary part of the operator in~\eqref{e:laggie}.

However, the analysis for~\eqref{e:g++1} and~\eqref{e:g++2} is considerably
more complicated since the symbols $a_\pm$ have very rough behavior in the directions transversal to $\Gamma_\pm$,
oscillating on the scale $h^\rho$~-- this corresponds to the fact that $t$ is almost \emph{twice}
the Ehrenfest time (since the maximal expansion rate for the geodesic flow is equal to 1,
the Ehrenfest time is just below ${1\over 2}\log(1/h)$~-- see for instance~\cite[Proposition~3.9]{qeefun}).
To solve this problem, we use the fact that $\Gamma_+$ is foliated by
the leaves of the weak unstable Lagrangian foliation $L_u$, while $\Gamma_-$ is foliated by the leaves
of the weak stable Lagrangian foliation $L_s$; therefore, we can make $a_+$ vary on scale $1$
along $L_u$ and $a_-$ vary on the scale $1$ along $L_s$. Then $a_+$ and $a_-$ can both be quantized
to some operators $\Op_h^{L_u}(a_+)$ and $\Op_h^{L_s}(a_-)$;
however, these operators will not be part of the same calculus~-- see~\S\ref{s:second-microlocalization} for details.

Next, the fractal uncertainty principle gives the following estimate
for some $a_\pm$ satisfying the conditions of~\eqref{e:g++1}, \eqref{e:g++2}:
\begin{equation}
  \label{e:fuppie}
\|\Op_h(a_-)\Op_h(a_+)\|_{L^2\to L^2}\leq Ch^{\beta-\varepsilon/2}.
\end{equation}
To see this, we conjugate by a Fourier integral operator whose underlying canonical transformation maps
an $h^\rho$ neighborhood of $\Gamma_+,\Gamma_-$ to an $h^\rho$ neighborhood of~\eqref{e:hor-lag},
\eqref{e:ver-lag} respectively
(strictly speaking, to the products of \eqref{e:hor-lag}, \eqref{e:ver-lag} with $(T^*\mathbb R^+)_{w,\theta}$
where $w$ corresponds to $|\xi|_g$ and $\partial_\theta$ corresponds to the generator of the geodesic flow).
Under this conjugation, $\Op_h(a_-)$ corresponds to $\indic_{\Lambda_\Gamma(h^\rho)}$
and $\Op_h(a_+)$ corresponds to $\mathcal B_\chi \indic_{\Lambda_\Gamma(h^\rho)} \mathcal B_\chi^*$,
therefore~\eqref{e:fuppie} follows from~\eqref{e:fup-standard}.
See~\S\ref{s:fun} for details.

Gathering together~\eqref{e:g++1}, \eqref{e:g++2},
and~\eqref{e:fuppie}, and recalling that $-\Im\lambda<\beta-\varepsilon$, we obtain a contradiction
for $\rho$ close enough to $1$ and $h\ll 1$ (thus finishing the proof):
$$
\begin{aligned}
C^{-1}h^{(\beta-\varepsilon)\rho}&\leq
C^{-1}h^{(-\Im\lambda)\rho}\leq \|\Op_h(a_-)u\|_{L^2}\\
&\leq \|\Op_h(a_-)\Op_h(a_+)u\|_{L^2}+\mathcal O(h^\infty)
\leq Ch^{\beta-\varepsilon/2}.
\end{aligned}
$$

It would be interesting to see if 
Theorem~\ref{t:fup-reduction} could be proved using transfer operator techniques such as the ones in~\cite{Naud}.
We however note that the microlocal argument presented above may be easier to adapt to a variable curvature
situation (see the Conjecture above) and it also provides an explicit polynomial bound on the resolvent~\eqref{e:essential-gap2}.

%%%%%%%%%%%%%%%%%%%%%%%%%%%%%%%%%%%%%%%%%%%%%%%%%%%%%%%%%%%%%%%%%%%%%%%%%%%%%%%%
\subsection{Fractal uncertainty principle via additive energy}

As remarked before (following Definition~\ref{d:fup}),
the fractal uncertainty principle holds with
$\beta={n-1\over 2}-\delta$. This corresponds to counting the total
area of the intersections of $h$-neighborhoods of~\eqref{e:hor-lag} and~\eqref{e:ver-lag}
(which in turn depend on $\delta$ by~\eqref{e:LG-upper-1})
and can be seen via an $L^1\to L^\infty$ norm bound on $\mathcal B_\chi$.
On the other hand, an $L^2\to L^2$ norm bound on $\mathcal B_\chi$
gives the fractal uncertainty principle with $\beta=0$.

If we only use the volume bound~\eqref{e:LG-upper-1}, then
no better value of $\beta$ can be obtained~-- for a non-rigorous explanation, one may
replace $\Lambda_\Gamma(C_1h^\rho)$ in~\eqref{e:fup-standard}
by a ball of volume $h^{n-1-\delta}$ in $\mathbb R^{n-1}$, replace $\mathcal B_\chi$
by the semiclassical Fourier transform, and calculate the corresponding $L^2\to L^2$ norm.

To get a better exponent $\beta$, we thus have to use the fractal structure of $\Lambda_\Gamma$. More precisely,
we will rely on the following combinatorial quantity:
%%%%%%%%%%%%%%%%%%%%%%%%%%%%%%%%%%%%%%%%%%%%%%%%%%%%%%%%%%%%%%%%%%%%%%%%%%%%%%%%
\begin{defi}
  \label{d:ae}
For $\mathcal X\subset\mathbb R^{n-1}$ and $\alpha>0$, define
the \textbf{$\alpha$-additive energy} of $\mathcal X$ by
$$
E_A(\mathcal X,\alpha)=\alpha^{4(1-n)}\mu_L(\{(\eta_1,\eta_2,\eta_3,\eta_4)\in \mathcal X(\alpha)^4\mid
|\eta_1-\eta_2+\eta_3-\eta_4|\leq\alpha\})
$$
where $\mathcal X(\alpha)$ is the $\alpha$-neighborhood of $\mathcal X$ and $\mu_L$ is the Lebesgue measure.
This definition trivially extends
from $\mathbb R^{n-1}$ to any $n-1$ dimensional vector space with an inner product.
\end{defi}
%%%%%%%%%%%%%%%%%%%%%%%%%%%%%%%%%%%%%%%%%%%%%%%%%%%%%%%%%%%%%%%%%%%%%%%%%%%%%%%%
Additive energy is intimately connected with the additive structure of finite sets, and it is one of the central concepts in the field of additive combinatorics. See~\cite{Tao} for further information on additive energy and related topics.

To explain the normalization of $E_A$, assume that $\mathcal X(\alpha)$ is the union of $N(\alpha)$ disjoint 
balls of radius $\alpha$, where the volume of $\mathcal X(\alpha)$ is proportional to $N(\alpha)\alpha^{n-1}$.
Then $E_A(\mathcal X,\alpha)$ is proportional to the number of combinations of four such balls such that the sum of the centers
of the first two balls is approximately equal to the sum of the centers of the other two.

Motivated by~\eqref{e:LG-upper-1}, we assume that $N(\alpha)\sim\alpha^{-\delta}$. Then
\begin{equation}
  \label{e:basicae}
\alpha^{-2\delta}\lesssim E_A(\mathcal X,\alpha)\lesssim \alpha^{-3\delta}.
\end{equation}
Indeed, the upper bound follows from the fact that the first three balls determine the fourth one uniquely,
and the lower bound follows from considering combinations of the form $(\eta_1,\eta_1,\eta_3,\eta_3)$.
%%%%%%%%%%%%%%%%%%%%%%%%%%%%%%%%%%%%%%%%%%%%%%%%%%%%%%%%%%%%%%%%%%%%%%%%%%%%%%%%
\begin{figure}
\includegraphics{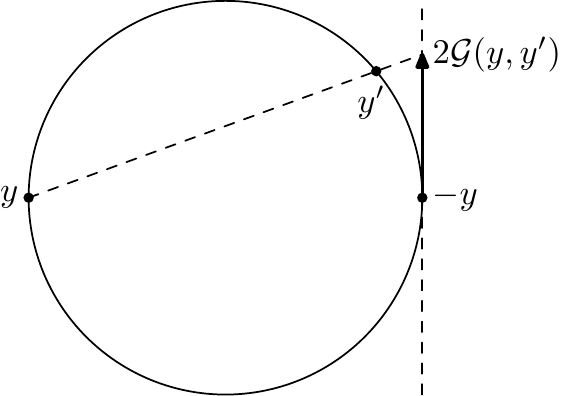}
\caption{The stereographic projection map $\mathcal G$.}
\label{f:stereographic0}
\end{figure}
%%%%%%%%%%%%%%%%%%%%%%%%%%%%%%%%%%%%%%%%%%%%%%%%%%%%%%%%%%%%%%%%%%%%%%%%%%%%%%%%

We will use the additive energy of the images of the limit set $\Lambda_\Gamma$ by the map
\begin{equation}
  \label{e:stpro}
\mathcal G(y,y')={y'-(y\cdot y')y\over 1-y\cdot y'}\in\mathbb R^n,\quad
y,y'\in\mathbb S^{n-1}\subset\mathbb R^n,\quad
y\neq y',
\end{equation}
which is half the stereographic projection of $y'$ with the base point $y$~-- see Figure~\ref{f:stereographic0}.
We have $\mathcal G(y,y')\perp y$, therefore we may think of it as a vector in $T_y\mathbb S^{n-1}$,
or (pairing with the round metric on the sphere) as a vector in $T_y^*\mathbb S^{n-1}$.
Note that $\mathcal G$ is related to the leaves of~\eqref{e:hor-lag} since
\begin{equation}
  \label{e:gide}
\partial_y\log (|y-y'|^2)=-\mathcal G(y,y').
\end{equation}
%%%%%%%%%%%%%%%%%%%%%%%%%%%%%%%%%%%%%%%%%%%%%%%%%%%%%%%%%%%%%%%%%%%%%%%%%%%%%%%%
\begin{defi}
  \label{d:ae-estimate}
We say that $\Lambda_\Gamma$ satisfies the \textbf{additive energy bound} with exponent
$\beta_E>0$, if for each $C_1>0$ there exists $C>0$ such that for all $\alpha\in (0,1)$,
\begin{equation}
  \label{e:ae-estimate}
\sup_{y_0\in \Lambda_\Gamma}E_A(\mathcal G(y_0,\Lambda_\Gamma)\cap B(0,C_1),\alpha)
\leq C\alpha^{-3\delta+\beta_E}.
\end{equation}
\end{defi}
%%%%%%%%%%%%%%%%%%%%%%%%%%%%%%%%%%%%%%%%%%%%%%%%%%%%%%%%%%%%%%%%%%%%%%%%%%%%%%%%
One can also interpret the sets~$\mathcal G(y_0,\Lambda_\Gamma)$ in terms of the dynamics
of the geodesic flow on $M$ using horocyclic flows~-- see~\eqref{e:cal-F} and~\eqref{e:cal-F-useful}.

Given an additive energy bound, we obtain a fractal uncertainty principle
and thus (by Theorem~\ref{t:fup-reduction}) an essential spectral gap:
%%%%%%%%%%%%%%%%%%%%%%%%%%%%%%%%%%%%%%%%%%%%%%%%%%%%%%%%%%%%%%%%%%%%%%%%%%%%%%%%
\begin{theo}
  \label{t:ae-reduction}
Assume that $\Lambda_\Gamma$ satisfies the additive energy bound with
exponent $\beta_E>0$.
Then $\Lambda_\Gamma$ satisfies the fractal uncertainty principle with exponent
\begin{equation}
  \label{e:beta-ae}
\beta={3\over 8}\Big({n-1\over 2}-\delta\Big)+{\beta_E\over 16}.
\end{equation}
\end{theo}
%%%%%%%%%%%%%%%%%%%%%%%%%%%%%%%%%%%%%%%%%%%%%%%%%%%%%%%%%%%%%%%%%%%%%%%%%%%%%%%%
\noindent\textbf{Remark}. Note that by~\eqref{e:basicae}, the maximal $\beta_E$ for which Definition~\ref{d:ae-estimate}
may hold is $\beta_E=\delta$. Plugged into~\eqref{e:beta-ae},
this gives an essential spectral gap of size ${3(n-1)-5\delta\over 16}$, which
improves over~\eqref{e:standard-gap} only when $\delta\in ({5\over 11}(n-1),{3\over 5}(n-1))$.

Theorem~\ref{t:ae-reduction} is proved using an $L^4$ estimate on the Fourier transforms
of $\indic_{\mathcal G(y_0,\Lambda_\Gamma(h^{\rho/2}))}$
for $y_0\in\Lambda_\Gamma$ obtained from the additive energy bound. Here we have
to replace the original $h^\rho$ neighborhood of $\Lambda_\Gamma$ by a bigger
$h^{\rho/2}$ neighborhood to approximate correlations between different
leaves of~\eqref{e:hor-lag} restricted to $\Lambda_\Gamma(h^{\rho/2})$ using the Fourier transform.
Roughly speaking,
the leaves which are farther than $h^{1/2}$ apart have an $\mathcal O(h^\infty)$ correlation
and for the leaves which are closer than $h^{1/2}$ to each other, the difference
of the phase functions in the resulting integral can be well approximated by its linear part~--
see the paragraph following~\eqref{e:eddie}.
The enlargement of the neighborhood to $\Lambda_\Gamma(h^{\rho/2})$
causes the loss of a factor of $1\over 2$ in the size of the gap;
together with a factor of $3\over 4$ coming from the use of the $L^4$ bound (rather than $L^\infty$)
this explains the factor of $3\over 8$ in~\eqref{e:beta-ae}.

%%%%%%%%%%%%%%%%%%%%%%%%%%%%%%%%%%%%%%%%%%%%%%%%%%%%%%%%%%%%%%%%%%%%%%%%%%%%%%%%
\subsection{Additive energy via Ahlfors-David regularity}
  \label{s:introad}

We now restrict to dimension $n=2$ and show that the limit set $\Lambda_\Gamma\subset\mathbb S^1$ of a convex co-compact Fuchsian
group $\Gamma$ with $\delta\in (0,1)$ satisfies the additive energy bound with some positive exponent.
For that we use the following regularity property:
%%%%%%%%%%%%%%%%%%%%%%%%%%%%%%%%%%%%%%%%%%%%%%%%%%%%%%%%%%%%%%%%%%%%%%%%%%%%%%%%
\begin{defi}
  \label{d:ad-regular}
Let $(\metSpace,\metricChar)$ be a complete metric space with more than one element. We say a closed set $\mathcal X \subset \metSpace$ is \textbf{$\setdim$--regular} with constant $C_{\mathcal X}$ if for all $x\in \mathcal X$ we have
\begin{equation}\label{defnADRegular}
C_{\mathcal X}^{-1}r^\setdim\ \leq\ \mu_{\setdim}(\mathcal X\cap B(x,r))\ \leq\ C_{\mathcal X} r^{\setdim},\quad
0<r<\diam(\mathcal M)
\end{equation}
where $B(x,r)$ is the metric ball of radius $r$ centered at $x$ and $\mu_{\setdim}$ is the $\setdim$--dimensional Hausdorff measure.
\end{defi}
%%%%%%%%%%%%%%%%%%%%%%%%%%%%%%%%%%%%%%%%%%%%%%%%%%%%%%%%%%%%%%%%%%%%%%%%%%%%%%%%
Sets with this property are also known as \emph{Ahlfors-David regular}. See~\cite{DS} for an introduction to $\setdim$--regular sets.
While Definition~\ref{d:ad-regular} is phrased using $\setdim$--dimensional Hausdorff measure, any other Borel outer measure could be used instead (in particular, for limit sets of convex co-compact Fuchsian groups the Patterson--Sullivan measure could be used).
This is discussed further in Lemma~\ref{equivOfADRegDefns} below.

The limit set $\Lambda_\Gamma\subset \mathbb S^{n-1}$
of a convex co-compact Fuchsian group $\Gamma$
is $\delta$--regular with $\delta$ defined in~\eqref{e:delta}~--
see~\cite[Theorem~7]{Sullivan} and~\cite[Lemma~14.13 and Theorem~14.14]{Borthwick}. We denote
the associated regularity constant by
\begin{equation}
  \label{e:ad-regular-limit}
\mathbf C:=C_{\Lambda_\Gamma}.
\end{equation}
Using $\delta$--regularity of $\Lambda_\Gamma$, we obtain the following additive energy bound.
Combined with Theorems~\ref{t:fup-reduction} and~\ref{t:ae-reduction}, it implies
Theorem~\ref{t:main} and thus Theorem~\ref{t:marketing}.
%%%%%%%%%%%%%%%%%%%%%%%%%%%%%%%%%%%%%%%%%%%%%%%%%%%%%%%%%%%%%%%%%%%%%%%%%%%%%%%%
\begin{theo}
  \label{t:ad-reduced}
Let $M=\Gamma\backslash\mathbb H^2$ be a convex co-compact hyperbolic surface
with limit set $\Lambda_\Gamma\subset\mathbb S^1$ of dimension $\delta\in (0,1)$.
Then $\Lambda_\Gamma$ satisfies the additive energy bound in the sense
of Definition~\ref{d:ae-estimate} with exponent
\begin{equation}\label{e:betaE}
\beta_E:=\betaXBoundBFC,
\end{equation}
where $\mathbf C$ is defined in~\eqref{e:ad-regular-limit} and $\mathbf K$ is a global constant.
\end{theo}
%%%%%%%%%%%%%%%%%%%%%%%%%%%%%%%%%%%%%%%%%%%%%%%%%%%%%%%%%%%%%%%%%%%%%%%%%%%%%%%%
%
%%%%%%%%%%%%%%%%%%%%%%%%%%%%%%%%%%%%%%%%%%%%%%%%%%%%%%%%%%%%%%%%%%%%%%%%%%%%%%%%
\noindent\textbf{Remarks}.
(i) The specifics of the bound \eqref{e:betaE} are not particularly important. The key point is that the exponent $\beta_E $ 
in~\eqref{e:ae-estimate}
is independent of $\h$, and it can be computed explicitly. We did not compute the value of $\mathbf{K}$, but in principle it can be done without much difficulty. 

\noindent (ii) In dimensions $n>2$, Theorem~\ref{t:ad-reduced} no longer holds in general as shown by the
example of the hyperbolic cylinder in three dimensions (see for instance~\cite[Appendix~A]{fwl}). In this example,
the limit set $\Lambda_\Gamma$ is a great circle on $\mathbb S^2$, and the stereographic projections
$\mathcal G(y_0,\Lambda_\Gamma)$ are straight lines, which saturate the upper bound in~\eqref{e:basicae}.
See~\S\ref{higherDimRemark} for possible generalizations to higher dimensions.
%%%%%%%%%%%%%%%%%%%%%%%%%%%%%%%%%%%%%%%%%%%%%%%%%%%%%%%%%%%%%%%%%%%%%%%%%%%%%%%%

Theorem~\ref{t:ad-reduced} follows from a general result bounding additive
energy of Ahlfors-David regular sets, stated as 
Theorem~\ref{t:ae-combinatorial} in~\S\ref{s:ae-combinatorial}; the proof of Theorem~\ref{t:ae-combinatorial}
can schematically be explained as follows (see~\S\ref{s:ae-ideas} for more details):
\begin{enumerate}
\item Ahlfors-David regular sets cannot contain large subsets of arithmetic progressions. This follows
by a direct argument using~\eqref{defnADRegular} and the fact that $\delta<1$.
\item A variant of Fre{\u\i}man's theorem from additive combinatorics asserts that any set with large additive energy must contain large subsets of generalized arithmetic progressions. Together with~(1) this implies that Ahlfors-David regular sets cannot have extremely large (i.e.~near maximal) additive energy.
\item Ahlfors-David regular sets also have a certain type of coarse self-similarity. This allows us to analyze them at many scales and at many different locations. Since Ahlfors-David regular sets cannot have extremely large additive energy at any scale or at any location, we can perform a multi-scale analysis to conclude that such sets must actually have small additive energy. 
\end{enumerate}
%%%%%%%%%%%%%%%%%%%%%%%%%%%%%%%%%%%%%%%%%%%%%%%%%%%%%%%%%%%%%%%%%%%%%%%%%%%%%%%%
\subsection{Structure of the paper}

\begin{itemize}
\item In~\S\ref{s:semiclassical}, we review certain notions in semiclassical analysis,
in particular pseudodifferential and Fourier integral operators.
\item In~\S\ref{s:second-microlocalization}, we study an anisotropic pseudodifferential
calculus associated to a Lagrangian foliation.
\item In~\S\ref{s:hyperbolic}, we study geometric and dynamical properties
of hyperbolic manifolds and, using the calculus of~\S\ref{s:second-microlocalization},
prove Theorem~\ref{t:fup-reduction}.
\item In~\S\ref{s:fup}, we discuss the fractal uncertainty principle
and prove Theorem~\ref{t:ae-reduction}.
\item In~\S\ref{s:ae-combinatorial}, we prove that
Ahlfors-David regular sets have small additive energy.
\item In~\S\ref{s:ae}, we establish Ahlfors-David regularity of the stereographic projections
of the limit set and prove Theorem~\ref{t:ad-reduced}. We also
obtain locally uniform bounds on the regularity constant for 3-funneled surfaces.
\item In Appendix~\ref{s:hyperbolic-technical}, we prove several technical
lemmas used in~\S\ref{s:hyperbolic} and~\S\ref{s:ae}.
\end{itemize}

%%%%%%%%%%%%%%%%%%%%%%%%%%%%%%%%%%%%%%%%%%%%%%%%%%%%%%%%%%%%%%%%%%%%%%%%%%%%%%%%
%%%%%%%%%%%%%%%%%%%%%%%%%%%%%%%%%%%%%%%%%%%%%%%%%%%%%%%%%%%%%%%%%%%%%%%%%%%%%%%%
\section{Semiclassical preliminaries}
  \label{s:semiclassical}

In this section, we give a brief review of semiclassical analysis.
For a comprehensive introduction to the subject, the reader is referred to~\cite{e-z}.
We partially follow the presentation of~\cite[Appendix~E]{dizzy} and~\cite{qeefun,fwl,nhp}.

%%%%%%%%%%%%%%%%%%%%%%%%%%%%%%%%%%%%%%%%%%%%%%%%%%%%%%%%%%%%%%%%%%%%%%%%%%%%%%%%
\subsection{Pseudodifferential operators}

Let $M$ be a manifold. For $k\in\mathbb R$, we say that $a(x,\xi)\in C^\infty(T^*M)$
lies in the symbol class $S^k_{1,0}(T^*M)$ if it satisfies the derivative bounds
\begin{equation}
  \label{e:basic-symbol}
|\partial^\alpha_x\partial^\beta_\xi a(x,\xi)|\leq C_{\alpha\beta K}\langle\xi\rangle^{k-|\beta|},\quad
x\in K,
\end{equation}
for each compact set $K\subset M$.
We restrict ourselves to the subset of \emph{polyhomogeneous, or classical, symbols} $S^k(T^*M)\subset S^k_{1,0}(T^*M)$ which have asymptotic expansions
$a(x,\xi)\sim\sum_{j=0}^\infty a_j(x,\xi)$ as
$|\xi|\to\infty$
where each $a_j$ is positively homogeneous in $\xi$ of degree $k-j$.

A family of symbols $b(x,\xi;h)\in S_{1,0}^k(T^*M)$ depending on a small parameter $h>0$
is said to lie in the class $S^k_h(T^*M)$ if it has the following expansion as $h\to 0$:
\begin{equation}
  \label{e:h-classical}
b(x,\xi;h)\sim\sum_{\ell=0}^\infty h^\ell a_\ell(x,\xi),\quad
a_\ell\in S^{k-\ell}(T^*M).
\end{equation}
See for instance~\cite[\S E.1.2]{dizzy} and~\cite[\S2]{vasy2} for details.

If $a\in S^k_{1,0}(T^*\mathbb R^n)$ satisfies~\eqref{e:basic-symbol} uniformly in $x\in\mathbb R^n$,
then we can quantize it by the following formula (see~\cite[\S4.1.1]{e-z}
and~\cite[\S E.1.4]{dizzy})
\begin{equation}
  \label{e:standard-quantization}
\Op_h(a)f(x)=(2\pi h)^{-n}\int_{\mathbb R^{2n}}e^{{i\over h}(x-y)\cdot\xi}a(x,\xi)f(y)\,dyd\xi,
\end{equation}
which gives an operator $\Op_h(a)$ acting on the space
$\mathscr S(\mathbb R^n)$ of Schwartz functions, as well as on the dual space
$\mathscr S'(\mathbb R^n)$ of tempered distributions.

Following~\cite[\S E.1.5]{dizzy} and~\cite[\S14.2.2]{e-z}, for a general manifold $M$ we consider the class $\Psi^k_h(M)$ of semiclassical
pseudodifferential operators with symbols in $S^k_h(T^*M)$. We denote by
$$
\sigma_h:\Psi^k_h(M)\to S^k(T^*M)
$$
the principal symbol map.
Operators in $\Psi^k_h$ act on semiclassical Sobolev spaces $H^s_{h,\comp}\to H^{s-k}_{h,\loc}$, see~\cite[\S E.1.6]{dizzy} and~\cite[\S14.2.4]{e-z}.
We will
often use the class $\Psi^{\comp}_h(M)$ of operators whose full symbols are essentially compactly supported in $T^*M$ and
whose Schwartz kernels are compactly supported in $M\times M$.

For $A\in\Psi^k_h(M)$, denote by $\WFh(A)$ its semiclassical wavefront set, which is the essential support
of its full symbol~-- see for instance~\cite[\S E.2.1]{dizzy} and~\cite[Appendix~C.1]{zeta}. Then $\WFh(A)$ is a closed
subset of the fiber-radially compactified cotangent bundle $\overline T^*M\supset T^*M$, see for instance~\cite[\S E.1.2]{dizzy}
or~\cite[\S2]{vasy2}.
For $A,B\in\Psi^k_h(M)$ and an open set $U\subset \overline T^*M$, we say that
$$
A=B+\mathcal O(h^\infty)\quad\text{microlocally in }U,
$$
if $\WFh(A-B)\cap U=\emptyset$.
We also use the notion of wavefront sets of $h$-tempered distributions and operators, see
for instance~\cite[\S E.2.3]{dizzy} or~\cite[\S2.3]{zeta}.  

Let $B=B(h):\mathcal D'(M)\to C_0^\infty(M)$ be an $h$-tempered family of smoothing operators
and assume that the wavefront set $\WF'_h(B)\subset \overline T^*(M\times M)$ is a compact
subset of $T^*(M\times M)$. We say that $B$ is \emph{pseudolocal} if $\WF'_h(B)$ is contained
in the diagonal $\Delta(T^*M)\subset T^*(M\times M)$. For a pseudolocal operator $B$,
we consider the set $\WF_h(B)\subset T^*M$ defined by
\begin{equation}
  \label{e:wf-pseudolocal}
  \WF'_h(B)=\{(x,\xi,x,\xi)\mid (x,\xi)\in \WF_h(B)\}.
\end{equation}
Note that operators in $\Psi^{\comp}_h(M)$ are pseudolocal and their definition of
wavefront set given in~\cite[\S E.2.1]{dizzy} agrees with the one given by~\eqref{e:wf-pseudolocal}.

%%%%%%%%%%%%%%%%%%%%%%%%%%%%%%%%%%%%%%%%%%%%%%%%%%%%%%%%%%%%%%%%%%%%%%%%%%%%%%%%
\subsection{Fourier integral operators}
  \label{s:fios}

We next introduce semiclassical Fourier integral operators.
Let $\varkappa:U_2\to U_1$ be a canonical transformation (that is,
a symplectomorphism), where $U_j\subset T^*M_j$ are open sets
and $M_j$ are manifolds of the same dimension. Define the graph of $\varkappa$ by
\begin{equation}
  \label{e:Graph}
\Graph(\varkappa):=\{(x,\xi,y,\eta)\mid (y,\eta)\in U_2,\ (x,\xi)=\varkappa(y,\eta)\}\ \subset\ T^*(M_1\times M_2).
\end{equation}
Let $\xi\,dx$ and $\eta\,dy$ be the canonical 1-forms on $T^*U_1$ and $T^*U_2$ respectively.
Since $\varkappa$ is a canonical transformation, the restriction
$(\xi\,dx-\eta\,dy)|_{\Graph(\varkappa)}$ is a closed 1-form. We require that
$\varkappa$ is \emph{exact} in the sense that this restriction is an exact form,
and fix an antiderivative
\begin{equation}
  \label{e:antiderivative}
F\in C^\infty(\Graph(\varkappa)),\quad
(\xi\,dx-\eta\,dy)|_{\Graph(\varkappa)}=dF.
\end{equation}
For a canonical transformation $\varkappa$ with a fixed antiderivative $F$,
we consider the class~$I^{\comp}_h(\varkappa)$ of compactly supported and microlocalized
Fourier integral operators associated to $\varkappa$~-- see for instance~\cite[Chapter~5]{gu-st0},
\cite[Chapter~8]{gu-st},
\cite[\S3.2]{qeefun},
\cite[\S3.2]{fwl},%
\footnote[2]{The presentation in~\cite[\S3.2]{fwl} contained an error because the Fourier
integral operators associated to the identity map were not necessarily pseudodifferential
operators with classical symbols due to a possible constant phase factor $e^{ic/h}$. We correct
it here by fixing the antiderivative, which is always possible locally.}
\cite[\S3.2]{nhp}, and the references there.
We adopt a convention that operators in $I^{\comp}_h(\varkappa)$
act $\mathcal D'(M_2)\to C_0^\infty(M_1)$.

We list some basic properties of the class $I^{\comp}_h(\varkappa)$:
\begin{itemize}
\item each $B\in I^{\comp}_h(\varkappa)$ is bounded
uniformly in~$h$ on the spaces $H^s_{h,\loc}(M_2)\to H^{s'}_{h,\comp}(M_1)$ for all $s,s'\in\mathbb R$,
and $\WF'_h(B)\subset \Graph(\varkappa)$;
\item if $\varkappa:U_2\to U_1$, $U'\subset U_2$, and $\varkappa':=\varkappa|_{U'}$,
then $B\in I^{\comp}_h(\varkappa')$ if and only if $B\in I^{\comp}_h(\varkappa)$ and
$\WF'_h(B)\subset\Graph(\varkappa')$;
\item if $\varkappa:T^*M\to T^*M$ is the identity map with the zero
antiderivative, then $B\in I^{\comp}_h(\varkappa)$ if and only if
$B\in\Psi^{\comp}_h(M)$;
\item if $B\in I^{\comp}_h(\varkappa)$, then $B^*\in I^{\comp}_h(\varkappa^{-1})$,
with the antiderivatives
on $\Graph(\varkappa)$ and $\Graph(\varkappa^{-1})$ summing up to zero;
\item if $\varkappa:U_2\to U_1$, $\varkappa':U_3\to U_2$, and $B\in I^{\comp}_h(\varkappa)$,
$B'\in I^{\comp}_h(\varkappa')$, then $BB'\in I^{\comp}_h(\varkappa\circ\varkappa')$,
with the antiderivative on $\Graph(\varkappa\circ\varkappa')$ chosen as the sum of the antiderivatives
on $\Graph(\varkappa)$ and $\Graph(\varkappa')$.
\end{itemize}
To give a concrete expression for elements of $I^{\comp}_h(\varkappa)$, assume
that $\varkappa$ is parametrized by a nondegenerate phase function
$\Phi(x,y,\zeta)\in C^\infty(U_\Phi;\mathbb R)$, $U_\Phi\subset M_1\times M_2\times\mathbb R^m$,
in the sense that the differentials
$d(\partial_{\zeta_1}\Phi),\dots,d(\partial_{\zeta_m}\Phi)$ are independent on the critical set
$$
\mathcal C_\Phi=\{(x,y,\zeta)\in U_\Phi\mid\partial_\zeta\Phi(x,y,\zeta)=0\}
$$
and the graph $\Graph(\varkappa)$ is given by
\begin{equation}
  \label{e:kappa-parametrized}
\Graph(\varkappa)=j_\Phi(\mathcal C_\Phi),\quad
j_\Phi:(x,y,\zeta)\mapsto
(x,\partial_x\Phi(x,y,\zeta),y,-\partial_y\Phi(x,y,\zeta)).
\end{equation}
The corresponding antiderivative is just the pullback of $\Phi$ from $\mathcal C_\Phi$ to $\Graph(\varkappa)$
by the map $j_\Phi$. Then any operator $B\in I^{\comp}_h(\varkappa)$ has the following
form modulo $\mathcal O(h^\infty)_{\Psi^{-\infty}}$:
\begin{equation}
  \label{e:fio-general-form}
Bf(x)=(2\pi h)^{-{m+n\over 2}}\int_{M_1\times\mathbb R^m}
e^{{i\over h}\Phi(x,y,\zeta)}b(x,y,\zeta;h)\,dyd\zeta
\end{equation}
where $n=\dim M_1=\dim M_2$ and $b$ is a compactly supported symbol on $U_\Phi$,
that is an $h$-dependent family of smooth functions with support
contained in some $h$-independent compact set which has
an asymptotic expansion in nonnegative integer powers of $h$.
Moreover, local principal symbol calculus shows that
\begin{equation}
  \label{e:principal-killed}
b(x,y,\zeta;0)=0\quad\text{for all }(x,y,\zeta)\in\mathcal C_\Phi\ \Longrightarrow\
B\in hI^{\comp}_h(\varkappa).
\end{equation}
See for example~\cite[\S3.2]{qeefun} for details.

A special case is when $M_2$ is an open subset of $\mathbb R^n$ and $\Graph(\varkappa)$
projects diffeomorphically onto the $(x,\eta)$ variables. Let $F\in C^\infty(\Graph(\varkappa))$
be the fixed antiderivative, and define the \emph{generating function}
$S(x,\eta)\in C^\infty(U_S;\mathbb R)$ by the formula $S(x,\eta)=F+y\cdot\eta$,
where $\Graph(\varkappa)$ is parametrized by $(x,\eta)\in U_S\subset M_1\times\mathbb R^n$.
Then
\begin{equation}
  \label{e:canonical-form}
\Graph(\varkappa)=\{\xi=\partial_x S(x,\eta),\
y=\partial_\eta S(x,\eta),\
(x,\eta)\in U_S\}
\end{equation}
implying that $\varkappa$ is parametrized in the sense of~\eqref{e:kappa-parametrized} by the function
$(x,y,\zeta)\mapsto S(x,\zeta)-y\cdot\zeta$.
Each $B\in I^{\comp}_h(\varkappa)$ has the following
form modulo $\mathcal O(h^\infty)_{\mathcal D'(M_2)\to C_0^\infty(M_1)}$:
\begin{equation}
  \label{e:fio-local-form}
Bf(x)=(2\pi h)^{-n}\int_{\mathbb R^{2n}}e^{{i\over h}(S(x,\eta)-y\cdot\eta)}b(x,\eta;h)\chi(y)f(y)\,dyd\eta,\quad
f\in \mathcal D'(M_2),
\end{equation}
where $n=\dim M_j$, $b(x,\eta;h)$ is a compactly supported symbol on $U_S$,
and $\chi\in C_0^\infty(M_2)$ is any function such that $\chi=1$ near 
$\partial_\eta S(\supp b)$. (The resulting operator is independent of the choice of $\chi$
modulo $\mathcal O(h^\infty)_{\mathcal D'(M_2)\to C_0^\infty(M_1)}$.)

As remarked in~\cite[\S3.2]{fwl}, $\varkappa$ can locally be written in the form~\eqref{e:canonical-form}
for some choice of local coordinates on $M_2$ as long as
its domain does not intersect the zero
section of $T^*M_2$. The latter condition can be arranged locally by composing $\varkappa$
with a transformation of the form $(y,\eta)\mapsto (y,\eta-d\psi(y))$ for some $\psi\in C^\infty(M_2)$,
which amounts to multiplying the resulting operators by $e^{i\psi(y)/ h}$~-- see Lemma~\ref{l:gauge-fio} below.

We next discuss microlocal inverses of Fourier integral operators.
Assume that $B\in I^{\comp}_h(\varkappa),B'\in I^{\comp}_h(\varkappa^{-1})$. Then
$BB'\in\Psi^{\comp}_h(M_1)$, $B'B\in\Psi^{\comp}_h(M_2)$, 
$\WFh(BB')\subset U_1$, $\WFh(B'B)\subset U_2$, and
(as shown in the case of~\eqref{e:fio-local-form} by an explicit application of the method
of stationary phase and in general is a form of Egorov's Theorem)
\begin{equation}
  \label{e:symbol-commutes}
\sigma_h(B'B)=\sigma_h(BB')\circ\varkappa.
\end{equation}
We call $B\in I^{\comp}_h(\varkappa)$ \emph{elliptic} at a point $(x,\xi,y,\eta)\in\Graph(\varkappa)$, if
there exists $B'\in I^{\comp}_h(\varkappa^{-1})$ such that $\sigma_h(BB')(x,\xi)\neq 0$
(in fact, this is equivalent to requiring that $\sigma_h(BB^*)(x,\xi)\neq 0$). For $B$ given by~\eqref{e:fio-general-form},
this simply means that $b(x,y,\zeta;0)\neq 0$ where $(x,y,\zeta)=j_\Phi^{-1}(x,\xi,y,\eta)\in \mathcal C_\Phi$.
For each point in $\Graph(\varkappa)$,
there exist operators in $I^{\comp}_h(\varkappa)$ elliptic at this point.

If $V_j\subset U_j$ are compact subsets such that $\varkappa(V_2)=V_1$, then
we say that $B,B'$ \emph{quantize}~$\varkappa$ near $V_1\times V_2$ if
\begin{equation}
  \label{e:quantized}
\begin{aligned}
BB'&=1+\mathcal O(h^\infty)\quad\text{microlocally near }V_1,\\
B'B&=1+\mathcal O(h^\infty)\quad\text{microlocally near }V_2.
\end{aligned}
\end{equation}
Such operators $B,B'$ exist if $V_2=\{(y,\eta)\}$ for
any given point $(y,\eta)\in U_2$ (and thus if $V_2$ is a sufficiently small
neighborhood of $(y,\eta)$). 
To show this, take $B\in I^{\comp}_h(\varkappa)$
elliptic at $(\varkappa(y,\eta),y,\eta)$ and $B'_0\in I^{\comp}_h(\varkappa^{-1})$
such that $\sigma_h(BB'_0)\neq 0$ on $V_1$. Multiplying $B'_0$ on the right
by an elliptic parametrix of $BB'_0$
(see for instance~\cite[\S E.2.2]{dizzy}
and~\cite[Proposition~2.4]{zeta}), we obtain $B'\in I^{\comp}_h(\varkappa)$
such that $BB'=1+\mathcal O(h^\infty)$ microlocally near $V_1$. By~\eqref{e:symbol-commutes},
we have $\sigma_h(B'_0B)\neq 0$ on $V_2$, so we can construct $B''\in I^{\comp}_h(\varkappa)$
such that $B''B=1+\mathcal O(h^\infty)$ microlocally near $V_2$. Then
$$
\WF'_h(B'-B'')\cap (V_1\times V_2)\ \subset\ \WF'_h((B''B)B'-B''(BB'))\ =\ \emptyset,
$$
therefore~\eqref{e:quantized} holds.
One could also define $B,B'$ as solutions of an evolution equation,
see~\cite[Theorem~11.5]{e-z} and~\cite[\S3.2]{fwl}.

One useful family of Fourier integral operators is given by the following
%%%%%%%%%%%%%%%%%%%%%%%%%%%%%%%%%%%%%%%%%%%%%%%%%%%%%%%%%%%%%%%%%%%%%%%%%%%%%%%%
\begin{lemm}
  \label{l:gauge-fio}
Let $\varphi:M_1\to M_2$ be a diffeomorphism
and $\psi\in C^\infty(M_1)$. Consider the operator
$$
B=B(h):\mathcal D'(M_2)\to \mathcal D'(M_1),\quad
Bf(x)=e^{i\psi(x)/h}f(\varphi(x)).
$$
Then for each $A_j\in\Psi^{\comp}_h(M_j)$, we have
$A_1B,BA_2\in I^{\comp}_h(\varkappa^{-1})$, where
$$
\varkappa:T^*M_1\to T^*M_2,\quad
\varkappa(x,\xi)=\big(\varphi(x),(d\varphi(x))^{-T}\cdot(\xi-d\psi(x))\big),
$$
and the antiderivative is given by $\psi(x)$.
\end{lemm}
%%%%%%%%%%%%%%%%%%%%%%%%%%%%%%%%%%%%%%%%%%%%%%%%%%%%%%%%%%%%%%%%%%%%%%%%%%%%%%%%
\begin{proof}
It suffices to consider the case when $M_1,M_2$ are open subsets of $\mathbb R^n$.
Let $A_2=\Op_h(a)\chi$, where $a(y,\eta;h)$ is compactly supported in $M_2\times\mathbb R^n$
and $\chi\in C_0^\infty(M_2)$ is equal to 1 near the projection of $\supp a$.
Then
$$
BA_2f(x)=(2\pi h)^{-n}\int_{\mathbb R^{2n}}e^{{i\over h}((\varphi(x)-y)\cdot\eta+\psi(x))}
a(\varphi(x),\eta;h)\chi(y)f(y)\,dyd\eta.
$$
This has the form~\eqref{e:fio-local-form} with
$$
S(x,\eta)=\varphi(x)\cdot\eta+\psi(x),\quad
b(x,\eta;h)=a(\varphi(x),\eta;h),
$$
and it is straightforward to see that $\varkappa^{-1}$ is given by~\eqref{e:canonical-form}.
The case of $A_1B$ is reduced to the case of $BA_2$ by considering adjoint operators.
\end{proof}
%%%%%%%%%%%%%%%%%%%%%%%%%%%%%%%%%%%%%%%%%%%%%%%%%%%%%%%%%%%%%%%%%%%%%%%%%%%%%%%%

%%%%%%%%%%%%%%%%%%%%%%%%%%%%%%%%%%%%%%%%%%%%%%%%%%%%%%%%%%%%%%%%%%%%%%%%%%%%%%%%
\section{Calculus associated to a Lagrangian foliation}
\label{s:second-microlocalization}

In this section, we define a class of exotic pseudodifferential operators associated to a Lagrangian
foliation. The symbols of these operators are allowed to vary on the constant scale along the foliation
and on the scale $h^\rho$, $0\leq \rho<1$, in the directions transversal to the foliation.
For $\rho>{1\over 2}$, the resulting operators will not generally lie in the exotic calculus
$\Psi_{1/2}$ (see for instance~\cite[\S5.1]{fwl}), yet they form an algebra with properties
similar to those of standard pseudodifferential operators.

A similar (in fact, sharper in certain ways
as it allowed for $\rho=1$ and $\Psi_{1/2}$ behavior in some directions) second microlocal calculus
associated to a hypersurface
has previously been developed by Sj\"ostrand--Zworski~\cite[\S5]{sj-zw};
for a calculus associated to a Lagrangian submanifold in the analytic category,
see~\cite[Chapter~2]{delort-book} and the references given there.

%%%%%%%%%%%%%%%%%%%%%%%%%%%%%%%%%%%%%%%%%%%%%%%%%%%%%%%%%%%%%%%%%%%%%%%%%%%%%%%%
\subsection{Foliations and symbols}

We start with the definition of a Lagrangian foliation:
%%%%%%%%%%%%%%%%%%%%%%%%%%%%%%%%%%%%%%%%%%%%%%%%%%%%%%%%%%%%%%%%%%%%%%%%%%%%%%%%
\begin{defi}
  \label{d:l-foli}
Let $M$ be a manifold, $U\subset T^*M$ be an open set, and 
$$
L_{(x,\xi)}\ \subset\ T_{(x,\xi)}(T^*M),\quad
(x,\xi)\in U
$$
a family of subspaces depending smoothly on $(x,\xi)$. We say that $L$ is a \textbf{Lagrangian foliation} on $U$ if
%%%%%%%%%%%%%%%%%%%%%%%%%%%%%%%%%%%%%%%%%%%%%%%%%%%%%%%%%%%%%%%%%%%%%%%%%%%%%%%%
\begin{itemize}
\item $L_{(x,\xi)}$ is integrable in the sense that if $X,Y$ are two vector fields on $U$ lying in $L$ at each
point (we denote this by $X,Y\in C^\infty(U;L)$), then the Lie bracket $[X,Y]$ lies in $C^\infty(U;L)$ as well;
\item $L_{(x,\xi)}$ is a Lagrangian subspace of $T_{(x,\xi)}(T^*M)$ for each $(x,\xi)\in U$.
\end{itemize}
\end{defi}
%%%%%%%%%%%%%%%%%%%%%%%%%%%%%%%%%%%%%%%%%%%%%%%%%%%%%%%%%%%%%%%%%%%%%%%%%%%%%%%%
Another way to think about a Lagrangian foliation is in terms of its
leaves, which are Lagrangian submanifolds whose tangent spaces are given by $L$.
The existence of these leaves follows from Frobenius's Theorem, see Lemma~\ref{l:canonical}
below.

We consider the following class of symbols:
%%%%%%%%%%%%%%%%%%%%%%%%%%%%%%%%%%%%%%%%%%%%%%%%%%%%%%%%%%%%%%%%%%%%%%%%%%%%%%%%
\begin{defi}
  \label{d:symbols}
Let $L$ be a Lagrangian foliation on $U\subset T^*M$, and fix $\rho\in [0,1)$. We say that a function
$a(x,\xi;h)$ is a (compactly supported) symbol of class $S_\rho$ with respect to $L$, and write
$$
a\in S^{\comp}_{L,\rho}(U),
$$
if for each $h\in (0,h_0)$, $(x,\xi)\mapsto a(x,\xi;h)$ is a smooth function
on $U$ supported inside some $h$-independent compact set and it satisfies the derivative bounds
(with the constant $C$ depending on $Y_j,Z_j$, but not on $h$)
\begin{equation}
  \label{e:symbols-def}
\sup_{x,\xi} |Y_1\dots Y_m Z_1\dots Z_k a(x,\xi;h)|\leq C h^{-\rho k},
\end{equation}
for each vector fields $Y_1,\dots,Y_m,Z_1,\dots,Z_k$ on $U$ such that
$Y_1,\dots,Y_m\in C^\infty(U;L)$.
\end{defi}
%%%%%%%%%%%%%%%%%%%%%%%%%%%%%%%%%%%%%%%%%%%%%%%%%%%%%%%%%%%%%%%%%%%%%%%%%%%%%%%%
The following statement is useful for constructing symbols in the class
$S^{\comp}_{L,\rho}$:
%%%%%%%%%%%%%%%%%%%%%%%%%%%%%%%%%%%%%%%%%%%%%%%%%%%%%%%%%%%%%%%%%%%%%%%%%%%%%%%%
\begin{lemm}
  \label{l:symbol-construction}
Let $M_1$ be a compact manifold and 
$V_0(h)\subset V_1(h)\subset M_1$ be $h$-dependent sets
satisfying
$$
d\big(V_0(h),M_1\setminus V_1(h)\big)>\varepsilon h^\rho
$$
for some fixed $\varepsilon>0,\rho\in [0,1)$
and all $h\in (0,1)$. Then there exists $\chi(h)\in C_0^\infty(M_1;[0,1])$
such that for all $h\in (0,1)$,
\begin{gather}
  \label{e:sc-1}
\supp(1-\chi(h))\cap V_0(h)=\emptyset,\quad
\supp\chi(h)\subset V_1(h);\\
  \label{e:sc-2}
\sup_{M_1}|\partial^\alpha\chi|\leq C_\alpha h^{-\rho|\alpha|}.
\end{gather}
\end{lemm}
%%%%%%%%%%%%%%%%%%%%%%%%%%%%%%%%%%%%%%%%%%%%%%%%%%%%%%%%%%%%%%%%%%%%%%%%%%%%%%%%
\begin{proof}
By a partition of
unity we reduce to the case when $V_1(h)$ is contained
in a small coordinate neighborhood on $M_1$; therefore,
it suffices to consider the case $M_1=\mathbb R^n$.
Let $d(\cdot,\cdot)$ be the Euclidean distance function.
Put
$$
V_2(h):=\{x\in\mathbb R^n\mid d(x,V_0(h))\leq \varepsilon h^\rho/2\},
$$
then (here $B(x,r)$ denotes the ball of radius $r$ centered at $x$)
\begin{equation}
  \label{e:sc-3}
\begin{aligned}
x\in V_0(h)\quad&\Longrightarrow\quad B(x,\varepsilon h^\rho/2)\subset V_2(h),\\
x\in V_2(h)\quad&\Longrightarrow\quad B(x,\varepsilon h^\rho/2)\subset V_1(h).
\end{aligned}
\end{equation}
Take nonnegative $\psi\in C_0^\infty(B(0,\varepsilon/2))$ such that $\int\psi=1$,
and put (here $m=\dim M_1$)
$$
\chi(x;h):=h^{-m\rho}\int_{V_2(h)}\psi\Big({x-y\over h^\rho}\Big)\,dy.
$$
It follows immediately from~\eqref{e:sc-3} that $\chi$ satisfies~\eqref{e:sc-1}.
Moreover, by putting derivatives on $\psi$ we obtain the derivative bounds~\eqref{e:sc-2},
finishing the proof.
\end{proof}
%%%%%%%%%%%%%%%%%%%%%%%%%%%%%%%%%%%%%%%%%%%%%%%%%%%%%%%%%%%%%%%%%%%%%%%%%%%%%%%%
To keep track of the essential supports of symbols in $S^{\comp}_{L,\rho}(U)$ in an $h$-dependent way,
we use the following
%%%%%%%%%%%%%%%%%%%%%%%%%%%%%%%%%%%%%%%%%%%%%%%%%%%%%%%%%%%%%%%%%%%%%%%%%%%%%%%%
\begin{defi}
  \label{d:rapid-decay}
Assume that $a(x,\xi;h)$ is an $h$-dependent family of smooth functions in $(x,\xi)\in U$,
and $h_j\to 0$, $(x_j,\xi_j)\in U$ are some sequences. We say that $a$ is $\mathcal O(h^\infty)$
along the sequence $(x_j,\xi_j,h_j)$, if for each $N$ and each vector fields
$Z_1,\dots,Z_L$ on $U$, there exists a constant $C$ such that
$$
|Z_1\dots Z_L a(x_j,\xi_j;h_j)|\leq C h_j^N.
$$  
\end{defi}
%%%%%%%%%%%%%%%%%%%%%%%%%%%%%%%%%%%%%%%%%%%%%%%%%%%%%%%%%%%%%%%%%%%%%%%%%%%%%%%%
We next introduce local canonical coordinates bringing an arbitrary Lagrangian foliation
to a normal form. Let $L_0$ be the Lagrangian foliation on $T^*\mathbb R^n$ given by
the fibers of the cotangent bundle; that is, in the standard coordinates $(y,\eta)$
on $T^*\mathbb R^n$,
$$
L_0=\Span(\partial_{\eta_1},\dots,\partial_{\eta_n})
$$
is the annihilator of $dy$.
%%%%%%%%%%%%%%%%%%%%%%%%%%%%%%%%%%%%%%%%%%%%%%%%%%%%%%%%%%%%%%%%%%%%%%%%%%%%%%%%
\begin{defi}
Let $L$ be a Lagrangian foliation on $U\subset T^*M$. A \textbf{Lagrangian chart}
is a symplectomorphism
$$
\varkappa:U_0\to V,\quad
U_0\subset U,\quad
V\subset T^*\mathbb R^n,
$$
such that $d\varkappa(x,\xi)\cdot L_{(x,\xi)}=(L_0)_{\varkappa(x,\xi)}$ for each $(x,\xi)\in U_0$.
\end{defi}
%%%%%%%%%%%%%%%%%%%%%%%%%%%%%%%%%%%%%%%%%%%%%%%%%%%%%%%%%%%%%%%%%%%%%%%%%%%%%%%%
The basic properties of Lagrangian charts are given by
%%%%%%%%%%%%%%%%%%%%%%%%%%%%%%%%%%%%%%%%%%%%%%%%%%%%%%%%%%%%%%%%%%%%%%%%%%%%%%%%
\begin{lemm}
  \label{l:canonical}
1. Let $L$ be a Lagrangian foliation on $U\subset T^*M$ and $(x_0,\xi_0)\in U$.
Then there exists a Lagrangian chart $\varkappa:U_0\to T^*\mathbb R^n$ on some neighborhood
$U_0\subset U$ of $(x_0,\xi_0)$. 

2. Assume that $\varkappa:V\to V'$, where $V,V'\subset T^*\mathbb R^n$ are open, is a symplectomorphism
which preserves the foliation $L_0$, and $(y_0,\eta_0)\in V$.
Then there exists $\varepsilon>0$ such that
\begin{equation}
  \label{e:gauge-transform}
\varkappa(y,\eta)=\big(\varphi(y),(d\varphi(y))^{-T}\cdot (\eta-d\psi(y))\big),\quad
(y,\eta)\in B(y_0,\varepsilon)\times B(\eta_0,\varepsilon),
\end{equation}
for some diffeomorphism $\varphi:B(y_0,\varepsilon)\to\mathbb R^n$ onto its image
and some function $\psi\in C^\infty(B(y_0,\varepsilon);\mathbb R)$.
\end{lemm}
%%%%%%%%%%%%%%%%%%%%%%%%%%%%%%%%%%%%%%%%%%%%%%%%%%%%%%%%%%%%%%%%%%%%%%%%%%%%%%%%
\begin{proof}
1. Since $L$ is integrable,
by Frobenius's Theorem~\cite[Theorem~C.1.1]{ho3} there exist local coordinates $(y,\tilde\eta)$
in a neighborhood of $(x_0,\xi_0)$ such that $L$ is the annihilator of
$dy$. Moreover, since $L$ is Lagrangian, we have $\{y_j,y_k\}=0$.
Now, by Darboux Theorem~\cite[Theorem~21.1.6]{ho3} there exists a set of functions $\eta_1,\dots,\eta_n$
defined near $(x_0,\xi_0)$ such that
$$
\{y_j,y_k\}=\{\eta_j,\eta_k\}=0,\quad
\{\eta_j,y_k\}=\delta_{jk}.
$$
The map $(x,\xi)\mapsto (y,\eta)$ is a Lagrangian chart in a neighborhood of $(x_0,\xi_0)$.

2. Define the functions
$y',\eta'$ on $V$ by setting
$\varkappa:(y,\eta)\mapsto (y',\eta')$.
Since the annihilators of $dy'$ and $dy$ are the same (and both equal to $L_0$),
we have $y'=\varphi(y)$ for
$(y,\eta)\in B(y_0,\varepsilon)\times B(\eta_0,\varepsilon)$,
some $\varepsilon>0$,
and some diffeomorphism onto its image $\varphi:B(y_0,\varepsilon)\to\mathbb R^n$.

Since $\{y'_j,\eta'_k\}=\delta_{jk}$, we have
$$
\eta'=(d\varphi(y))^{-T}\cdot(\eta-F(y)),\quad
(y,\eta)\in B(y_0,\varepsilon)\times B(\eta_0,\varepsilon),
$$
for some smooth map $F:B(y_0,\varepsilon)\to\mathbb R^n$.
Since $\{\eta'_j,\eta'_k\}=0$, we have
$F(y)=d\psi(y)$ for some $\psi:B(y_0,\varepsilon)\to\mathbb R$.
\end{proof}
%%%%%%%%%%%%%%%%%%%%%%%%%%%%%%%%%%%%%%%%%%%%%%%%%%%%%%%%%%%%%%%%%%%%%%%%%%%%%%%%

%%%%%%%%%%%%%%%%%%%%%%%%%%%%%%%%%%%%%%%%%%%%%%%%%%%%%%%%%%%%%%%%%%%%%%%%%%%%%%%%
\subsection{Calculus on $\mathbb R^n$}

We next develop the calculus for the case $L=L_0$.
We have $a(y,\eta;h)\in S_{L_0,\rho}^{\comp}(T^*\mathbb R^n)$ if and only if
$a$ is supported inside some $h$-independent compact set and satisfies the derivative bounds
\begin{equation}
  \label{e:s0-symb}
\sup_{y,\eta}|\partial_y^\alpha\partial_\eta^\beta a(y,\eta;h)|\leq C_{\alpha\beta}h^{-\rho|\alpha|}.
\end{equation}
We derive several basic properties of quantizations of symbols in $S_{L_0,\rho}^{\comp}(T^*\mathbb R^n)$
by the map $\Op_h$ defined in~\eqref{e:standard-quantization}:
%%%%%%%%%%%%%%%%%%%%%%%%%%%%%%%%%%%%%%%%%%%%%%%%%%%%%%%%%%%%%%%%%%%%%%%%%%%%%%%%
\begin{lemm}
  \label{l:l2-bdd}
For $a\in S_{L_0,\rho}^{\comp}(T^*\mathbb R^n)$, the operator
$\Op_h(a)$ is bounded on $L^2(\mathbb R^n)$ uniformly in $h$.
\end{lemm}
%%%%%%%%%%%%%%%%%%%%%%%%%%%%%%%%%%%%%%%%%%%%%%%%%%%%%%%%%%%%%%%%%%%%%%%%%%%%%%%%
\begin{proof}
We introduce the unitary rescaling operator
$$
T_\rho:L^2(\mathbb R^n)\to L^2(\mathbb R^n),\quad
T_\rho u(y)=h^{\rho/4}u(h^{\rho/2} y).
$$
It suffices to estimate the $L^2\to L^2$ norm of
$$
T_\rho\Op_h(a)T_\rho^{-1}=\Op_h(a_\rho),\quad
a_\rho(\tilde y,\tilde \eta;h):=a(h^{\rho/2} \tilde y,h^{-\rho/2}\tilde \eta;h).
$$
It follows from~\eqref{e:s0-symb} that $a_\rho\in S_{\rho/2}$, where the classes
$S_\delta$, $0\leq \delta\leq 1/2$, are defined in~\cite[(4.4.5)]{e-z}.
It remains to apply~\cite[Theorem~4.23(ii)]{e-z}.
\end{proof}
%%%%%%%%%%%%%%%%%%%%%%%%%%%%%%%%%%%%%%%%%%%%%%%%%%%%%%%%%%%%%%%%%%%%%%%%%%%%%%%%

%%%%%%%%%%%%%%%%%%%%%%%%%%%%%%%%%%%%%%%%%%%%%%%%%%%%%%%%%%%%%%%%%%%%%%%%%%%%%%%%
\begin{lemm}
  \label{l:quant-basic}
Let $a,b\in S_{L_0,\rho}^{\comp}(T^*\mathbb R^n)$. Then:

1. We have
$$
\Op_h(a)\Op_h(b)=\Op_h(a\#b)+\mathcal O(h^\infty)_{L^2\to L^2},
$$
where $a\# b\in S_{L_0,\rho}^{\comp}(T^*\mathbb R^n)$ and for each $N$,
$$
a\# b(y,\eta;h)=\sum_{j=0}^{N-1}{(-ih)^j\over j!} (\partial_\eta\cdot\partial_{y'})^j
\big(a(y,\eta;h)b(y',\eta';h)\big)|_{y'=y\atop \eta'=\eta}+\mathcal O(h^{(1-\rho)N})_{S_{L_0,\rho}^{\comp}(T^*\mathbb R^n)}.
$$

2. We have
$$
\Op_h(a)^*=\Op_h(a^*)+\mathcal O(h^\infty)_{L^2\to L^2},
$$
where $a^*\in S_{L_0,\rho}^{\comp}(T^*\mathbb R^n)$ and for each $N$,
$$
a^*(y,\eta;h)=\sum_{j=0}^{N-1}{(-ih)^j\over j!} (\partial_\eta\cdot\partial_y)^j \overline{a(y,\eta;h)}
+\mathcal O(h^{(1-\rho)N})_{S_{L_0,\rho}^{\comp}(T^*\mathbb R^n)}.
$$
\end{lemm}
%%%%%%%%%%%%%%%%%%%%%%%%%%%%%%%%%%%%%%%%%%%%%%%%%%%%%%%%%%%%%%%%%%%%%%%%%%%%%%%%
\begin{proof}
It suffices to apply~\cite[Theorems~4.14 and~4.17]{e-z}
to the rescaled symbols $a_\rho,b_\rho\in S_{\rho/2}$ introduced in the proof of Lemma~\ref{l:l2-bdd}.
The resulting symbols are $\mathcal O(h^\infty)$ outside of a compact set
and thus can be cut off to compactly supported symbols.
\end{proof}
%%%%%%%%%%%%%%%%%%%%%%%%%%%%%%%%%%%%%%%%%%%%%%%%%%%%%%%%%%%%%%%%%%%%%%%%%%%%%%%%
Lemma~\ref{l:quant-basic} (or rather its trivial extension to symbols which are not compactly
supported) implies that
\begin{equation}
  \label{e:funny-pseudolocal}
\Op_h(b_1)\Op_h(a)\Op_h(b_2)=\mathcal O(h^\infty)_{L^2\to L^2},\quad
\supp b_1\cap\supp b_2=\emptyset,
\end{equation}
for each $a\in S_{L_0,\rho}^{\comp}(T^*\mathbb R^n)$ and
$h$-independent $b_1,b_2\in C^\infty(\mathbb R^{2n})$ with all derivatives uniformly bounded.
This in turn implies that the operator $\Op_h(a)$ is pseudolocal
and $\WF'_h(\Op_h(a))$ is compactly contained in $T^*(\mathbb R^n\times\mathbb R^n)$.

The next two lemmas establish invariance of the class of operators of the form
$\Op_h(a)$, $a\in S_{L_0,\rho}^{\comp}(T^*\mathbb R^n)$, under conjugation by Fourier integral
operators whose canonical transformations preserve the foliation $L_0$:
%%%%%%%%%%%%%%%%%%%%%%%%%%%%%%%%%%%%%%%%%%%%%%%%%%%%%%%%%%%%%%%%%%%%%%%%%%%%%%%%
\begin{lemm}
\label{l:chvar}
Assume that $\varphi:V_1\to V_2$ is a diffeomorphism,
where $V_j\subset \mathbb R^n$ are open sets,
$\psi\in C^\infty(V_1)$, and
$\chi\in C_0^\infty(V_1)$. Define the operators
$B:C^\infty(V_1)\to C_0^\infty(V_2)$,
$B':C^\infty(V_2)\to C_0^\infty(V_1)$
by
$$
Bg(y')=e^{-i\psi(\varphi^{-1}(y'))/h}\chi(\varphi^{-1}(y'))f(\varphi^{-1}(y')),\quad
B'f(y)=e^{i\psi(y)/h}\chi(y)f(\varphi(y)).
$$
Then for each
$a\in S_{L_0,\rho}^{\comp}(T^*\mathbb R^n)$,
$$
B' \Op_h(a)B=\Op_h(\tilde a)+\mathcal O(h^\infty)_{L^2\to L^2},
$$
for some $\tilde a\in S_{L_0,\rho}^{\comp}(T^*\mathbb R^n)$ such that for each $N$,
$$
\tilde a(y,\eta;h)=\sum_{j=0}^{N-1}h^j L_j\big(\chi(y)\chi(y')a(\varphi(y),\theta;h)\big)\big|_{y'=y,\,
\theta=d\varphi(y)^{-T}(\eta-d\psi(y))}+\mathcal O(h^N)_{S^{\comp}_{L_0,\rho}(\mathbb R^n)}.
$$
where $L_j$ are differential operators of order $2j$ in $y',\theta$
depending on $\varphi,\psi$ and $L_0=1$.
\end{lemm}
%%%%%%%%%%%%%%%%%%%%%%%%%%%%%%%%%%%%%%%%%%%%%%%%%%%%%%%%%%%%%%%%%%%%%%%%%%%%%%%%
\begin{proof}
We write
$$
\begin{gathered}
B'\Op_h(a)B f(y)\\
=(2\pi h)^{-n}\int_{\mathbb R^{2n}}
e^{{i\over h}((\varphi(y)-\varphi(y'))\cdot\theta+\psi(y)-\psi(y'))}
\chi(y)\chi(y')J_\varphi(y')a(\varphi(y),\theta;h)f(y')\,dy'd\theta
\end{gathered}
$$
where $J_\varphi(y')=|\det d\varphi(y')|$. By oscillatory testing~\cite[Theorem~4.19]{e-z},
we have $B'\Op_h(a)B=\Op_h(b)$, where
$$
\begin{gathered}
b(y,\eta;h)=e^{-{i\over h} y\cdot\eta}B'\Op_h(a)B(e^{{i\over h}y'\cdot\eta})\\
=(2\pi h)^{-n}\int_{\mathbb R^{2n}}e^{{i\over h}((\varphi(y)-\varphi(y'))\cdot\theta+\psi(y)-\psi(y')-(y-y')\cdot\eta)}
\chi(y)\chi(y')J_\varphi(y')a(\varphi(y),\theta;h)\,dy'd\theta,
\end{gathered}
$$
as long as all derivatives of $b$ are bounded uniformly on $\mathbb R^{2n}$ for each fixed~$h$.
It then remains to establish the asymptotic expansion for $b$, which follows immediately by the method of stationary
phase~\cite[Theorem~3.16]{e-z}. The symbol $b$ is $\mathcal O(h^\infty)_{\mathscr S(\mathbb R^{2n})}$ outside
of a fixed compact set, therefore it can be cut off to a compactly supported symbol.
\end{proof}
%%%%%%%%%%%%%%%%%%%%%%%%%%%%%%%%%%%%%%%%%%%%%%%%%%%%%%%%%%%%%%%%%%%%%%%%%%%%%%%%
%
%%%%%%%%%%%%%%%%%%%%%%%%%%%%%%%%%%%%%%%%%%%%%%%%%%%%%%%%%%%%%%%%%%%%%%%%%%%%%%%%
\begin{lemm}
  \label{l:gauge}
Assume that $\varkappa:V\to V'$, where $V,V'\subset T^*\mathbb R^n$ are open, is a canonical transformation
which preserves the foliation $L_0$. 
Let
$$
B\in I^{\comp}_h(\varkappa),\quad
B'\in I^{\comp}_h(\varkappa^{-1}).
$$
Take $a\in S^{\comp}_{L_0,\rho}(T^*\mathbb R^n)$. Then there exists
$b\in S^{\comp}_{L_0,\rho}(T^*\mathbb R^n)$ such that
$$
\begin{aligned}
B'\Op_h(a)B&=\Op_h(b)+\mathcal O(h^\infty)_{L^2\to L^2},\\
b&=(a\circ\varkappa)\sigma_h(B'B)+\mathcal O(h^{1-\rho})_{S^{\comp}_{L_0,\rho}(T^*\mathbb R^n)}.
\end{aligned}
$$
Moreover, if $h_j\to 0$, $(y_j,\eta_j)\in T^*\mathbb R^n$ are some sequences
such that $a$ is $\mathcal O(h^\infty)$ along the sequence $(\varkappa(y_j,\eta_j),h_j)$
(in the sense of Definition~\ref{d:rapid-decay}), then
$b$ is $\mathcal O(h^\infty)$ along the sequence $(y_j,\eta_j,h_j)$.
\end{lemm}
%%%%%%%%%%%%%%%%%%%%%%%%%%%%%%%%%%%%%%%%%%%%%%%%%%%%%%%%%%%%%%%%%%%%%%%%%%%%%%%%
\begin{proof}
By applying a partition of unity to $B,B'$ and using pseudolocality of $\Op_h(a)$
(see~\eqref{e:funny-pseudolocal}) and part~2 of Lemma~\ref{l:canonical},
we reduce to the case when $\varkappa$ has the form~\eqref{e:gauge-transform}
for some $\varphi:B(y_0,\varepsilon)\to\mathbb R^n$, $\psi\in C^\infty(B(y_0,\varepsilon);\mathbb R)$;
we add a constant to $\psi$ to make sure that the fixed antiderivative on $\Graph(\varkappa)$
is equal to $\psi(x)$.
By Lemma~\ref{l:gauge-fio} and the composition property of Fourier integral operators, the products
$$
A:=e^{{i\over h}\psi}\varphi^* B,\quad
A':=B'(\varphi^{-1})^*e^{-{i\over h}\psi},
$$
lie in $\Psi^{\comp}_h(\mathbb R^n)$. (Lemma~\ref{l:gauge-fio} applies since we can insert an
element of $\Psi^{\comp}_h$ in between $B,B'$ and other factors.)
Since $\WF'_h(B)\subset\Graph(\varkappa)$ and $\WF'_h(B')\subset\Graph(\varkappa')$
are compact, there exists $\chi\in C_0^\infty(B(y_0,\varepsilon))$ such that
$$
B=(\chi\circ\varphi^{-1})B+\mathcal O(h^\infty)_{L^2\to L^2},\quad
B'=B'(\chi\circ\varphi^{-1})+\mathcal O(h^\infty)_{L^2\to L^2}.
$$
Then we write 
$$
B'\Op_h(a)B=A'\big(\chi e^{{i\over h}\psi}\varphi^*\Op_h(a)(\varphi^{-1})^*e^{-{i\over h}\psi}\chi\big) A+\mathcal O(h^\infty)_{L^2\to L^2}.
$$
By Lemma~\ref{l:chvar}, we can write the operator in parentheses on the right-hand side
as $\Op_h(\tilde a)+\mathcal O(h^\infty)_{L^2\to L^2}$ for some $\tilde a\in S^{\comp}_{L_0,\rho}(T^*\mathbb R^n)$;
by Lemma~\ref{l:quant-basic}, we have $A'\Op_h(\tilde a)A=\Op_h(b)+\mathcal O(h^\infty)_{L^2\to L^2}$
for some $b\in S^{\comp}_{L_0,\rho}(T^*\mathbb R^n)$. The expression for the principal part of $b$
and the microlocal vanishing statement follow directly from Lemmas~\ref{l:quant-basic} and~\ref{l:chvar}
and the fact that $\sigma_h(A')\sigma_h(A)=\sigma_h(A'A)=\sigma_h(B'B)$.
\end{proof}
%%%%%%%%%%%%%%%%%%%%%%%%%%%%%%%%%%%%%%%%%%%%%%%%%%%%%%%%%%%%%%%%%%%%%%%%%%%%%%%%

%%%%%%%%%%%%%%%%%%%%%%%%%%%%%%%%%%%%%%%%%%%%%%%%%%%%%%%%%%%%%%%%%%%%%%%%%%%%%%%%
\subsection{General calculus}
  \label{s:calculus-general}

We now introduce pseudodifferential operators associated to general Lagrangian foliations,
starting with the following 
%%%%%%%%%%%%%%%%%%%%%%%%%%%%%%%%%%%%%%%%%%%%%%%%%%%%%%%%%%%%%%%%%%%%%%%%%%%%%%%%
\begin{defi}
Let $M$ be a manifold, $U\subset T^*M$ an open set, $L$ a Lagrangian foliation on $U$, and
$\rho\in [0,1)$.
A family of operators
$$
A=A(h):\mathcal D'(M)\to C_0^\infty(M)
$$
is called a semiclassical pseudodifferential operator with symbol of class $S_{L,\rho}^{\comp}(U)$
(denoted $A\in\Psi^{\comp}_{h,L,\rho}(U)$)
if it can be written in the form
\begin{equation}
  \label{e:represent}
A=\sum_{\ell=1}^N B'_\ell \Op_h(a_\ell) B_\ell+\mathcal O(h^\infty)_{\mathcal D'(M)\to C_0^\infty(M)}
\end{equation}
for some Lagrangian charts $\varkappa_\ell$, Fourier integral operators
$B_\ell\in I^{\comp}_h(\varkappa_\ell),B'_\ell\in I^{\comp}_h(\varkappa_\ell^{-1})$,
and symbols $a_\ell\in S_{L_0,\rho}^{\comp}(T^*\mathbb R^n)$.
\end{defi}
%%%%%%%%%%%%%%%%%%%%%%%%%%%%%%%%%%%%%%%%%%%%%%%%%%%%%%%%%%%%%%%%%%%%%%%%%%%%%%%%
%
%%%%%%%%%%%%%%%%%%%%%%%%%%%%%%%%%%%%%%%%%%%%%%%%%%%%%%%%%%%%%%%%%%%%%%%%%%%%%%%%
\begin{lemm}
  \label{l:globallem}
Let $A\in \Psi^{\comp}_{h,L,\rho}(U)$. Then:

1. $A$ is bounded on $L^2$ uniformly in $h$, pseudolocal, and
$\WFh(A)\subset U$ is compact.

2. If $\varkappa:\widetilde U\to T^*\mathbb R^n$ is a Lagrangian chart
and $B\in I^{\comp}_h(\varkappa),B'\in I^{\comp}_h(\varkappa^{-1})$, then
$$
BAB'=\Op_h(a)+\mathcal O(h^\infty)_{L^2\to L^2}
$$
for some $a\in S_{L_0,\rho}^{\comp}(T^*\mathbb R^n)$,
$\supp a\subset\varkappa(\widetilde U)$, and for each
representation~\eqref{e:represent} of $A$,
\begin{equation}
  \label{e:symbol-transform}
a\circ\varkappa=\sigma_h(B'B)\sum_{\ell=1}^N \sigma_h(B'_\ell B_\ell)(a_\ell\circ\varkappa_\ell)+\mathcal O(h^{1-\rho})_{S^{\comp}_{L,\rho}(U)}.
\end{equation}
Moreover, if $h_j\to 0$
and $(x_j,\xi_j)\in U$ are sequences such that for each $\ell$, either
$(x_j,\xi_j)\notin \pi_1(\WF'_h (B_\ell))\cap \pi_2(\WF'_h(B'_\ell))$ for all $j$
or $a_\ell\circ\varkappa_\ell$ is $\mathcal O(h^\infty)$
along $(x_j,\xi_j,h_j)$ for all $\ell$ in the sense of Definition~\ref{d:rapid-decay}, then
$a\circ\varkappa$ is $\mathcal O(h^\infty)$ along $(x_j,\xi_j,h_j)$ as well.
\end{lemm}
%%%%%%%%%%%%%%%%%%%%%%%%%%%%%%%%%%%%%%%%%%%%%%%%%%%%%%%%%%%%%%%%%%%%%%%%%%%%%%%%
\begin{proof}
1. This follows immediately from Lemma~\ref{l:l2-bdd}
and the properties of $\Op_h(a)$, $a\in S^{\comp}_{L_0,\rho}(T^*\mathbb R^n)$ established
in the paragraph following~\eqref{e:funny-pseudolocal}.

2. We write $A$ in the form~\eqref{e:represent}, then
$$
BAB'=\sum_{j=1}^N (BB'_\ell)\Op_h(a_\ell)(B_\ell B')+\mathcal O(h^\infty)_{\mathcal D'\to C_0^\infty}.
$$
Now, we have $B_\ell B'\in I^{\comp}_h(\varkappa'_j)$, where
$\varkappa'_\ell=\varkappa_\ell\circ\varkappa^{-1}:\varkappa(\widetilde U\cap U_\ell)\to\mathbb T^*\mathbb R^n$
is a symplectomorphism onto its image preserving the foliation $L_0$ and
$U_\ell$ is the domain of $\varkappa_\ell$. Similarly $BB'_\ell \in I^{\comp}_h((\varkappa'_j)^{-1})$.
It remains to apply Lemma~\ref{l:gauge}. To see~\eqref{e:symbol-transform},
we use the following corollary of~\eqref{e:symbol-commutes}:
$\sigma_h(BB'_\ell B_\ell B')=(\sigma_h(B'_\ell B_\ell)\sigma_h(B'B))\circ\varkappa^{-1}$.
\end{proof}
%%%%%%%%%%%%%%%%%%%%%%%%%%%%%%%%%%%%%%%%%%%%%%%%%%%%%%%%%%%%%%%%%%%%%%%%%%%%%%%%
We now define the principal symbol and ($h$-dependent) microsupport of an operator in $\Psi^{\comp}_{h,L,\rho}(U)$:
%%%%%%%%%%%%%%%%%%%%%%%%%%%%%%%%%%%%%%%%%%%%%%%%%%%%%%%%%%%%%%%%%%%%%%%%%%%%%%%%
\begin{defi}
  \label{d:lag-prince}
Let $A\in \Psi^{\comp}_{h,L,\rho}(U)$. We define the principal symbol
$$
\sigma_h^L(A)\in S^{\comp}_{L,\rho}(U)/h^{1-\rho}S^{\comp}_{L,\rho}(U)
$$
by the following formula valid for any representation~\eqref{e:represent}:
\begin{equation}
  \label{e:symboldef}
\sigma_h^L(A)=\sum_{\ell=1}^N \sigma_h(B'_\ell B_\ell)(a_\ell\circ\varkappa_\ell).
\end{equation}
Moreover, if $h_j\to 0$ and $(x_j,\xi_j)\in U$ are some sequences, then we say
that $A=\mathcal O(h^\infty)$ microlocally along $(x_j,\xi_j,h_j)$, if for each
choice of $\varkappa,B,B'$ in part~2 of Lemma~\ref{l:globallem} and the corresponding
symbol $a\in S^{\comp}_{L_0,\rho}(T^*\mathbb R^n)$, the symbol
$a\circ\varkappa$ is $\mathcal O(h^\infty)$ along $(x_j,\xi_j,h_j)$.
\end{defi}
%%%%%%%%%%%%%%%%%%%%%%%%%%%%%%%%%%%%%%%%%%%%%%%%%%%%%%%%%%%%%%%%%%%%%%%%%%%%%%%%
It follows from Lemma~\ref{l:globallem} that $\sigma_h^L(A)$ does not depend
on the choice of the representation~\eqref{e:represent} of $A$.

To construct an operator with given principal symbol and microsupport, we use the \emph{quantization map}
\begin{equation}
  \label{e:op-h-l}
a\in S^{\comp}_{L,\rho}(U)\ \mapsto\ \Op_h^L(a):=\sum_\ell B'_\ell \Op_h(a_\ell)B_\ell,
\end{equation}
where the sum above has finitely many nonzero terms for each $a$ and
\begin{itemize}
\item $\varkappa_\ell: U_\ell\to T^*\mathbb R^n$ are Lagrangian charts and $U_\ell$,
$\ell\in\mathbb N$, form a locally finite covering of $U$;
\item $B_\ell\in I^{\comp}_h(\varkappa_\ell)$, $B'_\ell\in I^{\comp}_h(\varkappa_\ell^{-1})$
are Fourier integral operators such that $\sigma_h(B'_\ell B_\ell)\in C_0^\infty(U_\ell)$
form a partition of unity:
$$
\sum_\ell \sigma_h(B'_\ell B_\ell)=1\quad\text{on }U;
$$
\item $a_\ell=(\chi_\ell a)\circ\varkappa_\ell^{-1}\in S^{\comp}_{L_0,\rho}(T^*\mathbb R^n)$,
where $\chi_\ell\in C_0^\infty(U_\ell)$ are some functions equal to 1 near
$\supp \sigma_h(B'_\ell B_\ell)$.
\end{itemize}
One can choose $\varkappa_\ell$ with the required properties by Lemma~\ref{l:canonical};
for existence of $B_\ell,B'_\ell$ see the discussion following~\eqref{e:quantized}. The quantization map is not canonical
as it depends on the choice of $\varkappa_\ell,B_\ell,B'_\ell,\chi_\ell$.

Note that if $A\in\Psi^{\comp}_h(M)$ is a pseudodifferential operator in the standard calculus
and $\WFh(A)\subset U$, then $A\in\Psi^{\comp}_{h,L,\rho}(U)$ and $\sigma_h^L(A)=\sigma_h(A)$.
Also, if $a\in S^0_h(T^*M)$ is a symbol in the standard class supported inside $U$,
then
$\Op_h^L(a)\in\Psi^{\comp}_h(M)$. This follows from the composition property
of Fourier integral operators together with~\eqref{e:symbol-commutes}.

The basic properties of the symbol map and a quantization map are given by
%%%%%%%%%%%%%%%%%%%%%%%%%%%%%%%%%%%%%%%%%%%%%%%%%%%%%%%%%%%%%%%%%%%%%%%%%%%%%%%%
\begin{lemm}
  \label{l:globalprop}
1. For each $a\in S^{\comp}_{L,\rho}(U)$,
$$
\sigma_h^L(\Op_h^L(a))=a+\mathcal O(h^{1-\rho})_{S^{\comp}_{L,\rho}(U)}.
$$

2. For each $A\in\Psi^{\comp}_{h,L,\rho}(U)$, we have
$\sigma_h^L(A)=\mathcal O(h^{1-\rho})_{S^{\comp}_{L,\rho}(U)}$ if and only
if $A\in h^{1-\rho}\Psi^{\comp}_{h,L,\rho}(U)$.

3. For each $A\in\Psi^{\comp}_{h,L,\rho}(U)$,
there exists $a\in S^{\comp}_{L,\rho}(U)$ such that $A=\Op_h^L(a)+\mathcal O(h^\infty)_{\mathcal D'\to C_0^\infty}$.

4. For each $a\in S^{\comp}_{L,\rho}(U)$, if $h_j\to 0$ and $(x_j,\xi_j)\in U$
are sequences such that $a$ is $\mathcal O(h^\infty)$ along $(x_j,\xi_j,h_j)$
(in the sense of Definition~\ref{d:rapid-decay}), then
$\Op_h^L(a)$ is $\mathcal O(h^\infty)$ microlocally along $(x_j,\xi_j,h_j)$
(in the sense of Definition~\ref{d:lag-prince}).

5. Let $A,B\in\Psi^{\comp}_{h,L,\rho}(U)$. Then $AB,A^*\in\Psi^{\comp}_{h,L,\rho}(U)$ and
$$
\begin{aligned}
\sigma_h^L(AB)&=\sigma_h^L(A)\sigma_h^L(B)+\mathcal O(h^{1-\rho})_{S^{\comp}_{L,\rho}(U)},\\
\sigma_h^L(A^*)&=\overline{\sigma_h^L(A)}+\mathcal O(h^{1-\rho})_{S^{\comp}_{L,\rho}(U)}.
\end{aligned}
$$
\end{lemm}
%%%%%%%%%%%%%%%%%%%%%%%%%%%%%%%%%%%%%%%%%%%%%%%%%%%%%%%%%%%%%%%%%%%%%%%%%%%%%%%%
\begin{proof}
1. This follows immediately from~\eqref{e:symboldef}.

2. Assume that $\sigma_h^L(A)=\mathcal O(h^{1-\rho})_{S^{\comp}_{L,\rho}(U)}$;
we need to show that $A\in h^{1-\rho}\Psi^{\comp}_{h,L,\rho}(U)$ (the reverse implication
follows directly from~\eqref{e:symboldef}). Using a pseudodifferential partition of unity,
we may assume that $\WFh(A)$ is contained in a some open subset $\widetilde U\subset U$
such that there exists a Lagrangian chart $\varkappa:\widetilde U\to T^*\mathbb R^n$
and Fourier integral operators
$$
B\in I^{\comp}_h(\varkappa),\
B'\in I^{\comp}_h(\varkappa^{-1});\quad
B'B=1+\mathcal O(h^\infty)\quad\text{microlocally near }\WFh(A).
$$
Then
$$
A=B'(BAB')B+\mathcal O(h^\infty)_{\mathcal D'\to C_0^\infty}.
$$
However, by part~2 of Lemma~\ref{l:globallem} we have
$$
BAB'=\Op_h(a)+\mathcal O(h^\infty)_{L^2\to L^2},\quad
a\in S^{\comp}_{L_0,\rho}(T^*\mathbb R^n),
$$
and since $\sigma_h^L(A)=\mathcal O(h^{1-\rho})_{S^{\comp}_{L,\rho}(U)}$, we have
$a=\mathcal O(h^{1-\rho})_{S^{\comp}_{L_0,\rho}(T^*\mathbb R^n)}$. Therefore,
$a=h^{1-\rho}b$ for some $b\in S^{\comp}_{L_0,\rho}(T^*\mathbb R^n)$ and
$$
A=h^{1-\rho}(B'\Op_h(b)B+\mathcal O(h^\infty)_{\mathcal D'\to C_0^\infty})\in h^{1-\rho}\Psi^{\comp}_{h,L,\rho}(U).
$$

3. Put $a_0=\sigma_h^L(A)$. Then $A=\Op_h^L(a_0)+\mathcal O(h^{1-\rho})_{\Psi^{\comp}_{h,L,\rho}(U)}$, therefore
$A=\Op_h^L(a_0)+h^{1-\rho}A_1$ for some $A_1\in\Psi^{\comp}_{h,L,\rho}(U)$. By induction we construct a family of
operators $A_j\in \Psi^{\comp}_{h,L,\rho}(U)$, $j\in\mathbb N_0$, such that $A_0=A$
and $A_j=\Op_h^L(\sigma_h^L(A_j))+h^{1-\rho}A_{j+1}$. Then we have
$A=\Op_h^L(a)+\mathcal O(h^\infty)_{\mathcal D'\to C_0^\infty}$ where $a\in S^{\comp}_{h,L,\rho}(U)$ is the following asymptotic sum:
$$
a\sim\sum_{j=0}^\infty h^{(1-\rho)j}\sigma_h(A_j).
$$

4,5. These follow from Lemma~\ref{l:quant-basic} and part~2 of Lemma~\ref{l:globallem}.
\end{proof}
%%%%%%%%%%%%%%%%%%%%%%%%%%%%%%%%%%%%%%%%%%%%%%%%%%%%%%%%%%%%%%%%%%%%%%%%%%%%%%%%

%%%%%%%%%%%%%%%%%%%%%%%%%%%%%%%%%%%%%%%%%%%%%%%%%%%%%%%%%%%%%%%%%%%%%%%%%%%%%%%%
\subsection{Further properties}

We start with an improved bound on the $L^2$ operator norm:
%%%%%%%%%%%%%%%%%%%%%%%%%%%%%%%%%%%%%%%%%%%%%%%%%%%%%%%%%%%%%%%%%%%%%%%%%%%%%%%%
\begin{lemm}
  \label{l:l2-improved}
Let $A\in\Psi^{\comp}_{h,L,\rho}(U)$. Then as $h\to 0$,
$$
\|A\|_{L^2(M)\to L^2(M)}\leq \sup_U |\sigma_h^L(A)|+o(1).
$$
\end{lemm}
%%%%%%%%%%%%%%%%%%%%%%%%%%%%%%%%%%%%%%%%%%%%%%%%%%%%%%%%%%%%%%%%%%%%%%%%%%%%%%%%
\begin{proof}
Take $\varepsilon>0$ and let $a:=\sigma_h^L(A)$. It suffices to prove that
\begin{equation}
  \label{e:l2-bdd-new}
\limsup_{h\to 0}\|A\|_{L^2(M)\to L^2(M)}\leq C_\varepsilon:=\sup_U |a|+\varepsilon.
\end{equation}
Define the function $b$ by
$$
b=\sqrt{C_\varepsilon^2-|a|^2}.
$$
Note that $b=C_\varepsilon$ outside of $\supp a$ and
$C_\varepsilon-b\in S^{\comp}_{L,\rho}(U)$.
Take the following quantization of $b$:
$$
B:=C_\varepsilon-\Op_h^L(C_\varepsilon-b)
$$
and note that $B$ is bounded on $L^2$.
Since $|a|^2+|b|^2=C_\varepsilon^2$, we have
$$
A^*A+B^*B=C_\varepsilon^2+\mathcal O(h^{1-\rho})_{L^2\to L^2}.
$$
By applying this to $u\in L^2$ and taking the scalar product with $u$ itself, we get
$$
\|Au\|_{L^2}^2\leq C_\varepsilon^2\|u\|_{L^2}^2+\mathcal O(h^{1-\rho})\|u\|_{L^2}^2,
$$
which implies~\eqref{e:l2-bdd-new}.
\end{proof}
%%%%%%%%%%%%%%%%%%%%%%%%%%%%%%%%%%%%%%%%%%%%%%%%%%%%%%%%%%%%%%%%%%%%%%%%%%%%%%%%
We next give a version of the elliptic parametrix construction:
%%%%%%%%%%%%%%%%%%%%%%%%%%%%%%%%%%%%%%%%%%%%%%%%%%%%%%%%%%%%%%%%%%%%%%%%%%%%%%%%
\begin{lemm}
  \label{l:lag-elliptic}
Assume that $A,B\in\Psi^{\comp}_{h,L,\rho}(U)$ and $B$ is elliptic on the microsupport
of $A$ in the following sense: there exists $\varepsilon>0$ such that
for each sequences $h_j\to 0$ and $(x_j,\xi_j)\in U$,
if $|\sigma_h^L(B)(x_j,\xi_j;h_j)|\leq\varepsilon$, then $A$ is $\mathcal O(h^\infty)$
microlocally along $(x_j,\xi_j;h_j)$. Then there exists
$$
Q\in\Psi^{\comp}_{h,L,\rho}(U),\quad
A=QB+\mathcal O(h^\infty)_{\mathcal D'\to C_0^\infty}.
$$
\end{lemm}
%%%%%%%%%%%%%%%%%%%%%%%%%%%%%%%%%%%%%%%%%%%%%%%%%%%%%%%%%%%%%%%%%%%%%%%%%%%%%%%%
\begin{proof}
Let $a_0=\sigma_h^L(A),b_0=\sigma_h^L(B)$. We first show that
there exists $q_0\in S^{\comp}_{L,\rho}(U)$ such that
$a_0=q_0b_0+\mathcal O(h^{1-\rho})_{S^{\comp}_{L,\rho}(U)}$.
For that, let $\chi\in C_0^\infty(-\varepsilon,\varepsilon)$ be equal to 1 near the origin.
Then from the ellipticity assumption we have
$$
a_0 \chi(|b_0|)=\mathcal O(h^{1-\rho})_{S^{\comp}_{L,\rho}(U)}
$$
and it remains to put $q_0:=a_0\big(1-\chi(|b_0|)\big)/b_0$.

Put $Q_0=\Op_h^L(q_0)$, then
$$
A=Q_0B+\mathcal O(h^{1-\rho})_{\Psi^{\comp}_{h,L,\rho}(U)}.
$$
We write $A=Q_0B+h^{1-\rho}R_1$ for some $R_1\in \Psi^{\comp}_{h,L,\rho}(U)$.
By part~4 of Lemma~\ref{l:globalprop}, $B$ is elliptic on the microsupport of $Q_0$;
therefore, it is elliptic on the microsupport of $R_1$.
Repeating the above process, we find a sequence $Q_j=\Op_h^L(q_j),R_j\in \Psi^{\comp}_{h,L,\rho}(U)$
such that $A=R_0$ and
$$
R_j=Q_j B+h^{1-\rho}R_{j+1},\quad
j=0,1,\dots
$$
It remains to take $Q=\Op_h^L(q)$, where $q\in S^{\comp}_{L,\rho}(U)$ is the asymptotic sum
$$
q\sim\sum_{j=0}^\infty h^{(1-\rho)j}q_j.\qedhere
$$
\end{proof}
%%%%%%%%%%%%%%%%%%%%%%%%%%%%%%%%%%%%%%%%%%%%%%%%%%%%%%%%%%%%%%%%%%%%%%%%%%%%%%%%
%
Finally, we give a version of Egorov's theorem for the $\Psi^{\comp}_{h,L,\rho}$ calculus:
%
%%%%%%%%%%%%%%%%%%%%%%%%%%%%%%%%%%%%%%%%%%%%%%%%%%%%%%%%%%%%%%%%%%%%%%%%%%%%%%%%
\begin{lemm}
  \label{l:egorov}
Let $L$ be a Lagrangian foliation on $U\subset T^*M$, $P\in\Psi^{\comp}_h(M)$,
the principal symbol $p=\sigma_h(P)$ be real-valued (and $h$-independent), and
\begin{equation}
  \label{e:egorov-condition}
L_{(x,\xi)}\ \subset\ \ker dp(x,\xi)\quad\text{for each }(x,\xi)\in U.
\end{equation}
Let $A\in\Psi^{\comp}_{h,L,\rho}(U)$ and take $T>0$ such that
$e^{-tH_p}(\WFh(A))\subset U$ for all $t\in [0,T]$.
Then there exists a family of operators depending smoothly on $t$
$$
A_t\in\Psi^{\comp}_{h,L,\rho}(U),\quad
t\in [0,T],\quad
A_0=A+\mathcal O(h^\infty)_{\mathcal D'\to C_0^\infty},
$$
such that $\sigma_h^L(A_t)=\sigma_h^L(A)\circ e^{tH_p}+\mathcal O(h^{1-\rho})_{S^{\comp}_{L,\rho}(U)}$ and
\begin{equation}
  \label{e:egorov-equation}
ih\partial_t A_t+[P,A_t]=\mathcal O(h^\infty)_{\mathcal D'\to C_0^\infty}.
\end{equation}
Moreover, if $t$ is fixed and $h_j\to 0$, $(x_j,\xi_j)\in U$ are sequences such that
$A$ is $\mathcal O(h^\infty)$ microlocally along $(x_j,\xi_j,h_j)$, then
$A_t$ is $\mathcal O(h^\infty)$ microlocally along $(e^{-tH_p}(x_j,\xi_j),h_j)$.
\end{lemm}
%%%%%%%%%%%%%%%%%%%%%%%%%%%%%%%%%%%%%%%%%%%%%%%%%%%%%%%%%%%%%%%%%%%%%%%%%%%%%%%%
\begin{proof}
First of all, by~\eqref{e:egorov-condition} the Hamiltonian vector field $H_p$ lies in $L$.
Therefore, the Hamiltonian flow $e^{tH_p}$ preserves $L$, and
for each $a\in S^{\comp}_{L,\rho}(U)$, both $H_pa$ and $a\circ e^{tH_p}$ (as long as
$e^{-tH_p}(\supp a)\subset U$) lie in $S^{\comp}_{L,\rho}(U)$ as well.
Next, we have for each $\widetilde A\in\Psi^{\comp}_{h,L,\rho}(U)$
\begin{equation}
  \label{e:nice-commutator}
[P,\widetilde A]={h\over i}\Op_h^L(H_p \sigma_h^L(\widetilde A))+\mathcal O(h^{2-\rho})_{\Psi^{\comp}_{h,L,\rho}(U)}.
\end{equation}
Indeed, by a pseudodifferential partition of unity and part~2 of Lemma~\ref{l:globallem}, it suffices to prove%
\footnote{We do not obtain the error $\mathcal O(h^2)$ in~\eqref{e:nice-commutator}
because right-hand side is in $h\Psi^{\comp}_{h,L,\rho}(U)$
and thus Lemma~\ref{l:globallem} produces an $\mathcal O(h\cdot h^{1-\rho})$ error.}
that for each function $f\in C_0^\infty(\mathbb R^n)$,
$$
f(y)\#\tilde a-\tilde a\# f(y)={h\over i}\{f(y),\tilde a\}+\mathcal O(h^2)_{S^{\comp}_{h,L,\rho}(T^*\mathbb R^n)}
$$
which follows immediately from Lemma~\ref{l:quant-basic}.

Take $a\in S^{\comp}_{L,\rho}(U)$ such that $A=\Op_h^L(a)+\mathcal O(h^\infty)_{\mathcal D'\to C_0^\infty}$.
Consider the family of operators
$$
A^{(0)}_t:=\Op_h^L(a^{(0)}_t),\quad
a^{(0)}_t=a\circ e^{tH_p},\quad
t\in [0,T].
$$
Then~\eqref{e:nice-commutator} implies
$$
ih\partial_t A^{(0)}_t+[P,A^{(0)}_t]=h^{2-\rho}R^{(1)}_t,\quad
R^{(1)}_t\in\Psi^{\comp}_{h,L,\rho}(U).
$$
Next, put for $t\in [-T,T]$,
$$
A^{(1)}_t:=\Op_h^L(a^{(1)}_t),\quad
a^{(1)}_t:=\int_0^t i\sigma_h^L(R^{(1)}_s)\circ e^{(t-s)H_p}\,ds.
$$
Then
$$
a^{(1)}_0=0,\quad
\partial_t a^{(1)}_t=H_pa^{(1)}_t+i\sigma_h^L(R^{(1)}_t)
$$
and thus~\eqref{e:nice-commutator} implies
$$
ih\partial_t A^{(1)}_t+[P,A^{(1)}_t]+hR^{(1)}_t=h^{2-\rho}R^{(2)}_t,\quad
R^{(2)}_t\in\Psi^{\comp}_{h,L,\rho}(U).
$$
Arguing by induction, we construct operators
$A^{(j)}_t=\Op_h^L(a^{(j)}_t),R^{(j)}_t\in\Psi^{\comp}_{h,L,\rho}(U)$, $j\in\mathbb N_0$, such that
$R^{(0)}_t=0$, $A^{(0)}_0=A+\mathcal O(h^\infty)_{\mathcal D'\to C_0^\infty}$,
$A^{(j)}_0=0$ for $j>0$, and
$$
ih\partial_t A^{(j)}_t+[P,A^{(j)}_t]+hR^{(j)}_t=h^{2-\rho}R^{(j+1)}_t.
$$
It remains to put
$$
A_t:=\Op_h^L(a_t),\quad
a_t\sim\sum_{j=0}^\infty h^{(1-\rho)j} a_t^{(j)}.\qedhere
$$
\end{proof}
%%%%%%%%%%%%%%%%%%%%%%%%%%%%%%%%%%%%%%%%%%%%%%%%%%%%%%%%%%%%%%%%%%%%%%%%%%%%%%%%

%%%%%%%%%%%%%%%%%%%%%%%%%%%%%%%%%%%%%%%%%%%%%%%%%%%%%%%%%%%%%%%%%%%%%%%%%%%%%%%%
%%%%%%%%%%%%%%%%%%%%%%%%%%%%%%%%%%%%%%%%%%%%%%%%%%%%%%%%%%%%%%%%%%%%%%%%%%%%%%%%
\section{Hyperbolic manifolds}
  \label{s:hyperbolic}

In this section, we assume that $(M,g)$ is an $n$-dimensional convex co-compact hyperbolic
manifold, that is, a quotient $M=\Gamma\backslash\mathbb H^n$ of the hyperbolic space $\mathbb H^n$
by a convex co-compact (geometrically finite) subgroup $\Gamma$ of the isometry group $\PSO(1,n)$ of $\mathbb H^n$.
We refer the reader to~\cite{Borthwick} for the formal definition and properties
of these manifolds in the important special case of dimension $n=2$
and to~\cite{perry} for the case of general dimension.

We will use the calculus of~\S\ref{s:second-microlocalization} to obtain fine
microlocal bounds on the scattering resolvent on $M$, and prove Theorem~\ref{t:fup-reduction}
using these bounds.

%%%%%%%%%%%%%%%%%%%%%%%%%%%%%%%%%%%%%%%%%%%%%%%%%%%%%%%%%%%%%%%%%%%%%%%%%%%%%%%%
\subsection{Dynamical properties}
  \label{s:dynamical}

Define the function $p\in C^\infty(T^*M\setminus 0)$ by
\begin{equation}
  \label{e:p-symbol}
p(x,\xi)=|\xi|_g,\quad
(x,\xi)\in T^*M\setminus 0,
\end{equation}
and let $X$ be the Hamiltonian vector field of $p$. Then
\begin{equation}
  \label{e:g-flow}
e^{tX}:T^*M\setminus 0\to T^*M\setminus 0
\end{equation}
is the homogeneous rescaling of the geodesic flow.
Here homogeneity means that $[X,\xi\cdot\partial_\xi]=0$ where $\xi\cdot\partial_\xi$ is the generator
of dilations.

In what follows, we will identify the cotangent bundle $T^*M$ with the tangent bundle $TM$ using the metric $g$.

%%%%%%%%%%%%%%%%%%%%%%%%%%%%%%%%%%%%%%%%%%%%%%%%%%%%%%%%%%%%%%%%%%%%%%%%%%%%%%%%
\smallsection{Stable/unstable decomposition}
For $(x,\xi)\in T^*M\setminus 0$, we decompose the tangent space at $(x,\xi)$
as follows:
\begin{equation}
  \label{e:sudec}
T_{(x,\xi)}(T^*M)=\mathbb RX\oplus \mathbb R(\xi\cdot\partial_\xi)\oplus E_s(x,\xi)\oplus E_u(x,\xi)
\end{equation}
where $E_s,E_u$ are the $n-1$ dimensional stable and unstable bundles,
defined in the case $|\xi|_g=1$ for instance in~\cite[(3.14)]{rrh}
(recalling the identification $T^*M\simeq TM$),
and in general by requiring that they are homogeneous.
Note that $E_s,E_u$ are the images of the stable/unstable bundles
of $\mathbb H^n$ (which are also denoted $E_s,E_u$) under the
covering map
\begin{equation}
  \label{e:pi-Gamma}
\pi_\Gamma:T^*\mathbb H^n\to T^*M,
\end{equation}
and they are tangent to the level sets of $p$.

The subbundles $E_s,E_u$ are invariant under the flow $e^{tX}$.
Moreover, the projection map $T_{(x,\xi)}T^*M\to T_x M$
is an isomorphism from $E_s(x,\xi)$ onto the space
$\{\eta\in T_xM\mid \langle\xi,\eta\rangle=0\}$. Therefore, 
we can canonically pull back the metric $g_x$ to $E_s(x,\xi)$. Same
is true for $E_u(x,\xi)$, and we have (see for instance~\cite[\S3.3]{rrh})
\begin{equation}
  \label{e:stun}
|de^{tX}(x)v|_g=\begin{cases}
e^t |v|_g,&\quad v\in E_u(x,\xi);\\
e^{-t}|v|_g,&\quad v\in E_s(x,\xi).
\end{cases}
\end{equation}
For each $(x,\xi)\in T^*M\setminus 0$, consider the \emph{weak stable/unstable subspaces}
\begin{equation}
  \label{e:L-s-L-u}
L_s(x,\xi):=\mathbb R X(x,\xi)\oplus E_s(x,\xi),\quad
L_u(x,\xi):=\mathbb R X(x,\xi)\oplus E_u(x,\xi).
\end{equation}
Define the maps
\begin{equation}
  \label{e:B-pm}
B_\pm:T^*\mathbb H^n\setminus 0\to\mathbb S^{n-1}
\end{equation}
as follows: for $(x,\xi)\in T^*\mathbb H^n$,
$B_\pm(x,\xi)$ is the limit of the projection to the ball model of $\mathbb H^n$ of the geodesic $e^{tX}(x,\xi)$
as $t\to\pm\infty$~-- see for instance~\cite[\S3.4]{rrh}.
Then the lifts of $L_s,L_u$ to $T^*\mathbb H^n\setminus 0$ are given by~\cite[(3.25)]{rrh}
\begin{equation}
  \label{e:desker}
\begin{aligned}
\pi_\Gamma^*L_s(x,\xi)&=\ker dB_+(x,\xi)\cap \ker dp(x,\xi),\\
\pi_\Gamma^*L_u(x,\xi)&=\ker dB_-(x,\xi)\cap \ker dp(x,\xi).
\end{aligned}
\end{equation}
%%%%%%%%%%%%%%%%%%%%%%%%%%%%%%%%%%%%%%%%%%%%%%%%%%%%%%%%%%%%%%%%%%%%%%%%%%%%%%%%
\begin{lemm}
  \label{l:our-ls}
$L_s$ and $L_u$ are Lagrangian foliations on $T^*M\setminus 0$ in the sense of
Definition~\ref{d:l-foli}.
\end{lemm}
%%%%%%%%%%%%%%%%%%%%%%%%%%%%%%%%%%%%%%%%%%%%%%%%%%%%%%%%%%%%%%%%%%%%%%%%%%%%%%%%
\begin{proof}
We consider the case of $L_s$; the case of $L_u$ is handled similarly.
Using the covering map $\pi_\Gamma$, we reduce to the case $M=\mathbb H^n$.
The fact that $L_s$ is integrable follows immediately from~\eqref{e:desker}.
Since $\dim L_s=n$, it remains to show that
$\omega(Y_1,Y_2)=0$ for each $Y_1,Y_2\in L_s(x,\xi)$, where $\omega$
is the symplectic form on $T^*\mathbb H^n$. When $Y_1=X(x,\xi)$, this is immediate
since $L_s(x,\xi)\subset\ker dp(x,\xi)$. Therefore,
we may assume that $Y_1,Y_2\in E_s(x,\xi)$. Since $e^{tX}$ is a Hamiltonian
flow, it is a symplectomorphism, and we find
$$
\omega(Y_1,Y_2)=\omega(de^{tX}(x,\xi)Y_1,de^{tX}(x,\xi)Y_2)\leq
C|de^{tX}(x,\xi)Y_1|_g\cdot |de^{tX}(x,\xi)Y_2|_g
$$
where the constant $C$ in the last inequality is independent of $t$
since the isometry group $\PSO(1,n)$ acts transitively on $\mathbb H^n$
and the lifted action on $T^*\mathbb H^n\setminus 0$ preserves
$\omega$, $E_s$, and the induced metric on $E_s$.
Letting $t\to +\infty$ and using~\eqref{e:stun}, we get
$\omega(Y_1,Y_2)=0$ as required.
\end{proof}
%%%%%%%%%%%%%%%%%%%%%%%%%%%%%%%%%%%%%%%%%%%%%%%%%%%%%%%%%%%%%%%%%%%%%%%%%%%%%%%%
The next lemma states that the result of propagating
a compactly supported symbol up to almost twice the Ehrenfest time
lies in the anisotropic class $S^{\comp}_{L,\rho}$ from Definition~\ref{d:symbols},
where $L=L_s$ or $L=L_u$ depending on the direction of propagation.
See Appendix~\ref{s:hyperbolic-technical} for the proof.
%%%%%%%%%%%%%%%%%%%%%%%%%%%%%%%%%%%%%%%%%%%%%%%%%%%%%%%%%%%%%%%%%%%%%%%%%%%%%%%%
\begin{lemm}
  \label{l:propagated-okay}
Let $\chi_1,\chi_2\in C_0^\infty(T^*M\setminus 0)$ be independent of $h$
and fix $\rho\in [0,1)$. Then we have uniformly in $t\in [0,\rho\log(1/h)]$,
$$
\chi_2(\chi_1\circ e^{tX})\in S^{\comp}_{L_s,\rho}(T^*M\setminus 0),\quad
\chi_2(\chi_1\circ e^{-tX})\in S^{\comp}_{L_u,\rho}(T^*M\setminus 0).
$$
\end{lemm}
%%%%%%%%%%%%%%%%%%%%%%%%%%%%%%%%%%%%%%%%%%%%%%%%%%%%%%%%%%%%%%%%%%%%%%%%%%%%%%%%

%%%%%%%%%%%%%%%%%%%%%%%%%%%%%%%%%%%%%%%%%%%%%%%%%%%%%%%%%%%%%%%%%%%%%%%%%%%%%%%%
\smallsection{Infinity and trapping}
Since $M$ is a convex co-compact hyperbolic manifold, it is \emph{even asymptotically hyperbolic}
in the sense of~\cite[Definition~1.2]{guillarmou}; more precisely,
it is the interior of a compact manifold with boundary
$\overline M$ such that near $\partial\overline M$,
$$
g={d\tilde x^2+g_1(\tilde x^2,\tilde y)\over \tilde x^2},
$$
where $\tilde x\geq 0$ is a boundary defining function
and $(\tilde x,\tilde y)\in (0,\varepsilon)\times \partial\overline M$
are some product coordinates
on a collar neighborhood of $\partial\overline M$.

It is shown for example in~\cite[Lemma~7.1]{qeefun} that there exists a function
$r:M\to\mathbb R$ such that $\{r\leq R\}$ is compact for all $R$
and the sets $\{r\leq R\}$ are strictly convex for all $R\geq 0$; that is,
if we restrict $r$ to any geodesic on $M$ and denote by dots derivatives
with respect to the geodesic parameter, then at each point of the geodesic
we have
\begin{equation}
  \label{e:convex}
r\geq 0,\
\dot r=0\quad\Longrightarrow\quad
\ddot r>0.
\end{equation}
In fact, it suffices to take $r:=\tilde x^{-1}-r_1$ for a
boundary defining function $\tilde x$ of $\overline M$ and large enough constant $r_1>0$.

We now define the \emph{incoming/outgoing tails} $\Gamma_\pm$ by
\begin{equation}
  \label{e:GpmDef}
\Gamma_\pm=\{(x,\xi)\in T^*M\setminus 0\mid r(e^{tX}(x,\xi)) \text{ is bounded as }t\to\mp\infty\}.
\end{equation}
Define also the \emph{trapped set} $K=\Gamma_+\cap\Gamma_-$.
It follows from~\eqref{e:convex} that $\Gamma_\pm$ are closed subsets of $T^*M\setminus 0$ and $K\subset \{r<0\}$,
see for instance~\cite[\S4.1]{qeefun}. We assume that $K\neq\emptyset$ (in the case when $K=\emptyset$,
$M$ is known to have an arbitrarily large essential spectral gap, see for
instance~\cite[(1.1)]{vasy2}).

Recall that $M=\Gamma\backslash\mathbb H^n$, where $\Gamma$ is a convex co-compact group of hyperbolic
isometries.
Define the \emph{limit set} $\Lambda_\Gamma\subset \mathbb S^{n-1}$
as follows: for each $x\in \mathbb H^n$,
\begin{equation}
  \label{e:Lambda-Gamma}
\Lambda_\Gamma=\overline{\{\gamma.x\mid \gamma\in\Gamma\}}\cap \mathbb S^{n-1}
\end{equation}
where we use the ball model of the hyperbolic space and the closure is taken
in the closed ball in $\mathbb R^n$. The resulting
set is closed and independent of the choice of $x$; see for instance~\cite{Patterson}
and~\cite[Lemma~2.8]{Borthwick} for the case of $n=2$ and~\cite{Sullivan} for general $n$.

For each $(x,\xi)\in T^*\mathbb H^n\setminus 0$, we have
(see Appendix~\ref{s:hyperbolic-technical} for the proof)
\begin{equation}
  \label{e:Gpm-formula}
\begin{aligned}
\pi_\Gamma(x,\xi)\in\Gamma_+\quad\iff\quad B_-(x,\xi)\in\Lambda_\Gamma,\\
\pi_\Gamma(x,\xi)\in\Gamma_-\quad\iff\quad B_+(x,\xi)\in\Lambda_\Gamma,
\end{aligned}
\end{equation}
where the maps $B_\pm$ are defined in~\eqref{e:B-pm}
and $\pi_\Gamma$ is defined in~\eqref{e:pi-Gamma}.

The following statement, when combined with~\eqref{e:Gpm-formula},
implies that for a trajectory $(x(t),\xi(t))$ of $e^{tX}$ on $T^*M\setminus 0$
which stays in a fixed compact set for all $t\in [0,T]$,
the point $(x(0),\xi(0))$ is $\mathcal O(e^{-T})$ close to $\Gamma_-$
and the point $(x(T),\xi(T))$ is $\mathcal O(e^{-T})$ close to $\Gamma_+$.
See Appendix~\ref{s:hyperbolic-technical} for
the proof.
%%%%%%%%%%%%%%%%%%%%%%%%%%%%%%%%%%%%%%%%%%%%%%%%%%%%%%%%%%%%%%%%%%%%%%%%%%%%%%%%
\begin{lemm}
  \label{l:close-to-trapping}
Let $V\subset T^*\mathbb H^n\setminus 0$ be a compact set. Then there exists a constant
$C$ such that for each $t\geq 0$,
$$
(x,\xi)\in V,\quad
\pi_\Gamma(e^{\pm tX}(x,\xi))\in\pi_\Gamma(V)\quad\Longrightarrow\quad
d(B_\pm(x,\xi),\Lambda_\Gamma)\leq Ce^{-t}.
$$
Here $d(\cdot,\cdot)$ is the Euclidean distance function on $\mathbb S^{n-1}$.
\end{lemm}
%%%%%%%%%%%%%%%%%%%%%%%%%%%%%%%%%%%%%%%%%%%%%%%%%%%%%%%%%%%%%%%%%%%%%%%%%%%%%%%%

%%%%%%%%%%%%%%%%%%%%%%%%%%%%%%%%%%%%%%%%%%%%%%%%%%%%%%%%%%%%%%%%%%%%%%%%%%%%%%%%
%%%%%%%%%%%%%%%%%%%%%%%%%%%%%%%%%%%%%%%%%%%%%%%%%%%%%%%%%%%%%%%%%%%%%%%%%%%%%%%%
\subsection{Scattering resolvent}
  \label{s:scattering-resolvent}

Consider the Laplace--Beltrami operator $\Delta$ on $(M,g)$ and its $L^2$ resolvent
$$
R(\lambda)=\Big(-\Delta-{(n-1)^2\over 4}-\lambda^2\Big)^{-1}:L^2(M)\to H^2(M),\quad
\Im\lambda>0,
$$
which may have finitely many poles corresponding to eigenvalues of $-\Delta$ on the interval
$\big[0,{(n-1)^2\over 4}\big)$.
Then
$R(\lambda)$ continues meromorphically with poles of finite rank as a family
of operators
\begin{equation}
  \label{e:R-lambda}
R(\lambda):L^2_{\comp}(M)\to H^2_{\loc}(M),\quad
\lambda\in\mathbb C.
\end{equation}
A related question of continuation of Eisenstein series was studied by
Patterson~\cite{patterson1,patterson2} in dimension 2 and
Perry~\cite{perry,perry2} in higher dimensions. The continuation of $R(\lambda)$ was established by
Mazzeo--Melrose~\cite{mazzeo-melrose} and Guillarmou~\cite{guillarmou} for general
(even) asymptotically hyperbolic manifolds, and by
Guillop\'e--Zworski~\cite{guillope-zworski} for manifolds of constant curvature near infinity.
We refer the reader to~\cite[Chapter~6]{Borthwick} for the proof in dimension 2
and an overview of the history of the subject.

To study essential spectral gaps, we write
\begin{equation}
  \label{e:nu-def}
\lambda=h^{-1}-i\nu,\quad
\nu\in [-1,\beta-\varepsilon],
\end{equation}
where the semiclassical parameter $h>0$ needs to be small enough for the argument to work.
(Resolvent bounds for negative $\Re\lambda$ follow from bounds
for positive $\Re\lambda$ since $R(\lambda)^*=R(-\bar\lambda)$.)
We introduce the semiclassical resolvent
$$
R_h(\omega):=h^{-2}R(\lambda),\quad
\omega:=h\lambda=1-ih\nu.
$$
To derive high frequency estimates near infinity,
we use the construction of the meromorphic continuation
of the resolvent due to Vasy~\cite{vasy1,vasy2}.
See in particular~\cite[\S5.1]{vasy2}
and also~\cite[Lemma~2.1]{fwl}, \cite[\S4.4]{nhp}.
The book \cite[Chapter~5]{dizzy} provides a detailed account of a slightly
modified version of Vasy's method, which could be used
in the present paper,
and~\cite{vfd} gives a short self-contained introduction in the nonsemiclassical case.
(For the constant curvature case considered here, one could alternatively apply
complex scaling, see
Zworski~\cite{zworski-inventiones} and Datchev~\cite{trumpet-of-death}.)
Specifically, we write
$$
R_h(\omega)f=\psi_1 (\mathcal P_h(\omega)^{-1}\psi_2 f)|_M,\quad
f\in C_0^\infty(M).
$$
Here $\psi_j\in C^\infty(M)$, $j=1,2$, are certain nonvanishing functions depending on $\omega,h$
and $\mathcal P_h(\omega)\in\Psi^2_h(\Mext)$ is a certain family of semiclassical pseudodifferential
operators on a compact manifold $\Mext$ containing $M$ as an open subset; we have
\begin{equation}
  \label{e:p-h-def}
(\mathcal P_h(\omega)u)|_M=\psi_2\Big(-h^2\Delta-{h^2(n-1)^2\over 4}-\omega^2\Big)\psi_1 (u|_M),\quad
u\in C^\infty(\Mext).
\end{equation}
Moreover, $\mathcal P_h(\omega)$ is a Fredholm operator
between the spaces
$$
\{u\in H^s_h(\Mext)\mid \mathcal P_h(\omega)u\in H^{s-1}_h(\Mext)\}\ \to\ H^{s-1}_h(\Mext)
$$
provided that $s>0$ is large enough depending on $\beta$; the inverse
$\mathcal P_h(\omega)^{-1}:H^{s-1}_h(\Mext)\to H^s_h(\Mext)$ is meromorphic in $\omega$ with poles
of finite rank (if we treat $\omega$ and $h$ as independent parameters).

For each fixed $r_0>0$ we may arrange so that $\psi_1=\psi_2=1$ on $\{r\leq r_0\}$,
see for instance the paragraph preceding~\cite[(3.14)]{vasy2}. Therefore, to
show the resolvent bound~\eqref{e:essential-gap2}
it suffices to prove the estimate
\begin{equation}
  \label{e:est0}
\|u\|_{H^s_h(\Mext)}\leq Ch^{-1-2\max(0,\nu)-\varepsilon}\|f\|_{H^{s-1}_h(\Mext)}
\end{equation}
when $h$ is small enough depending on $\varepsilon$ and
\begin{equation}
  \label{e:eq0}
\mathcal P_h(\omega)u=f,\quad
u\in H^s_h(\Mext),\quad
f\in H^{s-1}_h(\Mext).
\end{equation}
(The resulting $H^{s-1}_h\to H^s_h$ estimate on $\chi R_h(\omega)\chi$
can be converted to an $L^2\to L^2$ estimate using the elliptic parametrix
of $-h^2\Delta-{h^2(n-1)^2\over 4}-\omega^2$ near the fiber infinity,
see for instance~\cite[Proposition~3.3]{nhp}.)

We use the following outgoing estimates on the operator $\mathcal P_h(\omega)$.
Their meaning is as follows: since $R_h(\omega)$ is the \emph{outgoing} resolvent
(in the sense that it maps compactly supported functions on $M$
to functions with outgoing behavior at the infinity of $M$),
it should be \emph{semiclassically outgoing}, that is propagate
singularities in the forward direction along the geodesic flow. In particular
if $\tilde u=R_h(\omega)\tilde f$ and $\tilde f=\mathcal O(h^\infty)$,
then $\WFh(\tilde u)$ is contained in the outgoing tail $\Gamma_+$
(as follows from~\eqref{e:pest-1} below). Moreover, if we control
$u$ near the trapped set then we can bound its norm everywhere
(as follows from~\eqref{e:pest-2} below).
%%%%%%%%%%%%%%%%%%%%%%%%%%%%%%%%%%%%%%%%%%%%%%%%%%%%%%%%%%%%%%%%%%%%%%%%%%%%%%%%
\begin{lemm}
  \label{l:outgoing}
For each $u,f$ satisfying~\eqref{e:eq0}, we have the following estimates:

1. Assume that $A_1\in\Psi^0_h(\Mext)$, $\WFh(A_1)\subset \{r\leq r_0\}\subset\overline T^*M$, and
\begin{equation}
  \label{e:pest-1-cond}
\WFh(A_1)\cap\Gamma_+\cap \{|\xi|_g=1\}=\emptyset.
\end{equation}
Then
\begin{equation}
  \label{e:pest-1}
\|A_1u\|_{H^s_h(\Mext)}\leq Ch^{-1}\|f\|_{H^{s-1}_h(\Mext)}+\mathcal O(h^\infty)\|u\|_{H^s_h(\Mext)}.
\end{equation}

2. Assume that $A_2\in\Psi^{\comp}_h(\Mext)$ is elliptic on $K\cap \{|\xi|_g=1\}$. Then
\begin{equation}
  \label{e:pest-2}
\|u\|_{H^s_h(\Mext)}\leq C\|A_2u\|_{L^2}+Ch^{-1}\|f\|_{H^{s-1}_h(\Mext)}.
\end{equation}
\end{lemm}
%%%%%%%%%%%%%%%%%%%%%%%%%%%%%%%%%%%%%%%%%%%%%%%%%%%%%%%%%%%%%%%%%%%%%%%%%%%%%%%%
\noindent\textbf{Remark.} The estimates~\eqref{e:pest-1}, \eqref{e:pest-2}
make it possible to treat the infinite ends of our manifold as a black box;
see~\cite[\S4]{nhp} for a more formal treatment. In particular, our results
would apply to any manifold with the same trapping structure as a convex co-compact
hyperbolic quotient and infinite ends which satisfy~\eqref{e:pest-1}, \eqref{e:pest-2};
this includes Euclidean ends~\cite[\S4.3]{nhp}
and general even asymptotically hyperbolic ends~\cite[\S4.4]{nhp}.
%%%%%%%%%%%%%%%%%%%%%%%%%%%%%%%%%%%%%%%%%%%%%%%%%%%%%%%%%%%%%%%%%%%%%%%%%%%%%%%%
\begin{proof}[Sketch of proof]
Both of these statements follow from the elliptic estimate~\cite[Proposition~3.2]{nhp}, propagation of singularities~\cite[Proposition~3.4]{nhp},
and radial points estimates~\cite[Propositions~2.10 and~2.11]{vasy1} applied to the dynamical picture
of the Hamiltonian flow of the principal symbol of $\mathcal P_h(\omega)$
as studied in~\cite{vasy1,vasy2}. More precisely, condition~\eqref{e:pest-1-cond}
guarantees that each point in $\WFh(A_1)$ either lies in the elliptic set of $\mathcal P_h(\omega)$
or the corresponding backwards Hamiltonian flow line converges to the radial sets,
near which $u$ is controlled when $s$ is large enough depending on $\beta$;
this yields~\eqref{e:pest-1}. Next, each backwards
Hamiltonian flow line of $\mathcal P_h(\omega)$ either passes through its elliptic set,
or converges to the radial sets, or passes through the elliptic set of $A_2$;
this yields~\eqref{e:pest-2}.

The proof (in a modified setting using domains with boundary,
which however works equally well for our purposes) is described in detail in~\cite[\S6.2.3]{dizzy}.
We also refer the reader to~\cite[Lemma~4.4]{fwl} and~\cite[Lemma~4.1]{nhp}
for more details on the dynamics of the flow
and to~\cite[Lemmas~4.4 and~4.6]{nhp} for slightly different proofs involving
a semiclassically outgoing parametrix for the resolvent.
\end{proof}
%%%%%%%%%%%%%%%%%%%%%%%%%%%%%%%%%%%%%%%%%%%%%%%%%%%%%%%%%%%%%%%%%%%%%%%%%%%%%%%%
Finally, we write a pseudodifferential equation (see~\eqref{e:eq1} below) which is a direct consequence of~\eqref{e:eq0}
but more convenient for Lemma~\ref{l:general-propagation} below
because the principal symbol of the associated operator
is the function $p$ given by~\eqref{e:p-symbol}.
Consider the set
\begin{equation}
  \label{e:W-0}
W_0:=\{r\leq r_0,\ |\xi|_g\in [1/2,2]\}\subset T^*M.
\end{equation}
Take $P\in\Psi^{\comp}_h(M)$ such that $P^*=P$ and
\begin{equation}
  \label{e:the-P}
\begin{gathered}
P^2=-h^2\Delta+{h^2(n-1)^2\over 4}+\mathcal O(h^\infty)\quad\text{microlocally near }
W_0,\\
\sigma_h(P)(x,\xi)=p(x,\xi)=|\xi|_g\quad\text{near }W_0.
\end{gathered}
\end{equation}
We can construct such an operator following~\cite[Lemma~4.6]{grigis-sjostrand}: first take
$P_0\in\Psi^{\comp}_h(M)$ such that $P_0^*=P_0$ and $\sigma_h(P_0)=p$ near $W_0$. Denote
$\mathbf P:=-h^2\Delta+{h^2(n-1)^2\over 4}$, then
$\sigma_h(P_0^2)=\sigma_h(\mathbf P)$ near $W_0$ and thus
$$
\mathbf P=P_0^2+hR_0+\mathcal O(h^\infty)\quad\text{microlocally near }W_0
$$
for some $R_0\in\Psi^{\comp}_h(M)$ and $R_0^*=R_0$. We next construct $P_1\in\Psi^{\comp}_h(M)$
such that $P_1^*=P_1$ and 
$$
\mathbf P=(P_0+hP_1)^2+h^2R_1+\mathcal O(h^\infty)\quad\text{microlocally near }W_0
$$
for some $R_1\in\Psi^{\comp}_h(M)$ and $R_1^*=R_1$; to do that, it suffices to put
$\sigma_h(P_1)=\sigma_h(R_0)/2p$ near $W_0$. Arguing by induction, we construct a family
of operators $P_j\in \Psi^{\comp}_h(M)$ such that $P_j^*=P_j$ and
$$
\mathbf P=(P_0+hP_1+\dots+h^m P_m)^2+\mathcal O(h^{m+1})\quad\text{microlocally near }W_0;
$$
it remains to take as $P$ the asymptotic sum $P\sim\sum_{j=0}^\infty h^jP_j$.

By~\eqref{e:p-h-def}, and since $\psi_1=\psi_2=1$ near $\{r\leq r_0\}$, we have
$$
\|B(P^2-\omega^2)u\|_{L^2}\leq C\|f\|_{H^{s-1}_h(\Mext)}+\mathcal O(h^\infty)\|u\|_{H^s_h(\Mext)}
$$
for $u,f$ satisfying~\eqref{e:eq0} and
each $B\in\Psi^{\comp}_h(M)$ such that $\WF_h(B)$ is contained in some small neighborhood
of $W_0$. We write $(P^2-\omega^2)=(P+\omega)(P-\omega)$ and note
that $P+\omega$ is elliptic on $W_0$; therefore, the elliptic estimate~\cite[Proposition~3.2]{nhp}
gives
\begin{equation}
  \label{e:eq1}
\|A(P-\omega)u\|_{L^2}\leq C\|f\|_{H^{s-1}_h(\Mext)}+\mathcal O(h^\infty)\|u\|_{H^s_h(\Mext)}
\end{equation}
for each $A\in\Psi^{\comp}_h(M)$ such that $\WFh(A)\subset W_0$.

%%%%%%%%%%%%%%%%%%%%%%%%%%%%%%%%%%%%%%%%%%%%%%%%%%%%%%%%%%%%%%%%%%%%%%%%%%%%%%%%
\subsection{Second microlocalization of the resolvent}

We now take the first step towards proving a spectral gap, which is to use the
calculus of~\S\ref{s:second-microlocalization} and the Lagrangian
foliations $L_u,L_s$ of~\eqref{e:L-s-L-u} to obtain fine microlocal estimates on
solutions to~\eqref{e:eq0}.
We start with a general propagation estimate:
%%%%%%%%%%%%%%%%%%%%%%%%%%%%%%%%%%%%%%%%%%%%%%%%%%%%%%%%%%%%%%%%%%%%%%%%%%%%%%%%
\begin{lemm}
  \label{l:general-propagation}
Let $a,b\in S^{\comp}_{L,\rho}(T^*M\setminus 0)$ where $L\in \{L_u,L_s\}$,
$\rho\in[0,1)$, and fix $T>0$. Assume that $|a|\leq 1$ everywhere and
$$
e^{-TX}(\supp a)\ \subset\ \{b=1\};\quad
e^{-tX}(\supp a)\ \subset\ W_0,\quad t\in [0,T],
$$
where $W_0\subset T^*M\setminus 0$ is defined in~\eqref{e:W-0}. Then
for each $\varepsilon_0>0$ and each $u,f$ satisfying~\eqref{e:eq0} we have
$$
\|\Op_h^L(a) u\|_{L^2}\leq (e^{\nu T}+\varepsilon_0)\|\Op_h^L(b)u\|_{L^2}+Ch^{-1}\|f\|_{H^s_h(\Mext)}+\mathcal O(h^\infty)\|u\|_{H^s_h(\Mext)}.
$$
where $\nu$ is defined in~\eqref{e:nu-def} and $\Op_h^L$ is a quantization procedure described in~\eqref{e:op-h-l}.
\end{lemm}
%%%%%%%%%%%%%%%%%%%%%%%%%%%%%%%%%%%%%%%%%%%%%%%%%%%%%%%%%%%%%%%%%%%%%%%%%%%%%%%%
\begin{proof}
Let $P\in\Psi^{\comp}_h(M)$ be the operator defined in~\eqref{e:the-P}.
Consider the family of operators $A_t\in\Psi^{\comp}_{h,L,\rho}(T^*M\setminus 0)$, $t\in [0,T]$, constructed in
Lemma~\ref{l:egorov}, with $A_0=\Op_h^L(a)+\mathcal O(h^\infty)$;
here~\eqref{e:egorov-condition} holds since $\sigma_h(P)=p$ near $W_0$ and
$L_u,L_s\subset\ker dp$.

Using~\eqref{e:egorov-equation}, \eqref{e:eq1}, and the fact that $P^*=P$, we write
$$
\begin{aligned}
\partial_t\|A_t u\|_{L^2}^2&=2\Re\langle \partial_t A_t u,A_t u\rangle\\
&=-{2\over h}\Im\langle [P,A_t] u,A_t u\rangle+\mathcal O(h^\infty)\|u\|_{H^s_h(\Mext)}\cdot \|A_t u\|_{L^2}\\
&={2\over h}\Im\langle A_t Pu, A_t u\rangle+\mathcal O(h^\infty)\|u\|_{H^s_h(\Mext)}\cdot \|A_t u\|_{L^2}\\
&=-2\nu\|A_t u\|^2_{L^2}+(Ch^{-1}\|f\|_{H^{s-1}_h(\Mext)}+\mathcal O(h^\infty)\|u\|_{H^s_h(\Mext)})\|A_tu\|_{L^2}.
\end{aligned}
$$
Integrating this, we get
\begin{equation}
  \label{e:prop1}
\|\Op_h^L(a)u\|_{L^2}\leq e^{\nu T}\|A_Tu\|_{L^2}+Ch^{-1}\|f\|_{H^{s-1}_h(\Mext)}+\mathcal O(h^\infty)\|u\|_{H^s_h(\Mext)}.
\end{equation}
Now, it follows from part~4 of Lemma~\ref{l:globalprop} and Lemma~\ref{l:egorov} that
for each sequences $h_j\to 0$ and $(x_j,\xi_j)\in T^*M\setminus 0$ such that
$e^{TX}(x_j,\xi_j)\notin\supp a(\bullet;h_j)$, the operator $A_T$ is $\mathcal O(h^\infty)$
microlocally along $(x_j,\xi_j,h_j)$ in the sense of Definition~\ref{d:lag-prince}.
We then apply Lemma~\ref{l:lag-elliptic} to write
$$
A_T=Q\Op_h^L(b)+\mathcal O(h^\infty)_{\mathcal D'\to C_0^\infty},\quad
Q\in \Psi^{\comp}_{h,L,\rho}(T^*M\setminus 0).
$$
Moreover, Lemma~\ref{l:egorov} and the proof of Lemma~\ref{l:lag-elliptic} give
$$
\sigma_h^L(Q)=(a\circ e^{TX})/b+\mathcal O(h^{1-\rho})_{S^{\comp}_{L,\rho}(T^*M\setminus 0)}
=a\circ e^{TX}+\mathcal O(h^{1-\rho})_{S^{\comp}_{L,\rho}(T^*M\setminus 0)}.
$$
By Lemma~\ref{l:l2-improved}, we have $\|Q\|_{L^2\to L^2}\leq 1+\varepsilon_1$ for each
$\varepsilon_1>0$ and $h$ small enough depending on $\varepsilon_1$. Therefore,
$$
e^{\nu T}\|A_T u\|_{L^2}\leq (e^{\nu T}+\varepsilon_0)\|\Op_h^L(b)u\|_{L^2}
+\mathcal O(h^\infty)\|u\|_{H^s_h(\Mext)}
$$
which together with~\eqref{e:prop1} finishes the proof.
\end{proof}
%%%%%%%%%%%%%%%%%%%%%%%%%%%%%%%%%%%%%%%%%%%%%%%%%%%%%%%%%%%%%%%%%%%%%%%%%%%%%%%%
We can now prove second microlocal estimates on solutions to~\eqref{e:eq0}.
Roughly speaking, in the case $f=0$ and $t=\rho\log(1/h)$ the estimate~\eqref{e:second1} below states that
 $u$ is concentrated $h^\rho$ close to $\Gamma_+$
(for each $\rho<1$) and the estimate~\eqref{e:second2} states that
the $u$ has to be of size at least $h^{\nu\rho+}$
in an $h^\rho$ neighborhood of $\Gamma_-$~-- see~\eqref{e:g++1} and~\eqref{e:g++2}.
In~\S\ref{s:fun}, we will
see that the combination of these two facts
with the fractal uncertainty principle
implies that $\nu$ cannot be too small, giving an essential spectral gap.
%%%%%%%%%%%%%%%%%%%%%%%%%%%%%%%%%%%%%%%%%%%%%%%%%%%%%%%%%%%%%%%%%%%%%%%%%%%%%%%%
\begin{lemm}
  \label{l:second}
Let $\chi\in C_0^\infty(T^*M\setminus 0;[0,1])$ be equal to 1 near
$K\cap \{|\xi|_g=1\}$. Fix $\rho\in [0,1)$. Then there exists
$T>0$ such that we have for each $\varepsilon_0>0$, uniformly in
$t\in [T,\rho\log(1/h)]$, and $u,f$ satisfying~\eqref{e:eq0}
\begin{gather}
  \label{e:second1}
\big\|\Op_h^{L_u}\big(\chi(1-\chi\circ e^{-tX})\big)u\big\|_{L^2}\ \leq\ Ch^{-1}e^{(\max(0,\nu)+\varepsilon_0)t}\|f\|_{H^{s-1}_h}
+\mathcal O(h^\infty)\|u\|_{H^s_h},\\
  \label{e:second2}
\|u\|_{H^s_h}\ \leq\ Ce^{(\nu+\varepsilon_0)t}\big\|\Op_h^{L_s}\big(\chi(\chi\circ e^{tX})\big)u\big\|_{L^2}+Ch^{-1}e^{(\max(0,\nu)+\varepsilon_0)t}\|f\|_{H^{s-1}_h}.
\end{gather}
Here $\chi(1-\chi\circ e^{-tX})\in S^{\comp}_{L_u,\rho}(T^*M\setminus 0)$,
$\chi(\chi\circ e^{tX})\in S^{\comp}_{L_s,\rho}(T^*M\setminus 0)$ by Lemma~\ref{l:propagated-okay}.
\end{lemm}
%%%%%%%%%%%%%%%%%%%%%%%%%%%%%%%%%%%%%%%%%%%%%%%%%%%%%%%%%%%%%%%%%%%%%%%%%%%%%%%%
\begin{proof}
Denote
$$
\begin{aligned}
F_+(t)&:=\big\|\Op_h^{L_u}\big(\chi(1-\chi\circ e^{-tX})\big)u\big\|_{L^2},\\
F_-(t)&:=\big\|\Op_h^{L_s}\big(\chi(\chi\circ e^{tX})\big)u\big\|_{L^2},
\end{aligned}
$$
then it suffices to show that for each $\varepsilon_0>0$ there exists $T>0$ such that for all $t_0\in [T/2,T]$ and
$t\in [0,\rho\log(1/h)]$, we have (with constants uniform in $t_0,t$)
\begin{gather}
  \label{e:second1.1}
F_+(t+t_0)\ \leq\   e^{(\nu+\varepsilon_0)t_0} F_+(t)+Ch^{-1}\|f\|_{H^{s-1}_h}+\mathcal O(h^\infty)\|u\|_{H^s_h},\\
  \label{e:second1.2}
F_-(t)\ \leq\ e^{(\nu+\varepsilon_0)t_0} F_-(t+t_0)+Ch^{-1}\|f\|_{H^{s-1}_h}+\mathcal O(h^\infty)\|u\|_{H^s_h}.
\end{gather}
Indeed, iterating these estimates we get for all $t\in [T,\rho\log(1/h)]$
\begin{align}
  \label{e:second2.1}
F_+(t)\ &\leq\ e^{(\nu+\varepsilon_0)t}F_+(T/2)+Ch^{-1}e^{(\max(0,\nu)+\varepsilon_0)t}\|f\|_{H^{s-1}_h}+\mathcal O(h^\infty)\|u\|_{H^s_h},\\
  \label{e:second2.2}
F_-(0)\ &\leq\ e^{(\nu+\varepsilon_0)t}F_-(t)+Ch^{-1}e^{(\max(0,\nu)+\varepsilon_0)t}\|f\|_{H^{s-1}_h}+\mathcal O(h^\infty)\|u\|_{H^s_h}.
\end{align}
By~\eqref{e:dynamo-1} below, the wavefront set of $\Op_h^{L_u}\big(\chi(1-\chi\circ e^{-TX/2})\big)\in \Psi^{\comp}_h(M)$ 
does not intersect $\Gamma_+\cap \{|\xi|_g=1\}$. By~\eqref{e:pest-1} (where $r_0$ is chosen large enough depending on $\chi$)
we see that
$$
F_+(T/2)\leq Ch^{-1}\|f\|_{H^{s-1}_h}+\mathcal O(h^\infty)\|u\|_{H^s_h}
$$
and~\eqref{e:second1} follows from here and~\eqref{e:second2.1}.

Next, $\Op_h^{L_s}(\chi^2)\in\Psi^{\comp}_h(M)$ is elliptic on $K\cap \{|\xi|_g=1\}$. By~\eqref{e:pest-2}
we get
$$
\|u\|_{H^s_h}\leq CF_-(0)+Ch^{-1}\|f\|_{H^{s-1}_h}
$$
and~\eqref{e:second2} follows from here and~\eqref{e:second2.2}.

%%%%%%%%%%%%%%%%%%%%%%%%%%%%%%%%%%%%%%%%%%%%%%%%%%%%%%%%%%%%%%%%%%%%%%%%%%%%%%%%
\begin{figure}
\includegraphics{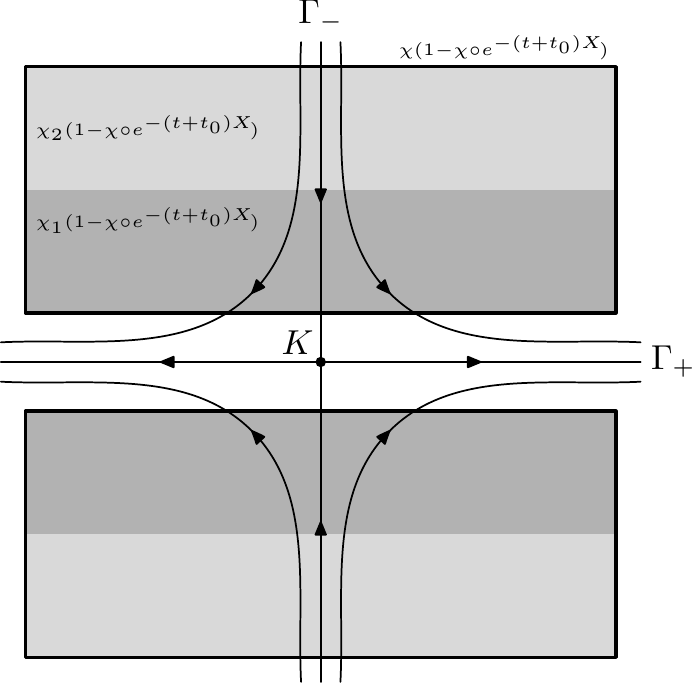}
\qquad
\includegraphics{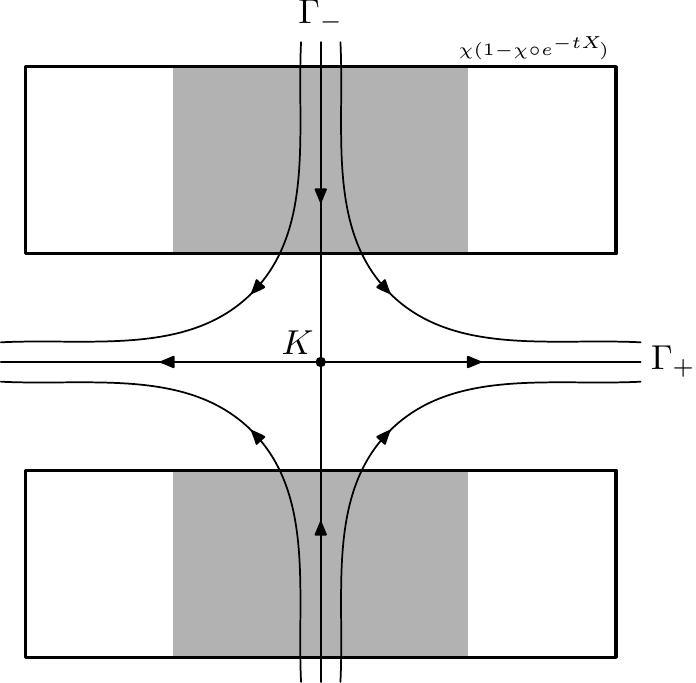}
\caption{An illustration of the proof of~\eqref{e:second1.1}. The function
$\chi(1-\chi\circ e^{-(t+t_0)X})$ is split into two parts.
The part corresponding to $\chi_2$ is estimated by~\eqref{e:pest-1}
and the darker shaded part corresponding to $\chi_1$ is transported backwards by the flow
to the right half of the figure, where it is covered
by $\chi(1-\chi\circ e^{-tX})$.}
\label{f:prop-1}
\end{figure}
%%%%%%%%%%%%%%%%%%%%%%%%%%%%%%%%%%%%%%%%%%%%%%%%%%%%%%%%%%%%%%%%%%%%%%%%%%%%%%%%
We now prove~\eqref{e:second1.1}. We put $T:=NT_0$, where $N$ is a large constant to be chosen
later and for each $(x,\xi)\in \{|\xi|_g=1\}$ and each $t,t_1,t_2\geq T_0>0$ we have
\begin{align}
  \label{e:dynamo-1}
(x,\xi)\in\Gamma_+\cap \supp\chi\ &\Longrightarrow\ e^{-tX}(x,\xi)\notin\supp (1-\chi),\\
  \label{e:dynamo-2}
(x,\xi)\in e^{t_1X}(\supp\chi)\cap e^{-t_2X}(\supp \chi)\ &\Longrightarrow\ (x,\xi)\notin\supp (1-\chi).
\end{align}
The existence of such $T_0$ follows from~\cite[Lemmas~2.3 and~2.4]{rnc} and the fact
that $\chi=1$ near $K\cap \{|\xi|_g=1\}$.

We write $\chi=\chi_1+\chi_2$ where $\chi_j\in C_0^\infty(T^*M\setminus 0;[0,1])$, $\supp\chi_2\cap \Gamma_+\cap \{|\xi|_g=1\}=\emptyset$,
and for each $t\in [T_0,T+3T_0]$, $t_1,t_2\geq T_0$, and $(x,\xi)\in T^*M\setminus 0$
\begin{align}
  \label{e:dynamo2-1}
(x,\xi)\in\supp\chi_1\ &\Longrightarrow\ e^{-tX}(x,\xi)\notin\supp (1-\chi),\\
  \label{e:dynamo2-2}
(x,\xi)\in e^{t_1X}(\supp\chi)\cap e^{-t_2X}(\supp \chi_1)\ &\Longrightarrow\ (x,\xi)\notin\supp (1-\chi),\\
  \label{e:dynamo2-3}
(x,\xi)\in e^{t_1X}(\supp\chi_1)\cap e^{-t_2X}(\supp \chi)\ &\Longrightarrow\ (x,\xi)\notin\supp (1-\chi).
\end{align}
Note that~\eqref{e:dynamo2-2} and~\eqref{e:dynamo2-3} follow immediately from~\eqref{e:dynamo-2}
as long as $\supp\chi_1\subset\supp\chi$.

Take $\chi'_2\in C_0^\infty(T^*M)$ such that $\chi'_2=1$ near $\supp\chi_2$ and
$\supp\chi'_2\cap \Gamma_+\cap \{|\xi|_g=1\}=\emptyset$.
By Lemma~\ref{l:lag-elliptic} and~\eqref{e:pest-1} we have
\begin{equation}
  \label{e:chi-2-est}
\begin{gathered}
\big\|\Op_h^{L_u}\big(\chi_2(1-\chi\circ e^{-(t+t_0)X})\big)u\|_{L^2}
\leq C\|\Op_h^{L_u}(\chi'_2)u\|_{L^2}+\mathcal O(h^\infty)\|u\|_{H^s_h}
\\\leq Ch^{-1}\|f\|_{H^{s-1}_h}+\mathcal O(h^\infty)\|u\|_{H^s_h}.
\end{gathered}
\end{equation}
Next, we have (see Figure~\ref{f:prop-1})
\begin{equation}
  \label{e:prop-containment-1}
e^{-(t_0+T_0)X}\big(\supp(\chi_1(1-\chi\circ e^{-(t+t_0)X}))\big)\ \subset\ \{\chi(1-\chi\circ e^{-tX})=1\}.
\end{equation}
Indeed, let $(x,\xi)\in\supp(\chi_1(1-\chi\circ e^{-(t+t_0)X}))$. Since $t_0+T_0\in [T_0,T+T_0]$,
by~\eqref{e:dynamo2-1} we have $\chi(e^{-(t_0+T_0)X}(x,\xi))=1$. It remains
to show that $\chi(e^{-(t+t_0+T_0)X}(x,\xi))=0$. This follows from~\eqref{e:dynamo2-2} applied to
$e^{-(t+t_0)X}(x,\xi)\in\supp (1-\chi)$, $t_1=T_0$, $t_2=t+t_0$.

We now apply Lemma~\ref{l:general-propagation} to~\eqref{e:prop-containment-1}
(where we choose $r_0$ large enough depending on $\chi$ and $T$
and make $\supp\chi_1\subset \{|\xi_g|\in [1/2,2]\}$) and get for each fixed $\varepsilon_1>0$,
$$
\big\|\Op_h^{L_u}\big(\chi_1(1-\chi\circ e^{-(t+t_0)X})\big)u\big\|_{L^2}
\leq (e^{\nu (t_0+T_0)}+\varepsilon_1)F_+(t)+Ch^{-1}\|f\|_{H^{s-1}_h}+\mathcal O(h^\infty)\|u\|_{H^s_h}.
$$
Together with~\eqref{e:chi-2-est} this implies~\eqref{e:second1.1} as long as we have
\begin{equation}
  \label{e:rookie}
e^{\nu(t_0+T_0)}+\varepsilon_1\leq e^{(\nu+\varepsilon_0)t_0}.
\end{equation}
By choosing $\varepsilon_1$ small enough, this reduces to $\nu T_0<\varepsilon_0t_0$,
which follows from the fact that $t_0\geq T/2=NT_0/2$ if we choose $N$ large enough depending on $\varepsilon_0,\beta$.

%%%%%%%%%%%%%%%%%%%%%%%%%%%%%%%%%%%%%%%%%%%%%%%%%%%%%%%%%%%%%%%%%%%%%%%%%%%%%%%%
\begin{figure}
\includegraphics{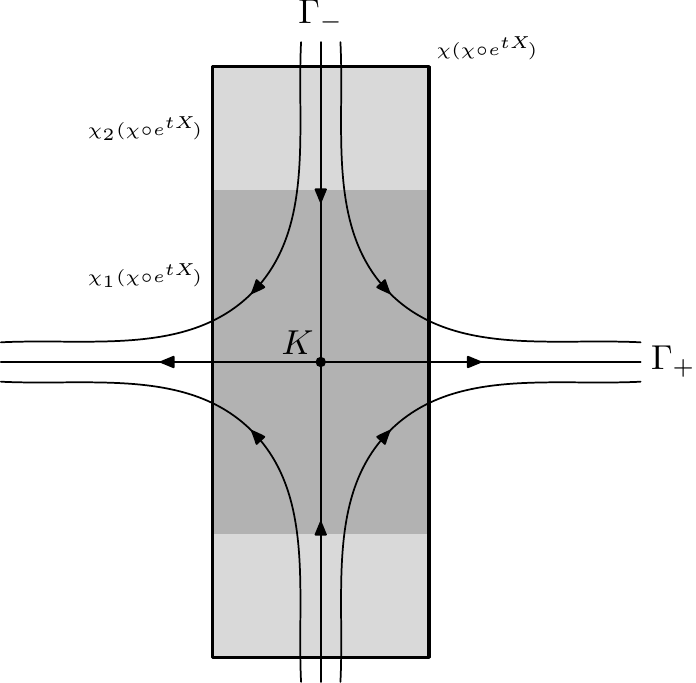}
\qquad
\includegraphics{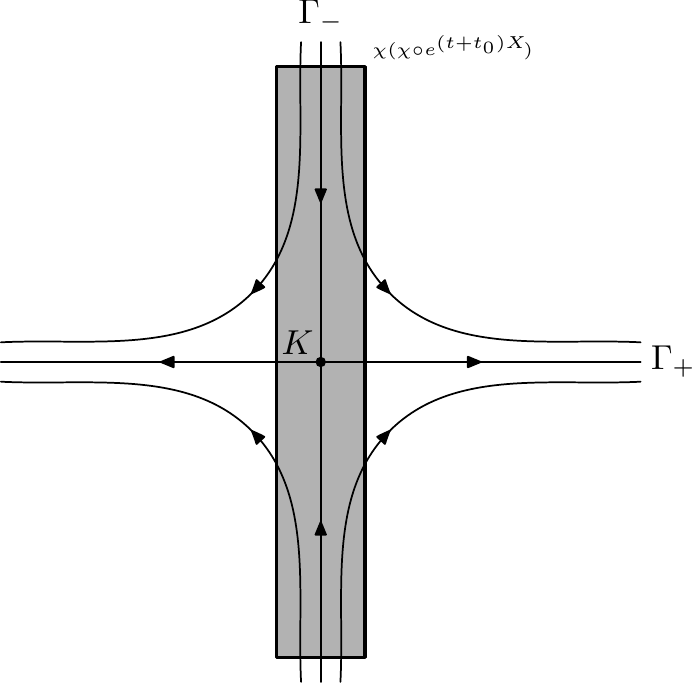}
\caption{An illustration of the proof of~\eqref{e:second1.2}. The function
$\chi(\chi\circ e^{tX})$ is split into two parts.
The part corresponding to $\chi_2$ is estimated by~\eqref{e:pest-1}
and the darker shaded part corresponding to $\chi_1$ is transported backwards by the flow
to the right half of the figure, where it is covered
by $\chi(\chi\circ e^{(t+t_0)X})$.}
\label{f:prop-2}
\end{figure}
%%%%%%%%%%%%%%%%%%%%%%%%%%%%%%%%%%%%%%%%%%%%%%%%%%%%%%%%%%%%%%%%%%%%%%%%%%%%%%%%

To show~\eqref{e:second1.2}, we first note that similarly to~\eqref{e:chi-2-est},
\begin{equation}
  \label{e:chi-2-est2}
\big\|\Op_h^{L_s}\big(\chi_2(\chi\circ e^{tX})\big)u\big\|_{L^2}
\leq Ch^{-1}\|f\|_{H^{s-1}_h}+\mathcal O(h^\infty)\|u\|_{H^s_h}.
\end{equation}
Next, there exists $T_1\in [T_0,3T_0]$ such that (see Figure~\ref{f:prop-2})
\begin{equation}
  \label{e:prop-containment-2}
e^{-(t_0+T_1)X}\big(\supp(\chi_1(\chi\circ e^{tX}))\big)\ \subset\ 
\{\chi(\chi\circ e^{(t+t_0)X}) =1\}.
\end{equation}
Indeed, let $(x,\xi)\in \supp(\chi_1(\chi\circ e^{tX}))$. By~\eqref{e:dynamo2-1}, we have
$\chi(e^{-(t_0+T_1)X}(x,\xi))=1$. It remains to show that $\chi(e^{(t-T_1)X}(x,\xi))=1$.
If $t\leq 2T_0$, then we put $T_1:=t+T_0$ and use~\eqref{e:dynamo2-1}.
If $t\geq 2T_0$, then we put $T_1:=T_0$ and apply~\eqref{e:dynamo2-3}
to $e^{(t-T_0)X}(x,\xi)$, $t_1=t-T_0$, $t_2=T_0$.

Applying Lemma~\ref{l:general-propagation} to~\eqref{e:prop-containment-2} we get for each fixed
$\varepsilon_1>0$,
$$
\big\|\Op_h^{L_s}\big(\chi_1(\chi\circ e^{tX})\big)\|_{L^2}
\leq (e^{\nu(t_0+T_1)}+\varepsilon_1)F_-(t+t_0)+Ch^{-1}\|f\|_{H^{s-1}_h}+\mathcal O(h^\infty)\|u\|_{H^s_h}.
$$
Together with~\eqref{e:chi-2-est2} this implies~\eqref{e:second1.2} as long as we have
$$
e^{\nu(t_0+T_1)}+\varepsilon_1\leq e^{(\nu+\varepsilon_0)t_0}
$$
which is achieved by taking $N$ large enough similarly to~\eqref{e:rookie}.
\end{proof}
%%%%%%%%%%%%%%%%%%%%%%%%%%%%%%%%%%%%%%%%%%%%%%%%%%%%%%%%%%%%%%%%%%%%%%%%%%%%%%%%

%%%%%%%%%%%%%%%%%%%%%%%%%%%%%%%%%%%%%%%%%%%%%%%%%%%%%%%%%%%%%%%%%%%%%%%%%%%%%%%%
\subsection{Reduction to a fractal uncertainty principle}
  \label{s:fun}

In this section, we prove Theorem~\ref{t:fup-reduction}.
We start by constructing symplectomorphisms
\begin{equation}
  \label{e:kappa-M-domains}
\varkappa^\pm:T^*\mathbb H^n\setminus 0\to T^*(\mathbb R^+_w\times\mathbb S^{n-1}_y)
\end{equation}
which map the weak stable/unstable Lagrangian foliations $L_s,L_u$ defined in~\eqref{e:L-s-L-u} to
the vertical foliation on $T^*(\mathbb R^+\times\mathbb S^{n-1})$:
\begin{equation}
  \label{e:lafol}
(\varkappa^+)_*L_u=(\varkappa^-)_* L_s=L_V:=\ker(dw)\cap\ker(dy).
\end{equation}
Recall the symbol $p:T^*\mathbb H^n\setminus 0\to (0,\infty)$ and
the maps $B_\pm:T^*\mathbb H^n\setminus 0\to \mathbb S^{n-1}$ defined in~\eqref{e:p-symbol} and~\eqref{e:B-pm}.
For $(x,\xi)\in T^*\mathbb H^n\setminus 0$, put
$$
G_\pm(x,\xi)=p(x,\xi)\mathcal G(B_\pm(x,\xi),B_\mp(x,\xi))\in T^*_{B_\pm(x,\xi)}\mathbb S^{n-1}
$$
where $\mathcal G$ is defined in~\eqref{e:stpro}. See Figure~\ref{f:stereographic}.

%%%%%%%%%%%%%%%%%%%%%%%%%%%%%%%%%%%%%%%%%%%%%%%%%%%%%%%%%%%%%%%%%%%%%%%%%%%%%%%%
\begin{figure}
\includegraphics{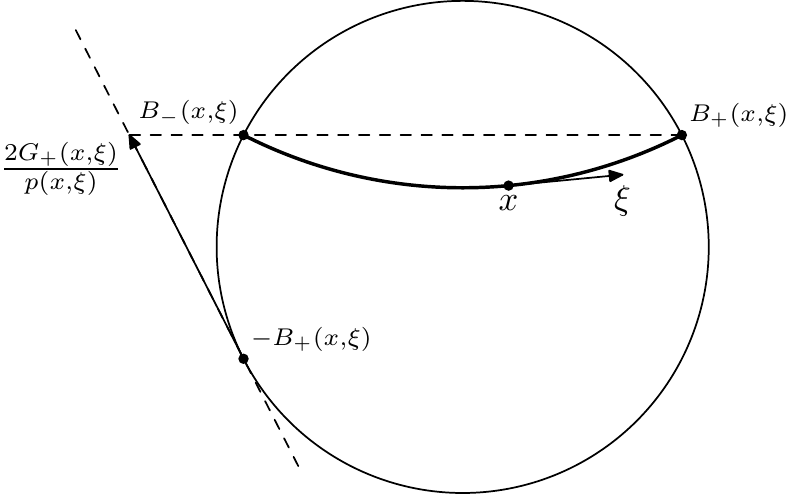}
\caption{The points $B_\pm(x,\xi)$ and the vector $G_+(x,\xi)$ in the
ball model of the hyperbolic space.}
\label{f:stereographic}
\end{figure}
%%%%%%%%%%%%%%%%%%%%%%%%%%%%%%%%%%%%%%%%%%%%%%%%%%%%%%%%%%%%%%%%%%%%%%%%%%%%%%%%

Denote by $\mathcal P(x,y)$ the (two-dimensional version of)
Poisson kernel, defined on the ball model of $\mathbb H^n$ by
\begin{equation}
  \label{e:Pker}
\mathcal P(x,y)={1-|x|^2\over |x-y|^2},\quad
x\in\mathbb H^n,\
y\in\mathbb S^{n-1}.
\end{equation}
The symplectomorphisms $\varkappa^\pm$ are constructed in the following lemma; see Appendix~\ref{s:hyperbolic-technical}
for the proof. Note that~\eqref{e:lafol} follows immediately from~\eqref{e:kappa-M-def} and~\eqref{e:desker}.
%%%%%%%%%%%%%%%%%%%%%%%%%%%%%%%%%%%%%%%%%%%%%%%%%%%%%%%%%%%%%%%%%%%%%%%%%%%%%%%%
\begin{lemm}
  \label{l:kappa-M}
The maps
\begin{equation}
  \label{e:kappa-M-def}
\varkappa^\pm:(x,\xi)\mapsto \big(p(x,\xi),B_\mp(x,\xi),\pm\log\mathcal P(x,B_\mp(x,\xi)),\pm G_\mp(x,\xi)\big)
\end{equation}
are exact symplectomorphisms from $T^*\mathbb H^n\setminus 0$ onto $T^*(\mathbb R^+\times\mathbb S^{n-1})$.
\end{lemm}
%%%%%%%%%%%%%%%%%%%%%%%%%%%%%%%%%%%%%%%%%%%%%%%%%%%%%%%%%%%%%%%%%%%%%%%%%%%%%%%%
\noindent\textbf{Remark}. The coordinates $\varkappa^\pm(x,\xi)=(w,y,\theta,\eta)$ can be interpreted as follows:
\begin{itemize}
\item $y,\eta$ determine the geodesic $\gamma(t)=e^{tX}(x,\xi)$ up to shifting $t$ and rescaling $\xi$,
in particular $y$ gives the limit of the geodesic $\gamma(t)$ as $t\to\pm\infty$;
\item $w$ is the length of $\xi$, corresponding to the energy of the geodesic $\gamma(t)$;
\item $\theta$ satisfies $\theta(\gamma(t))=\theta(\gamma(0))-t$ and thus determines
the position of $(x,\xi)$ on the geodesic $\gamma(t)$.
\end{itemize}
%%%%%%%%%%%%%%%%%%%%%%%%%%%%%%%%%%%%%%%%%%%%%%%%%%%%%%%%%%%%%%%%%%%%%%%%%%%%%%%%
We next consider the symplectomorphism
\begin{equation}
  \label{e:kappa-hat}
\widehat\varkappa:=\varkappa^+\circ(\varkappa^-)^{-1}:T^*(\mathbb R^+\times\mathbb S^{n-1})\to T^*(\mathbb R^+\times\mathbb S^{n-1}).
\end{equation}
The next lemma, proved in Appendix~\ref{s:hyperbolic-technical}, constructs
a generating function for $\widehat\varkappa$:
%%%%%%%%%%%%%%%%%%%%%%%%%%%%%%%%%%%%%%%%%%%%%%%%%%%%%%%%%%%%%%%%%%%%%%%%%%%%%%%%
\begin{lemm}
  \label{l:kappa-hat}
Consider the following function on $\mathbb R^+_w\times \mathbb S^{n-1}_\Delta$:
\begin{equation}
  \label{e:Theta}
\Theta(w,y,y')=w\log {|y-y'|^2\over 4}
\end{equation}
where $|y-y'|$ denotes Euclidean distance on $\mathbb S^{n-1}\subset\mathbb R^n$.
Then for each $(w,y,\theta,\eta)$ and $(w,y',\theta',\eta')$ in $T^*(\mathbb R^+\times\mathbb S^{n-1})$,
the following two statements are equivalent:
\begin{gather}
  \label{e:kappa-hat1}
(w,y',\theta',\eta')=\widehat\varkappa(w,y,\theta,\eta);\\
  \label{e:kappa-hat2}
\theta-\theta'=\partial_w\Theta(w,y,y'),\quad
\eta=\partial_y\Theta(w,y,y'),\quad
\eta'=-\partial_{y'}\Theta(w,y,y').
\end{gather}
Moreover, the antiderivative for $\widehat\varkappa$ defined as the sum of antiderivatives
for $\varkappa^+$ and $(\varkappa^-)^{-1}$ (see~\S\ref{s:fios})
is equal to the pullback of $\Theta$ to the graph $\Graph(\widehat\varkappa)$.
\end{lemm}
%%%%%%%%%%%%%%%%%%%%%%%%%%%%%%%%%%%%%%%%%%%%%%%%%%%%%%%%%%%%%%%%%%%%%%%%%%%%%%%%
Using Lemma~\ref{l:kappa-hat} and the theory presented in~\S\ref{s:fios}, we characterize
Fourier integral operators associated to $\widehat\varkappa^{-1}$:
%%%%%%%%%%%%%%%%%%%%%%%%%%%%%%%%%%%%%%%%%%%%%%%%%%%%%%%%%%%%%%%%%%%%%%%%%%%%%%%%
\begin{lemm}
  \label{l:B-hat}
Assume that $B\in I^{\comp}_h(\widehat\varkappa^{-1})$. Then we have
$$
B=A\widetilde{\mathcal B}_\chi+\mathcal O(h^\infty)_{\Psi^{-\infty}}
$$
for some $A\in\Psi^{\comp}_h(\mathbb R^+\times\mathbb S^{n-1})$,
$\chi\in C_0^\infty(\mathbb S^{n-1}_\Delta)$, and
$$
\widetilde{\mathcal B}_\chi v(w,y)=(2\pi h)^{1-n\over 2}\int_{\mathbb S^{n-1}} \Big|{y-y'\over 2}\Big|^{2iw/h}\chi(y,y')v(w,y')\,dy'
$$
where $\mathbb S^{n-1}_\Delta$ is defined in~\eqref{e:s-diag},
$|y-y'|$ denotes the Euclidean distance, and $dy'$ is the standard volume form on the sphere.
\end{lemm}
%%%%%%%%%%%%%%%%%%%%%%%%%%%%%%%%%%%%%%%%%%%%%%%%%%%%%%%%%%%%%%%%%%%%%%%%%%%%%%%%
\begin{proof}
For the function $\Theta$ defined in~\eqref{e:Theta}, we have
$$
\widetilde{\mathcal B}_\chi v(w,y)=(2\pi h)^{1-n\over 2}\int_{\mathbb S^{n-1}}e^{{i\over h}\Theta(w,y,y')}\chi(y,y')v(w,y')\,dy'.
$$
For $A\in\Psi^{\comp}_h(\mathbb R^+\times\mathbb S^{n-1})$, the
operator $A\widetilde{\mathcal B}_\chi$ is given by the following formula
modulo an $\mathcal O(h^\infty)_{\Psi^{-\infty}}$ remainder:
$$
A\widetilde{\mathcal B}_\chi v(w,y)=(2\pi h)^{-{n+1\over 2}}\int\limits_{\mathbb S^{n-1}\times T^*\mathbb R^+}e^{{i\over h}((w-w')\theta+\Theta(w',y,y'))}
b(w,\theta,y,y';h) v(w',y')\,  dy' dw' d\theta
$$
where $b$ is a compactly supported symbol on $\mathbb R^+_w\times\mathbb R_\theta\times \mathbb S^{n-1}_\Delta$
such that
$$
b(w,\theta,y,y';0)=\sigma_h(A)(w,y,\theta,\partial_y\Theta(w,y,y'))\chi(y,y').
$$
To see this, it suffices to choose some local coordinates on $\mathbb S^{n-1}$,
take $A=\Op_h(a)$ for some compactly supported symbol $a(w,y,\theta,\eta;h)$,
write
$$
\begin{aligned}
A\widetilde{\mathcal B}_\chi v(w,y)=\,&(2\pi h)^{1-3n\over 2}\int
e^{{i\over h}((w-w')\theta+(y-y'')\cdot\eta+\Theta(w',y'',y'))}\\
&a(w,y,\theta,\eta;h)\chi(y'',y')v(w',y')\,dw'dy''d\theta d\eta dy'
\end{aligned}
$$
where the integral is taken over $\mathbb R^+_{w'}\times\mathbb S^{n-1}_{y''}\times\mathbb R^n_{\theta,\eta}\times\mathbb S^{n-1}_{y'}$, and apply the method of stationary phase in the $(y'',\eta)$ variables.

Now, let $B\in I^{\comp}_h(\widehat\varkappa^{-1})$. Fix $\chi\in C_0^\infty(\mathbb S^{n-1}_\Delta)$
such that
\begin{equation}
  \label{e:WF-B-cond}
\WFh(B)\ \subset\  U_\chi:=\varphi_\Theta(\{(w,\theta,y,y')\mid \chi(y,y')\neq 0\})
\end{equation}
where $\varphi_\Theta:R^+_w\times\mathbb R_\theta\times \mathbb S^{n-1}_\Delta\to\Graph(\widehat\varkappa^{-1})$
is the diffeomorphism constructed using~\eqref{e:kappa-hat2}.
By Lemma~\ref{l:kappa-hat}, the function
$$
\Phi:(w,y,w',y',\theta)\in \mathbb R^+_w\times\mathbb R^+_{w'}\times \mathbb S^{n-1}_\Delta\times\mathbb R_\theta
\ \mapsto\ (w-w')\theta+\Theta(w',y,y')
$$
parametrizes $\widehat\varkappa^{-1}$ in the sense of~\eqref{e:kappa-parametrized}. Therefore,
by~\eqref{e:fio-general-form}
we may write for some compactly supported symbol $\tilde b$ on the domain of $\Phi$,
modulo $\mathcal O(h^\infty)_{\Psi^{-\infty}}$,
$$
Bv(w,y)=(2\pi h)^{-{n+1\over 2}}\int\limits_{\mathbb S^{n-1}\times T^*\mathbb R^+}
e^{{i\over h}((w-w')\theta+\Theta(w',y,y'))}\tilde b(w,y,w',y',\theta;h)v(w',y')\,dy'dw'd\theta.
$$
Moreover, by~\eqref{e:WF-B-cond} (which implies that $B$ is a Fourier
integral operator associated to a restriction of $\widehat\varkappa^{-1}$)
we can take $\tilde b$ supported inside $\{\chi(y,y')\neq 0\}$.

Take $A_0\in\Psi^{\comp}_h(\mathbb R^+\times\mathbb S^{n-1})$ such that
for all $(w,y,\theta,\eta,w,y',\theta',\eta')\in \Graph(\widehat\varkappa^{-1})$,
$$
\sigma_h(A_0)(w,y,\theta,\eta)={\tilde b(w,y,w,y',\theta;0)\over \chi(y,y')}.
$$
Comparing the oscillatory integral expressions for $A_0\widetilde{\mathcal B}_\chi$ and $B$
and using~\eqref{e:principal-killed}, we get
$$
B-A_0\widetilde{\mathcal B}_\chi\in hI^{\comp}_h(\widehat\varkappa^{-1}).
$$
Moreover, we may choose $A_0$ so that
$\WFh(B-A_0\widetilde{\mathcal B}_\chi)\subset U_\chi$. Replacing
$B$ with $h^{-1}(B-A_0\widetilde{\mathcal B}_\chi)$ and arguing by induction, we construct
$A_j\in \Psi^{\comp}_h(\mathbb R^+\times\mathbb S^{n-1})$ such that
$$
B-\sum_{j=0}^N h^jA_j\widetilde{\mathcal B}_\chi\in h^{N+1}I^{\comp}_h(\widehat\varkappa^{-1}).
$$
Then $B=A\widetilde{\mathcal B}_\chi+\mathcal O(h^\infty)_{\Psi^{-\infty}}$,
where $A\in\Psi^{\comp}_h(\mathbb R^+\times\mathbb S^{n-1})$
is defined by the asymptotic sum $A\sim \sum_{j=0}^\infty h^j A_j$. 
\end{proof}
%%%%%%%%%%%%%%%%%%%%%%%%%%%%%%%%%%%%%%%%%%%%%%%%%%%%%%%%%%%%%%%%%%%%%%%%%%%%%%%%
We now reformulate the fractal uncertainty principle of
Definition~\ref{d:fup} as follows:
%%%%%%%%%%%%%%%%%%%%%%%%%%%%%%%%%%%%%%%%%%%%%%%%%%%%%%%%%%%%%%%%%%%%%%%%%%%%%%%%
\begin{lemm}
  \label{l:fup-revisited}
Assume $\Lambda_\Gamma,\beta,\varepsilon>0,\rho\in (0,1)$
are such that~\eqref{e:fup-standard} holds for all $C_1,\chi$.
With $L_V$ defined in~\eqref{e:lafol}, let
$$
A_\pm\in \Psi^{\comp}_{h,L_V,\rho}(T^*(\mathbb R^+\times\mathbb S^{n-1})),\quad
B\in I^{\comp}_h(\widehat\varkappa^{-1})
$$
and assume that for some constant $C_2$,
$A_+$ and $A_-$ are $\mathcal O(h^\infty)$, in the sense of Definition~\ref{d:lag-prince},
along every sequence $(w_j,y_j,\theta_j,\eta_j,h_j)$ such that
$d(y_j,\Lambda_\Gamma)>C_2h_j^\rho$.
Then
$$
\|A_-BA_+\|_{L^2\to L^2}\leq Ch^{\beta-\varepsilon}.
$$
\end{lemm}
%%%%%%%%%%%%%%%%%%%%%%%%%%%%%%%%%%%%%%%%%%%%%%%%%%%%%%%%%%%%%%%%%%%%%%%%%%%%%%%%
\begin{proof}
We first use Lemma~\ref{l:B-hat} to write $B=A\widetilde{\mathcal B}_\chi+\mathcal O(h^\infty)_{L^2\to L^2}$ for some
$\chi\in C_0^\infty(\mathbb S^{n-1}_\Delta)$
and $A\in\Psi^{\comp}_h(\mathbb R^+\times\mathbb S^{n-1})$.
The operator $A_-A\in\Psi^{\comp}_{h,L_V,\rho}$ satisfies the same microlocal
vanishing assumption as $A_-$, therefore it suffices to show that
\begin{equation}
  \label{e:meimei}
\|A_-\chi'(w)\widetilde{\mathcal B}_\chi A_+\|_{L^2\to L^2}\leq Ch^{\beta-\varepsilon}.
\end{equation}
Here we may insert some cutoff function $\chi'\in C_0^\infty(\mathbb R^+)$
since $A$ is compactly supported.

Using Lemma~\ref{l:symbol-construction}, take a function
$\chi_0(y;h)\in C^\infty(\mathbb S^{n-1})$ such that
$$
\supp\chi_0\subset \Lambda_\Gamma(2C_2h^\rho),\quad
\supp(1-\chi_0)\cap\Lambda_\Gamma(C_2h^\rho)=\emptyset,\quad
|\partial^\alpha_y \chi_0(y)|\leq C_\alpha h^{-\rho|\alpha|}.
$$
Here $\Lambda_\Gamma(\cdot)$ is defined in~\eqref{e:limit-nbhd}. We claim that
\begin{equation}
  \label{e:meimei2}
A_-(1-\chi_0(y;h))=\mathcal O(h^\infty)_{L^2\to L^2},\quad
(1-\chi_0(y;h))A_+=\mathcal O(h^\infty)_{L^2\to L^2}.
\end{equation}
Indeed, by a partition of unity we may assume that $\WFh(A_-)$ is contained 
in the cotangent bundle of a coordinate chart on $\mathbb R^+\times\mathbb S^{n-1}$.
By part~2 of Lemma~\ref{l:globallem} (where $B,B'$ are pullback operators),
we can write $A_-=\Op_h(a_-)+\mathcal O(h^\infty)_{L^2\to L^2}$ for some $a_-\in S^{\comp}_{L_0,\rho}$.
Moreover, by the assumption on $A_-$, we see that $a_-$ is $\mathcal O(h^\infty)$,
in the sense of Definition~\ref{d:rapid-decay}, along every sequence
$(w_j,y_j,\theta_j,\eta_j,h_j)$ such that $y_j\notin\Lambda_\Gamma(C_2h_j^\rho)$.
Then the first estimate in~\eqref{e:meimei2} follows from Lemma~\ref{l:quant-basic}
(or rather, its trivial adaptation to the non-compactly supported symbol $1-\chi_0(y;h)$);
the second estimate in~\eqref{e:meimei2} is proved similarly.

By~\eqref{e:meimei2}, since $\|A_\pm\|_{L^2\to L^2}$ is bounded uniformly in $h$,
in order to prove~\eqref{e:meimei} it suffices to show
\begin{equation}
  \label{e:meimei3}
\|\chi_0(y;h)\chi'(w)\widetilde{\mathcal B}_\chi\chi_0(y;h)\|_{L^2(\mathbb R^+\times\mathbb S^{n-1})\to L^2(\mathbb R^+\times\mathbb S^{n-1})}\leq Ch^{\beta-\varepsilon}.
\end{equation}
We calculate
$$
\chi_0(y;h)\chi'(w)\widetilde{\mathcal B}_\chi\chi_0(y;h)v(w,y)=\chi'(w)4^{-iw/h}\mathcal B'_w(v(w,\cdot))(y),
$$
where $\mathcal B'_w:L^2(\mathbb S^{n-1})\to L^2(\mathbb S^{n-1})$ is given by
$$
\mathcal B'_w f(y)=(2\pi h)^{1-n\over 2}\int_{\mathbb S^{n-1}}|y-y'|^{2iw/h}
\chi(y,y')\chi_0(y;h)\chi_0(y';h)f(y')\,dy'.
$$
Replacing $h$ by $h/w$ in~\eqref{e:fup-standard} and using that $\chi_0$ is bounded
and supported in $\Lambda_\Gamma(2C_2h^\rho)$, we see that uniformly in $w\in \supp\chi'$,
$$
\|\mathcal B'_w\|_{L^2(\mathbb S^{n-1})\to L^2(\mathbb S^{n-1})}\leq Ch^{\beta-\varepsilon}
$$
and~\eqref{e:meimei3} follows from here by integration.
\end{proof}
%%%%%%%%%%%%%%%%%%%%%%%%%%%%%%%%%%%%%%%%%%%%%%%%%%%%%%%%%%%%%%%%%%%%%%%%%%%%%%%%
We are now ready to give
%%%%%%%%%%%%%%%%%%%%%%%%%%%%%%%%%%%%%%%%%%%%%%%%%%%%%%%%%%%%%%%%%%%%%%%%%%%%%%%%
\begin{proof}[Proof of Theorem~\ref{t:fup-reduction}]
In order to prove~\eqref{e:essential-gap2}, it suffices to show the estimate~\eqref{e:est0}
for all functions $u,f$ satisfying~\eqref{e:eq0},
where (see~\eqref{e:nu-def})
$$
\omega=1-ih\nu,\quad\nu\in [-1,\beta-\varepsilon].
$$
Take $\chi_\pm\in C_0^\infty(T^*M\setminus 0;[0,1])$
such that $\chi_\pm=1$ near $K\cap \{|\xi|_g=1\}$.
We also take $\varepsilon_0>0$ small enough
depending on $\varepsilon$ and fix $\rho\in (0,1)$ so that~\eqref{e:fup-standard} is satisfied
with $\varepsilon$ replaced by $\varepsilon_0$.
Put
$$
\begin{gathered}
t:=\rho\log(1/h),\quad
\nu^+:=\max(0,\nu);\\
A'_+:=\Op_h^{L_u}\big(\chi_+(\chi_+\circ e^{-tX})\big),\quad
A_0:=\Op_h^{L_u}(\chi_+),\quad
A'_-:=\Op_h^{L_s}\big(\chi_-(\chi_-\circ e^{tX})\big).
\end{gathered}
$$
Note that by Lemma~\ref{l:propagated-okay},
$$
A'_+\in\Psi^{\comp}_{h,L_u,\rho}(T^*M\setminus 0),\quad
A_0\in\Psi^{\comp}_h(M),\quad
A'_-\in\Psi^{\comp}_{h,L_s,\rho}(T^*M\setminus 0).
$$
By Lemma~\ref{l:second} we obtain
\begin{gather}
  \label{e:together1}
\|(A_0-A'_+)u\|_{L^2}\leq Ch^{-1-\rho(\nu^++\varepsilon_0)}\|f\|_{H^{s-1}_h}
+\mathcal O(h^\infty)\|u\|_{H^s_h},\\
  \label{e:together2}
\|u\|_{H^s_h}\leq Ch^{-\rho(\nu+\varepsilon_0)}\|A'_-u\|_{L^2}
+Ch^{-1-\rho(\nu^++\varepsilon_0)}\|f\|_{H^{s-1}_h}.
\end{gather}
We choose $\chi_\pm$ so that $\chi_+=1$ near $\supp\chi_-$.
Let $Q\in\Psi^{\comp}_h(M)$ be an elliptic parametrix of $A_0$
near $\supp\chi_-$ (see for instance~\cite[\S E.2.2]{dizzy} or~\cite[Proposition~2.4]{zeta});
in particular,
$$
QA_0=1+\mathcal O(h^\infty)\quad\text{microlocally near }\supp\chi_-.
$$
Since $A'_-$ is pseudolocal and its wavefront set is contained
inside $\supp\chi_-$, we have
\begin{equation}
  \label{e:together4}
\|A'_-(1-QA_0)u\|_{L^2}=\mathcal O(h^\infty)\|u\|_{H^s_h}.
\end{equation}
Take $\varepsilon_0<\varepsilon/2$. We claim that it suffices to prove the bound
\begin{equation}
  \label{e:prodib}
\|A'_-QA'_+\|_{L^2(M)\to L^2(M)}\leq Ch^{\beta-\varepsilon_0}.
\end{equation}
Indeed, putting together~\eqref{e:together1}--\eqref{e:prodib} and using that $\nu\leq\beta-\varepsilon$, we have
$$
\begin{aligned}
\|u\|_{H^s_h}&\leq Ch^{-\rho(\nu+\varepsilon_0)}\|A'_-u\|_{L^2}
+Ch^{-1-\rho(\nu^++\varepsilon_0)}\|f\|_{H^{s-1}_h}\\
&\leq Ch^{-\rho(\nu+\varepsilon_0)}\big(\|A'_-QA'_+u\|_{L^2}+
\|A'_-Q(A_0-A'_+)u\|_{L^2}\\
&\quad+\|A'_-(1-QA_0)u\|_{L^2}\big)
+Ch^{-1-\rho(\nu^++\varepsilon_0)}\|f\|_{H^{s-1}_h}\\
&\leq Ch^{\varepsilon-2\varepsilon_0}\|u\|_{H^s_h}
+Ch^{-1-2(\nu^++\varepsilon_0)}\|f\|_{H^{s-1}_h}
+\mathcal O(h^\infty)\|u\|_{H^s_h}
\end{aligned}
$$
giving~\eqref{e:est0} for $h$ small enough.

It remains to deduce~\eqref{e:prodib} from the fractal uncertainty principle. We may assume that $\WFh(Q)$ is contained in
a small neighborhood of $\supp\chi_-$; by an appropriate choice of $\chi_-$,
we may assume that $\WFh(Q)$ is contained in a small neighborhood
of $K\cap \{|\xi|_g=1\}$. By a partition of unity, it suffices to show that
for each $(x_0,\xi_0)\in K\cap \{|\xi|_g=1\}$, there exists a neighborhood $V$
of $(x_0,\xi_0)$ such that~\eqref{e:prodib} holds for all
$Q\in\Psi^{\comp}_h(M)$ with $\WFh(Q)\subset V$.

Fix $(x_0,\xi_0)\in K\cap \{|\xi|_g=1\}$. Composing the maps $\varkappa^\pm$ constructed in Lemma~\ref{l:kappa-M}
with a local inverse of the covering map $\pi_\Gamma$ defined in~\eqref{e:pi-Gamma}, we obtain exact symplectomorphisms
\begin{equation}
  \label{e:kappa-0}
\varkappa_0^\pm:U\to U'_\pm,\quad
U\subset T^* M\setminus 0,\quad
U'_\pm\subset T^*(\mathbb R^+\times\mathbb S^{n-1}),
\end{equation}
for some small neighborhood $U$ of $(x_0,\xi_0)$ and
some small neighborhoods $U'_\pm$ of
$$
(1,y_0^\pm,\theta_0^\pm,\eta_0^\pm):=\varkappa_0^\pm(x_0,\xi_0).
$$
Here $(w,y)$ are coordinates on $\mathbb R^+\times\mathbb S^{n-1}$
and $(\theta,\eta)$ are the corresponding dual variables.

Take $\mathcal B_\pm\in I^{\comp}_h(\varkappa_0^\pm)$,
$\mathcal B'_\pm\in I^{\comp}_h((\varkappa_0^\pm)^{-1})$ quantizing
$\varkappa_0^\pm$ near $V\times \varkappa_0^\pm(V)$ in the sense of~\eqref{e:quantized},
where $V\Subset U$ is a small neighborhood of $(x_0,\xi_0)$.
Since $\WFh(Q)\subset V$ and $A'_\pm$ are pseudolocal, we have
$$
A'_-QA'_+=(\mathcal B'_-\mathcal B_-)A'_-(\mathcal B'_-\mathcal B_-)(\mathcal B'_+\mathcal B_+)QA'_+(\mathcal B'_+\mathcal B_+)+\mathcal O(h^\infty)_{L^2\to L^2};
$$
therefore, to prove~\eqref{e:prodib} it suffices to show that
\begin{equation}
  \label{e:prodib2}
\|A_- BA_+\|_{L^2(\mathbb R^+\times\mathbb S^{n-1})\to L^2(\mathbb R^+\times\mathbb S^{n-1})}\leq Ch^{\beta-\varepsilon_0},
\end{equation}
where, with $\widehat\varkappa$ defined in~\eqref{e:kappa-hat},
$$
A_-=\mathcal B_-A'_-\mathcal B'_-,\quad
A_+=\mathcal B_+QA'_+\mathcal B'_+,\quad
B=\mathcal B_-\mathcal B'_+\in I^{\comp}_h(\widehat\varkappa^{-1}).
$$
By the construction of $\Psi^{\comp}_{h,L,\rho}$ calculus in~\S\ref{s:calculus-general},
together with~\eqref{e:lafol}, we see that
$$
A_\pm\in\Psi^{\comp}_{h,L_V,\rho}(T^*(\mathbb R^+\times\mathbb S^{n-1})).
$$
Moreover, $A_\pm=\mathcal O(h^\infty)$ in the sense of Definition~\ref{d:lag-prince}
along each sequence $(w_j,y_j,\theta_j,\eta_j,h_j)$ such that
\begin{equation}
  \label{e:killit-1}
(w_j,y_j,\theta_j,\eta_j)\ \notin\ \varkappa_0^\pm(U\cap e^{\pm\rho\log(1/h_j)X}(\supp\chi_\pm)).
\end{equation}
By Lemma~\ref{l:close-to-trapping} and~\eqref{e:kappa-M-def}, we see that~\eqref{e:killit-1} is satisfied when
$d(y_j,\Lambda_\Gamma)> C_2h_j^\rho$ for some fixed constant $C_2$.
Now~\eqref{e:prodib2} follows from Lemma~\ref{l:fup-revisited}.
\end{proof}
%%%%%%%%%%%%%%%%%%%%%%%%%%%%%%%%%%%%%%%%%%%%%%%%%%%%%%%%%%%%%%%%%%%%%%%%%%%%%%%%
\noindent\textbf{Remark}. As follows from the proof of Lemma~\ref{l:kappa-M}
in Appendix~\ref{s:hyperbolic-technical}, the Fourier integral operators $\mathcal B_\pm$ used in the proof
of Theorem~\ref{t:fup-reduction}
are microlocalized versions of the Poisson operators
$$
v(w,y)\ \mapsto\ u(x)=\int_{\mathbb R^+}\int_{\mathbb S^{n-1}} \mathcal P(x,y)^{\mp i w/h} v(w,y)\,dy dw.
$$
Therefore, conjugation by $\mathcal B_\pm$ is related to the representation
of resonant states as images under the Poisson operator of distributions supported on the limit
set, see for instance~\cite[(14.9)]{Borthwick} or~\cite{Bunke-Olbrich2}.

%%%%%%%%%%%%%%%%%%%%%%%%%%%%%%%%%%%%%%%%%%%%%%%%%%%%%%%%%%%%%%%%%%%%%%%%%%%%%%%%
\section{Fractal uncertainty principle}
  \label{s:fup}
  
In this section, we prove Theorem~\ref{t:ae-reduction}.
We will not use directly the geometry or the dynamics of the manifold $M$,
instead relying on the additive structure of the limit set $\Lambda_\Gamma$
defined in~\eqref{e:Lambda-Gamma} and basic harmonic analysis.

%%%%%%%%%%%%%%%%%%%%%%%%%%%%%%%%%%%%%%%%%%%%%%%%%%%%%%%%%%%%%%%%%%%%%%%%%%%%%%%%
\subsection{Basic properties}

We start with some basic facts regarding the fractal uncertainty principle
of Definition~\ref{d:fup}. First of all, since $\mathcal B_\chi(h)$ is a semiclassical Fourier
integral operator associated to a canonical transformation
(see~\eqref{e:fio-general-form} and Lemma~\ref{l:kappa-hat}), it is bounded
on $L^2$ uniformly in $h$. This gives the bound
\begin{equation}
  \label{e:tb-1}
\|\indic_{\Lambda_\Gamma(C_1h^\rho)}\mathcal B_\chi\indic_{\Lambda_\Gamma(C_1 h^\rho)}\|_{L^2\to L^2}\leq C.
\end{equation}
Combined with Theorem~\ref{t:fup-reduction}, this bound translates to the well-known statement
that there are no resonances in the upper half-plane away from the imaginary axis, which
is a direct consequence of self-adjointness of the Laplacian on $L^2(M)$.

To formulate the next bound, we introduce the parameter
\begin{equation}
  \label{e:delta}
\delta\in [0,n-1]
\end{equation}
defined as the exponent of convergence of Poincar\'e series: that is,
$\delta$ is the smallest number such that for all $x,x'\in\mathbb H^n$ and $\Re s>\delta$, we have
$$
\Sigma_p(s;x,x'):=\sum_{\gamma\in\Gamma} \exp\big(-sd_{\mathbb H^n}(x,\gamma.x')\big)<\infty.
$$
Here $d_{\mathbb H^n}(\cdot,\cdot)$ stands for the distance function on $\mathbb H^n$
induced by the hyperbolic metric.

The constant $\delta$ is the Hausdorff dimension of the limit set $\Lambda_\Gamma$, see~\cite[Theorem~14.14]{Borthwick}
for the case $n=2$ and~\cite{Sullivan} for general dimensions. It is also the Minkowski dimension of $\Lambda_\Gamma$;
in fact, we have the following more precise estimate (which is a form of Ahlfors-David regularity):
\begin{equation}
  \label{e:AD-estimate-Lebesgue}
\begin{gathered}
C^{-1}\alpha^{n-1-\delta}(\alpha')^\delta\ \leq\ \mu_L(\Lambda_\Gamma(\alpha)\cap B(y_0,\alpha'))
\ \leq\ C\alpha^{n-1-\delta}(\alpha')^\delta,\\
0<\alpha\leq \alpha'\leq 1,\quad
y_0\in\Lambda_\Gamma,
\end{gathered}
\end{equation}
where $\Lambda_\Gamma(\alpha)$ is defined in~\eqref{e:limit-nbhd},
$B(y_0,\alpha')$ denotes the ball of radius $\alpha'$ centered at $y_0$,
$\mu_L$ denotes the Lebesgue measure on $\mathbb S^{n-1}$, and the constant $C>0$ does not depend
on $y_0,\alpha,\alpha'$.
See~\S\ref{s:minkowski} for the proof. By putting $\alpha'=1$ in~\eqref{e:AD-estimate-Lebesgue},
we obtain in particular~\eqref{e:LG-upper-1}.

Given~\eqref{e:LG-upper-1}, we may use Schur's lemma~\cite[Lemma~18.1.12]{ho3}:
the estimate
$$
\sup_{y\in\Lambda_\Gamma(C_1h^\rho)}\int_{\Lambda_\Gamma(C_1h^\rho)} |\chi(y,y')|\,dy'\leq Ch^{\rho(n-1-\delta)}
$$
implies by~\eqref{e:B-chi} that
\begin{equation}
  \label{e:tb-2}
\|\indic_{\Lambda_\Gamma(C_1h^\rho)}\mathcal B_\chi\indic_{\Lambda_\Gamma(C_1 h^\rho)}\|_{L^2\to L^2}\leq Ch^{1-n\over 2}h^{\rho(n-1-\delta)}.
\end{equation}
Since we may choose $\rho$ arbitrarily close to 1, this gives the fractal uncertainty principle
with exponent ${n-1\over 2}-\delta$. By Theorem~\ref{t:fup-reduction},
the bounds \eqref{e:tb-1} and~\eqref{e:tb-2} together translate (with a loss of $\varepsilon$)
to the standard spectral gap~\eqref{e:standard-gap}.

We finally show that the fractal uncertainty principle cannot hold with
$\beta>{n-1\over 2}-{\delta\over 2}$. This threshold corresponds to the Jakobson--Naud
conjecture, see~\S\ref{s:spfup}. More precisely, we claim that for each $\varepsilon>0$,
there exists $\chi\in C_0^\infty(\mathbb S^{n-1}_\Delta)$ and
a family $v(h)\in L^2(\mathbb S^{n-1})$ such that for $h$ small enough,
\begin{equation}
  \label{e:JN}
\|\indic_{\Lambda_\Gamma(h)}\mathcal B_\chi(h) \indic_{\Lambda_\Gamma(h)}v(h)\|_{L^2}\geq h^{{n-1\over 2}-{\delta\over 2}+\varepsilon}\|v(h)\|_{L^2}.
\end{equation}
To prove~\eqref{e:JN}, take small $\tilde\varepsilon>0$,
fix $y_0,y_1\in\Lambda_\Gamma$, $y_0\neq y_1$, and let $v(h)=\indic_{B(y_0,h^{1+\tilde\varepsilon})}$.
Then
\begin{equation}
  \label{e:vvvnnn}
\|v(h)\|_{L^2}\leq C h^{(1+\tilde\varepsilon)(n-1)/2},\quad
\indic_{\Lambda_\Gamma(h)}v(h)=v(h).
\end{equation}
Using~\eqref{e:B-chi}, we compute
$$
\mathcal B_\chi(h) v(y;h)=(2\pi h)^{1-n\over 2}|y-y_0|^{2i/h}\int_{\mathbb S^{n-1}} \Big({|y-y'|\over |y-y_0|}\Big)^{2i/h}\chi(y,y')v(y')\,dy'.
$$
We take $h$-independent $\chi$ such that $\chi(y,y')=1$ for $(y,y')$ near $(y_1,y_0)$. For $y$ in a fixed neighborhood
of $y_1$ and $y'\in\supp v$, we have $\big({|y-y'|\over |y-y_0|}\big)^{2i/h}=1+\mathcal O(h^{\tilde\varepsilon})$, therefore
for $\tilde\varepsilon_1$ small enough,
$$
|\mathcal B_\chi(h) v(y;h)|\geq C^{-1}h^{1-n\over 2} h^{(1+\tilde\varepsilon)(n-1)}\quad\text{for }
|y-y_1|<\tilde\varepsilon_1.
$$
By~\eqref{e:AD-estimate-Lebesgue}, we then have
$$
\|\indic_{\Lambda_\Gamma(h)}\mathcal B_\chi(h)v(h)\|_{L^2}\geq C^{-1}
h^{1-n\over 2}h^{(1+\tilde\varepsilon)(n-1)}h^{n-1-\delta\over 2},
$$
which together with~\eqref{e:vvvnnn} implies~\eqref{e:JN} as long as $\tilde\varepsilon<{2\over n-1}\varepsilon$.

%%%%%%%%%%%%%%%%%%%%%%%%%%%%%%%%%%%%%%%%%%%%%%%%%%%%%%%%%%%%%%%%%%%%%%%%%%%%%%%%
\subsection{Reduction to additive energy}
  \label{s:fup-ae}

We now prove Theorem~\ref{t:ae-reduction}. We will take $\rho$ very close to 1, in particular $\rho>1/2$.
Using Lemma~\ref{l:symbol-construction}, take $\psi_0(y;h)\in C^\infty(\mathbb S^{n-1};[0,1])$ such that
\begin{gather}
\label{e:psi0-1}
\supp(1-\psi_0)\cap \Lambda_\Gamma(C_1h^\rho)=\emptyset,\quad
\supp\psi_0\subset \Lambda_\Gamma(h^{\rho/2});\\
\label{e:psi0-3}
\sup_{\mathbb S^{n-1}}|\partial^\alpha\psi_0|\leq C_\alpha h^{-{\rho\over 2}|\alpha|}.
\end{gather}
To show~\eqref{e:fup-standard}, it suffices to prove that
$$
\|\sqrt{\psi_0}\,\mathcal B_\chi(h)\indic_{\Lambda_\Gamma(C_1h^\rho)}\|_{L^2\to L^2}\leq Ch^{\beta-\varepsilon},
$$
in fact it is enough to prove the following $T^*T$-bound:
$$
\|\indic_{\Lambda_\Gamma(C_1 h^\rho)}\mathcal B_\chi(h)^*\psi_0\mathcal B_\chi(h)\indic_{\Lambda_\Gamma(C_1h^\rho)}\|_{L^2\to L^2}\leq Ch^{2(\beta-\varepsilon)}.
$$
By Schur's lemma~\cite[Lemma~18.1.12]{ho3} and~\eqref{e:B-chi} it is enough to prove the Schwartz kernel bound
\begin{equation}
  \label{e:skb1}
\sup_{y''\in\Lambda_\Gamma(C_1 h^\rho)}\int\limits_{\Lambda_\Gamma(C_1 h^\rho)}|\mathcal K(y,y'';h)|\,dy
\leq Ch^{2(\beta-\varepsilon)}
\end{equation}
where the integral kernel $\mathcal K(y,y'';h)$ of the operator $\mathcal B_\chi(h)^*\psi_0\mathcal B_\chi(h)$
is given by
\begin{equation}
  \label{e:skb2}
\mathcal K(y,y'';h)=h^{1-n}\int\limits_{\mathbb S^{n-1}}
\bigg({|y'-y''|\over |y'-y|}\bigg)^{2i/h} \chi(y',y'')\overline{\chi(y',y)}\psi_0(y';h)\,dy'.
\end{equation}
Morally speaking, $\mathcal K(y,y'';h)$ is the correlation on $\Lambda_\Gamma(h^{\rho/2})$ between Lagrangian
states corresponding to two levels in~\eqref{e:hor-lag} given by $y_j=y$ and $y_j=y''$.

To capture cancellations in the expression~\eqref{e:skb1}, we use the following precise
version of the method of nonstationary phase (see for instance~\cite[Theorem~7.7.1]{ho1} for
the standard version):
%%%%%%%%%%%%%%%%%%%%%%%%%%%%%%%%%%%%%%%%%%%%%%%%%%%%%%%%%%%%%%%%%%%%%%%%%%%%%%%%
\begin{lemm}
  \label{l:ibp}
Let $U\subset \mathbb R^m$ be open and bounded,
$\widetilde\varphi\in C^\infty(U;\mathbb R)$, and $a\in C_0^\infty(U)$.
Assume that the following inequalities hold:
\begin{equation}
  \label{e:ibp1}
\begin{aligned}
C_1^{-1}\tilde h^{-1}\leq |d\widetilde\varphi(x)|\leq C_1\tilde h^{-1}&\quad\text{for all }x\in\supp a;\\
|\partial^\alpha \widetilde\varphi(x)|\leq C_{|\alpha|}\tilde h^{-\tilde\rho|\alpha|}&\quad\text{for all }x\in\supp a\text{ and }|\alpha|\geq 2;\\
|\partial^\alpha a(x)|\leq C_{|\alpha|}\tilde h^{-\tilde\rho|\alpha|}&\quad\text{for all }x\in U\text{ and all }\alpha.
\end{aligned}
\end{equation}
Here $\tilde\rho,\tilde h\in (0,1)$ and $C_0,C_1,C_2,\dots$ are positive constants. Then for each $N\in\mathbb N_0$
\begin{equation}
  \label{e:ibp2}
\bigg|\int_U e^{i\widetilde\varphi(x)}a(x)\,dx\bigg|\leq
C'_N\tilde h^{N(1-\tilde\rho)},
\end{equation}
where the constant $C'_N$ depends only on $U,C_0,C_1,\dots,C_{N+1}$.
\end{lemm}
%%%%%%%%%%%%%%%%%%%%%%%%%%%%%%%%%%%%%%%%%%%%%%%%%%%%%%%%%%%%%%%%%%%%%%%%%%%%%%%%
\noindent\textbf{Remark.} Using coordinate charts and a partition of unity for $a(x)$,
we see that Lemma~\ref{l:ibp} also applies when $U$ is a manifold;
we will typically use it for $U=\mathbb S^{n-1}$.
%%%%%%%%%%%%%%%%%%%%%%%%%%%%%%%%%%%%%%%%%%%%%%%%%%%%%%%%%%%%%%%%%%%%%%%%%%%%%%%%
\begin{proof}
Consider the first order differential operator
$$
L=-i\sum_{j=1}^m {\partial_j\widetilde\varphi\over |d\widetilde\varphi|^2}\partial_j,\quad
|d\widetilde\varphi|^2:=\sum_{j=1}^m|\partial_j\widetilde\varphi|^2.
$$
Then $e^{i\widetilde\varphi}=L(e^{i\widetilde\varphi})$. Integrating
by parts $N$ times, we obtain
$$
\bigg|\int_{U} e^{i\widetilde\varphi(x)}a(x)\,dx\bigg|=
\bigg|\int_{U} e^{i\widetilde\varphi(x)}(L^t)^Na(x)\,dx\bigg|
\leq \mu_L(U)\sup_x |(L^t)^Na(x)|,
$$
where $\mu_L$ is the Lebesgue measure and
$$
L^t f=i\sum_{j=1}^m\partial_j\Big({\partial_j\widetilde\varphi\over |d\widetilde\varphi|^2}f\Big).
$$
Now, the first two bounds in~\eqref{e:ibp1} imply that
$$
\sup_{x\in\supp a}\Big|\partial^\alpha\Big({\partial_j\widetilde\varphi\over |d\widetilde\varphi|^2}\Big)\Big|\leq
C'_\alpha\tilde h^{1-\tilde\rho|\alpha|}
$$
where the constants $C'_\alpha$ depend only on $C_0,C_1,\dots,C_{|\alpha|+1}$.
This together with the last bound in~\eqref{e:ibp1} implies the estimate~\eqref{e:ibp2}.
\end{proof}
%%%%%%%%%%%%%%%%%%%%%%%%%%%%%%%%%%%%%%%%%%%%%%%%%%%%%%%%%%%%%%%%%%%%%%%%%%%%%%%%
Armed with Lemma~\ref{l:ibp}, we establish decay of the kernel $\mathcal K$.
We first consider the case when $y$ and $y''$ are sufficiently far away from each other
so that the corresponding Lagrangian leaves almost do not correlate:
%%%%%%%%%%%%%%%%%%%%%%%%%%%%%%%%%%%%%%%%%%%%%%%%%%%%%%%%%%%%%%%%%%%%%%%%%%%%%%%%
\begin{lemm}
  \label{l:ibpl1}
We have uniformly in $y,y''\in\mathbb S^{n-1}$ such that $|y-y''|>{1\over 2}h^{1/2}$,
$$
\mathcal K(y,y'';h)=\mathcal O(h^\infty).
$$
\end{lemm}
%%%%%%%%%%%%%%%%%%%%%%%%%%%%%%%%%%%%%%%%%%%%%%%%%%%%%%%%%%%%%%%%%%%%%%%%%%%%%%%%
\begin{proof} We rewrite~\eqref{e:skb2} as
\begin{equation}
  \label{e:K-rewritten}
\begin{gathered}
\mathcal K(y,y'';h)=h^{1-n}\int_{\mathbb S^{n-1}}e^{i\varphi(y,y',y'')/h}a(y,y',y'';h)\,dy',\\
\varphi=2(\log |y'-y''|-\log |y'-y|),\quad
a=\chi(y',y'')\overline{\chi(y',y)}\psi_0(y';h).
\end{gathered}
\end{equation}
Due to the cutoff $\chi$, the amplitude $a$ is supported inside some fixed compact set
which does not intersect $\{y=y'\}$ and $\{y'=y''\}$; in particular $\varphi$
is smooth near $\supp a$. We next have for all $N$,
$$
\|\varphi\|_{C^N_{y'}(\supp a)}\leq C_N |y-y''|,
$$
where the constants $C_N$ are independent of $y,y'',h$.
Moreover, for some constant $C$ independent of $y,y',y'',h$,
$$
|\partial_{y'}\varphi(y,y',y'')|\geq C^{-1}|y-y''|\quad\text{for }(y,y',y'')\in\supp a.
$$
This follows immediately from~\eqref{e:gide}
and the fact that the map $y\mapsto \mathcal G(y',y)$ is a diffeomorphism
from $\mathbb S^{n-1}\setminus \{y'\}$ to $T_{y'}^*\mathbb S^{n-1}$.

It remains to apply Lemma~\ref{l:ibp} with
$$
\widetilde\varphi:={\varphi\over h},\quad \tilde h:={h\over|y-y''|}<2h^{1/2},\quad
\tilde\rho:=\rho,
$$ 
and use~\eqref{e:psi0-3}.
\end{proof}
%%%%%%%%%%%%%%%%%%%%%%%%%%%%%%%%%%%%%%%%%%%%%%%%%%%%%%%%%%%%%%%%%%%%%%%%%%%%%%%%
Given Lemma~\ref{l:ibpl1}, in order to show~\eqref{e:skb1} it suffices to prove
the following bound:
\begin{equation}
  \label{e:skb3}
\sup_{y_0\in\Lambda_\Gamma}\,\sup_{y''\in B(y_0,h^{1/2})}\int\limits_{\Lambda_\Gamma(C_1 h^\rho)\cap B(y_0,h^{1/2})}|\mathcal K(y,y'';h)|\,dy
\leq Ch^{2(\beta-\varepsilon)}.
\end{equation}
We claim that it is enough to prove the following $L^4$ estimate:
\begin{equation}
  \label{e:horror1}
\sup_{y_0\in\Lambda_\Gamma}\,\sup_{y''\in B(y_0,h^{1/2})}
\int\limits_{B(y_0,h^{1/2})}|\mathcal K(y,y'';h)|^4\,dy
\leq Ch^{8(\beta-\varepsilon)-3\rho(n-1-\delta)-{3\delta\over 2}}.
\end{equation}
Indeed, \eqref{e:skb3} follows by H\"older's inequality from~\eqref{e:horror1} 
and the following corollary of~\eqref{e:AD-estimate-Lebesgue}:
\begin{equation}
  \label{e:eddie}
\|\indic_{\Lambda_\Gamma(C_1h^\rho)\cap B(y_0,h^{1/2})}\|_{L^{4/3}}\leq Ch^{{3\rho\over 4}(n-1-\delta)+{3\delta\over 8}}.
\end{equation}
The proof of~\eqref{e:horror1} is based on taking
the Taylor expansion of the phase function $\varphi$ in~\eqref{e:K-rewritten} around $y=y''=y_0$.
The first term in the expansion is linear in $y-y''$ and gives the Fourier transform of a distorted
version of $\psi_0$; the next terms are $\mathcal O(|y-y_0|^2+|y''-y_0|^2)=\mathcal O(h)$ and can be put into the amplitude
in the integral. 
The $L^4$ norm of the Fourier transform can next be estimated via the additive energy of the distorted
support of $\psi_0$. The proof below relies on this argument, though it does not explicitly use the Fourier transform.
Note that to reduce our integral to Fourier transform we needed to restrict
to $y,y''=y_0+\mathcal O(h^{1/2})$. To show that the contributions
of other $y,y''$ are negligible in Lemma~\ref{l:ibpl1}
we needed the derivative bounds~\eqref{e:psi0-3}, explaining the need
to make $\psi_0$ live on an $h^{\rho/2}$ neighborhood of the limit set
rather than on an $h^{\rho}$ neighborhood.

Using Lemma~\ref{l:symbol-construction}, take $\psi_1(y;h)\in C^\infty(\mathbb S^{n-1};[0,1])$ such that
\begin{equation}
  \label{e:psi1}
\begin{gathered}
\supp(1-\psi_1)\cap B(y_0,h^{1/2})=\emptyset,\quad
\supp\psi_1\subset B(y_0,2h^{1/2});\\
\sup_{\mathbb S^{n-1}}|\partial^\alpha \psi_1|\leq C_\alpha h^{-|\alpha|/2}.
\end{gathered}
\end{equation}
Then to prove~\eqref{e:horror1} is enough to show that
\begin{equation}
  \label{e:horror2}
\sup_{y_0\in\Lambda_\Gamma}\,\sup_{y''\in B(y_0,h^{1/2})}
\int\limits_{\mathbb S^{n-1}}\psi_1(y;h)|\mathcal K(y,y'';h)|^4\,dy
\leq Ch^{8(\beta-\varepsilon)-3\rho(n-1-\delta)-{3\delta\over 2}}.
\end{equation}
By Fubini's Theorem and~\eqref{e:skb2} we have
$$
\int\limits_{\mathbb S^{n-1}}\psi_1(y;h)|\mathcal K(y,y'';h)|^4\,dy
=\int\limits_{(\mathbb S^{n-1})^4}\mathcal K_1(y_1,y_2,y_3,y_4,y'';h)\,dy_1dy_2dy_3dy_4,
$$
where
$$
\begin{aligned}
\mathcal K_1&=h^{4(1-n)}\psi_0(y_1;h)\psi_0(y_2;h)\psi_0(y_3;h)\psi_0(y_4;h)\mathcal K_2,\\
\mathcal K_2&=\Big({|y_1-y''|\cdot |y_3-y''|\over |y_2-y''|\cdot |y_4-y''|}\Big)^{2i/h}
\chi(y_1,y'')\overline{\chi(y_2,y'')}\chi(y_3,y'')\overline{\chi(y_4,y'')}\mathcal K_3,\\
\mathcal K_3&=\int_{\mathbb S^{n-1}}\Big({|y_2-y|\cdot |y_4-y|\over |y_1-y|\cdot |y_3-y|}\Big)^{2i/h}\,\,
\overline{\chi(y_1,y)}\chi(y_2,y)\overline{\chi(y_3,y)}\chi(y_4,y)\psi_1(y;h)\,dy.
\end{aligned}
$$
The next statement shows that $\mathcal K_1$ is very small unless $y_1,y_2,y_3,y_4$
satisfy a certain additive relation. The measure of the set of quadruples $(y_1,y_2,y_3,y_4)$
which do satisfy this relation will later be estimated using additive energy.
%%%%%%%%%%%%%%%%%%%%%%%%%%%%%%%%%%%%%%%%%%%%%%%%%%%%%%%%%%%%%%%%%%%%%%%%%%%%%%%%
\begin{lemm}
  \label{l:ibpl2}
Let $\eta_j=\mathcal G(y_0,y_j)\in T^*_{y_0}\mathbb S^{n-1}$, with $\mathcal G$ defined in~\eqref{e:stpro}, and assume that
\begin{equation}
  \label{e:etacon}
|\eta_1-\eta_2+\eta_3-\eta_4|\geq h^{\rho/2}.
\end{equation}
Then $\mathcal K_1(y_1,y_2,y_3,y_4,y'';h)=\mathcal O(h^\infty)$,
uniformly in $y_0,y_1,y_2,y_3,y_4,y''$.
\end{lemm}
%%%%%%%%%%%%%%%%%%%%%%%%%%%%%%%%%%%%%%%%%%%%%%%%%%%%%%%%%%%%%%%%%%%%%%%%%%%%%%%%
\begin{proof}
It is enough to show that $\mathcal K_3=\mathcal O(h^\infty)$. For that, we write
$$
\begin{gathered}
\mathcal K_3=\int_{\mathbb S^{n-1}}e^{i\varphi/h}a\,dy,\quad
\varphi(y_1,y_2,y_3,y_4,y)=2\sum_{j=1}^4 (-1)^j\log |y_j-y|,\\
a(y_1,y_2,y_3,y_4,y;h)=\overline{\chi(y_1,y)}\chi(y_2,y)\overline{\chi(y_3,y)}\chi(y_4,y)\psi_1(y;h).
\end{gathered}
$$
Put $\eta:=\eta_1-\eta_2+\eta_3-\eta_4\in T^*_{y_0}\mathbb S^{n-1}$.
By~\eqref{e:gide},
$$
\partial_y\varphi(y_1,y_2,y_3,y_4,y_0)=\eta.
$$
Since $\psi_1$ is supported in $B(y_0,2h^{1/2})$, we have for some global constant $C$,
$$
|\eta|-Ch^{1/2}\leq|\partial_y \varphi|\leq |\eta|+Ch^{1/2}\quad\text{on }\supp a.
$$
By~\eqref{e:etacon}, we get for $h$ small enough,
$$
|\eta|/2\leq |\partial_y \varphi|\leq 2|\eta|\quad\text{on }\supp a.
$$
It remains to apply Lemma~\ref{l:ibp} with
$$
\widetilde\varphi:={\varphi\over h},\quad
\tilde h:={h\over |\eta|}\leq h^{1-\rho/2},\quad
\tilde\rho:={1\over 2-\rho},
$$
and use~\eqref{e:psi1}.
\end{proof}
%%%%%%%%%%%%%%%%%%%%%%%%%%%%%%%%%%%%%%%%%%%%%%%%%%%%%%%%%%%%%%%%%%%%%%%%%%%%%%%%
Since $\chi$ is supported away from the diagonal, there exists a constant $C_1$
such that on $\supp \mathcal K_1$, we have $|\mathcal G(y_0,y_j)|\leq C_1$
for $j=1,2,3,4$. By~\eqref{e:psi0-1},
the additive energy bound~\eqref{e:ae-estimate} with $\alpha=Ch^{\rho/2}$
implies that uniformly in $y''$,
\begin{equation}
  \label{e:ae-mesb}
\mu_L\big(\{(y_1,y_2,y_3,y_4)\in \supp\mathcal K_1\mid
|\eta_1-\eta_2+\eta_3-\eta_4|\leq h^{\rho/2}\}\big)\leq Ch^{2\rho(n-1)-{3\rho\over 2}\delta+{\rho\over 2}\beta_E}
\end{equation}
where $\eta_j:=\mathcal G(y_0,y_j)\in T^*_{y_0}\mathbb S^{n-1}$.
We also have by~\eqref{e:psi1}
$$
\sup|\mathcal K_1|\leq Ch^{{7\over 2}(1-n)}.
$$
Together with~\eqref{e:ae-mesb} and Lemma~\ref{l:ibpl2}, this gives
$$
\sup_{y_0\in\Lambda_\Gamma}\sup_{y''\in B(y_0,h^{1/2})}
\int_{\mathbb S^{n-1}}\psi_1(y;h)|\mathcal K(y,y'';h)|^4\,dy\leq
Ch^{(2\rho-{7\over 2})(n-1)-{3\rho\over 2}\delta+{\rho\over 2}\beta_E}.
$$
This implies~\eqref{e:horror2} as long as
$$
\Big(2\rho-{7\over 2}\Big)(n-1)-{3\rho\over 2}\delta+{\rho\over 2}\beta_E\geq
8(\beta-\varepsilon)-3\rho(n-1-\delta)-{3\delta\over 2}.
$$
Recalling~\eqref{e:beta-ae}, this inequality becomes
$$
\Big(5(n-1)-{9\over 2}\delta+{1\over 2}\beta_E\Big)(1-\rho)\leq 8\varepsilon.
$$
The last inequality holds when $\rho$ is close to 1 depending on $\varepsilon$,
finishing the proof of Theorem~\ref{t:ae-reduction}.

%%%%%%%%%%%%%%%%%%%%%%%%%%%%%%%%%%%%%%%%%%%%%%%%%%%%%%%%%%%%%%%%%%%%%%%%%%%%%%%%
\section{General bounds on additive energy}
  \label{s:ae-combinatorial}
  
In this section, we prove a new bound (Theorem~\ref{t:ae-combinatorial})
on the additive energy of general Ahlfors-David regular sets
(not just those arising from hyperbolic quotients).

There is substantial conflict of notation between
the current section and the rest of the paper. However, this should not cause a problem
since the two are completely decoupled from each other.
This makes it possible to use simpler notation in the current section.
\renewcommand{\colon}{\ |\ }
\newcommand{\E}{\mathcal E_A}

We first recall the definition of Ahlfors-David regularity:

%%%%%%%%%%%%%%%%%%%%%%%%%%%%%%%%%%%%%%%%%%%%%%%%%%%%%%%%%%%%%%%%%%%%%%%%%%%%%%%%
\begin{ADRegularDefn}
Let $(\metSpace,\metricChar)$ be a complete metric space with more than one element. We say a closed set $\mathcal X \subset \metSpace$ is \textbf{$\setdim$--regular} with constant $C_{\mathcal X}$ if for all $x\in \mathcal X$ we have
\begin{equation*}
C_{\mathcal X}^{-1}r^\setdim\ \leq\ \mu_{\setdim}(\mathcal X\cap B(x,r))\ \leq\ C_{\mathcal X} r^{\setdim},\quad
0<r<\diam(\mathcal M)
\end{equation*}
where $B(x,r)$ is the metric ball of radius $r$ centered at $x$ and $\mu_{\setdim}$ is the $\setdim$--dimensional Hausdorff measure.
\end{ADRegularDefn}
%%%%%%%%%%%%%%%%%%%%%%%%%%%%%%%%%%%%%%%%%%%%%%%%%%%%%%%%%%%%%%%%%%%%%%%%%%%%%%%%

%%%%%%%%%%%%%%%%%%%%%%%%%%%%%%%%%%%%%%%%%%%%%%%%%%%%%%%%%%%%%%%%%%%%%%%%%%%%%%%%
\begin{example}[Cantor set]
Let $\X\subset[0,1]$ be the middle third Cantor set. Then $\X$ is $(\log2/\log3)$--regular.
\end{example}
%%%%%%%%%%%%%%%%%%%%%%%%%%%%%%%%%%%%%%%%%%%%%%%%%%%%%%%%%%%%%%%%%%%%%%%%%%%%%%%%
%
%
%%%%%%%%%%%%%%%%%%%%%%%%%%%%%%%%%%%%%%%%%%%%%%%%%%%%%%%%%%%%%%%%%%%%%%%%%%%%%%%%
\begin{example}[The limit set of a hyperbolic group]\label{HGroupExample}
Let $\Gamma$ be an $n$-dimensional convex co-compact hyperbolic group, let $\Lambda_{\Gamma}\subset S^{n-1}$ be the limit set, and let $\setdim=\setdim(\Gamma)$ be the critical exponent. Then $\Lambda_{\Gamma}$ is $\setdim$--regular. See~\S\ref{s:introad} and~\cite[Theorem 7]{Sullivan}.
\end{example}
%%%%%%%%%%%%%%%%%%%%%%%%%%%%%%%%%%%%%%%%%%%%%%%%%%%%%%%%%%%%%%%%%%%%%%%%%%%%%%%%
In this section, it is convenient to use the following definition of additive energy, which is different from Definition~\ref{d:ae}. In~\S\ref{s:minkowski} we will reduce one quantity to the other.
%%%%%%%%%%%%%%%%%%%%%%%%%%%%%%%%%%%%%%%%%%%%%%%%%%%%%%%%%%%%%%%%%%%%%%%%%%%%%%%%
\begin{defn}[Additive energy]
\label{d:aespecial}
Let $\X\subset[0,1]^n$ and let $\mu$ be an outer measure on~$\X$ with $0<\mu(\X)<\infty$.
For $\h>0$, define the (scale $\h$) \textbf{additive energy}
\begin{equation}
  \label{e:AESpecial}
\E(\X,\mu,\h)=\mu^4(\{(x_1,x_2,x_3,x_4)\in \X^4\colon |x_1-x_2 +x_3-x_4|<\h\}).
\end{equation}
If the measure $\mu$ is apparent from context, we will write $\E(\X,\h)$ in place of $\E(\X,\mu,\h)$.
\end{defn}
%%%%%%%%%%%%%%%%%%%%%%%%%%%%%%%%%%%%%%%%%%%%%%%%%%%%%%%%%%%%%%%%%%%%%%%%%%%%%%%%

We will now restrict attention to the metric space $[0,1]$. The main result of this section is
%%%%%%%%%%%%%%%%%%%%%%%%%%%%%%%%%%%%%%%%%%%%%%%%%%%%%%%%%%%%%%%%%%%%%%%%%%%%%%%%
\begin{theo}[Regular sets have small additive energy]\label{ADRegularSmallAddEnergyThm}\label{t:ae-combinatorial}
Let $\X\subset[0,1]$ be a $\setdim$--regular set with regularity constant $C_\X$ and $\setdim<1$. Then
\begin{equation}
\E(\X,\h) =\E(\X,\mu_{\setdim},\h)\leq \widetilde C\h^{\setdim+\beta_\X}
\end{equation}
for some $\beta_\X>0$ and some $\widetilde C>0$. In particular, we can choose
\begin{equation}\label{finalEpsBd}
\beta_\X= \betaXBound,
\end{equation}
where $\mathbf{K}$ is a large absolute constant; the constant $\widetilde C$ depends only on $\setdim$ and $C_\X$.
\end{theo}
%%%%%%%%%%%%%%%%%%%%%%%%%%%%%%%%%%%%%%%%%%%%%%%%%%%%%%%%%%%%%%%%%%%%%%%%%%%%%%%%
The exact bound \eqref{finalEpsBd} is not important and can certainly be improved. The key point is that $\beta_\X$ does not depend on $\h$. 

Heuristically, Theorem~\ref{ADRegularSmallAddEnergyThm} says that if we choose points $x_1,x_2,x_3\in \X$ at random (using normalized $\setdim$--dimensional Hausdorff measure on $\X$), then the probability that there will exist a point $x_4$ with $|x_1-x_2+x_3-x_4|<\h$ is at most $\widetilde C\h^{\beta_\X}$. For $\beta_\X>0$ and $\h<\!\!<\widetilde C^{-1/\beta_\X}$, this quantity is much smaller than 1.

\noindent\textbf{Remarks}. (i) If $I\subset\RR$ is an interval of finite length and $\X\subset I$ is a $\setdim$--regular set, we can apply Theorem \ref{t:ae-combinatorial} to $\X$ by rescaling $I$ appropriately. More interestingly, if $I\subset\RR$ is a (possibly infinite) interval and $\X\subset I$ is a is a $\setdim$--regular set, then $\X^\prime=\X\cap[0,1]$ might not be $\setdim$--regular (for example, it might be the union of a $\setdim$--regular set $\X_0$ and a point far from $\X_0$), but in many instances we can still apply Theorem \ref{t:ae-combinatorial}. This is discussed further in~\S\ref{s:minkowski}.

\noindent (ii) Theorem~\ref{ADRegularSmallAddEnergyThm} is a statement about the additive energy of $\setdim$--regular sets. In Proposition \ref{VarADRegularSmallAddEnergyThm} below, we state an alternate version of the theorem that bounds the additive energy of sets that are unions of intervals of length $\h$ (and which satisfy conditions analogous to being $\setdim$--regular).

The main application of Theorem~\ref{ADRegularSmallAddEnergyThm} will be a bound on the additive energy of the limit set of a Fuchsian group. Informally, the result is as follows. Let $\Lambda_{\Gamma}\subset S^1$ be the limit set of a Fuchsian group with critical exponent $\setdim(\Gamma)$ and let $\mathcal{G}(y_0,\cdot): S^1\to\RR$ be the stereographic projection defined by~\eqref{e:stpro}. Then the image of $\Lambda_{\Gamma}$ is a $\setdim(\Gamma)$-regular set. The set $\mathcal{G}(y_0,\Lambda_{\Gamma})$ transforms naturally under a certain type of group operation, and this allows us to restrict attention to the interval $[0,1]$. We then apply Theorem~\ref{ADRegularSmallAddEnergyThm} to bound the additive energy of $\mathcal{G}(y_0,\Lambda_{\Gamma})\cap[0,1]$. The exact statement and proof are in~\S\ref{s:ae}. 

%%%%%%%%%%%%%%%%%%%%%%%%%%%%%%%%%%%%%%%%%%%%%%%%%%%%%%%%%%%%%%%%%%%%%%%%%%%%%%%%
\subsection{Ideas behind the proof of Theorem~\ref{t:ae-combinatorial}}
\label{s:ae-ideas}

%%%%%%%%%%%%%%%%%%%%%%%%%%%%%%%%%%%%%%%%%%%%%%%%%%%%%%%%%%%%%%%%%%%%%%%%%%%%%%%%
\subsubsection{Ahlfors-David regularity and arithmetic progressions}
A $\setdim$--regular subset of $[0,1]$ cannot contain long arithmetic progressions. More precisely, suppose $P\subset \X$ is an arithmetic progression of length $|P|$ and spacing $t>0$. Let $I\subset[0,1]$ be the interval of length $t|P|$ centered around $P$. If we place an interval of radius $t/2$ around each point of $P$, then $\mu_\delta(\X\cap I)\geq C_\X^{-1}|P| (t/2)^{\setdim}$. On the other hand,
$\mu_\delta(\X\cap I)\leq C_\X(|P| t)^\setdim$. If $|P|$ is sufficiently large (depending on $C_\X$ and $\setdim$), we arrive at a contradiction, provided that $\delta<1$. In fact more is true. If $P$ is not contained in $\X$ but merely meets $\X$ in many points, the argument still applies as well. Finally, the argument is not affected if we perturb the points of $P$ slightly. We say that $\X$ \emph{strongly avoids long arithmetic progressions}. 

Note that this argument relies on the fact that $\delta<1$. If instead $\X$ is a subset of $[0,1]^n$ for $n>1$ (or a more general metric space) 
and $\delta\geq 1$ then the argument fails. In~\S\ref{higherDimRemark} we will discuss this phenomenon further.

%%%%%%%%%%%%%%%%%%%%%%%%%%%%%%%%%%%%%%%%%%%%%%%%%%%%%%%%%%%%%%%%%%%%%%%%%%%%%%%%
\subsubsection{Small doubling and additive structure}
If $A\subset\ZZ$ is a finite set and $|A+A|<K|A|$, what can we say about $A$? There is a family of theorems in additive combinatorics that say that if $K$ is small then $A$ must have additive structure. The most famous of these is Fre{\u\i}man's theorem~\cite{freiman}, which says that $A$ must be contained in a generalized arithmetic progression. For our purposes however, we will obtain stronger results by using a variant of Fre{\u\i}man's theorem due to
Sanders~\cite{Sanders}, which makes the weaker claim that $A$ has large intersection with a convex progression.

When combined with the ideas discussed above, we conclude that if $\X$ is a regular set then $\X$ cannot have maximally large additive energy. Unfortunately, the sort of bounds that one obtains from this argument are very weak---far too weak to obtain the polynomial
in $\alpha$ improvement of Theorem~\ref{ADRegularSmallAddEnergyThm}\footnote{However, if the polynomial Fre{\u\i}man-Ruzsa conjecture is proved then this theorem may be employed directly, and the subsequent steps would not be needed.}.
%
%%%%%%%%%%%%%%%%%%%%%%%%%%%%%%%%%%%%%%%%%%%%%%%%%%%%%%%%%%%%%%%%%%%%%%%%%%%%%%%%
\subsubsection{Multiscale analysis of Ahlfors-David regular sets}
If $\X$ is a regular set, we can examine $\X$ at many intermediate scales between $\h$ and 1; there will be roughly $|\log \h|$ scales total. We use the arguments above to get a small gain in the scale-$\h$ additive energy of $\X$ at each intermediate scale. These gains will compound with each intermediate scale, and the total gain will be large enough to obtain Theorem~\ref{ADRegularSmallAddEnergyThm}.
%
%
%

%%%%%%%%%%%%%%%%%%%%%%%%%%%%%%%%%%%%%%%%%%%%%%%%%%%%%%%%%%%%%%%%%%%%%%%%%%%%%%%%
\subsection{Ahlfors-David regular sets and additive structure}

We start the proof of Theorem~\ref{ADRegularSmallAddEnergyThm} by exploring the implications of $\delta$-regularity for the additive structure of the set $\mathcal X$.
%%%%%%%%%%%%%%%%%%%%%%%%%%%%%%%%%%%%%%%%%%%%%%%%%%%%%%%%%%%%%%%%%%%%%%%%%%%%%%%%
\begin{defn}
An \textbf{arithmetic progression} is a set of the form
$$
\{a-\ell q, a-(\ell-1)q,\ldots, a, a+q,\ldots,a+\ell q\}\subset\ZZ,
$$
where $a,q\in\ZZ$, $q\neq 0$, and $\ell\geq 0$.
\end{defn}
%%%%%%%%%%%%%%%%%%%%%%%%%%%%%%%%%%%%%%%%%%%%%%%%%%%%%%%%%%%%%%%%%%%%%%%%%%%%%%%%
%
%%%%%%%%%%%%%%%%%%%%%%%%%%%%%%%%%%%%%%%%%%%%%%%%%%%%%%%%%%%%%%%%%%%%%%%%%%%%%%%%
\begin{defn}
  \label{strongAvoid}
Let $A\subset\ZZ$ be a finite set. We say that $A$ \textbf{strongly avoids long arithmetic progressions} (with parameter $S$) if for all $\epsilon>0$, there is a number $S=S(\epsilon)$ so that for all arithmetic progressions $P\subset\ZZ$ with $|P\cap A|\geq\epsilon|P|$, we have $|P|\leq S(\varepsilon).$ 
\end{defn}
%%%%%%%%%%%%%%%%%%%%%%%%%%%%%%%%%%%%%%%%%%%%%%%%%%%%%%%%%%%%%%%%%%%%%%%%%%%%%%%%
If $\X\subset t\ZZ$ for $0<t<1$, then we say that $\X$ strongly avoids long arithmetic progressions if $t^{-1}\X$ does, i.e.~we simply re-scale $\X$ so that it lies on the integer lattice.
%%%%%%%%%%%%%%%%%%%%%%%%%%%%%%%%%%%%%%%%%%%%%%%%%%%%%%%%%%%%%%%%%%%%%%%%%%%%%%%%
%
%%%%%%%%%%%%%%%%%%%%%%%%%%%%%%%%%%%%%%%%%%%%%%%%%%%%%%%%%%%%%%%%%%%%%%%%%%%%%%%%
\begin{defn}  \label{convexProgression}
A  ($d$-dimensional) \textbf{centered convex progression} is a triple $(B,\Lambda,\varphi)$, where $B\subset\RR^d$ is a centrally symmetric convex set, $\Lambda\subset\RR^d$ is a lattice, and $\varphi : \Lambda\to\ZZ$ is a linear map. We will primarily be interested in the image $\varphi(B\cap\Lambda)$. Sets of this form are generalizations of %generalized
arithmetic progressions. If $\varphi$ is injective on $B\cap\Lambda$, we say that $(B,\Lambda,\varphi)$ is \textbf{proper}.
\end{defn}
%%%%%%%%%%%%%%%%%%%%%%%%%%%%%%%%%%%%%%%%%%%%%%%%%%%%%%%%%%%%%%%%%%%%%%%%%%%%%%%%
The next lemma shows that a $d$-dimensional centered convex progression
$(B,\Lambda,\varphi)$
of \emph{cardinality} $N:=|\varphi(B\cap\Lambda)|$ can be embedded into
a centered convex progression $(B',\Lambda',\varphi')$ whose \emph{size}
$|B'\cap\Lambda'|$ is bounded by $N$ times a $d$-dependent constant.
%%%%%%%%%%%%%%%%%%%%%%%%%%%%%%%%%%%%%%%%%%%%%%%%%%%%%%%%%%%%%%%%%%%%%%%%%%%%%%%%
\begin{lemm}[Cardinality vs. size]\label{containingProgressionInWellBehavedSet}
Let $(B,\Lambda, \varphi)$ be a $d$-dimensional centered convex progression. Then there is some $d^\prime\leq d$ and a $d^\prime$-dimensional centered convex progression $(B^\prime,\Lambda^\prime,\varphi^\prime)$ with
\begin{equation}
|B^\prime\cap\Lambda^\prime|\leq  2^{(d+2)^2} |\varphi(B\cap\Lambda)|\label{BpSize}
\end{equation}
such that
\begin{equation}
\varphi(B\cap\Lambda)\subset \varphi^\prime(B^\prime\cap\Lambda^\prime).\label{BpInclusion}
\end{equation}
\end{lemm}
%%%%%%%%%%%%%%%%%%%%%%%%%%%%%%%%%%%%%%%%%%%%%%%%%%%%%%%%%%%%%%%%%%%%%%%%%%%%%%%%
\begin{proof}
We recall \cite[Corollary 4.2]{Tao2}:
%%%%%%%%%%%%%%%%%%%%%%%%%%%%%%%%%%%%%%%%%%%%%%%%%%%%%%%%%%%%%%%%%%%%%%%%%%%%%%%%
\begin{prop}
Let  $(B,\Lambda,\varphi)$ be a $d$--dimensional centered convex progression. Then there exists a $d'$-dimensional \textbf{proper} centered convex progression  $(B^{\prime\prime},\Lambda^{\prime\prime},\varphi^{\prime\prime})$ for some $d'\leq d$ such that we have the inclusions
\begin{align}
&\varphi(B\cap\Lambda)\subset \varphi^{\prime\prime}\big((2^{d-d'+1}B^{\prime\prime})\cap\Lambda^{\prime\prime}\big),\label{BBppInclusion1}\\
& \varphi\big((2B^{\prime\prime})\cap\Lambda^{\prime\prime}\big)\subset \varphi(B\cap\Lambda),\label{BBppInclusion2}
\end{align}
where $tB=\{x\in\RR^d\colon t^{-1}x\in B\}$.
\end{prop}
%%%%%%%%%%%%%%%%%%%%%%%%%%%%%%%%%%%%%%%%%%%%%%%%%%%%%%%%%%%%%%%%%%%%%%%%%%%%%%%%
Let $B^\prime:=2^{d-d'+1}B^{\prime\prime}$, $\Lambda^\prime:=\Lambda^{\prime\prime}$, and $\varphi^\prime:=\varphi^{\prime\prime}$. Then \eqref{BpInclusion} is satisfied. 
Since $(B^{\prime\prime},\Lambda^{\prime\prime},\varphi^{\prime\prime})$ is proper, we have 
\begin{equation}
  \label{sizeContainment}
\big|(2^{-d+d^\prime-1}B^\prime)\cap\Lambda^\prime\big|=\big|\varphi^\prime\big((2^{-d+d^\prime-1}B^\prime)\cap\Lambda^\prime\big)\big|\leq |\varphi(B\cap\Lambda)|. 
\end{equation}
By~\cite[Lemma 3.3]{Tao2}, we have
\begin{equation}\label{ContainBpInDilate}
|B^\prime\cap\Lambda^\prime| \leq 2^{(d+2)^2} \big|(2^{-d+d^\prime-1}B^\prime)\cap\Lambda^\prime\big|.
\end{equation}
Combining \eqref{sizeContainment} and \eqref{ContainBpInDilate} we obtain \eqref{BpSize}.
\end{proof}
%%%%%%%%%%%%%%%%%%%%%%%%%%%%%%%%%%%%%%%%%%%%%%%%%%%%%%%%%%%%%%%%%%%%%%%%%%%%%%%%
We next recall \cite[Lemma 3.36]{Tao}:
%%%%%%%%%%%%%%%%%%%%%%%%%%%%%%%%%%%%%%%%%%%%%%%%%%%%%%%%%%%%%%%%%%%%%%%%%%%%%%%%
\begin{prop}
Let $B\subset\RR^d$ be a centrally symmetric convex set and let $\Lambda\subset\RR^d$ be a lattice. Suppose that the $\RR$ span of $B\cap\Lambda$ has dimension $r$. Then there exists an $r$--tuple $w=(w_1,\ldots,w_r)$ with $w_1,\ldots,w_r$ linearly independent vectors in $\Lambda$, and an $r$--tuple of integers $N=(N_1,\ldots,N_r)$ so that
\begin{equation}
[-N,N]\cdot w\ \subset\ B\cap\Lambda\ \subset\ [-r^{2r}N,r^{2r}N]\cdot w.
\end{equation}
Here
\begin{equation*}
[-N,N]\cdot w= \{ \ell_1w_1+\ldots+\ell_rw_r\colon -N_j\leq \ell_j\leq N_j,\ j=1,\ldots,r\}.
\end{equation*}
\end{prop}
%%%%%%%%%%%%%%%%%%%%%%%%%%%%%%%%%%%%%%%%%%%%%%%%%%%%%%%%%%%%%%%%%%%%%%%%%%%%%%%%
%
%
%%%%%%%%%%%%%%%%%%%%%%%%%%%%%%%%%%%%%%%%%%%%%%%%%%%%%%%%%%%%%%%%%%%%%%%%%%%%%%%%
\begin{corr}\label{trappedInAGAP}
Let $B\subset\RR^d$ be a centrally symmetric convex set and let $\Lambda\subset\RR^d$ be a lattice. Suppose that the $\RR$ span of $B\cap\Lambda$ has dimension $r$. Then there exists an $r$--tuple $w=(w_1,\ldots,w_r)$ with $w_1,\ldots,w_r$ linearly independent vectors in $\Lambda$, and an $r$--tuple of integers $N=(N_1,\ldots,N_r)$ so that
\begin{equation}
B\cap\Lambda\ \subset\ [-N,N]\cdot w,
\end{equation}
and
\begin{equation}
\prod_{j=1}^r (2N_j+1) \leq 3^r r^{2r^2}|B\cap\Lambda|. 
\end{equation}
\end{corr}
%%%%%%%%%%%%%%%%%%%%%%%%%%%%%%%%%%%%%%%%%%%%%%%%%%%%%%%%%%%%%%%%%%%%%%%%%%%%%%%%
%
%%%%%%%%%%%%%%%%%%%%%%%%%%%%%%%%%%%%%%%%%%%%%%%%%%%%%%%%%%%%%%%%%%%%%%%%%%%%%%%%
\begin{lemm}[Arithmetic progressions inside convex progressions]\label{convexProgressionsContainArithmeticProgressions}
Let  $(B,\Lambda,\varphi)$ be a $d$--dimensional centered convex progression and let
$$
A\subset \varphi(B\cap\Lambda),\quad
|A|\geq\epsilon|\varphi(B\cap\Lambda)|.
$$
Then there exists a (one dimensional) arithmetic progression $P\subset\ZZ$ with
\begin{equation}\label{PisBig}
|P|\geq |A|^{1/d}
\end{equation}
and
\begin{equation}\label{PMeetsA}
|P\cap A|\geq d^{-3(d+2)^2}\epsilon|P|.
\end{equation}
\end{lemm}
%%%%%%%%%%%%%%%%%%%%%%%%%%%%%%%%%%%%%%%%%%%%%%%%%%%%%%%%%%%%%%%%%%%%%%%%%%%%%%%%
\begin{proof}
Let $(B^\prime, \Lambda^\prime,\varphi^\prime)$ be a centered convex progression obeying \eqref{BpSize} and \eqref{BpInclusion}. In particular, we have $A\subset\varphi^\prime(B^\prime\cap\Lambda^\prime)$ and
\begin{equation}
|B^\prime\cap\Lambda^\prime|\leq 2^{(d+2)^2}\epsilon^{-1}|A|.
\end{equation}
By Corollary \ref{trappedInAGAP}, there is a number $r\leq d$, an $r$--tuple $w=(w_1,\ldots,w_r)$ of
linearly independent vectors in~$\Lambda'$, and an $r$--tuple of integers  $N=(N_1,\ldots,N_r)$ so that
\begin{equation}
  \label{varphiInvInBox}
(\varphi^\prime)^{-1}(A)\cap B^\prime\cap\Lambda^\prime\ \subset\
B^\prime\cap\Lambda^\prime\ \subset\ [-N,N]\cdot w,
\end{equation}
and
\begin{equation}\label{boundOnSizeOfNi}
\prod_{j=1}^r (2N_j+1)\ \leq\ 3^r r^{2r^2}|B^\prime\cap\Lambda^\prime|\ \leq\ 2^{(d+3)^2} d^{2d^2}\epsilon^{-1} |A|. 
\end{equation}
Now, \eqref{varphiInvInBox} implies that
\begin{equation}\label{AInBox}
A \subset \varphi'([-N,N]\cdot w).
\end{equation}
We can assume that $\varphi'(w_i)\neq 0$ for each index $i$ for which $N_i\neq 0$, since if $\varphi'(w_i)=0$ then we could set $N_i=0$ and both \eqref{boundOnSizeOfNi} and \eqref{AInBox} remain satisfied.

By re-indexing if necessary, we can assume that $N_1\geq N_j$ for all $j=2,\ldots,r$. Partition the set $[-N,N]\cdot w$ into $\prod_{j=2}^r (2N_j+1)$ disjoint sets of the form
\begin{equation*}
\Big\{\ell w_1 + \sum_{j=2}^r x_jw_j \,\,\Big|\,\, -N_1 \leq \ell \leq N_1\Big\}.
\end{equation*}
Each of these sets has cardinality 
\begin{equation*}
(2N_1+1) \geq  \Big(\prod_{j=1}^r (2N_j+1)\Big)^{1/r}\geq|A|^{1/d}.
\end{equation*}
By \eqref{boundOnSizeOfNi}, \eqref{AInBox}, and pigeonholing, at least one of these sets $Y$ must satisfy
\begin{equation}
|Y\cap (\varphi^\prime)^{-1}(A)|\ \geq\ 2^{-(d+3)^2} d^{-2d^2}\epsilon|Y|\ \geq\  d^{-3(d+2)^2}\epsilon|Y|.
\end{equation}
For the last inequality we may assume that $d\geq 2$, since otherwise we may put $P:=\varphi(B\cap\Lambda)$.

Let $P=\varphi^\prime(Y)$. By assumption $\varphi^\prime$ is injective on $Y$, so $P$ satisfies \eqref{PisBig} and~\eqref{PMeetsA}.
\end{proof}
%%%%%%%%%%%%%%%%%%%%%%%%%%%%%%%%%%%%%%%%%%%%%%%%%%%%%%%%%%%%%%%%%%%%%%%%%%%%%%%%
The following Proposition is a direct corollary of~\cite[Theorem 1.4]{Sanders}:
%%%%%%%%%%%%%%%%%%%%%%%%%%%%%%%%%%%%%%%%%%%%%%%%%%%%%%%%%%%%%%%%%%%%%%%%%%%%%%%%
\begin{prop}[Additive structure]
  \label{SchoenProp}
Let $A\subset\ZZ$ and suppose $|A+A|\leq K|A|$. Then there is a $d\leq K_0 \log^4K$ dimensional centered convex progression $(B,\ZZ^d,\varphi)$ and an offset $x\in\ZZ$ so that
\begin{equation}
|\varphi(B\cap\ZZ^d)|\leq \exp[K_0 \log^4 K]\cdot |A|
\end{equation}
and
\begin{equation}
|(A-x)\cap \varphi(B\cap\ZZ^d)|\geq 	\exp[-K_0\log^4 K]\cdot |A|.
\end{equation}
Here $K_0>0$ is an absolute constant.
\end{prop}
%%%%%%%%%%%%%%%%%%%%%%%%%%%%%%%%%%%%%%%%%%%%%%%%%%%%%%%%%%%%%%%%%%%%%%%%%%%%%%%%
Applying Lemma~\ref{convexProgressionsContainArithmeticProgressions}, we obtain the following corollary.
%%%%%%%%%%%%%%%%%%%%%%%%%%%%%%%%%%%%%%%%%%%%%%%%%%%%%%%%%%%%%%%%%%%%%%%%%%%%%%%%
\begin{corr}
  \label{corrAdditiveStrBigGAPs}
Let $A\subset \ZZ$ and suppose that $|A+A|\leq K|A|$. Then there is an arithmetic progression $P$ so that
\begin{equation}
|P|\geq |A|^{(K_1 \log^4K)^{-1}},
\end{equation}
and
\begin{equation}\label{ACapPSize}
|A\cap P |\geq e^{-K_1 \log^9 K}|P|.
\end{equation}
Here $K_1>0$ is an absolute constant. The term $ \log^9$ in  \eqref{ACapPSize} could be replaced with $\log^{8+o(1)}$, but we will not worry about these small optimizations.
\end{corr}
%%%%%%%%%%%%%%%%%%%%%%%%%%%%%%%%%%%%%%%%%%%%%%%%%%%%%%%%%%%%%%%%%%%%%%%%%%%%%%%%
%
%%%%%%%%%%%%%%%%%%%%%%%%%%%%%%%%%%%%%%%%%%%%%%%%%%%%%%%%%%%%%%%%%%%%%%%%%%%%%%%%
\begin{corr}
  \label{corrAPAP}
Let $A\subset\ZZ$. Suppose that $A$ strongly avoids long arithmetic progressions (with parameter $S(\varepsilon)$), and $|A+A|\leq K|A|$. Then
\begin{equation}\label{corrBoundOnSizeA}
|A|\leq \Big(S(e^{-K_1 \log^9 K})\Big)^{K_1\log^4 K},
\end{equation}
where $K_1$ is the absolute constant from Corollary~\ref{corrAdditiveStrBigGAPs}.
\end{corr}
%%%%%%%%%%%%%%%%%%%%%%%%%%%%%%%%%%%%%%%%%%%%%%%%%%%%%%%%%%%%%%%%%%%%%%%%%%%%%%%%
%
%%%%%%%%%%%%%%%%%%%%%%%%%%%%%%%%%%%%%%%%%%%%%%%%%%%%%%%%%%%%%%%%%%%%%%%%%%%%%%%%
\begin{prop}[Ahlfors-David regular sets avoid arithmetic progressions]
  \label{ADStronglyAvoids} 
Let $\X\subset[0,1]$ be a $\setdim$--regular set with regularity constant $C$ and let $0<\h<1$. Then $\X(\h)\cap \h\ZZ$ strongly avoids long arithmetic progressions
(here $\X(\h)$ is the $\h$-neighborhood of $\X$). In particular, we can take $S(\epsilon)=(10C^2\epsilon^{-1})^{\frac{1}{1-\setdim}}.$
\end{prop}
%%%%%%%%%%%%%%%%%%%%%%%%%%%%%%%%%%%%%%%%%%%%%%%%%%%%%%%%%%%%%%%%%%%%%%%%%%%%%%%%
\begin{proof}
Let $P\subset[0,1]\cap \h\ZZ$ be a proper arithmetic progression. In particular, $P$ has spacing $t\geq \h$.
Assume that $|P\cap\X(\h)|>\epsilon|P|$.
For each point $y\in P\cap \X(\h)$, the ball $B(y,2t)$ contains the ball of radius
$t$ centered at some point of $\X$. 
By \eqref{defnADRegular} we have
\begin{equation}
\sum_{y\in P\cap\X(\h)}\mu_{\setdim}(\X\cap B(y,2t))\geq C^{-1}\epsilon|P|t^{\setdim}.
\end{equation}
On the other hand, the balls $B(y,2t)$ are at most five-fold overlapping, and $P$ is contained in an interval $J$ of length $(|P|+3)t\leq 2|P|t$ (unless $|P|<3$ in which case $|P|<S(\epsilon)$ automatically). By \eqref{defnADRegular} we have
\begin{equation}
\sum_{y\in P\cap\X(\h)}\mu_{\setdim}(\X\cap B(y,2t))\leq 5\mu_{\setdim}(\X\cap J)\leq  10C(t|P|)^{\setdim}.
\end{equation}
We conclude that $|P|\leq (10C^2\epsilon^{-1})^{\frac{1}{1-\setdim}}$.
\end{proof}
%%%%%%%%%%%%%%%%%%%%%%%%%%%%%%%%%%%%%%%%%%%%%%%%%%%%%%%%%%%%%%%%%%%%%%%%%%%%%%%%
Combining Corollary~\ref{corrAPAP} and Proposition~\ref{ADStronglyAvoids}, we get
%
%
%%%%%%%%%%%%%%%%%%%%%%%%%%%%%%%%%%%%%%%%%%%%%%%%%%%%%%%%%%%%%%%%%%%%%%%%%%%%%%%%
\begin{corr}\label{ADRegularSetsInIntervals}
Let $\X\subset[0,1]$ be a $\setdim$--regular set with regularity constant $C$. Let $0<\h<1$ and $A\subset \X(\h)\cap\h\ZZ$, and suppose $|A+A|\leq K|A|$. Then
\begin{equation}
|A|\leq \exp\big[K_3(1-\setdim)^{-1}(\log C)\log^{13}K\big]
\end{equation}
for some absolute constant $K_3$.
\end{corr}
%%%%%%%%%%%%%%%%%%%%%%%%%%%%%%%%%%%%%%%%%%%%%%%%%%%%%%%%%%%%%%%%%%%%%%%%%%%%%%%%
%
%
%

%%%%%%%%%%%%%%%%%%%%%%%%%%%%%%%%%%%%%%%%%%%%%%%%%%%%%%%%%%%%%%%%%%%%%%%%%%%%%%%%
\subsection{Ahlfors-David regular trees}
The key to proving Theorem \ref{ADRegularSmallAddEnergyThm} will be to analyze $\X$ at many scales. Heuristically, if $\h=M^{-N}$ for $M,N$ positive integers, then $\X$ can naturally be analyzed at the $N$ scales $1,M^{-1},\ldots,M^{-N}$. On each scale, we will get a small gain in the scale-$\h$ additive energy of $\X$. In order to keep track of $\X(\h)$ at these different scales we will construct an object called a tree.

%%%%%%%%%%%%%%%%%%%%%%%%%%%%%%%%%%%%%%%%%%%%%%%%%%%%%%%%%%%%%%%%%%%%%%%%%%%%%%%%
\begin{defn} [Trees]
A \textbf{(rooted) tree} of height $\ell\in\mathbb Z_{\geq 0}$ is a connected acyclic graph with a distinguished vertex (called the root). Once we have specified the root of a tree, each vertex has a well-defined height (i.e.~its distance from the root), and we say that one vertex $v$ is a parent of another vertex $v^\prime$ if $v$ and $v^\prime$ are adjacent and $v$ has smaller height. 

More formally, a (rooted) tree is a quadruple $(V,H,p,\ell)$, where
\begin{itemize}
\item $V$ is a finite set of \textbf{vertices};
\item $H:V\to \{0,\dots,\ell\}$ is the \textbf{height} function, and
$H^{-1}(0)$ consists of a single vertex called the \textbf{root};
\item $p:V\setminus H^{-1}(0)\to V$ is the \textbf{parent} function,
and $p(H^{-1}(t))\subset H^{-1}(t-1)$ for
$t=1,\dots,\ell$.
\end{itemize}
We denote by $\mathcal {V}_t(T):=H^{-1}(t)$ the set of vertices of height $t$.
\end{defn}
%%%%%%%%%%%%%%%%%%%%%%%%%%%%%%%%%%%%%%%%%%%%%%%%%%%%%%%%%%%%%%%%%%%%%%%%%%%%%%%%

Let $T$ be a tree of height $\ell$. The set of \emph{leaves} of $T$ is defined as
$$
\mathcal L(T)=H^{-1}(\ell)\ \subset\ V.
$$

For each vertex $v\in V$, we say that $v'\in V$ is a \emph{child} of $v$, if $p(v')=v$.
More generally, we say that $v'$ is \emph{below} $v$, and write $v'\prec v$, if there is a sequence of vertices $v_1,\ldots,v_m$ so that $v_1=v$, $v_m=v'$, and $v_{i+1}$ is a child of $v_i$ for each $i=1,\ldots,m-1$.

If $T$ is a tree and $v\in V$, we define $T_v$ to be the \emph{subtree} of $T$ rooted at $v$. This is a tree of height $\ell-H(v)$. Its vertices are the set $\{v'\in V\mid v'\prec v\}$. The height function is $v'\mapsto H(v')-H(v)$,
and the parent function is inherited from the original tree.
%
%%%%%%%%%%%%%%%%%%%%%%%%%%%%%%%%%%%%%%%%%%%%%%%%%%%%%%%%%%%%%%%%%%%%%%%%%%%%%%%%
\begin{defn}[Regular trees]
Let $\ell,B,C\geq 0$. We say a tree $T$ is an ``Ahlfors-David regular tree of height~$\ell$, branching~$B$, and regularity constant~$C$'' if $T$ is a tree of height~$\ell$ and 
for each $v\in T$,
\begin{equation}\label{ADregularTreeEqn}
C^{-1} B^{\ell-H(v)}\leq |\mathcal{L}(T_v)|\leq C B^{\ell-H(v)},
\end{equation}
where $T_v$ is the sub-tree of $T$ rooted at $v$.
\end{defn}
%%%%%%%%%%%%%%%%%%%%%%%%%%%%%%%%%%%%%%%%%%%%%%%%%%%%%%%%%%%%%%%%%%%%%%%%%%%%%%%%

%%%%%%%%%%%%%%%%%%%%%%%%%%%%%%%%%%%%%%%%%%%%%%%%%%%%%%%%%%%%%%%%%%%%%%%%%%%%%%%%
\noindent\textbf{Remark}.
If $T$ is an Ahlfors-David regular tree with height $\ell$, branching $B$, and regularity constant $C$, then each vertex of $T$ has between $C^{-2} B$ and $C^2B$ children. However, much more is true--if a vertex of $T$ has (relatively) few children, then \emph{its} children must have many children, and vice versa. Thus the tree $T$ might not be perfectly balanced, but it cannot become extremely unbalanced either.
%%%%%%%%%%%%%%%%%%%%%%%%%%%%%%%%%%%%%%%%%%%%%%%%%%%%%%%%%%%%%%%%%%%%%%%%%%%%%%%%

%%%%%%%%%%%%%%%%%%%%%%%%%%%%%%%%%%%%%%%%%%%%%%%%%%%%%%%%%%%%%%%%%%%%%%%%%%%%%%%%
\begin{lemm}\label{subTreeIsADReg}
Let $T$ be an Ahlfors-David regular tree  with height $\ell$, branching $B$, and regularity constant $C$. Let $v\in T$. Then $T_v$ (the sub-tree rooted at $v$) is an Ahlfors-David regular tree of height $\ell-H(v)$, branching $B$, and regularity constant $C$.
\end{lemm}
%%%%%%%%%%%%%%%%%%%%%%%%%%%%%%%%%%%%%%%%%%%%%%%%%%%%%%%%%%%%%%%%%%%%%%%%%%%%%%%%

If $T$ is a tree of height $\ell$ and $j\in\mathbb N$, then we can define the $j$-th \emph{power} of $T$, denoted~$T^j$,
as the following tree of height $\ell$:
\begin{itemize}
\item the vertices of $T^j$ are ordered pairs
$(v_1,\dots,v_j)$, where $v_1,\dots,v_j$ are vertices of the tree $T$ of the same height;
\item the height of a vertex $(v_1,\dots,v_j)$ is equal to the height of each of $v_i$;
\item the parent of a vertex $(v_1,\dots,v_j)$ is equal to $(p(v_1),\dots,p(v_j))$,
where $p$ is the parent function of the original tree.
\end{itemize}

If $T$ is an Ahlfors-David regular tree with height $\ell$, branching $B$, and regularity constant $C$, then $T^j$ is an Ahlfors-David regular tree with with height $\ell$, branching $B^j$, and regularity constant $C^j$.
%
%
%
%
%
%%%%%%%%%%%%%%%%%%%%%%%%%%%%%%%%%%%%%%%%%%%%%%%%%%%%%%%%%%%%%%%%%%%%%%%%%%%%%%%%
\subsection{Discretization}
  \label{s:discretization}

The trees discussed in the previous section are useful for describing the multi-scale structure of $\setdim$--regular sets.

Let $\X\subset[0,1]$ be a $\setdim$--regular set with regularity constant $C$. Define
\begin{equation}\label{defnC1}
C_1:=(10C^2)^{1\over 1-\setdim}.
\end{equation}
Let $M,N$ be positive integers; we will fix $M$ and study asymptotic behavior as $N\to\infty$. We will describe a process that divides $[0,1)$ into sub-intervals of length roughly $M^{-j},\ j=0,\ldots N,$ and assembles these intervals into a tree.  

For each $j=0,\ldots,N$, divide $[0,1)$ into $M^j$ intervals of the form $[iM^{-j}, (i+1)M^{-j})$. If $I$ is an interval of this form, we say $I$ is \emph{empty} if $I\cap \X=\emptyset$. Otherwise $I$ is \emph{non-empty}. If several non-empty intervals are adjacent, merge them into a single (longer) interval. By Lemma~\ref{ADStronglyAvoids} with $\alpha:=M^{-j}$ and $\varepsilon:=1$, at most $C_1$ intervals can be merged into a single interval in this fashion. Let $\mathcal{I}_j$ be the set of non-empty intervals obtained after the merger process is complete. Each interval in $\mathcal{I}_j$ has length between $M^{-j}$ and $C_1M^{-j}$. Furthermore, if $I$ is an interval in $\mathcal{I}_j$, then there is a gap of size $\geq M^{-j}$ on either side of $I$ that is disjoint from $\X$.

%%%%%%%%%%%%%%%%%%%%%%%%%%%%%%%%%%%%%%%%%%%%%%%%%%%%%%%%%%%%%%%%%%%%%%%%%%%%%%%%
\begin{lemm}\label{uniqueParent}
Let $I\in \mathcal{I}_j$. Then there is a unique interval $\tilde I\in\mathcal{I}_{j-1}$ that intersects $I$. Furthermore, $I\subset\tilde I$.
\end{lemm}
\begin{proof}
First, note that $I\subset\bigcup_{\tilde I\in \mathcal{I}_{j-1}}\tilde I$. Thus it suffices to show that $I$ intersects at most one interval from $\mathcal{I}_{j-1}$. Suppose that $I$ intersects two intervals, $\tilde I=[\tilde i_0 M^{1-j},\ \tilde i_1 M^{1-j})$ and $\tilde I^\prime=[\tilde i_0^\prime M^{1-j},\ \tilde i_1^\prime M^{1-j})$ from $\mathcal{I}_{j-1}$. If we write $I=[i_0M^{-j},i_1M^{-j})$, then $i_0 \leq M\tilde i_1\leq Mi_0^\prime\leq i_1$. 

Since no two intervals in $\mathcal{I}_{j-1}$ can be adjacent, there must be some interval of the form $[i^\prime M^{1-j},(i^\prime+1)M^{1-j})$ that is disjoint from $\X$, with $\tilde i_1 < i^\prime < i^\prime+1 < \tilde i_0^\prime$. This implies that $[i^\prime M^{1-j},(i^\prime+1)M^{1-j})\subset I$, and in particular, $[(Mi^\prime) M^{-j},(Mi^\prime+1)M^{-j})\subset I.$ But this implies that $[(Mi^\prime) M^{-j},(Mi^\prime+1)M^{-j})\cap I\cap\X=\emptyset,$ which is a contradiction---by the construction of the intervals in $\mathcal{I}_j$, every sub-interval of $I$ of the form $[iM^{-j},(i+1)M^{-j})$ must intersect $\X$. We conclude that there is at most one interval from $\mathcal{I}_{j-1}$ that intersects $I$.  
\end{proof}
%%%%%%%%%%%%%%%%%%%%%%%%%%%%%%%%%%%%%%%%%%%%%%%%%%%%%%%%%%%%%%%%%%%%%%%%%%%%%%%%

We now construct the tree $T_{\X;M,N}$ as follows. The root vertex of $T_{\X;M,N}$ corresponds to the interval $[0,1)\in \mathcal I_0$. For each $j=1,\ldots,N$, the vertices of $T_{\X;M,N}$ of height $j$ correspond to the intervals in $\mathcal{I}_j$.
The parent of an interval in $\mathcal I_j$ is the unique interval in $\mathcal I_{j-1}$ containing that interval.
If $v$ is a vertex of $T_{\X;M,N}$, let $I_v$ be the corresponding interval. 
Note that $v'\prec v$ if and only if $I_{v^\prime}\subset I_v$. See Figure~\ref{f:thetree}.

\begin{figure}
\includegraphics{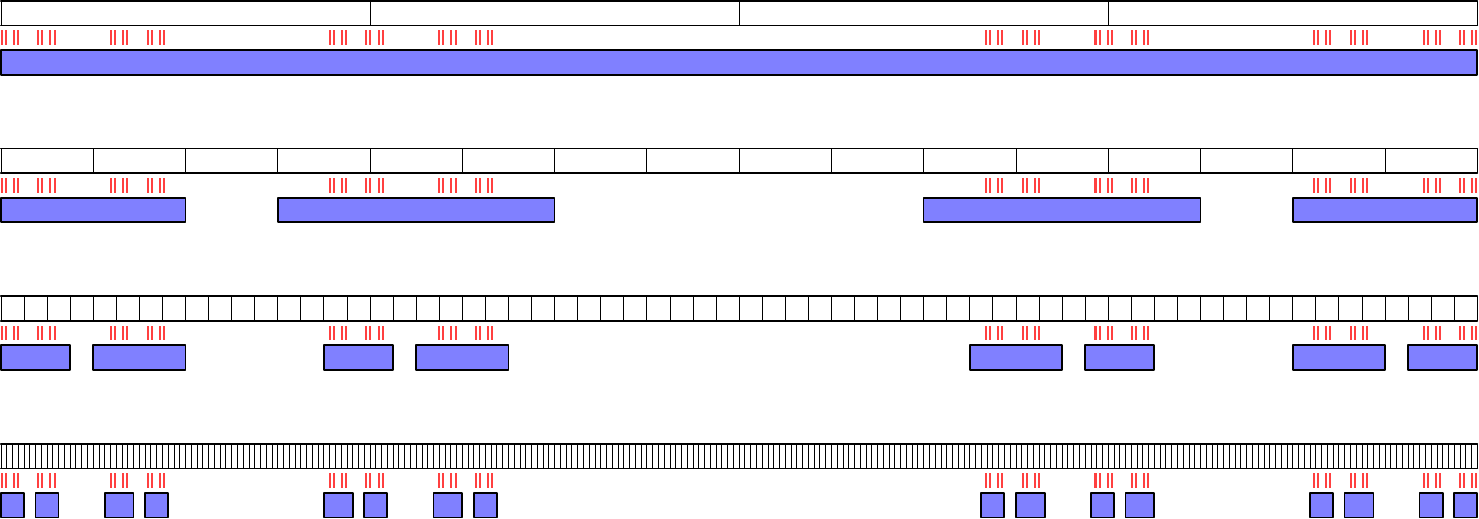}
\caption{The discretization of the middle third Cantor set (pictured in red) with $M=4$.
The intervals of the discretization for $j=1,\dots, 4$ are pictured in blue and
form a tree.}
\label{f:thetree}
\end{figure}

%%%%%%%%%%%%%%%%%%%%%%%%%%%%%%%%%%%%%%%%%%%%%%%%%%%%%%%%%%%%%%%%%%%%%%%%%%%%%%%%
\begin{lemm}
  \label{l:treecons}
Let $\X$ be a $\setdim$--regular set with regularity constant $C$. Then $T_{\X;M,N}$ is an Ahlfors-David regular tree with height $N$, branching $M^{\setdim}$ and regularity constant 
\begin{equation}\label{defnC2}
C_2:=C^2 C_1^\setdim=10^{\setdim\over 1-\setdim} C^{\frac{2}{1-\setdim}}. 
\end{equation}
\end{lemm}
%%%%%%%%%%%%%%%%%%%%%%%%%%%%%%%%%%%%%%%%%%%%%%%%%%%%%%%%%%%%%%%%%%%%%%%%%%%%%%%%
\begin{proof}
Let $v$ be a vertex of $T=T_{\X;M,N}$ and let $I_v\subset[0,1)$ be the corresponding interval.
Choose any $x_0\in I_v\cap \mathcal X$.
Since $I_v$ has length no more than $C_1M^{-H(v)}$, it is contained
in the ball $B(x_0,C_1M^{-H(v)})$.
On the other hand,%
\footnote{This is one of the places where we use the the merging of consecutive intervals; otherwise,
one of the intervals may intersect $\mathcal X$ at an endpoint and the resulting tree
may not be regular.}
$B(x_0,M^{-H(v)})\cap \mathcal X\subset I_v$.
Since $\mathcal X$ is $\delta$-regular,
we have
\begin{equation}
C^{-1} M^{-H(v)\setdim}\leq \mu_{\setdim}(\X\cap I_v)\leq C(C_1 M^{-H(v)})^{\setdim}.
\end{equation}
For each leaf $v^\prime\in \mathcal{L}(T)$ with $v^\prime\prec v$, let $I_{v^\prime}$ be the corresponding interval. Again, we have
\begin{equation}
C^{-1} M^{-N\setdim}\leq \mu_{\setdim}(\X\cap I_{v^\prime})\leq C(C_1 M^{-N})^{\setdim}.
\end{equation}
On the other hand, by Lemma \ref{uniqueParent} we have
\begin{equation*}
\mu_{\setdim}(\X\cap I_v)=\sum_{v^\prime\in \mathcal{L}(T_v)}\mu_{\setdim}(\X\cap I_{v^\prime}).
\end{equation*}
It follows that
\begin{equation*}
C_2^{-1}(M^\setdim)^{N-H(v)}\leq |\mathcal{L}(T_v)|\leq C_2(M^\setdim)^{N-H(v)},\ \quad C_2 = C^2 C_1^\setdim. \qedhere
\end{equation*}
\end{proof}
%%%%%%%%%%%%%%%%%%%%%%%%%%%%%%%%%%%%%%%%%%%%%%%%%%%%%%%%%%%%%%%%%%%%%%%%%%%%%%%%
%
%%%%%%%%%%%%%%%%%%%%%%%%%%%%%%%%%%%%%%%%%%%%%%%%%%%%%%%%%%%%%%%%%%%%%%%%%%%%%%%%
We will mainly be interested in the tree $T_{\X;M,N}^3$. A vertex of this tree is a triple of intervals $(I_1,I_2,I_3)$ where each interval $I_i$ meets $\X$ and all three intervals lie in $\mathcal I_j$ for some $j$ (and thus they have comparable lengths).
% 
% 
%
%	

%%%%%%%%%%%%%%%%%%%%%%%%%%%%%%%%%%%%%%%%%%%%%%%%%%%%%%%%%%%%%%%%%%%%%%%%%%%%%%%%
\subsection{Additive structure and pruning the tree}
  \label{AddEnergTreePrunSec}

Let $\X$ be a $\setdim$--regular set with regularity constant $C$ and let $T_{\X;M,N}$ be the Ahlfors-David regular tree described in~\S\ref{s:discretization}. Let $v=(I_1,I_2,I_3)$ be a vertex of $T_{\X;M,N}^3$. Consider the interval
$$
I_1-I_2+I_3:=\{x_1-x_2+x_3\mid x_1\in I_1,\ x_2\in I_2,\ x_3\in I_3\}.
$$
We say that $v$ \emph{misses} $\X$ if $(I_1-I_2+I_3)\cap \X(M^{-H(v)})=\emptyset$. Otherwise $v$ \emph{hits} $\X$.

The following result shows that if $M$ is sufficiently large then we have an improvement in additive energy on each level of the tree.

%%%%%%%%%%%%%%%%%%%%%%%%%%%%%%%%%%%%%%%%%%%%%%%%%%%%%%%%%%%%%%%%%%%%%%%%%%%%%%%%
\begin{prop}\label{missProp}
Let $\X$ be a $\setdim$--regular set with regularity constant $C$, and let $T_{\X;M,N}$ be the Ahlfors-David regular tree described above.
Assume that $M\geq M_0$, where
\begin{equation}\label{defnC3}
M_0:=\exp\big[K_5\delta^{-1}(1-\delta)^{-14}(1+\log^{14}C)\big],
\end{equation}
and $K_5$ is a large absolute constant. Let $v\in T_{\X;M,N}^3$ be a vertex that is not a leaf. Then at least one of the children of $v$ misses $\X$. 
\end{prop}
%%%%%%%%%%%%%%%%%%%%%%%%%%%%%%%%%%%%%%%%%%%%%%%%%%%%%%%%%%%%%%%%%%%%%%%%%%%%%%%%
\begin{proof}
Suppose $M\geq M_0$. Let $v\in T_{\X;M,N}^3$ be a vertex that is not a leaf and suppose all of the children of $v$ hit $\X$; we will obtain a contradiction.  Write $v=(v_1,v_2,v_3)$ and let $I_1,I_2,I_3\in \mathcal{I}_{H(v)}$ be the corresponding triple of intervals. Let $\h_1 = M^{-H(v)-1}$ and define
$$
\begin{aligned}
A_i&:=I_i \cap \X(\h_1) \cap \h_1\ZZ,\quad i=1,2,3;\\
A_4&:=(I_1-I_2+I_3) \cap \X(4C_1\h_1) \cap \h_1\ZZ.
\end{aligned}
$$
Here $C_1$ is defined in~\eqref{defnC1}. 
\begin{lemm}
If every child of $v$ hits $\X$, then
\begin{equation}
  \label{e:ally}
A_1-A_2+A_3\ \subset\ A_4. 
\end{equation}
\end{lemm}
\begin{proof}
Take $a_i\in A_i$, $i=1,2,3$. Then $|a_i-b_i|\leq\alpha_1$
for some $b_i\in\X$; since $I_i$ is surrounded
by size $M\alpha_1$ intervals which do not intersect $\X$,
we have $b_i\in I_i\cap\X$. Then $b_i$ lies in some child $I_i'$
of $I_i$, which is an interval of size no more than $C_1\alpha_1$.
We have $I_1'-I_2'+I_3'\cap\X(\alpha_1)\neq\emptyset$ and thus
$b_1-b_2+b_3\in\X((3C_1+1)\alpha_1)$. Then $a_1-a_2+a_3\in \X(4C_1\alpha_1)$
and~\eqref{e:ally} follows.
\end{proof}
We next claim that
\begin{equation}
  \label{e:zuotian}
\begin{aligned}
|A_i|&\geq (2C^2)^{-1}M^\delta,\quad i=1,2,3;\\
|A_4|&\leq 44C^2C_1M^\delta.
\end{aligned}
\end{equation}
These inequalities follow immediately from~\eqref{defnADRegular} and the following observations:
\begin{itemize}
\item The balls $B(a,\alpha_1)$ centered at $a\in A_i$ cover $\mathcal X\cap I_i$,
and $\mathcal X\cap I_i$
contains the intersection of $\mathcal X$ with a ball of radius $M\alpha_1$ centered at a point of $\mathcal X$.
\item For each $a\in A_4$, the ball $B(a,5C_1\alpha_1)$
contains a ball of radius $C_1\alpha_1$ centered at a point of $\mathcal X$,
the balls $B(a,5C_1\alpha_1)$ for different $a$ have overlapping at most~$11C_1$,
and their union is contained in an interval of length $(3M+10)C_1\alpha_1\leq 4MC_1\alpha_1$.
\end{itemize}
We now use the Ruzsa sum inequality~\cite[(4.6)]{Ruzsa} (see also Petridis~\cite{Petridis}):
$$
|A+C|\leq {|A+B|\cdot |B+C|\over |B|}
$$
valid for nonempty finite sets $A,B,C\subset\mathbb Z$.
Putting $A:=A_1,B:=A_3,C:=A_1$ and using~\eqref{e:zuotian} and the fact that
$|A_1+A_3|\leq |A_4|$ we obtain
$$
|A_1+A_1|\leq {|A_4|^2\over |A_1|\cdot |A_3|}\cdot |A_1|\leq C_3|A_1|,
$$
where
$$
C_3:=7744C^8C_1^2\leq (10C^2)^{6\over 1-\delta}.
$$
Apply Corollary~\ref{ADRegularSetsInIntervals} to the set $A_1$ with $K:=C_3$. We conclude that
\begin{equation}
|A_1|\leq \exp\big[K_4(1-\delta)^{-1}(\log C)\big((1-\delta)^{13}(1+\log C)^{13}\big)\big],
\end{equation}
where $K_4$ is an absolute constant.
If $M\geq M_0$, we obtain a contradiction with the first bound in~\eqref{e:zuotian}.
\end{proof}
%%%%%%%%%%%%%%%%%%%%%%%%%%%%%%%%%%%%%%%%%%%%%%%%%%%%%%%%%%%%%%%%%%%%%%%%%%%%%%%%

%
%
%%%%%%%%%%%%%%%%%%%%%%%%%%%%%%%%%%%%%%%%%%%%%%%%%%%%%%%%%%%%%%%%%%%%%%%%%%%%%%%%
\subsection{Analyzing the pruned tree}
  \label{s:pruned}

To take advantage of Proposition~\ref{missProp}, we prove the following general
fact about pruned subtrees of Ahlfors--David regular trees.

For two trees $T=(V,H,p,\ell),T'=(V',H',p',\ell)$ of same height, we say that $T'$ is a \emph{subtree} of $T$
if $V'\subset V$ and $H',p'$ are the restrictions of $H,p$ to $V'$. We say that
$T'$ is a \emph{pruned} subtree, if for each $v\in V'$ which is not a leaf,
there exists a child of $v$ in $T$ which does not lie in $V'$. See Figure~\ref{f:pruned}.

%%%%%%%%%%%%%%%%%%%%%%%%%%%%%%%%%%%%%%%%%%%%%%%%%%%%%%%%%%%%%%%%%%%%%%%%%%%%%%%%
\begin{lemm}[Pruned trees have few leaves]\label{treePrudingLem}
Let $T$ be an Ahlfors-David regular tree with height $\ell$, branching $B$, and regularity constant $C$. Let $T^\prime$
be a pruned subtree of~$T$.
Then
\begin{equation}
|\mathcal{L}(T^\prime)|\leq\big(1-C^{-2}B^{-1}\big)^{\ell} |\mathcal{L}(T)|.
\end{equation}
\end{lemm}
%%%%%%%%%%%%%%%%%%%%%%%%%%%%%%%%%%%%%%%%%%%%%%%%%%%%%%%%%%%%%%%%%%%%%%%%%%%%%%%%
\begin{proof}
We will prove the lemma by induction on $\ell$, the height of the tree. If $\ell=0$ the result is trivial. Now assume the result has been proved for all Ahlfors-David regular trees  with height $\ell-1$, branching $B$, and regularity constant $C$. Let $T$ be an Ahlfors-David regular tree  with height $\ell$, branching $B$, and regularity constant $C$, and let $T'$ be a pruned subtree of $T$.
Then at least one of the vertices in $\mathcal{V}_1(T)$ is missing from~$T^\prime$. We call this vertex $v^*$.

By Lemma \ref{subTreeIsADReg}, each of the trees $\{T_v\colon v\in \mathcal{V}_1(T^\prime)\}$ is a regular tree with height $\ell-1$, branching $B$, and regularity constant $C$. Thus we can apply the induction hypothesis to each such tree to obtain
%%%%%%%%%%%%%%%%%%%%%%%%%%%%%%%%%%%%%%%%%%%%%%%%%%%%%%%%%%%%%%%%%%%%%%%%%%%%%%%%
\begin{figure}
\includegraphics{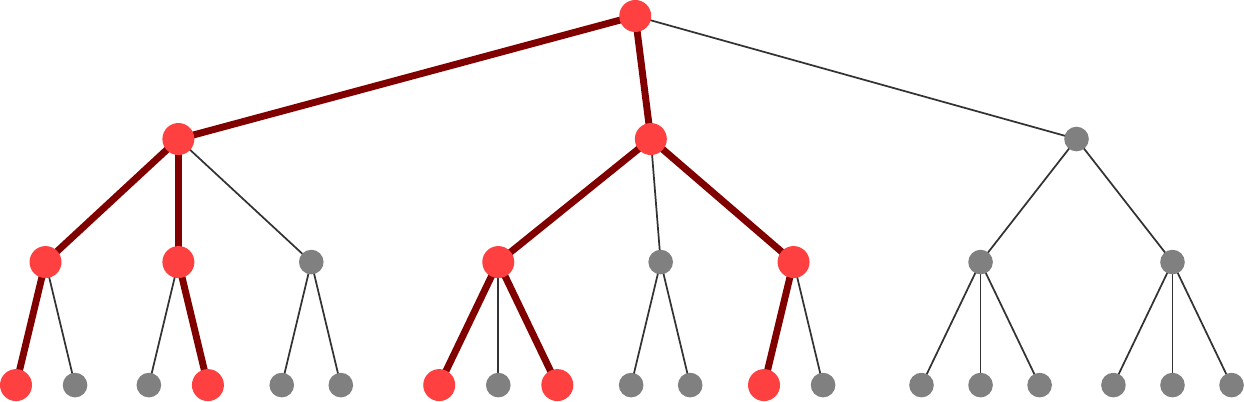}
\caption{An Ahlfors-David regular tree and a pruned subtree (in red).}
\label{f:pruned}
\end{figure}
%%%%%%%%%%%%%%%%%%%%%%%%%%%%%%%%%%%%%%%%%%%%%%%%%%%%%%%%%%%%%%%%%%%%%%%%%%%%%%%%
\begin{equation*}
\begin{split}
|\mathcal L(T^\prime)|&=\sum_{v\in \mathcal{V}_1(T^\prime)}|\mathcal{L}(T_v^\prime)|\\
&\leq   (1-C^{-2}B^{-1})^{\ell-1} \sum_{v\in \mathcal{V}_1(T^\prime)}|\mathcal{L}(T_v)|\\
&\leq   (1-C^{-2}B^{-1})^{\ell-1}\Big(|\mathcal{L}(T)|-|\mathcal{L}(T_{v^*})|\Big).
\end{split}
\end{equation*}
By \eqref{ADregularTreeEqn} we have
\begin{equation}
|\mathcal{L}(T_{v^*})|\geq C^{-2}B^{-1} |\mathcal{L}(T)|.
\end{equation}
Thus
\begin{equation*}
|\mathcal{L}(T)|-|\mathcal{L}(T_{v^*})|\leq (1-C^{-2}B^{-1})\cdot|\mathcal{L}(T)|,
\end{equation*}
so
\begin{equation*}
|\mathcal{L}(T^\prime)|\leq (1-C^{-2}B^{-1})^{\ell}|\mathcal{L}(T)|.\qedhere
\end{equation*}
\end{proof}
%%%%%%%%%%%%%%%%%%%%%%%%%%%%%%%%%%%%%%%%%%%%%%%%%%%%%%%%%%%%%%%%%%%%%%%%%%%%%%%%
%

%
%
%%%%%%%%%%%%%%%%%%%%%%%%%%%%%%%%%%%%%%%%%%%%%%%%%%%%%%%%%%%%%%%%%%%%%%%%%%%%%%%%
\subsection{Finishing the proof of Theorem~\ref{ADRegularSmallAddEnergyThm}}

Let $\X$ be a $\setdim$--regular set with regularity constant $C$ and let $\h>0$. Let
$$
M=\lceil M_0 \rceil,\quad
N=\bigg\lfloor {\log(1/\h)\over \log M}\bigg\rfloor,
$$
where $M_0$ is defined in \eqref{defnC3}, and let $T_{\X;M,N}$ be the associated Ahlfors-David regular tree
constructed in~\S\ref{s:discretization}.

Let $T^\prime$ be the subtree of $T_{\X;M,N}^3$ which consists of triples of intervals
that hit $\mathcal X$ (see~\S\ref{AddEnergTreePrunSec}). By Proposition~\ref{missProp},
$T'$ is pruned in the sense of~\S\ref{s:pruned}. By Lemmas~\ref{subTreeIsADReg}
and~\ref{treePrudingLem} and the inequality $(1-t)\leq e^{-t}$, we have
\begin{equation}
  \label{boundOnPrunedTree}
 \begin{split}
|\mathcal L(T^\prime)|&\leq\Big(1-\big(C_2^6M^{3\setdim}\big)^{-1}\Big)^{N}C_2^3M^{3\setdim N} \\
&\leq C_2^3 M^{N(3\setdim-\rho)}, \quad \rho=(C_2^6 M^{3\setdim}\log M)^{-1}.
\end{split}
\end{equation}
Expanding the definition of~$C_2$ and~$M_0$ from~\eqref{defnC2} and~\eqref{defnC3}, we have
\begin{equation*}
\rho\geq \delta\exp\big[-K_6(1-\delta)^{-14}(1+\log^{14}C)\big],
\end{equation*}
where $K_6$ is a large absolute constant. 

Now, assume that $(x_1,x_2,x_3)\in\X^3$ satisfies $x_1-x_2+x_3\in \X(\h)$. Then
$x_i\in I_i$, $i=1,2,3$, where $I_i$ are intervals corresponding
to some leaves $v_i$ of $T_{\X;M,N}$. Moreover, since
$\h\leq M^{-N}$, the vertex $(v_1,v_2,v_3)$ hits
$\X$ and thus is a leaf of $T'$. By~\eqref{defnADRegular}, for each leaf $(v_1,v_2,v_3)$ of $T'$ we have
$$
\mu_\delta^4\big(\big\{(x_1,x_2,x_3,x_4)\in \X^4\mid |x_1-x_2+x_3-x_4|\leq \alpha,\
x_i\in I_{v_i},\ i=1,2,3\big\}\big)\leq C_4\alpha^{4\delta}
$$
for some constant $C_4$ depending on $C$ and $\setdim$ but not on $\alpha$. Combining this with~\eqref{boundOnPrunedTree}
and recalling~\eqref{e:AESpecial}, we conclude that
$$
\mathcal E_A(\mathcal X,\mu_\delta,\alpha)\leq C_4C_2^3\alpha^{\delta+\rho}\leq C_4C_2^3\alpha^{\delta+\beta_\X}
$$
where $\beta_\X=\betaXBound$ for $\mathbf{K}$ a large absolute constant.
This finishes the proof of Theorem~\ref{ADRegularSmallAddEnergyThm}.

%%%%%%%%%%%%%%%%%%%%%%%%%%%%%%%%%%%%%%%%%%%%%%%%%%%%%%%%%%%%%%%%%%%%%%%%%%%%%%%%
\subsection{Further remarks}\label{furtherRemarks}

%%%%%%%%%%%%%%%%%%%%%%%%%%%%%%%%%%%%%%%%%%%%%%%%%%%%%%%%%%%%%%%%%%%%%%%%%%%%%%%%
\subsubsection{A discretized additive energy bound}

We have chosen to phrase Theorem \ref{ADRegularSmallAddEnergyThm} in the language of Ahlfors-David regular sets. However, we only examine these sets at scales between $\h$ and 1. Our proof of Theorem \ref{ADRegularSmallAddEnergyThm} also gives the following variant:
\begin{prop}\label{VarADRegularSmallAddEnergyThm}
Let $\X\subset[0,1]$ be a union of intervals of length $\h$. Suppose that for all $\h\leq r\leq 1$ and all $x\in \X$ we have the bounds
\begin{equation}
C_{\X}^{-1} r^{\setdim}\h^{1-\setdim}\leq \mu_L(\X\cap B(x,r))\leq C_{\X} r^{\setdim}\h^{1-\setdim},
\end{equation}
where $\mu_L$ is the one-dimensional Lebesgue measure. Then
\begin{equation}
\mu_L^4\big(\{(x_1,x_2,x_3,x_4)\in\X\colon |x_1-x_2+x_3-x_4|<\h\}\big)\leq\widetilde C\h^{4-3\setdim+\beta_{\X}},
\end{equation}
where $\beta_X$ is as given in \eqref{finalEpsBd} and $\widetilde C$ is some constant.
\end{prop}

%%%%%%%%%%%%%%%%%%%%%%%%%%%%%%%%%%%%%%%%%%%%%%%%%%%%%%%%%%%%%%%%%%%%%%%%%%%%%%%%
\subsubsection{Higher dimensions}
  \label{higherDimRemark}

Most of the arguments in this section extend to higher dimensions without difficulty. The real issue is extending Theorem \ref{ADRegularSmallAddEnergyThm} to $\setdim\geq 1$. This may be challenging because Proposition \ref{ADStronglyAvoids} may not be true if $\setdim\geq 1$. For example, a unit line segment in $[0,1]^2$ is a 1--regular set, but it contains arbitrarily long arithmetic progressions. A unit line segment in $[0,1]^2$ also has maximal additive energy, so no variant of Theorem  \ref{ADRegularSmallAddEnergyThm} can hope to hold for that example. If $\setdim\geq 1$, it is not clear what the proper hypotheses for the theorem should be.

One possible avenue is the following. In \cite{bond}, the authors study $1$--regular sets that satisfy an additional property called $(\rho,C_1)$--unrectifiability. If a 1--regular set is $(\rho,C_1)$ unrectifiable, then for all rectangles $R$ of dimensions $r_1\geq r_2$, we have $\mu_1(\X\cap R)\leq C r_1^{1-\rho}r_2^\rho$. If $\rho>0$ then sets with this property strongly avoid arithmetic progressions. It is possible that this property can be generalized to other $\setdim$--regular sets for $\setdim>1$.

%%%%%%%%%%%%%%%%%%%%%%%%%%%%%%%%%%%%%%%%%%%%%%%%%%%%%%%%%%%%%%%%%%%%%%%%%%%%%%%%
\subsubsection{Improving the bounds on $\beta_\X$}

It is likely that the bound on $\beta_\X$ from Theorem \ref{ADRegularSmallAddEnergyThm} can be substantially improved. A modest improvement would be to replace the bound in \eqref{finalEpsBd} by $C_{\X}^{-K/(1-\setdim)^K}$ where $K$ is an absolute constant.

However, the following example shows that the constant $\beta_\X$ must go to 0 as $C\to\infty$. Let $C>1$ be an integer and let 
\begin{equation}
\X=\{x\in[0,1]\colon\ x\ \textrm{has a base } C\ \textrm{expansion of the form}\ x=0.a_10a_30a_5\ldots\}.
\end{equation}
Then $\X$ is a $1/2$--regular set with regularity constant $C$. On the other hand, 
\begin{equation*}
\mathcal{E}_A(\X,\h)\geq \h^{1/2+\frac{1}{10\log C}}.
\end{equation*}
A similar example can be constructed for other values of $\setdim$. This shows that in Theorem~\ref{ADRegularSmallAddEnergyThm} we cannot take $\beta_\X>0$ to be independent of $C$. 

%%%%%%%%%%%%%%%%%%%%%%%%%%%%%%%%%%%%%%%%%%%%%%%%%%%%%%%%%%%%%%%%%%%%%%%%%%%%%%%%
\section{Regularity and additive energy of limit sets}
  \label{s:ae}

In this section, we study regularity of limit sets of convex co-compact groups and
their stereographic projections (\S\ref{s:ae-dynamical}). We next use the results
of~\S\ref{s:ae-combinatorial} to prove Theorem~\ref{t:ad-reduced} (\S\ref{s:minkowski}).

%%%%%%%%%%%%%%%%%%%%%%%%%%%%%%%%%%%%%%%%%%%%%%%%%%%%%%%%%%%%%%%%%%%%%%%%%%%%%%%%
\subsection{Regularity of limit sets}
  \label{s:ae-dynamical}
  
In this section, we consider a convex co-compact hyperbolic quotient
$M=\Gamma\backslash\mathbb H^n$ and
use the Ahlfors-David regularity of the limit set $\Lambda_\Gamma\subset\mathbb S^{n-1}$
to establish Ahlfors-David regularity of the stereographic projections
$\mathcal G(y_0,\Lambda_\Gamma)\subset T_{y_0}\mathbb S^{n-1}$, $y_0\in\Lambda_\Gamma$,
where $\mathcal G$ is defined in~\eqref{e:stpro}:
%%%%%%%%%%%%%%%%%%%%%%%%%%%%%%%%%%%%%%%%%%%%%%%%%%%%%%%%%%%%%%%%%%%%%%%%%%%%%%%%
\begin{lemm}
  \label{l:adred}
Let $\mathbf C$ be the constant in Ahlfors-David regularity of the limit set,
as defined in~\eqref{e:ad-regular-limit}.
Then for each $y_0,y_1\in \Lambda_\Gamma$, $y_0\neq y_1$, we have
\begin{equation}
  \label{e:adred}
\mathbf K_0^{-1}\mathbf C^{-1}r^\delta\ \leq\ \mu_\delta\big(\mathcal G(y_0,\Lambda_\Gamma)\cap B(\mathcal G(y_0,y_1),r)\big)
\ \leq\ \mathbf K_0\mathbf Cr^\delta,\quad
r>0
\end{equation}
where $\mu_\delta$ is the $\delta$--Hausdorff measure
on $T_{y_0}\mathbb S^{n-1}$ and $\mathbf K_0$ is a global constant
(depending only on the dimension).
\end{lemm}
%%%%%%%%%%%%%%%%%%%%%%%%%%%%%%%%%%%%%%%%%%%%%%%%%%%%%%%%%%%%%%%%%%%%%%%%%%%%%%%%
The case of bounded $r$ in the above lemma is an immediate consequence of the regularity
of $\Lambda_{\Gamma}$. To show that the constants in the regularity statement do not deteriorate when $r\to \infty$, we will prove that if we shrink the set $\mathcal G(y_0,\Lambda_{\Gamma})-\mathcal G(y_0,y_1)$, we obtain
an isometric image of
the set $\mathcal G(y'_{0},\Lambda_{\Gamma})-\mathcal G(y'_0,y'_1)$ for some other $y'_0,y'_1\in\Lambda_{\Gamma}$. For that we
use the group action, or equivalently, argue on the quotient manifold
$M$ rather than on $\mathbb H^{n}$.

We first write the sets $\mathcal G(y_0,\Lambda_\Gamma)\subset T_{y_0}\mathbb S^{n-1}$
as subsets of the unstable spaces on $M$
via a map $\mathscr U_-$ constructed using horocyclic flows (see Appendix~\ref{s:hyperbolic-technical} for the proof):
%%%%%%%%%%%%%%%%%%%%%%%%%%%%%%%%%%%%%%%%%%%%%%%%%%%%%%%%%%%%%%%%%%%%%%%%%%%%%%%%
\begin{lemm}
  \label{l:horocyclic}
Let $S^*\mathbb H^n\subset T^*\mathbb H^n$ be the unit cotangent bundle and $E_u$ the unstable foliation,
see~\eqref{e:sudec}. Then there exists a smooth map
$$
\mathscr U_-:\{(x,\xi,\eta)\mid (x,\xi)\in S^*\mathbb H^n,\ \eta\in E_u(x,\xi)\}\to S^*\mathbb H^n
$$
such that for $(\tilde x,\tilde\xi):=\mathscr U_-(x,\xi,\eta)$ and $B_\pm:S^*\mathbb H^n\to \mathbb S^{n-1}$ defined in~\eqref{e:B-pm},
\begin{gather}
  \label{e:scroo-1}
B_-(\tilde x,\tilde\xi)=B_-(x,\xi),\quad
\mathcal P(\tilde x,B_-(x,\xi))=\mathcal P(x,B_-(x,\xi)),\\
  \label{e:scroo-2}
\mathcal G\big(B_-(x,\xi),B_+(\tilde x,\tilde\xi)\big)-\mathcal G\big(B_-(x,\xi),B_+(x,\xi)\big)=
\mathcal P(x,B_-(x,\xi))\mathcal T_-(x,\xi)\eta,
\end{gather}
where $\mathcal T_-(x,\xi):E_u(x,\xi)\to T_{B_-(x,\xi)}\mathbb S^{n-1}$ is some linear isometry
and $\mathcal P$ is the Poisson kernel defined in~\eqref{e:Pker}.
Moreover, the map $\mathscr U_-$ commutes with the natural action of the isometry group $\PSO(1,n)$
and thus descends to a map
$$
\mathscr U_-:\{(x,\xi,\eta)\mid (x,\xi)\in S^*M,\ \eta\in E_u(x,\xi)\}\to S^* M.
$$
\end{lemm}
%%%%%%%%%%%%%%%%%%%%%%%%%%%%%%%%%%%%%%%%%%%%%%%%%%%%%%%%%%%%%%%%%%%%%%%%%%%%%%%%
\begin{figure}
\includegraphics{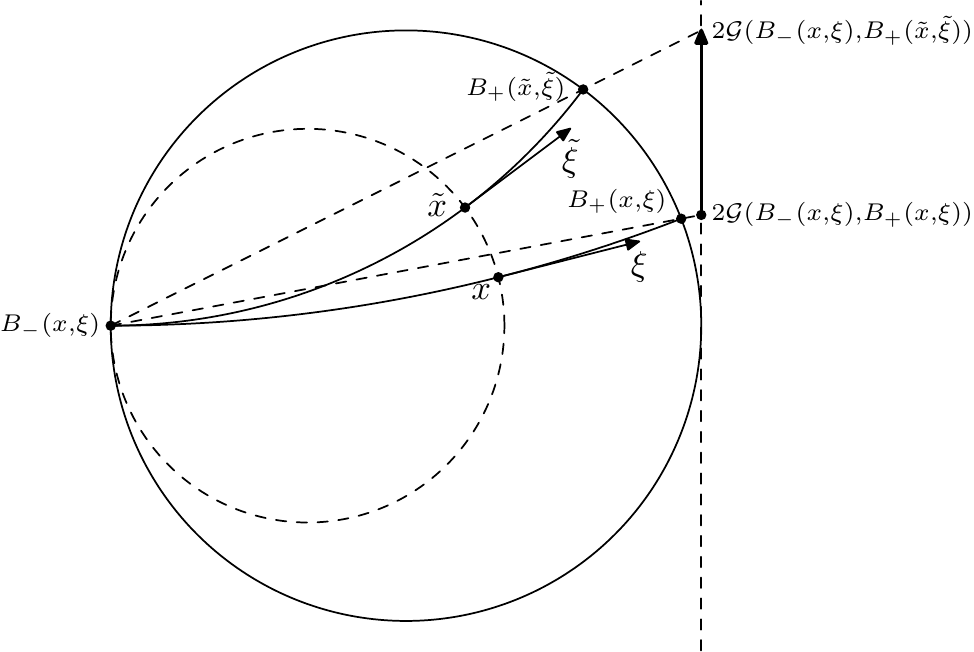}
\caption{The points $(x,\xi)\in S^*\mathbb H^n$ and $(\tilde x,\tilde\xi)=\mathscr U_-(x,\xi,\eta)$.
The dashed circle is a horocycle and the solid arcs are geodesics through $(x,\xi)$ and $(\tilde x,\tilde\xi)$.}
\label{f:horocyclic}
\end{figure}
%%%%%%%%%%%%%%%%%%%%%%%%%%%%%%%%%%%%%%%%%%%%%%%%%%%%%%%%%%%%%%%%%%%%%%%%%%%%%%%%
\noindent\textbf{Remarks}. (i) For each $(x,\xi)\in S^*M$,
the set $\{\mathscr U_-(x,\xi,\eta)\mid \eta\in E_u(x,\xi)\}$ is the unstable
manifold passing through $(x,\xi)$, and the differential of the map $\eta\mapsto \mathscr U_-(x,\xi,\eta)$
at $\eta=0$ is the embedding $E_u(x,\xi)\to T_{(x,\xi)}(S^*M)$. See Figure~\ref{f:horocyclic}. The map
$\mathcal T_-$ is related to the parametrization of $E_u(x,\xi)$ by the orthogonal complement
$\mathcal E(x,\xi)\subset T_x M$ of~$\xi$ (see the paragraph preceding~\eqref{e:stun}) and
to the parallel transport map $\mathcal E(x,\xi)\to T_{B_-(x,\xi)}\mathbb S^{n-1}$ (see~\cite[\S3.6]{rrh}), but
we do not need an explicit expression for $\mathcal T_-$ here.

\noindent (ii) In dimension $2$, $\mathscr U_-$ is given by the
flow of the unstable horocyclic vector field $U_-$
(see for instance~\cite[(2.1)]{rrh}):
$$
\mathscr U_-(x,\xi,\eta)=e^{sU_-}(x,\xi),\quad
\eta=sU_-(x,\xi)\in E_u(x,\xi),\quad
s\in\mathbb R.
$$
%%%%%%%%%%%%%%%%%%%%%%%%%%%%%%%%%%%%%%%%%%%%%%%%%%%%%%%%%%%%%%%%%%%%%%%%%%%%%%%%

\smallskip

For each point $(x,\xi)\in K\cap S^*M$, where $K=\Gamma_+\cap\Gamma_-$ is the trapped set, define
\begin{equation}
  \label{e:cal-F}
\mathcal F_{(x,\xi)}:=\{\eta\mid \mathscr U_-(x,\xi,\eta)\in K\}\ \subset\ E_u(x,\xi).
\end{equation}
Note that $(\tilde x,\tilde \xi):=\mathscr U_-(x,\xi,\eta)$ lies in $\Gamma_+$ for all $\eta$, since the geodesics
starting at $(x,\xi)$ and $(\tilde x,\tilde\xi)$ converge to each other as $t\to -\infty$ by~\eqref{e:scroo-1}
(see also~\eqref{e:Gpm-formula});
therefore, $K$ can be replaced in~\eqref{e:cal-F} by $\Gamma_-$. By~\eqref{e:scroo-2} and~\eqref{e:Gpm-formula},
for each $(x,\xi)\in S^*\mathbb H^n$ we have
\begin{equation}
  \label{e:cal-F-useful}
\mathcal G(y_-,\Lambda_\Gamma)-\mathcal G(y_-,y_+)=
\mathcal P(x,y_-)\mathcal T_-(x,\xi)\mathcal F_{\pi_\Gamma(x,\xi)},\quad
y_\pm :=B_\pm(x,\xi),
\end{equation}
with $\pi_\Gamma$ defined in~\eqref{e:pi-Gamma}.
%%%%%%%%%%%%%%%%%%%%%%%%%%%%%%%%%%%%%%%%%%%%%%%%%%%%%%%%%%%%%%%%%%%%%%%%%%%%%%%%
\begin{proof}[Proof of Lemma~\ref{l:adred}]
Let $y_0,y_1\in\Lambda_\Gamma$ and $y_0\neq y_1$.

We first prove~\eqref{e:adred} for the case $0<r<1$.
Define the diffeomorphism
$$
\Phi:\mathbb S^{n-1}\setminus \{y_0\}\to T_{y_0}\mathbb S^{n-1},\quad
\Phi(y):=\mathcal G(y_0,y).
$$
Then $d\Phi(y)$ is conformal with factor
${1\over 2}(1+|\Phi(y)|^2)$. It follows that
$$
\sup_{\Phi(y)\in B(\eta_1,1)} \|d\Phi(y)\|\leq {3\over 2} (1+|\eta_1|^2),\quad
\sup_{\Phi(y)\in B(\eta_1,1)} \|d\Phi(y)^{-1}\|\leq {6\over 1+|\eta_1|^2}
$$
where $\eta_1:=\Phi(y_1)$, and therefore
$$
B\Big(y_1,{2r\over 3(1+|\eta_1|^2)}\Big)\ \subset\ \Phi^{-1}(B(\eta_1,r))\ \subset\ B\Big(y_1,{6r\over 1+|\eta_1|^2}\Big).
$$
We now have by~\eqref{defnADRegular}
$$
\begin{aligned}
\mu_\delta\big(\Phi(\Lambda_\Gamma)\cap B(\eta_1,r)\big)&\leq \Big({3\over 2}(1+|\eta_1|^2)\Big)^\delta\mu_\delta\big(\Lambda_\Gamma
\cap \Phi^{-1}(B(\eta_1,r))\big)\\
&\leq \Big({3\over 2}(1+|\eta_1|^2)\Big)^\delta\mu_\delta\Big(\Lambda_\Gamma\cap B\Big(y_1,{6r\over 1+|\eta_1|^2}\Big)\Big)\\
&\leq \mathbf C (9r)^\delta
\end{aligned}
$$
and similarly
$$
\mu_\delta\big(\Phi(\Lambda_\Gamma)\cap B(\eta_1,r)\big)
\geq \Big({1+|\eta_1|^2\over 6}\Big)^\delta\mu_\delta\big(\Lambda_\Gamma
\cap \Phi^{-1}(B(\eta_1,r))\big)
\geq \mathbf C^{-1} (r/9)^\delta
$$
which gives~\eqref{e:adred} for $0<r<1$ with $\mathbf K_0:=9^{n-1}\geq 9^\delta$.

Now, assume that $r\geq 1$. Take some $(x,\xi)\in S^*\mathbb H^n$ on the geodesic connecting
$y_0$ and $y_1$, that is
$$
B_-(x,\xi)=y_0,\quad
B_+(x,\xi)=y_1.
$$
Let $\gamma\in\Gamma$ and put
$$
(x',\xi'):=\gamma.(x,\xi),\quad
B_-(x',\xi')=y'_0,\quad
B_+(x',\xi')=y'_1;
$$
note that $y'_0,y'_1\in\Lambda_\Gamma$. We choose $(x,\xi)$ and $\gamma$ such that 
\begin{equation}
  \label{e:huangpu0}
r':={\mathcal P(x',y_0')\over \mathcal P(x,y_0)}r\ <\ 1.
\end{equation}
To do that, we first remark
that there exists $R>0$ depending on the quotient $M$ such that
for each $(\tilde x,\tilde\xi)\in K\cap S^*M$, there exists
\begin{equation}
  \label{e:huangpu}
(x',\xi')\in\pi_\Gamma^{-1}(\tilde x,\tilde\xi),\quad
\mathcal P(x',B_-(x',\xi'))\leq R.
\end{equation}
This follows immediately from the compactness of $K\cap S^*M$. To ensure~\eqref{e:huangpu0},
it remains to take $(x,\xi)$ such that $\mathcal P(x,y_0)>rR$ (which is always possible
since the function $P(x,y_0)$ grows exponentially along the backwards geodesic flow)
and choose $(x',\xi')$ using~\eqref{e:huangpu} with $(\tilde x,\tilde\xi):=\pi_\Gamma(x,\xi)$.

From~\eqref{e:cal-F-useful} and the fact that $\pi_\Gamma(x,\xi)=\pi_\Gamma(x',\xi')$,
we have
\begin{equation}
  \label{e:huangpu1}
\mathcal G(y_0,\Lambda_\Gamma)-\mathcal G(y_0,y_1)={\mathcal P(x,y_0)\over \mathcal P(x',y'_0)}\widetilde{\mathcal T}\big(\mathcal G(y'_0,\Lambda_\Gamma)-\mathcal G(y'_0,y'_1)\big)
\end{equation}
where $\widetilde{\mathcal T}:T_{y_0'}\mathbb S^{n-1}\to T_{y_0}\mathbb S^{n-1}$ is an isometry.
Since~\eqref{e:adred} is already known for $r<1$, we have using~\eqref{e:huangpu0},
$$
\mathbf K_0^{-1}\mathbf C^{-1}(r')^\delta\ \leq\
\mu_\delta\big(\mathcal G(y'_0,\Lambda_\Gamma)\cap B(\mathcal G(y'_0,y'_1),r')\big)\ \leq\ 
\mathbf K_0\mathbf C (r')^\delta.
$$
Combining this with~\eqref{e:huangpu1}, we obtain~\eqref{e:adred}.
\end{proof}
%%%%%%%%%%%%%%%%%%%%%%%%%%%%%%%%%%%%%%%%%%%%%%%%%%%%%%%%%%%%%%%%%%%%%%%%%%%%%%%%
  
%%%%%%%%%%%%%%%%%%%%%%%%%%%%%%%%%%%%%%%%%%%%%%%%%%%%%%%%%%%%%%%%%%%%%%%%%%%%%%%%
\subsection{Regularity, additive energy, and Minkowski dimension}
  \label{s:minkowski}

In this section, we state a few results estimating the Lebesgue measure of neighborhoods
of Ahlfors-David regular sets. This establishes  bounds on Minkowski dimensions
of these sets. We rely on the following
%%%%%%%%%%%%%%%%%%%%%%%%%%%%%%%%%%%%%%%%%%%%%%%%%%%%%%%%%%%%%%%%%%%%%%%%%%%%%%%%
\begin{defi}
Assume that $(\mathcal M,d)$ is a metric space.
For $\mathcal X\subset\mathcal M$ and $\alpha>0$, define
the maximal number of $\alpha$-separated points in $\mathcal X$:
$$
\mathcal N(\mathcal X,\alpha)=\max\{N\mid x_1,\dots, x_N\in \mathcal X,\ d(x_i,x_j)> \alpha
\quad\text{for }i\neq j\}.
$$
\end{defi}
%%%%%%%%%%%%%%%%%%%%%%%%%%%%%%%%%%%%%%%%%%%%%%%%%%%%%%%%%%%%%%%%%%%%%%%%%%%%%%%%

For regular sets, the quantity $\mathcal N(\mathcal X,\alpha)$ establishes a link
between the Hausdorff and Minkowski dimensions:
%%%%%%%%%%%%%%%%%%%%%%%%%%%%%%%%%%%%%%%%%%%%%%%%%%%%%%%%%%%%%%%%%%%%%%%%%%%%%%%%
\begin{lemm}
  \label{l:alphasep}
Let $(\mathcal M,d)$ be a metric space and let $\mathcal X\subset\mathcal M$ be compact.

1. If $\mathcal X$ is $\delta$--regular in the sense of Definition~\ref{d:ad-regular} with
constant $C_\mathcal X$, then for each $x\in \mathcal X$,
$$
C_{\mathcal X}^{-2}\Big({\alpha'\over \alpha}\Big)^\delta\leq \mathcal N(\mathcal X\cap B(x,\alpha'),\alpha)\leq C_{\mathcal X}^2\Big(1+{2\alpha'\over \alpha}\Big)^\delta,\quad
0<\alpha,\alpha'<\diam(\mathcal M).
$$

2. If $\mathcal M$ is an $m$-dimensional Riemannian manifold and $\mathcal X(\alpha)$ is the $\alpha$-neighborhood of $\mathcal X$,
then
$$
C^{-1}\alpha^m\mathcal N(\mathcal X,\alpha) \leq \mu_L(\mathcal X(\alpha))\leq C\alpha^m\mathcal N(\mathcal X,\alpha),\quad
0<\alpha<1,
$$
where $\mu_L$ is the Lebesgue measure induced by the metric and $C$ is some constant independent of $\alpha$.
\end{lemm}
%%%%%%%%%%%%%%%%%%%%%%%%%%%%%%%%%%%%%%%%%%%%%%%%%%%%%%%%%%%%%%%%%%%%%%%%%%%%%%%%
\begin{proof}
We will begin with the first statement. Fix a ball $B(x,\alpha')$ with $x\in\mathcal{X}$. Let $\{x_1,\ldots,x_N\}\subset \X\cap B(x,\alpha')$ be a maximal collection of $\alpha$--separated points. Then
$$
\X\cap B(x,\alpha')\ \subset\ \bigcup_{i=1}^N B(x_i,\alpha ),\qquad
\bigsqcup_{i=1}^N B\Big(x_i,{\alpha\over 2}\Big)\ \subset\ B\Big(x,\alpha'+{\alpha\over 2}\Big).
$$
It follows that
\begin{equation*}
\begin{split}
N &\geq C_{\X}^{-1}   \h^{-\setdim} \sum_{i=1}^N \mu_{\setdim}(\X\cap B(x_i,\h)) \geq C_{\X}^{-1}  \h^{-\setdim} \mu_{\setdim}(\X\cap B(x,\alpha^\prime))\\&\geq C_{\X}^{-2}  \Big({\alpha'\over\alpha}\Big)^{\setdim},\\
N &\leq \Big({2\over\h}\Big)^{\setdim}C_{\X}   \sum_{i=1}^N \mu_{\setdim}\Big(\X\cap B\Big(x_i,{\alpha\over 2}\Big)\Big) \leq \Big({2\over\h}\Big)^{\setdim}C_{\X}   \mu_{\setdim}\Big(\X\cap B\Big(x,\alpha'+{\alpha\over 2}\Big)\Big)\\&\leq C_{\X}^2\Big(1+{2\alpha'\over\alpha}\Big)^{\setdim};
\end{split}
\end{equation*}
this finishes the proof of the first statement. 

The second statement follows similarly from the inclusions
$$
\mathcal X(\alpha)\ \subset\ \bigcup_{i=1}^N B(x_i,2\alpha),\qquad
\bigsqcup_{i=1}^N B\Big(x_i,{\alpha\over 2}\Big)\ \subset\ \mathcal X(\alpha),
$$
where $\{x_1,\dots,x_N\}\subset\mathcal X$ is a maximal collection of $\alpha$-separated points,
plus the observation that on any compact subset $\Omega\subset \mathcal M$, the Lebesgue measure of a ball of radius~$\alpha$ is between $C_{\Omega}^{-1}\alpha^m$ and $C_{\Omega}\alpha^m$.
\end{proof}
%%%%%%%%%%%%%%%%%%%%%%%%%%%%%%%%%%%%%%%%%%%%%%%%%%%%%%%%%%%%%%%%%%%%%%%%%%%%%%%%
As an application, we obtain
%%%%%%%%%%%%%%%%%%%%%%%%%%%%%%%%%%%%%%%%%%%%%%%%%%%%%%%%%%%%%%%%%%%%%%%%%%%%%%%%
\begin{proof}[Proof of~\eqref{e:AD-estimate-Lebesgue}]
Follows directly from $\delta$-regularity of the limit set (see~\S\ref{s:introad}),
Lemma~\ref{l:alphasep}, and the fact that
$\Lambda_\Gamma(\alpha)\cap B(y_0,\alpha')$ contains the $\alpha$-neighborhood
of $\Lambda_\Gamma\cap B(y_0,\alpha'-\alpha)$ and is contained in the $\alpha$-neighborhood
of $\Lambda_\Gamma\cap B(y_0,\alpha'+\alpha)$.
\end{proof}
%%%%%%%%%%%%%%%%%%%%%%%%%%%%%%%%%%%%%%%%%%%%%%%%%%%%%%%%%%%%%%%%%%%%%%%%%%%%%%%%
\noindent\textbf{Remark}. In Definition \ref{d:ad-regular} of Ahlfors-David regularity we used the Hausdorff measure. However, any other Borel measure could be used instead:
%%%%%%%%%%%%%%%%%%%%%%%%%%%%%%%%%%%%%%%%%%%%%%%%%%%%%%%%%%%%%%%%%%%%%%%%%%%%%%%%
\begin{lemm}[\cite{DS}, Lemma 1.2]\label{equivOfADRegDefns}
Let $(\metSpace,\metricChar)$ be a complete metric space with more than one element and let $\mathcal X \subset \metSpace$. Let $\mu$ be a Borel measure on $\metSpace$ with the property that for all $x\in \mathcal X$, 
\begin{equation}
C_{\mathcal{X}}^{-1}r^{\setdim}\leq \mu(\mathcal{X}\cap B(x,r))\leq C_{\mathcal{X}}r^{\setdim},\quad
0<r<\diam(\metSpace).
\end{equation}
Then $\mathcal{X}$ is $\setdim$--regular. The regularity constant depends only on $\setdim$ and $C_{\mathcal X}$. 
\end{lemm}
%%%%%%%%%%%%%%%%%%%%%%%%%%%%%%%%%%%%%%%%%%%%%%%%%%%%%%%%%%%%%%%%%%%%%%%%%%%%%%%%

We now give the proof of Theorem~\ref{t:ad-reduced}. The first step is the following
%%%%%%%%%%%%%%%%%%%%%%%%%%%%%%%%%%%%%%%%%%%%%%%%%%%%%%%%%%%%%%%%%%%%%%%%%%%%%%%%
\begin{lemm}
  \label{l:ad-helper}
Let $\Lambda_\Gamma$ be the limit set of a convex co-compact hyperbolic surface,
$\mathbf C$ be defined in~\eqref{e:ad-regular-limit},
$\mathbf K_0$ be given in Lemma~\ref{l:adred},
and $\mathcal G$ be defined in~\eqref{e:stpro}.
Take $y_0\in \Lambda_\Gamma$ and $R>0$.
Then there exists an interval $[-1,1] \subset I \subset [-2,2]$ such that
$$
\mathcal{X}:=I\cap R^{-1}\mathcal G(y_0,\Lambda_\Gamma)\ \subset\ T_{y_0}\mathbb S^1\simeq \mathbb R
$$
is $\setdim$--regular with regularity constant $\mathbf{C}_2:=(50\mathbf{K}_0\mathbf{C})^{{1+\delta\over 1-\delta}}$.
\end{lemm}
%%%%%%%%%%%%%%%%%%%%%%%%%%%%%%%%%%%%%%%%%%%%%%%%%%%%%%%%%%%%%%%%%%%%%%%%%%%%%%%%
\begin{proof}
By Lemma~\ref{l:adred}, the set
$$
\mathcal Y:=R^{-1}\mathcal G(y_0,\Lambda_\Gamma)\ \subset\ \mathbb R
$$
is $\delta$-regular with constant $\mathbf K_0 \mathbf C$.
Divide $[-2,-1]$ and $[1,2]$ into $\mathbf C_1$
intervals of size $\mathbf C_1^{-1}$ each,
where $\mathbf{C}_1:=\lceil(10\mathbf{K}_0^2\mathbf{C}^2)^{\frac{1}{1-\setdim}}\rceil$.
By Proposition~\ref{ADStronglyAvoids},
at least one of the sub-intervals in $[-2,-1]$ and at least one of the sub-intervals in $[1,2]$ must be disjoint from~$\mathcal{Y}$. Call these intervals $I_1$ and $I_2$. Let $I$ be the convex hull of the midpoints of $I_1$ and~$I_2$. 

We will show that $\mathcal{X}:=I\cap\mathcal Y$ is $\setdim$--regular. Let $x\in \mathcal{X}$ and $0<r\leq 4$. We immediately have 
\begin{equation*}
\mu_{\setdim}\big(\mathcal{X}\cap B(x,r)\big)\leq\mu_{\setdim}\big(\mathcal Y\cap B(x,r)\big)\leq  \mathbf{K}_0\mathbf{C}r^\setdim. 
\end{equation*}
It remains to prove a lower bound. 
If $r<\mathbf{C}_1^{-1}$ then $ \mathcal Y\cap B(x,r)\subset I$ and thus 
\begin{equation*}
\mu_{\setdim}(\mathcal{X}\cap B(x,r))=\mu_{\setdim}\big(\mathcal Y\cap B(x,r)\big)\geq  (\mathbf{K}_0\mathbf{C})^{-1}r^\setdim. 
\end{equation*}
On the other hand, if $\mathbf{C}_1^{-1}\leq r\leq 4$, then 
\begin{equation*}
\begin{split}
\mu_{\setdim}(\mathcal{X}\cap B(x,r))&\geq \mu_{\setdim}(\mathcal{X}\cap B(x,\mathbf{C}_1^{-1}))\\
&\geq \mathbf{K}_0^{-1}\mathbf{C}^{-1}\mathbf{C}_1^{-\setdim}\\
&\geq \mathbf C_2^{-1} r^{\setdim}.\qedhere
\end{split}
\end{equation*}

\end{proof}
%%%%%%%%%%%%%%%%%%%%%%%%%%%%%%%%%%%%%%%%%%%%%%%%%%%%%%%%%%%%%%%%%%%%%%%%%%%%%%%%
We finally combine Theorem~\ref{t:ae-combinatorial} from~\S\ref{s:ae-combinatorial},
Lemma~\ref{l:alphasep}, and Lemma~\ref{l:ad-helper} to obtain
%%%%%%%%%%%%%%%%%%%%%%%%%%%%%%%%%%%%%%%%%%%%%%%%%%%%%%%%%%%%%%%%%%%%%%%%%%%%%%%%
\begin{proof}[Proof of Theorem~\ref{t:ad-reduced}]
Let $\Lambda_\Gamma$ be the limit set of a convex co-compact hyperbolic group. Let $y_0\in\Lambda_\Gamma.$ Let $\h>0$ (small) and $C_1\geq 1$ (large),
and put $\alpha_1:=C_1^{-1}\alpha$.
First, note that 
\begin{equation}\label{boundOnRescaledBallAE}
E_A\big(\mathcal G(y_0,\Lambda_\Gamma)\cap B(0,C_1),\alpha\big)
= E_A\big(C_1^{-1}\mathcal G(y_0,\Lambda_\Gamma)\cap B(0,1),\alpha_1\big)\leq
E_A(\mathcal X,\alpha_1),
\end{equation}
where $\mathcal X:=I\cap C_1^{-1}\mathcal G(y_0,\Lambda_\Gamma)\subset[-2, 2]$ is defined in Lemma~\ref{l:ad-helper}.

Each point $(\eta_1,\eta_2,\eta_3,\eta_4)\in \mathcal X(\alpha_1)^4$ satisfying
$|\eta_1-\eta_2+\eta_3-\eta_4|\leq\alpha_1$ lies in the $4\alpha_1$-neighborhood
of the set
$$
\mathcal Z_{\alpha_1}=\{(\eta_1,\eta_2,\eta_3,\eta_4)\in\mathcal X^4\mid
|\eta_1-\eta_2+\eta_3-\eta_4|\leq 5\alpha_1\}.
$$
By Definition~\ref{d:ae} and part~2 of Lemma~\ref{l:alphasep}, we have for some global constant $C$,
\begin{equation}
  \label{e:allai1}
E_A(\mathcal X,\alpha_1)\leq C \mathcal N(\mathcal Z_{\alpha_1},4\alpha_1).
\end{equation}
We next claim that 
\begin{equation}
  \label{e:allai2}
\mathcal N(\mathcal Z_{\alpha_1},4\alpha_1)\leq C\alpha_1^{-4\delta}\mathcal E_A(\mathcal X,\mu_\delta,9\alpha_1)
\end{equation}
where $\mathcal E_A$ is given by Definition~\ref{d:aespecial} and $C$ is some constant
depending on $\mathbf C$. Indeed, let $z_1,\dots,z_N\in \mathcal Z_{\alpha_1}$ be a $4\alpha_1$-separated set of points.
Then
$$
\bigsqcup_{j=1}^N B(z_j,2\alpha_1)\ \subset\ \{|\eta_1-\eta_2+\eta_3-\eta_4|\leq 9\alpha_1\}.
$$
Taking the $\mu_\delta^4$ measure of the intersection of both sides with $\mathcal X^4$
and arguing similarly to the proof of part~1 of Lemma~\ref{l:alphasep}, we obtain~\eqref{e:allai2}.

Finally, applying Theorem~\ref{t:ae-combinatorial} to ${1\over 2}\mathcal X$,
which is $\delta$-regular by Lemma~\ref{l:ad-helper}, we obtain
\begin{equation}\label{boundOnSlightlyEnlargedBallAE}
\mathcal E_A(\mathcal X,\mu_\delta,9\h_1)\leq C\alpha_1^{\delta+\beta_E}.
\end{equation}
Here $\beta_E=\betaXBoundBFC,$ where $\mathbf{K}$ is an absolute constant,
and $C$ is some constant depending on $\mathbf C$.

Combining \eqref{boundOnRescaledBallAE}--\eqref{boundOnSlightlyEnlargedBallAE}, we conclude that $\Lambda_{\Gamma}$ satisfies the additive energy bound with exponent $\beta_E$ in the sense of Definition~\ref{d:ae-estimate}.
\end{proof}
%%%%%%%%%%%%%%%%%%%%%%%%%%%%%%%%%%%%%%%%%%%%%%%%%%%%%%%%%%%%%%%%%%%%%%%%%%%%%%%%

%%%%%%%%%%%%%%%%%%%%%%%%%%%%%%%%%%%%%%%%%%%%%%%%%%%%%%%%%%%%%%%%%%%%%%%%%%%%%%%%
\subsection{Example: three-funneled surfaces}
  \label{s:3fun}

We now consider a particular family of convex co-compact hyperbolic surfaces
and show that the regularity constants
in Lemma~\ref{l:adred} for
the corresponding limit sets
have a uniform upper bound when the surface varies in a compact
set in the moduli space. (Similar reasoning is expected to work for general convex co-compact hyperbolic surfaces.)

More precisely, we study the family of \emph{three-funneled surfaces}
$M_{\ell}$,
parametrized by the Fenchel--Nielsen coordinates
$\ell:=(\ell_1,\ell_2,\ell_3)\in (0,\infty)^3$. To construct $M_\ell$,
we start with a right-angled hyperbolic hexagon with
sides $\ell_1/2,q_3,\ell_2/2,q_1,\ell_3/2,q_2$. This hexagon is unique
up to isometry, and $q_1,q_2,q_3>0$ are determined by a formula
involving only $\ell_1,\ell_2,$ and $\ell_3$ (see for example~\cite[Theorem~3.5.13]{Ratcliffe}). Gluing two such
hexagons along the $q_3$ side, we obtain a right-angled hyperbolic octagon
with sides $\ell_1,q_2,\ell_3/2,q_1,\ell_2,q_1,\ell_3/2,q_2$.

Attaching funnel ends along the $\ell_1,\ell_2,\ell_3/2$ sides to the above octagon,
we obtain a fundamental domain $\mathcal F_{\ell}\subset\mathbb H^2$.
The complement of $\mathcal F_{\ell}$ is the disjoint union of four open geodesic
half-disks (i.e. regions of $\mathbb H^2$ bounded by a geodesic)
$D_1,D_2,D_3,D_4$, where $D_1,D_3$ are bounded by the geodesics containing
the two $q_2$ sides of the octagon and $D_2,D_4$ are bounded by the geodesics
containing its $q_1$ sides. See Figure~\ref{f:schottky}.

We next define the group $\Gamma_\ell\subset\PSL(2,\mathbb R)$ using the
following Schottky representation.
Let $\omega_1,\omega_2$ be the geodesics on $\mathbb H^2$ containing the
$\ell_1$ and $\ell_2$ sides of the octagon.
Then there exist unique $\gamma_1,\gamma_2\in \PSL(2,\mathbb R)$ satisfying
\begin{equation}
  \label{e:schottky}
\gamma_j(\overline{D_j})=\mathbb H^2\setminus D_{j+2},\quad
\gamma_j(\omega_j)=\omega_j,\quad
j=1,2.
\end{equation}
%
%%%%%%%%%%%%%%%%%%%%%%%%%%%%%%%%%%%%%%%%%%%%%%%%%%%%%%%%%%%%%%%%%%%%%%%%%%%%%%%%
\begin{figure}
\includegraphics{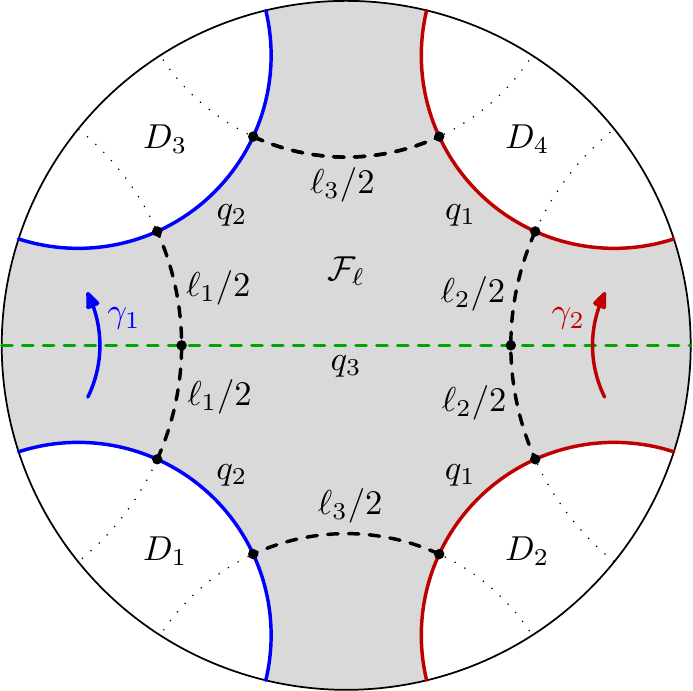}
\qquad
\includegraphics{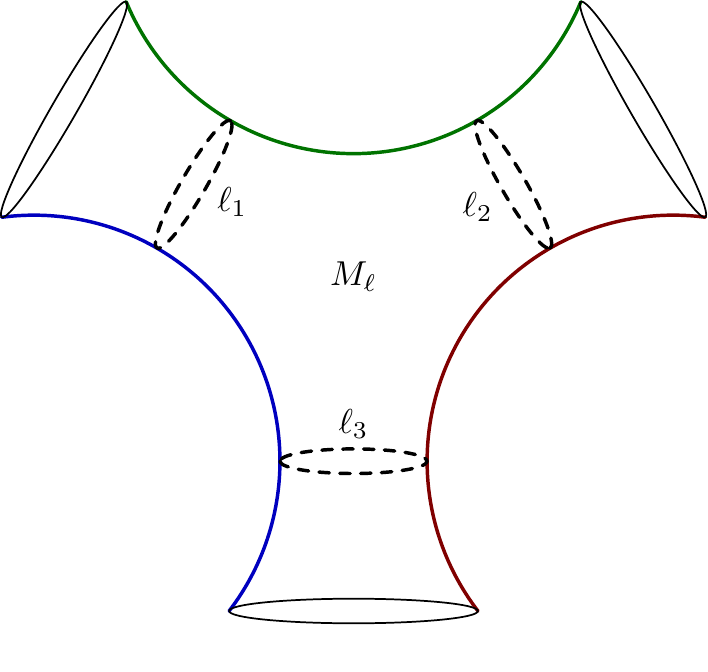}
\caption{A three-funneled surface (on the right) and its fundamental domain
in the Poincar\'e disk model (on the left).}
\label{f:schottky}
\end{figure}%
%%%%%%%%%%%%%%%%%%%%%%%%%%%%%%%%%%%%%%%%%%%%%%%%%%%%%%%%%%%%%%%%%%%%%%%%%%%%%%%%
Let $\Gamma_\ell$ be the group generated by $\gamma_1,\gamma_2$.
Then $\Gamma_\ell$ is a free group, and the quotient
$$
M_\ell:=\Gamma_\ell\backslash\mathbb H^2
$$
is a convex co-compact hyperbolic surface with fundamental domain $\mathcal F_\ell$.
The numbers $\ell_1,\ell_2,\ell_3$ are the lengths of the geodesic necks
separating the funnels of $M_\ell$ from the convex core.
See for instance~\cite[\S15.1]{Borthwick} for details.
The octagon constructed above, the disks $D_j$, and the group elements $\gamma_j$
depend continuously on the choice of $\ell=(\ell_1,\ell_2,\ell_3)$.

Denote by $\Lambda_\ell=\Lambda_{\Gamma_\ell}$ the limit set of $\Gamma_\ell$ and
by $\delta_\ell$ its dimension.
The following proposition states that $\Lambda_\ell$ is Ahlfors--David
regular with regularity constant locally uniform in $\ell$; we use same notation
as in Lemma~\ref{l:adred}.
%%%%%%%%%%%%%%%%%%%%%%%%%%%%%%%%%%%%%%%%%%%%%%%%%%%%%%%%%%%%%%%%%%%%%%%%%%%%%%%%
\begin{prop}
  \label{l:3fun}
Let $\mathscr K\subset (0,\infty)^3$ be a compact set. Then there exists a constant
$\mathbf C_{\mathscr K}>0$ such that for each $\ell\in\mathscr K$
and each $y_0,y_1\in\Lambda_\ell$, $y_0\neq y_1$ we have
\begin{equation}
  \label{e:3fun}
\mathbf C_{\mathscr K}^{-1}c_\ell r^\delta\ \leq\ \mu_\delta\big(\mathcal G(y_0,\Lambda_\ell)\cap B(\mathcal G(y_0,y_1),r)\big)
\ \leq\ \mathbf C_{\mathscr K}c_\ell r^\delta,\quad
r>0
\end{equation}
where $c_\ell>0$ depends only on $\ell$.
\end{prop}
%%%%%%%%%%%%%%%%%%%%%%%%%%%%%%%%%%%%%%%%%%%%%%%%%%%%%%%%%%%%%%%%%%%%%%%%%%%%%%%%
Before proving Proposition~\ref{l:3fun}, we use it to show the following
%%%%%%%%%%%%%%%%%%%%%%%%%%%%%%%%%%%%%%%%%%%%%%%%%%%%%%%%%%%%%%%%%%%%%%%%%%%%%%%%
\begin{theo}
  \label{t:3funny}
There exists an open set $\mathscr U\subset (0,\infty)^3$ such that:
\begin{itemize}
\item if $\ell\in (0,\infty)^3$ and $\delta_\ell=1/2$, then $\ell\in\mathscr U$;
\item if $\ell\in\mathscr U$, then $M_\ell$ has an essential spectral gap
in the sense of~\eqref{e:essential-gap2} of size
$$
\beta=\beta_\ell>\max\Big(0,{1\over 2}-\delta_\ell\Big).
$$
\end{itemize}
\end{theo}
%%%%%%%%%%%%%%%%%%%%%%%%%%%%%%%%%%%%%%%%%%%%%%%%%%%%%%%%%%%%%%%%%%%%%%%%%%%%%%%%
\noindent\textbf{Remark}.
It is well-known that $\delta_\ell$ depends continuously on $\ell$~--- see
for example~\cite{Anderson-Rocha}. Moreover, there exist $\ell_-,\ell_+\in (0,\infty)^3$
such that $\delta_{\ell_-}<1/2<\delta_{\ell_+}$.
In fact, for the case $\ell_1=\ell_2=\ell_3$, we have
$\delta_\ell\to 1$ as $\ell_j\to 0$ and $\delta_\ell\to 0$
as $\ell_j\to \infty$~--- see~\cite[Theorem~3.5]{McMullen}.
By considering a path
connecting $\ell_-$ with $\ell_+$ and applying Theorem~\ref{t:3funny},
we see that there exist $\ell$ such that $\delta_\ell>1/2$, yet 
$M_\ell$ has an essential spectral
gap of size $\beta_\ell>0$. 
%%%%%%%%%%%%%%%%%%%%%%%%%%%%%%%%%%%%%%%%%%%%%%%%%%%%%%%%%%%%%%%%%%%%%%%%%%%%%%%%
\begin{proof}
It suffices to show that for each $\tilde\ell$ with $\delta_{\tilde\ell}=1/2$,
there exists $\tilde\beta>0$ and a neighborhood $U_{\tilde\ell}$ of $\tilde\ell$ such that
for each $\ell\in U_{\tilde\ell}$, $M_\ell$ has an essential spectral gap of size $\tilde\beta$.
Indeed, it follows from here that there is an open neighborhood $U'_{\tilde\ell}$ of $\tilde\ell$
such that for each $\ell\in U'_{\tilde\ell}$, $M_\ell$ has an essential gap of
size $\beta_\ell>\max(0,1/2-\delta_\ell)$. It remains to let $\mathscr U$ be the
union of all $U'_{\tilde\ell}$.

To show the existence of the neighborhood $U_{\tilde\ell}$, by Theorems~\ref{t:fup-reduction} and~\ref{t:ae-reduction}
it suffices to
show that there exists a constant $\beta_E>0$
such that for all $\ell$ sufficiently close to $\tilde\ell$,
the set~$\Lambda_\ell$ satisfies the additive energy bound with
exponent $\beta_E$ in the sense of Definition~\ref{d:ae-estimate}.
To show this, we argue as in the proof of Theorem~\ref{t:ad-reduced} in~\S\ref{s:minkowski}.
The only difference is that Lemma~\ref{l:adred} is replaced by Proposition~\ref{l:3fun}.
The constant $c_\ell$ in~\eqref{e:3fun} can be removed by Lemma~\ref{equivOfADRegDefns};
alternatively, we may argue using the measure $c_\ell^{-1}\mu_\delta$
instead of $\mu_\delta$ since the proof of Theorem~\ref{t:ae-combinatorial}
never used that $\mu_\delta$ is the Hausdorff measure.
\end{proof}
%%%%%%%%%%%%%%%%%%%%%%%%%%%%%%%%%%%%%%%%%%%%%%%%%%%%%%%%%%%%%%%%%%%%%%%%%%%%%%%%
We now prove Proposition~\ref{l:3fun}. Assume that $\ell$ varies
in a compact subset of $(0,\infty)^3$; the constants below will
depend on that subset.
We start with the following
%%%%%%%%%%%%%%%%%%%%%%%%%%%%%%%%%%%%%%%%%%%%%%%%%%%%%%%%%%%%%%%%%%%%%%%%%%%%%%%%
\begin{lemm}
  \label{l:bashas}
Assume that
$$
\gamma=\begin{pmatrix} a&b\\c&d\end{pmatrix}\in \Gamma_\ell\subset\PSL(2,\mathbb R),\quad
|\gamma|:=a^2+b^2+c^2+d^2.
$$
Then for each $A\subset \mathbb S^1$, we have
\begin{equation}
  \label{e:bashas}
\mu_{\delta_\ell}(\Lambda_\ell\cap \gamma(A))\leq (2|\gamma|)^{\delta_\ell} \mu_{\delta_\ell}(\Lambda_\ell\cap A).
\end{equation}
\end{lemm}
%%%%%%%%%%%%%%%%%%%%%%%%%%%%%%%%%%%%%%%%%%%%%%%%%%%%%%%%%%%%%%%%%%%%%%%%%%%%%%%%
\begin{proof}
We identify the upper half-plane model $\{\Im z>0\}$ with the
disk model $\{|w|<1\}$ by the M\"obius transformation
$$
w={z-i\over z+i},\quad
z=i{1+w\over 1-w}.
$$
With the M\"obius transformation $\gamma$ given in the $z$ variable
by $\gamma.z={az+b\over cz+d}$, its derivative in the $w$ variable
on $\mathbb S^1$ satisfies
$$
|\gamma'(w)|={1+z^2\over (az+b)^2+(cz+d)^2},\quad
w\in\mathbb S^1.
$$
It follows that
$|\gamma'(w)|^{-1}\leq 2|\gamma|$;
substituting $\gamma^{-1}$ instead of $\gamma$,
we get $|\gamma'(w)|\leq 2|\gamma|$.
The estimate~\eqref{e:bashas} follows from here
and the fact that $\Lambda_\ell\cap\gamma(A)=\gamma(\Lambda_\ell\cap A)$.
\end{proof}
%%%%%%%%%%%%%%%%%%%%%%%%%%%%%%%%%%%%%%%%%%%%%%%%%%%%%%%%%%%%%%%%%%%%%%%%%%%%%%%%
Using Lemma~\ref{l:bashas}, we next prove
%%%%%%%%%%%%%%%%%%%%%%%%%%%%%%%%%%%%%%%%%%%%%%%%%%%%%%%%%%%%%%%%%%%%%%%%%%%%%%%%
\begin{lemm}
  \label{l:bashas2}
There exists a constant $c>0$ such that
\begin{equation}
  \label{e:bashas2}
\mu_{\delta_\ell}(\Lambda_\ell\cap \overline{D_j})\geq c\mu_{\delta_\ell}(\Lambda_\ell),\quad
j=1,2,3,4.
\end{equation}
\end{lemm}
%%%%%%%%%%%%%%%%%%%%%%%%%%%%%%%%%%%%%%%%%%%%%%%%%%%%%%%%%%%%%%%%%%%%%%%%%%%%%%%%
\begin{proof}
We consider the case $j=1$, the other cases are treated similarly.

By~\eqref{e:schottky}, we have $\overline{D_1\cup D_2\cup D_4}\subset\gamma_1(\overline{D_1})$.
Since $|\gamma_1|$ is bounded by some constant depending on $\mathscr K$, by Lemma~\ref{l:bashas}
we obtain for some constant $C$,
\begin{equation}
  \label{e:bashas2.1}
\mu_{\delta_\ell}\big(\Lambda_\ell\cap (\overline{D_1\cup D_2\cup D_4})\big)\leq C\mu_{\delta_\ell}(\Lambda_\ell\cap\overline{D_1}).
\end{equation}
Next, $\overline{D_3}\subset\gamma_2(\overline{D_2})$. Therefore,
similarly to~\eqref{e:bashas2.1} we get
\begin{equation}
  \label{e:bashas2.2}
\mu_{\delta_\ell}(\Lambda_\ell\cap \overline{D_3})\leq C\mu_{\delta_\ell}(\Lambda_\ell\cap \overline{D_2}).
\end{equation}
Combining~\eqref{e:bashas2.1} and~\eqref{e:bashas2.2} and using
that $\Lambda_\ell\subset \overline{D_1\cup D_2\cup D_3\cup D_4}$,
we obtain~\eqref{e:bashas2}.
\end{proof}
%%%%%%%%%%%%%%%%%%%%%%%%%%%%%%%%%%%%%%%%%%%%%%%%%%%%%%%%%%%%%%%%%%%%%%%%%%%%%%%%
We are now ready to give
%%%%%%%%%%%%%%%%%%%%%%%%%%%%%%%%%%%%%%%%%%%%%%%%%%%%%%%%%%%%%%%%%%%%%%%%%%%%%%%%
\begin{proof}[Proof of Lemma~\ref{l:3fun}]
We argue similarly to the proof of Lemma~\ref{l:adred}.
Take $y_0,y_1\in\Lambda_\ell$, $y_0\neq y_1$, and $r>0$.
Fix a large constant $C_1>0$ depending only on $\mathscr K$,
to be chosen later.
Let $(x,\xi)\in S^*\mathbb H^2$ be the unique point satisfying
$$
B_-(x,\xi)=y_0,\quad
B_+(x,\xi)=y_1,\quad
\mathcal P(x,y_0)=r/C_1.
$$
By~\eqref{e:Gpm-formula}, the projection $\pi_\Gamma(x,\xi)\in S^*M$
lies in the trapped set. Take $\gamma\in\Gamma_\ell$ such that for
$(x',\xi'):=\gamma.(x,\xi)$, the point $x'$ lies in the fundamental domain $\mathcal F_\ell$,
and denote
$$
y'_0:=B_-(x',\xi'),\quad
y'_1:=B_+(x',\xi').
$$
Since $\pi_\Gamma(x',\xi')$ is in the trapped set, $x'$
lies in the convex core, which is the octagon used in the construction of
$\mathcal F_\ell$. Thus there exists a constant $C_2>0$ depending
only on $\mathscr K$ such that
$$
C_2^{-1}\leq \mathcal P(x',y'_0)\leq C_2.
$$
Applying~\eqref{e:huangpu1}, we see that
$$
\begin{aligned}
&(C_1C_2)^{-\delta} r^\delta \mu_\delta\big(\mathcal G(y'_0,\Lambda_\ell)\cap B(\mathcal G(y'_0,y'_1),C_1/C_2)\big)\\
\leq\ &\mu_\delta\big(\mathcal G(y_0,\Lambda_\ell)\cap B(\mathcal G(y_0,y_1),r)\big)\\
\leq\ &
(C_2/C_1)^\delta r^\delta \mu_\delta\big(\mathcal G(y'_0,\Lambda_\ell)\cap B(\mathcal G(y'_0,y'_1),C_1C_2)\big).
\end{aligned}
$$
Therefore, in order to prove~\eqref{e:3fun}
it is enough to verify the inequalities
$$
\begin{aligned}
\mu_\delta\big(\mathcal G(y'_0,\Lambda_\ell)\cap B(\mathcal G(y'_0,y'_1),C_1C_2)\big)&\leq
Cc_\ell,\\
\mu_\delta\big(\mathcal G(y'_0,\Lambda_\ell)\cap B(\mathcal G(y'_0,y'_1),C_1/C_2)\big)&\geq
C^{-1}c_\ell
\end{aligned}
$$
for some constants $C$ depending only on $\mathscr K$ and $c_\ell$ depending on $\ell$.

We have $y'_0,y'_1\in\Lambda_\ell$, thus $y'_0\in \overline {D_j},y'_1\in\overline{D_k}$
for some $j,k\in\{1,2,3,4\}$. Moreover, $x'\in\mathcal F_\ell$ implies that
$j\neq k$. Thus $|y'_0-y'_1|$ is bounded away from zero uniformly in $\ell\in\mathscr K$.
Since $\mathcal G$ is a smooth map away from the diagonal, we see
that it is enough to prove the inequalities
\begin{align}
  \label{e:fuego2}
\mu_\delta\big(\{y\in\Lambda_\ell\mid \mathcal G(y'_0,y)\in B(\mathcal G(y'_0,y'_1),C_1C_2)\}\big)&\leq Cc_\ell,
\\
  \label{e:fuego3}
\mu_\delta\big(\{y\in\Lambda_\ell\mid \mathcal G(y'_0,y)\in B(\mathcal G(y'_0,y'_1),C_1/C_2)\}\big)&\geq C^{-1}c_\ell.
\end{align}
We put $c_\ell:=\mu_\delta(\Lambda_\ell)$, then~\eqref{e:fuego2} follows automatically.
To show~\eqref{e:fuego3}, we note that for $C_1$ large enough depending on $\mathscr K$,
the set on the left-hand side contains $\Lambda_\ell\cap \overline{D_k}$. It remains
to use Lemma~\ref{l:bashas2}.
\end{proof}
%%%%%%%%%%%%%%%%%%%%%%%%%%%%%%%%%%%%%%%%%%%%%%%%%%%%%%%%%%%%%%%%%%%%%%%%%%%%%%%%

%%%%%%%%%%%%%%%%%%%%%%%%%%%%%%%%%%%%%%%%%%%%%%%%%%%%%%%%%%%%%%%%%%%%%%%%%%%%%%%%
\appendix

%%%%%%%%%%%%%%%%%%%%%%%%%%%%%%%%%%%%%%%%%%%%%%%%%%%%%%%%%%%%%%%%%%%%%%%%%%%%%%%%
\section{Calculations on hyperbolic quotients}
  \label{s:hyperbolic-technical}

In this appendix, we prove several technical lemmas from~\S\ref{s:hyperbolic} and~\S\ref{s:ae}
regarding the geometry of convex co-compact hyperbolic quotients $M=\Gamma\backslash\mathbb H^n$.
Here
$$
\Gamma\ \subset\ G:=\PSO(1,n)
$$
is a group of hyperbolic isometries.

A useful tool is the \emph{coframe bundle} $F^* M$ whose points have the form
$(x,\xi_1,\dots,\xi_n)$, where $x\in M$ and $\xi_1,\dots,\xi_n\in T^*_xM$ form a positively oriented orthonormal basis.
We identify $F^*M$ with $\Gamma\backslash G$
by the diffeomorphism
$$
[\gamma]\in\Gamma\backslash G\ \mapsto\ \pi^F_\Gamma(\gamma(0),d\gamma(0)^{-T}\cdot dx_1,\dots,d\gamma(0)^{-T}\cdot dx_n)
$$
where we use the ball model for $\mathbb H^n$ and
$$
\pi^F_\Gamma:F^*\mathbb H^n\to F^* M
$$
is the covering map. Using this identification, we obtain a right action of $G$
on $F^*M$. Then each element of the Lie algebra of $G$, viewed as a left invariant vector field,
induces a vector field on $F^*M$. We use the induced vector fields
\begin{equation}
  \label{e:horvec}
U_1^\pm,\dots,U_{n-1}^\pm,\quad
\widetilde X,\quad R_{ab},\ 2\leq a,b\leq n
\end{equation}
on $F^*M$ defined in~\cite[\S3.2]{rrh} ($\widetilde X$ is denoted by simply $X$ in~\cite{rrh}). The following
commutation relations hold~\cite[(3.8)]{rrh}:
\begin{equation}
  \label{e:horcom}
[\widetilde X,U^\pm_j]=\pm U^\pm_j,\quad
[\widetilde X,R_{ab}]=0,\ 2\leq a,b\leq n.
\end{equation}
The vector field $\widetilde X$ projects down to $X$ by $\pi_S$; in fact,
$\widetilde X$ is the generator of parallel transport along geodesics.
The vector fields $R_{ab}$ generate rotations of the vectors $(\xi_2,\dots,\xi_n)$;
in fact, the kernel of $d\pi_S$ is spanned by $(R_{ab})$
where $\pi_S$ is the following submersion:
\begin{equation}
  \label{e:pi-S}
\pi_S:F^*M\to S^*M,\quad
(x,\xi_1,\dots,\xi_n)\mapsto (x,\xi_1).
\end{equation}

%%%%%%%%%%%%%%%%%%%%%%%%%%%%%%%%%%%%%%%%%%%%%%%%%%%%%%%%%%%%%%%%%%%%%%%%%%%%%%%%
\begin{proof}[Proof of Lemma~\ref{l:propagated-okay}]
We show the first statement, the second one is proved similarly.
Since $L_s$ is an integrable foliation, induction on $m+k$ shows that the order of applying vector fields
in~\eqref{e:symbols-def} does not matter for establishing the bound. By~\eqref{e:sudec}, it suffices
to prove the bounds
\begin{equation}
  \label{e:derby}
\begin{gathered}
\sup_{T^*M\setminus 0}|\chi_2  Y_1\dots Y_m Z_1\dots Z_k(\xi\cdot\partial_\xi)^iX^j(\chi_1\circ e^{tX})|\leq Ch^{-\rho k},\\
Y_1,\dots, Y_m\in C^\infty(T^*M\setminus 0;E_s),\quad
Z_1,\dots, Z_k\in C^\infty(T^*M\setminus 0;E_u).
\end{gathered}
\end{equation}
Since $X$ commutes with itself and with $\xi\cdot\partial_\xi$, we have
$$
(\xi\cdot\partial_\xi)^iX^j(\chi_1\circ e^{tX})=((\xi\cdot\partial_\xi)^iX^j\chi_1)\circ e^{tX}
$$
and thus we may assume that $i=j=0$.
Since $e^{tX}$ is a homogeneous flow and $E_s,E_u$ are homogeneous foliations,
we may restrict to the cosphere bundle $S^*M$.

On $F^*M$, we have (see~\cite[\S3.3]{rrh})
$$
\pi_S^*E_s=\Span(\{U^+_j\}\cup\{R_{ab}\}),\quad
\pi_S^*E_u=\Span(\{U^-_j\}\cup\{R_{ab}\}),
$$
(The fields $U^\pm_j$ do not project down to vector fields on $S^*M$ for $n\geq 3$,
hence the use of the coframe bundle.)

Then \eqref{e:derby} reduces to the following bound on $F^*M$:
$$
\sup_{F^*M}|(\pi_S^*\chi_2) U^+_{j_1}\dots U^+_{j_m}U^-_{r_1}\dots U^-_{r_k}((\pi_S^*\chi_1)\circ e^{t\widetilde X})|\leq Ch^{-\rho k}.
$$
Using~\eqref{e:horcom}, we see that
$$
U^+_{j_1}\dots U^+_{j_m}U^-_{r_1}\dots U^-_{r_k}((\pi_S^*\chi_1)\circ e^{t\widetilde X})
=e^{(k-m)t}(U^+_{j_1}\dots U^+_{j_m}U^-_{r_1}\dots U^-_{r_k}(\pi_S^*\chi_1))\circ e^{t\widetilde X}
$$
and the proof is finished as $e^{(k-m)t}\leq e^{kt}\leq h^{-\rho k}$ when $t\in [0,\rho\log(1/h)]$.
\end{proof}
%%%%%%%%%%%%%%%%%%%%%%%%%%%%%%%%%%%%%%%%%%%%%%%%%%%%%%%%%%%%%%%%%%%%%%%%%%%%%%%%
%
%%%%%%%%%%%%%%%%%%%%%%%%%%%%%%%%%%%%%%%%%%%%%%%%%%%%%%%%%%%%%%%%%%%%%%%%%%%%%%%%
\begin{proof}[Proof of~\eqref{e:Gpm-formula}]
We prove the second statement, the first one is proved similarly.
Assume first that $(x,\xi)\in T^*\mathbb H^n\setminus 0$ and 
$\pi_\Gamma(x,\xi)\in \Gamma_-$, where $\pi_\Gamma$ is defined in~\eqref{e:pi-Gamma}
and descends to a covering map $\mathbb H^n\to M$ (also called $\pi_\Gamma$).
Let $x(t)\in \mathbb H^n$ be the projection of $e^{tX}(x,\xi)$ to the base.
Since $(x,\xi)\in \Gamma_-$, the curve $\pi_\Gamma(x(t))$ stays in a compact
subset of $M$ when $t\geq 0$; therefore, there exists a sequence $t_j\to \infty$
and $x_\infty\in \mathbb H^n$ such that
$$
\pi_\Gamma(x(t_j))\to \pi_\Gamma(x_\infty)\quad\text{as }j\to\infty.
$$
Take $\gamma_j\in\Gamma$ such that
$d_{\mathbb H^n}(x(t_j),\gamma_j.x_\infty)\to 0$ as $j\to\infty$.
Since $x(t_j)\to B_+(x,\xi)$ in the topology of the closed disk model $\overline{\mathbb H^n}$,
we have $\gamma_j.x_\infty\to B_+(x,\xi)$ in $\overline{\mathbb H^n}$ and thus
$B_+(x,\xi)\in \Lambda_\Gamma$ by~\eqref{e:Lambda-Gamma}.

Assume now that $\pi_\Gamma(x,\xi)\notin\Gamma_-$. Then (see for instance~\cite[\S7.1 and Appendix~A.1]{qeefun})
the trajectory $\pi_\Gamma(x(t))$ converges in the compactified space $\overline M$
to some point $y\in \partial\overline M$ on the conformal boundary.
The point $B_+(x,\xi)$ projects to $y$ by an extension of $\pi_\Gamma$ to the conformal
boundary and there exists a neighborhood $U$ of $B_+(x,\xi)$ in $\overline{\mathbb H^n}$ such that
$\pi_\Gamma$ is a local diffeomorphism on $U$. This implies that $B_+(x,\xi)$
cannot be in the limit set $\Lambda_\Gamma$, finishing the proof.
\end{proof}
%%%%%%%%%%%%%%%%%%%%%%%%%%%%%%%%%%%%%%%%%%%%%%%%%%%%%%%%%%%%%%%%%%%%%%%%%%%%%%%%
%
Next, we prove Lemma~\ref{l:close-to-trapping}. We will use the following
%%%%%%%%%%%%%%%%%%%%%%%%%%%%%%%%%%%%%%%%%%%%%%%%%%%%%%%%%%%%%%%%%%%%%%%%%%%%%%%%
\begin{lemm}
  \label{l:reallytech}
Assume that $(x,\xi_1,\dots,\xi_n)\in F^*\mathbb H^n$, $y\in\mathbb S^{n-1}$,
and
\begin{equation}
  \label{e:rt-1}
y\neq B_-(x,\xi_1).
\end{equation}
Then there exists $\mathbf s=(s_1,\dots,s_{n-1})\in\mathbb R^{n-1}$
such that, putting (see~\eqref{e:horvec})
\begin{equation}
  \label{e:horoflow}
e^{\mathbf s U^-}:=\exp(s_1U^-_1+\dots+s_nU^-_n):F^*\mathbb H^n\to F^*\mathbb H^n,
\end{equation}
we have (with $\pi_S$ defined in~\eqref{e:pi-S})
\begin{equation}
  \label{e:rt-2}
B_+(\pi_S(e^{\mathbf s U^-}(x,\xi_1,\dots,\xi_n)))=y.
\end{equation}
Moreover, if $(x,\xi_1,y)$ varies in some fixed compact subset of $S^*\mathbb H^n\times\mathbb S^{n-1}$
satisfying~\eqref{e:rt-1}, then we may choose
$s_1,\dots,s_{n-1}\in [-C,C]$ where $C$ is independent of $x,\xi_1,\dots,\xi_n,y$.   
\end{lemm}
%%%%%%%%%%%%%%%%%%%%%%%%%%%%%%%%%%%%%%%%%%%%%%%%%%%%%%%%%%%%%%%%%%%%%%%%%%%%%%%%
\begin{proof}
We use the hyperboloid model of $\mathbb H^n$, see~\cite[\S3.1]{rrh}. In particular,
we consider elements of $G=\PSO(1,n)$ as linear automorphisms of the Minkowski
space $\mathbb R^{1,n}$. If $e_0,\dots,e_n$ is the canonical basis of $\mathbb R^{1,n}$
and $\gamma\in G$ is defined by
$$
\gamma.(e_0,\dots,e_n)=(x,\xi_1,\dots,\xi_n),
$$
then a direct calculation using~\cite[(3.7) and (3.16)]{rrh} shows that
for some $q>0$,
$$
\big(q,qB_+(\pi_S(e^{\mathbf sU^-}(x,\xi_1,\dots,\xi_n)))\big)=\gamma.(1+|\mathbf s|^2,
1-|\mathbf s|^2,-2s_1,\dots,-2s_{n-1})
$$
where both sides of the equation are vectors in the positive half
of the light cone in $\mathbb R^{1,n}$. The vector $\gamma^{-1}.(1,y)$ lies
in the positive half of the light cone,
therefore we have for some $r>0$ and $v\in\mathbb S^{n-1}$,
$$
\gamma^{-1}.(1,y)=(r,rv).
$$
By~\cite[(3.16)]{rrh}, we see that~\eqref{e:rt-1} implies that
\begin{equation}
  \label{e:rti-1}
v\neq (-1,0,\dots,0). 
\end{equation}
To obtain~\eqref{e:rt-2}, we have to find $\mathbf s$ such that
$$
v_1={1-|\mathbf s|^2\over 1+|\mathbf s|^2};\quad
v_{j+1}=-{2s_j\over 1+|\mathbf s|^2},\
1\leq j\leq n-1.
$$
The existence of such $\mathbf s$ follows directly from~\eqref{e:rti-1},
and the fact that $\mathbf s$ can be chosen bounded uniformly in $(x,\xi_1,\dots,\xi_n,y)$
follows immediately because under~\eqref{e:rt-1},
the distance from $v$ to $(-1,0,\dots,0)$ is bounded away from zero.
\end{proof}
%%%%%%%%%%%%%%%%%%%%%%%%%%%%%%%%%%%%%%%%%%%%%%%%%%%%%%%%%%%%%%%%%%%%%%%%%%%%%%%%

%%%%%%%%%%%%%%%%%%%%%%%%%%%%%%%%%%%%%%%%%%%%%%%%%%%%%%%%%%%%%%%%%%%%%%%%%%%%%%%%
\begin{proof}[Proof of Lemma~\ref{l:close-to-trapping}]
Assume that $t\geq 0$ and
\begin{equation}
  \label{e:henhao}
(x,\xi)\in V,\quad
\pi_\Gamma(e^{tX}(x,\xi))\in\pi_\Gamma(V).
\end{equation}
(The case of propagation for negative time is handled similarly.)
Since $e^{tX}$ is a homogeneous flow, we may additionally
assume that $|\xi|_g=1$.

By~\eqref{e:henhao}, there exists $\gamma'\in\Gamma$ such that
$\gamma'.e^{tX}(x,\xi)\in V$.
Since $K=\Gamma_+\cap\Gamma_-\neq\emptyset$,
by~\eqref{e:Gpm-formula} $\Lambda_\Gamma$ has
at least two elements. Then we may choose $y\in\Lambda_\Gamma$ depending on $\gamma'$ and $(x,\xi)$
such that for some $\varepsilon>0$ independent of $(x,\xi)$,
$$
d(B_-(\gamma'.(x,\xi)),y)>\varepsilon.
$$
Choose $\Xi:=(\xi_2,\dots,\xi_n)$ such that
$(x,\xi,\Xi)\in F^*\mathbb H^n$.
Applying Lemma~\ref{l:reallytech} to $\gamma'.e^{t\widetilde X}(x,\xi,\Xi)$, we see that
there exists a constant $C$ independent of $x,\xi,\Xi$ such that
\begin{equation}
  \label{e:zhidao}
B_+\big(\pi_S\big(e^{\mathbf sU^-}(\gamma'.e^{t\widetilde X}(x,\xi,\Xi))\big)\big)=y
\end{equation}
for some $\mathbf s\in [-C,C]^{n-1}$. Rewriting~\eqref{e:zhidao} using~\eqref{e:horcom}
and since $B_+$ stays constant along the geodesic flow, we get
$$
B_+\big(\pi_S\big(\exp(e^{-t}\mathbf sU^-)(\gamma'.(x,\xi,\Xi))\big)\big)=y.
$$
It follows that
\begin{equation}
  \label{e:dianhua}
B_+\big(\pi_S\big(\exp(e^{-t}\mathbf sU^-)(x,\xi,\Xi)\big)\big)=(\gamma')^{-1}.y\in\Lambda_\Gamma
\end{equation}
where $(\gamma')^{-1}\in\Gamma $ acts on the conformal infinity $\mathbb S^{n-1}$ as in~\cite[\S3.5]{rrh}
and $\Lambda_\Gamma$ is invariant by this action (as can be seen from either~\eqref{e:Lambda-Gamma} or~\eqref{e:Gpm-formula}).
It follows from~\eqref{e:dianhua} that $d(B_+(x,\xi),\Lambda_\Gamma)\leq Ce^{-t}$ for some
constant $C$, finishing the proof.
\end{proof}
%%%%%%%%%%%%%%%%%%%%%%%%%%%%%%%%%%%%%%%%%%%%%%%%%%%%%%%%%%%%%%%%%%%%%%%%%%%%%%%%

%%%%%%%%%%%%%%%%%%%%%%%%%%%%%%%%%%%%%%%%%%%%%%%%%%%%%%%%%%%%%%%%%%%%%%%%%%%%%%%%
\begin{proof}[Proof of Lemma~\ref{l:kappa-M}]
Consider the generating function
\begin{equation}
  \label{e:gen-function}
\mathcal S\in C^\infty(\mathbb H^n_x\times \mathbb R^+_w\times\mathbb S^{n-1}_y;\mathbb R),\quad
\mathcal S(x,w,y)=w\log \mathcal P(x,y).
\end{equation}
Let $\xi,\theta,\eta$ be the momenta on $T^*(\mathbb H^n\times \mathbb R^+\times\mathbb S^{n-1})$ corresponding to $x,w,y$.
We claim that the following two statements are equivalent for each
choice of $(x,w,y,\xi,\theta,\eta)$:
\begin{gather}
  \label{e:haode-1}
\xi=\mp\partial_x\mathcal S,\ \theta=\pm\partial_w\mathcal S,\ \eta=\pm\partial_y\mathcal S;\\
  \label{e:haode-2}
w=p(x,\xi),\ y=B_\mp(x,\xi),\ \theta=\pm\log \mathcal P(x,B_\mp(x,\xi)),\ \eta=\pm G_\mp(x,\xi).
\end{gather}
We compute
$$
\partial_x \mathcal S(x,w,y)=-2w\bigg({x\over 1-|x|^2}+{x-y\over |x-y|^2}\bigg),\quad
\partial_w \mathcal S(x,w,y)=\log \mathcal P(x,y),
$$
and the tangent vector in $T_y\mathbb S^{n-1}\subset\mathbb R^n$ identified
with $\partial_y\mathcal S(x,w,y)$ by the round metric on $\mathbb S^{n-1}$ is
\begin{equation}
  \label{e:zetazeta}
\zeta(x,w,y)=2w{x-(x\cdot y)y\over |x-y|^2}.
\end{equation}
For $(x,\xi)\in T^*\mathbb H^n\setminus 0$, put
$$
\Phi_\pm(x,\xi)=p(x,\xi)\mathcal P\big(x,B_\pm(x,\xi)\big).
$$
We calculate (using for example~\cite[\S3.4]{rrh}; note that in that paper,
$x$ denotes a point in the hyperboloid model and $y$ a point in the ball model;
we relate the two models by~\cite[(3.2)]{rrh}, identify tangent and cotangent vectors
by the metric and extend the formulas from the cosphere bundle to the
whole $T^*\mathbb H^{n+1}\setminus 0$ by homogeneity)
\begin{equation}
  \label{e:hyperfor}
\begin{aligned}
p(x,\xi)&={1-|x|^2\over 2}|\xi|,\\
\Phi_\pm(x,\xi)&={1+|x|^2\over 2}|\xi|\pm x\cdot\xi,\\
\Phi_\pm(x,\xi)B_\pm(x,\xi)&=(|\xi|\pm x\cdot\xi)x\pm {1-|x|^2\over 2}\xi.
\end{aligned}
\end{equation}
We now show that~\eqref{e:haode-1} and~\eqref{e:haode-2} are equivalent.
Assume first that~\eqref{e:haode-1} holds. Then
$$
|\xi|={2w\over 1-|x|^2},\quad
x\cdot\xi=\mp w\Big({1-|x|^2\over |x-y|^2}-{1+|x|^2\over 1-|x|^2}\Big).
$$
It follows that $p(x,\xi)=w$, $\Phi_\mp(x,\xi)=w\mathcal P(x,y)$, and $B_\mp(x,\xi)=y$,
which gives the first three equalities in~\eqref{e:haode-2}. Next,
$$
\begin{aligned}
\Phi_\pm(x,\xi)&=2w{1+|x|^2\over 1-|x|^2}-w\mathcal P(x,y),\\
\Phi_\pm(x,\xi)B_\pm(x,\xi)&={4w\over 1-|x|^2}x-w\mathcal P(x,y)y,\\
\Phi_\pm(x,\xi)(B_+(x,\xi)\cdot B_-(x,\xi))&=\Phi_\mp-{2w\over P(x,y)}.
\end{aligned}
$$
From here we see that $G_\mp(x,\xi)=\zeta(x,w,y)$, with $\zeta$ defined in~\eqref{e:zetazeta};
this finishes the proof of~\eqref{e:haode-2}.

Assume now that~\eqref{e:haode-2} holds. We have
$$
\begin{aligned}
\Phi_\mp(x,\xi)(x\cdot y)&=\mp{1+|x|^2\over 2}(x\cdot\xi)+|\xi|\cdot |x|^2,\\
\Phi_\mp(x,\xi)|x-y|^2&={(1-|x|^2)^2\over 2}|\xi|,\\
\Phi_\mp(x,\xi)(x-y)&=-{1-|x|^2\over 2}(|\xi|x\mp\xi),\\
{2\Phi_\mp(x,\xi)^2\over 1-|x|^2}(x-(x\cdot y)y)&=\Big({1-|x|^2\over 2}|\xi|^2+(x\cdot\xi)^2
\mp|\xi|(x\cdot\xi)\Big)x\\
&-\Big({1+|x|^2\over 2}(x\cdot\xi)\mp|\xi|\cdot |x|^2\Big)\xi.
\end{aligned}
$$
It follows that $w\mathcal P(x,y)=\Phi_\mp(x,\xi)$, $\partial_x\mathcal S(x,w,y)=\mp\xi$, and
$\zeta(x,w,y)=G_\mp(x,\xi)$; this gives~\eqref{e:haode-1}.

The equivalence of~\eqref{e:haode-1} and~\eqref{e:haode-2} implies that $\varkappa^\pm$ is a local symplectomorphism,
since it has a generating function $\mp\mathcal S$.
Moreover, $\varkappa^\pm$ is exact, as $\mp\mathcal S(x,w,y)$ is an antiderivative,
see~\eqref{e:antiderivative}.
To show that $\varkappa^\pm$ is a global symplectomorphism, it remains
to prove that for each $(w,y,\theta,\eta)\in T^*(\mathbb R^+\times\mathbb S^{n-1})$, there exists
unique $(x,\xi)\in T^*\mathbb H^n\setminus 0$ such that $\varkappa^\pm(x,\xi)=(w,y,\theta,\eta)$.
This is a direct consequence of~\eqref{e:kappa-M-def} and the geometry of the hyperbolic space:
the values of $w,\theta$ determine uniquely $p(x,\xi),\Phi_\mp(x,\xi)$ and
then $y,\eta$ determine uniquely $B_+(x,\xi),B_-(x,\xi)$;
then $B_+(x,\xi),B_-(x,\xi)$ determine the normalized geodesic passing through $(x,\xi)$
and $\Phi_\mp(x,\xi),p(x,\xi)$ determine the unique point $(x,\xi)$ on that geodesic.
\end{proof}
%%%%%%%%%%%%%%%%%%%%%%%%%%%%%%%%%%%%%%%%%%%%%%%%%%%%%%%%%%%%%%%%%%%%%%%%%%%%%%%%

%%%%%%%%%%%%%%%%%%%%%%%%%%%%%%%%%%%%%%%%%%%%%%%%%%%%%%%%%%%%%%%%%%%%%%%%%%%%%%%%
\begin{proof}[Proof of Lemma~\ref{l:kappa-hat}]
By~\eqref{e:kappa-hat}, \eqref{e:kappa-hat1} is equivalent to
the fact that
\begin{equation}
  \label{e:kappa-hat0}
(w,y,\theta,\eta)=\varkappa^-(x,\xi),\quad
(w,y',\theta',\eta')=\varkappa^+(x,\xi)
\end{equation}
for some $(x,\xi)\in T^*\mathbb H^n\setminus 0$,
which by~\eqref{e:kappa-M-def} is equivalent to
\begin{equation}
  \label{e:morgul1}
\begin{aligned}
\theta'&=\theta+\log(1+|\eta|^2/w^2),\\
y'&={(|\eta|^2-w^2)y-2w\eta\over |\eta|^2+w^2},\\
\eta'&={2w|\eta|^2y+(|\eta|^2-w^2)\eta\over |\eta|^2+w^2}.
\end{aligned}
\end{equation}
Here we identified vectors and covectors on $\mathbb S^{n-1}$
by the round metric and
used the following identities true for $(x,\xi)\in T^*\mathbb H^n\setminus 0$
and $(y,y')\in\mathbb S^{n-1}_\Delta$:
\begin{gather}
  \label{e:din1}
\mathcal P(x,B_+(x,\xi))\mathcal P(x,B_-(x,\xi))(1-B_+(x,\xi)\cdot B_-(x,\xi))=2,\\
  \label{e:din2}
|\mathcal G(y,y')|^2={1+y\cdot y'\over 1-y\cdot y'}.
\end{gather}
We next see that~\eqref{e:kappa-hat2} is equivalent to
\begin{equation}
  \label{e:morgul2}
\begin{aligned}
\theta'&=\theta+\log{4\over |y-y'|^2},\\
\eta&=-{2w(y'-(y\cdot y')y)\over |y-y'|^2},\\
\eta'&={2w(y-(y\cdot y')y')\over |y-y'|^2}.
\end{aligned}
\end{equation}
A direct calculation shows that~\eqref{e:morgul1} and~\eqref{e:morgul2} are equivalent;
therefore, \eqref{e:kappa-hat1} and~\eqref{e:kappa-hat2} are equivalent as well.
By the proof of Lemma~\ref{l:kappa-M}, the antiderivatives for $\varkappa^+$ and $(\varkappa^-)^{-1}$
are given by $-w\log\mathcal P(x,y')$ and $-w\log\mathcal P(x,y)$, where $x$ is defined by~\eqref{e:kappa-hat0};
their sum is then
$$
-w\log{2\over 1-y\cdot y'}=\Theta(w,y,y').
$$
\end{proof}
%%%%%%%%%%%%%%%%%%%%%%%%%%%%%%%%%%%%%%%%%%%%%%%%%%%%%%%%%%%%%%%%%%%%%%%%%%%%%%%%
%
%%%%%%%%%%%%%%%%%%%%%%%%%%%%%%%%%%%%%%%%%%%%%%%%%%%%%%%%%%%%%%%%%%%%%%%%%%%%%%%%
\begin{proof}[Proof of Lemma~\ref{l:horocyclic}]
We use the coframe bundle $F^*\mathbb H^n$ introduced at the beginning
of this appendix. Let $(x,\xi)\in S^*\mathbb H^n$,
$\eta\in E_u(x,\xi)$, and choose $\Xi=(\xi_2,\dots,\xi_n)$
such that $(x,\xi,\Xi)\in F^*\mathbb H^n$.

As in the proof
of Lemma~\ref{l:reallytech}, we use the hyperboloid model of $\mathbb H^n$
and take $\gamma\in G$ which maps
the standard frame to $(x,\xi,\Xi)$.
Then a direct calculation using~\cite[(3.7)]{rrh} shows that for each
$\mathbf s\in\mathbb R^{n-1}$, we have
\begin{equation}
  \label{e:horocycle-flowed}
\pi_S(e^{\mathbf s U_-}(x,\xi,\Xi))=\bigg(\gamma.\Big(1+{|\mathbf s|^2\over 2},-{|\mathbf s|^2\over 2},-\mathbf s\Big),
\gamma.\Big({|\mathbf s|^2\over 2},1-{|\mathbf s|^2\over 2},-\mathbf s\Big)\bigg)
\end{equation}
where $e^{\mathbf sU_-}:F^*\mathbb H^n\to F^*\mathbb H^n$ is defined in~\eqref{e:horoflow}
and we consider elements of $S^*\mathbb H^n$ as pairs of vectors in the Minkowski space.

Take $\mathbf s\in\mathbb R^{n-1}$ such that
\begin{equation}
  \label{e:sss-1}
\partial_r|_{r=0}\,\pi_S(e^{r\mathbf sU_-}(x,\xi,\Xi))=\eta.
\end{equation}
As follows from a direct calculation using~\eqref{e:horocycle-flowed} and~\cite[(3.14)]{rrh},
such $\mathbf s$ exists, is unique, and depends linearly on $\eta$; moreover
\begin{equation}
  \label{e:sss-2}
|\mathbf s|=|\eta|
\end{equation}
with $|\eta|$ defined in the paragraph preceding~\eqref{e:stun}.

With $\mathbf s$ solving~\eqref{e:sss-1}, we put
\begin{equation}
  \label{e:sss-3}
\mathscr U_-(x,\xi,\eta):=\pi_S(e^{\mathbf sU_-}(x,\xi,\Xi)).
\end{equation}
Then $\mathscr U_-(x,\xi,\eta)$ does not depend on the choice of $\Xi$.
Indeed, take another $\Xi'$ such that $(x,\xi,\Xi')\in F^*\mathbb H^n$,
and let $\gamma'\in G$ map the standard basis of the Minkowski space
to $(x,\xi,\Xi')$. Then $\gamma'=\gamma\tilde\gamma$ where $\tilde \gamma$ lies in the subgroup
$\SO(n-1)\subset G=\PSO(1,n)$ consisting of isometries
of the Minkowski space which fix the first two basis vectors.
Viewing $\tilde\gamma$ as an isometry on $\mathbb R^{n-1}$, we get from~\eqref{e:horocycle-flowed}
$$
\pi_S(e^{\mathbf s U_-}(x,\xi,\Xi'))=\bigg(\gamma.\Big(1+{|\mathbf s|^2\over 2},-{|\mathbf s|^2\over 2},-\tilde\gamma.\mathbf s\Big),
\gamma.\Big({|\mathbf s|^2\over 2},1-{|\mathbf s|^2\over 2},-\tilde\gamma.\mathbf s\Big)\bigg).
$$
Then $s':=(\tilde\gamma)^{-1}.\mathbf s$ solves~\eqref{e:sss-1} with $\Xi$ replaced by $\Xi'$,
and the right-hand sides of~\eqref{e:sss-3} for $\Xi,\mathbf s$ and for $\Xi',\mathbf s'$
are equal.

The fact that $\mathscr U_-$ commutes with the action of $G$ follows immediately from its construction
and the fact that $U_i^-$ are $G$-invariant vector fields on $F^*\mathbb H^n$. The identity~\eqref{e:scroo-1}
follows directly from~\eqref{e:sss-1} and~\cite[(3.16)]{rrh}.

To see~\eqref{e:scroo-2}, note that
by~\eqref{e:stpro} and \eqref{e:din1}
$$
\mathcal G(B_-(x,\xi),B_+(x,\xi))={\Phi_+(x,\xi)\Phi_-(x,\xi)\over 2}B_+(x,\xi)
+\Big(1-{\Phi_+(x,\xi)\Phi_-(x,\xi)\over 2}\Big)B_-(x,\xi)
$$
where $\Phi_\pm(x,\xi):=\mathcal P(x,B_\pm(x,\xi))$.
Put $(\tilde x,\tilde\xi):=\mathscr U_-(x,\xi,\eta)$.
By~\eqref{e:sss-1} and~\cite[(3.16)]{rrh}, we have
$\Phi_-(\tilde x,\tilde\xi)=\Phi_-(x,\xi)$. Then
$$
\begin{gathered}
\mathcal G(B_-(x,\xi),B_+(\tilde x,\tilde \xi))-\mathcal G(B_-(x,\xi),B_+(x,\xi))\\
={\Phi_-(x,\xi)\over 2}(\Phi_+(\tilde x,\tilde\xi)B_+(\tilde x,\tilde\xi)-
\Phi_+(x,\xi)B_+(x,\xi))\\
-{\Phi_+(\tilde x,\tilde\xi)-\Phi_+(x,\xi)\over 2}\Phi_-(x,\xi)B_-(x,\xi).
\end{gathered}
$$
Now, by~\cite[(3.16)]{rrh}, we have in the hyperboloid model
$x\pm \xi=\Phi_\pm(x,\xi)(1,B_\pm(x,\xi))$, therefore
$$
\begin{gathered}
\big(0,\mathcal G(B_-(x,\xi),B_+(\tilde x,\tilde \xi))-\mathcal G(B_-(x,\xi),B_+(x,\xi))\big)\\
={\Phi_-(x,\xi)\over 2}(\tilde x+\tilde \xi-x-\xi)
-{\tilde x_0+\tilde \xi_0-x_0-\xi_0\over 2}\Phi_-(x,\xi)(1,B_-(x,\xi))\\
=\Phi_-(x,\xi)\big(\mathbf w-\mathbf w_0(1,B_-(x,\xi))\big),\quad
\mathbf w=\gamma.\Big({|\mathbf s|^2\over 2},-{|\mathbf s|^2\over 2},-\mathbf s\Big)\in\mathbb R^{1,n}
\end{gathered}
$$
where the last identity follows from~\eqref{e:horocycle-flowed} and~\eqref{e:sss-3}.
Now, since $\gamma.(1,-1,0)=\Phi_-(x,\xi)(1,B_-(x,\xi))$, we have
$$
\mathbf w={|\mathbf s|^2\over 2}\Phi_-(x,\xi)(1,B_-(x,\xi))-\gamma.(0,0,\mathbf s)
$$
therefore~\eqref{e:scroo-2} holds with
$$
(0,\mathcal T_-(x,\xi)\eta)=\mathbf v-\mathbf v_0(1,B_-(x,\xi)),\quad
\mathbf v:=-\gamma.(0,0,\mathbf s)\in\mathbb R^{1,n},
$$
and the fact that $\mathcal T_-$ is an isometry follows from~\eqref{e:sss-2} and the fact that $(1,B_-(x,\xi))$
is a null vector in $\mathbb R^{1,n}$ orthogonal to $\mathbf v$.
\end{proof}
%%%%%%%%%%%%%%%%%%%%%%%%%%%%%%%%%%%%%%%%%%%%%%%%%%%%%%%%%%%%%%%%%%%%%%%%%%%%%%%%

%%%%%%%%%%%%%%%%%%%%%%%%%%%%%%%%%%%%%%%%%%%%%%%%%%%%%%%%%%%%%%%%%%%%%%%%%%%%%%%%
%%%%%%%%%%%%%%%%%%%%%%%%%%%%%%%%%%%%%%%%%%%%%%%%%%%%%%%%%%%%%%%%%%%%%%%%%%%%%%%%
\medskip\noindent\textbf{Acknowledgements.}
We would like to thank Maciej Zworski for constant support and encouragement throughout
the project; Dmitry Dolgopyat,
Fr\'ed\'eric Naud, St\'ephane Nonnenmacher, and Peter Sarnak for several insightful conversations on spectral gaps;
Larry Guth, Alex Iosevich, Hans Parshall, and Tomasz Schoen for several helpful discussions regarding
Theorem~\ref{ADRegularSmallAddEnergyThm};
Alexander Razborov and Jean Bourgain for comments and corrections to an earlier version of this manuscript;
Colin Guillarmou and Peter Perry for a review of the history of hyperbolic scattering;
Jeffrey Galkowski and Dmitry Jakobson for many interesting discussions regarding
this project; and David Borthwick and Tobias Weich for providing the data
for Figure~\ref{f:gaps}(b).

During the period this research was conducted, the first author served as
a Clay Research Fellow and the second author was supported by a NSF mathematical sciences postdoctoral research fellowship.

% arXiv bibliography macro
\def\arXiv#1{\href{http://arxiv.org/abs/#1}{arXiv:#1}}


\begin{thebibliography}{0}

\bibitem[AnRo]{Anderson-Rocha} James Anderson and Andr\'e Rocha,
	\emph{Analyticity of Hausdorff dimension of limit sets of Kleinian groups,\/}
	Ann. Acad. Sci. Fen. Math. \textbf{22}(1997), 349--364.

\bibitem[BWPSKZ]{prl} Sonja Barkhofen, Tobias Weich, Alexander Potzuweit, Hans-J\"urgen St\"ockmann,
	Ulrich Kuhl, and Maciej Zworski,
	\emph{Experimental observation of the spectral gap in microwave $n$-disk systems,\/}
	Phys. Rev. Lett. \textbf{110}(2013), 164102.

\bibitem[B{\L}Z]{bond} Matthew Bond, Izabella~{\L}aba, and Joshua Zahl,
	\emph{Quantitative visibility estimates for unrectifiable sets in the plane,\/}
	preprint, \arXiv{1306.5469}.
	
\bibitem[BoMi]{BonyMichel} Jean-Fran\c cois Bony and Laurent Michel,
	\emph{Microlocalization of resonant states and estimates of the residue of the scattering amplitude,\/}
	Comm. Math. Phys. \textbf{246}(2004), 375--402.

\bibitem[Bo07]{Borthwick} David Borthwick,
	\emph{Spectral theory of infinite-area hyperbolic surfaces,\/}
	Birkh\"auser, 2007.
	
\bibitem[Bo14]{BorthwickNum} David Borthwick,
	\emph{Distribution of resonances for hyperbolic surfaces,\/}
	Experimental Math. \textbf{23}(2014), 25--45.

\bibitem[BoWe]{Borthwick-Weich} David Borthwick and Tobias Weich,
	\emph{Symmetry reduction of holomorphic iterated function schemes
		and factorization of Selberg zeta functions,\/}
	J. Spectr. Th. \textbf{6}(2016), 267--329.

\bibitem[BGS]{BGS} Jean Bourgain, Alex Gamburd, and Peter Sarnak,
	\emph{Generalization of Selberg's 3/16 theorem and affine sieve,\/}
	Acta Math. \textbf{207}(2011), 255--290.
	
\bibitem[BuOl97]{Bunke-Olbrich2} Ulrich Bunke and Martin Olbrich,
	\emph{Fuchsian groups of the second kind and representations carried by the limit set,\/}
	Invent. Math. \textbf{127}(1997), 127--154.
	
\bibitem[BuOl99]{Bunke-Olbrich} Ulrich Bunke and Martin Olbrich,
	\emph{Group cohomology and the singularities of the Selberg zeta function associated to a Kleinian group,\/}
	Ann. of Math. (2) \textbf{149}(1999), 627--689.

\bibitem[Co]{Coornaert} Michel Coornaert,
	\emph{Mesures de Patterson--Sullivan sur le bord d'un espace hyperbolique au sens de Gromov,\/}
	Pacific J. Math. \textbf{159}(1993), 241--270.

\bibitem[Da]{trumpet-of-death} Kiril Datchev,
	\emph{Resonance free regions for nontrapping manifolds with cusps,\/}
	to appear in Analysis\&PDE; \arXiv{1210.7736}.

\bibitem[DaDy]{fwl} Kiril Datchev and Semyon Dyatlov,
	\emph{Fractal Weyl laws for asymptotically hyperbolic manifolds,\/}
	Geom. Funct. Anal. \textbf{23}(2013), 1145--1206.

\bibitem[DaSe]{DS} Guy David and Stephen Semmes,
	\emph{Fractured fractals and broken dreams: Self-similar geometry through metric and measure,\/}
	Oxford University Press, 1997.  

\bibitem[De]{delort-book} Jean-Marc Delort,
	\emph{FBI transformation, second microlocalization, and semilinear caustics,\/}
	Springer, 1992.

\bibitem[Do]{Dolgopyat} Dmitry Dolgopyat,
	\emph{On decay of correlations in Anosov flows,\/}
	Ann. Math. (2) \textbf{147}(1998), 357--390.

\bibitem[Dy]{nhp} Semyon Dyatlov,
	\emph{Resonance projectors and asymptotics for $r$-normally hyperbolic trapped sets,\/}
	J. Amer. Math. Soc. \textbf{28}(2015), 311--381.
	
\bibitem[DFG]{rrh} Semyon Dyatlov, Fr\'ed\'eric Faure, and Colin Guillarmou,
	\emph{Power spectrum of the geodesic flow on hyperbolic manifolds,\/}
	Analysis\&PDE \textbf{8}(2015), 923--1000.
	
\bibitem[DyGu14a]{qeefun} Semyon Dyatlov and Colin Guillarmou,
	\emph{Microlocal limits of plane waves and Eisenstein functions,\/}
	Ann. de l'ENS (4) \textbf{47}(2014), 371--448.
	
\bibitem[DyGu14b]{rnc} Semyon Dyatlov and Colin Guillarmou,
	\emph{Pollicott--Ruelle resonances for open systems,\/}
	Ann. Henri Poincar\'e, published online; \arXiv{1410.5516}.

\bibitem[DyZw16]{zeta} Semyon Dyatlov and Maciej Zworski,
	\emph{Dynamical zeta functions for Anosov flows via microlocal analysis,\/}
	Ann. Sc. Ec. Norm. Sup. \textbf{49}(2016), 543--577.
	
\bibitem[DyZw]{dizzy} Semyon Dyatlov and Maciej Zworski,
	\emph{Mathematical theory of scattering resonances,\/}
	book in progress,
	\url{http://math.mit.edu/~dyatlov/res/}
	
\bibitem[Fr]{freiman} Grigory Fre{\u\i}man,
	\emph{Foundations of a structural theory of set addition,\/}
	Translations of mathematical monographs \textbf{37}, AMS, 1973.
	
\bibitem[GaRi]{GaspardRice} Pierre Gaspard and Stuart Rice,
	\emph{Scattering from a classically chaotic repeller,\/}
	J. Chem. Phys. \textbf{90}(1989), 2225--2241.
	
\bibitem[GrSj]{grigis-sjostrand} Alain Grigis and Johannes Sj\"ostrand,
	\emph{Microlocal analysis for differential operators: an introduction,\/}
	Cambridge University Press, 1994.
	
\bibitem[Gu]{guillarmou} Colin Guillarmou,
	\emph{Meromorphic properties of the resolvent on asymptotically hyperbolic manifolds,\/}
	Duke Math. J. \textbf{129}(2005), 1--37.
	
\bibitem[GuSt77]{gu-st0} Victor Guillemin and Shlomo Sternberg,
	\emph{Geometric asymptotics,\/}
	Mathematical Surveys \textbf{14}, AMS, 1977.

\bibitem[GuSt13]{gu-st} Victor Guillemin and Shlomo Sternberg,
	\emph{Semi-classical analysis,\/}
	International Press, 2013.
	
\bibitem[GuZw]{guillope-zworski} Laurent Guillop\'e and Maciej Zworski,
	\emph{Polynomial bounds on the number of resonances for some complete spaces of constant negative curvature
	near infinity,\/}
	Asymp. Anal. \textbf{11}(1995), 1--22.

\bibitem[H\"oI]{ho1} Lars H\"ormander,
	\emph{The Analysis of Linear Partial Differential Operators I. Distribution Theory and Fourier Analysis,\/}
	Springer, 1990.
	
\bibitem[H\"oIII]{ho3} Lars H\"ormander,
	\emph{The Analysis of Linear Partial Differential Operators III. Pseudo-Differential Operators,\/}
	Springer, 1994.
	
\bibitem[Ik]{Ikawa} Mitsuru Ikawa,
	\emph{Decay of solutions of the wave equation in the exterior of several convex bodies,\/}
	Ann. Inst. Fourier \textbf{38}(1988), 113--146.
	
\bibitem[JaNa]{Jakobson-Naud2} Dmitry Jakobson and Fr\'ed\'eric Naud,
	\emph{On the critical line of convex co-compact hyperbolic surfaces,\/}
	Geom. Funct. Anal. \textbf{22}(2012), 352--368.
	
\bibitem[MOW]{MOW} Michael Magee, Hee Oh, and Dale Winter,
	\emph{Expanding maps and continued fractions,\/}
	preprint, \arXiv{1412.4284}.
	
\bibitem[MaMe]{mazzeo-melrose} Rafe Mazzeo and Richard Melrose,
	\emph{Meromorphic extension of the resolvent on complete spaces with asymptotically constant negative curvature,\/}
	J. Funct. Anal. \textbf{75}(1987), 260--310.

\bibitem[Mc]{McMullen} Curtis McMullen,
	\emph{Hausdorff dimension and conformal dynamics III: computation of dimension,\/}
	Amer. J. Math. \textbf{120}(1998), 691--721.

\bibitem[Na]{Naud} Fr\'ed\'eric Naud,
	\emph{Expanding maps on Cantor sets and analytic continuation of zeta functions,\/}
	Ann. de l'ENS (4) \textbf{38}(2005), 116--153.

\bibitem[NoZw09]{NoZwActa} St\'ephane Nonnenmacher and Maciej Zworski,
	\emph{Quantum decay rates in chaotic scattering,\/}
	Acta Math. \textbf{203}(2009), 149--233.

\bibitem[NoZw15]{NoZwInv} St\'ephane Nonnenmacher and Maciej Zworski,
	\emph{Decay of correlations for normally hyperbolic trapping,\/}
	Invent. Math. \textbf{200}(2015), 345--438.

\bibitem[OhWi]{Oh-Winter} Hee Oh and Dale Winter,
	\emph{Uniform exponential mixing and resonance free regions for convex cocompact congruence subgroups of $\SL_2(\mathbb Z)$,\/}
	to appear in J. Amer. Math. Soc., \arXiv{1410.4401}.

\bibitem[Pa75]{patterson1} Samuel James Patterson,
	\emph{The Laplacian operator on a Riemann surface. I,\/}
	Compos. Math. \textbf{31}(1975), 83--107.
	
\bibitem[Pa76a]{patterson2} Samuel James Patterson,
	\emph{The Laplacian operator on a Riemann surface. II,\/}
	Compos. Math. \textbf{32}(1976), 71--112.

\bibitem[Pa76b]{Patterson} Samuel James Patterson,
	\emph{The limit set of a Fuchsian group,\/}
	Acta Math. \textbf{136}(1976), 241--273.

\bibitem[PaPe]{Patterson-Perry} Samuel James Patterson and Peter Perry,	
	\emph{The divisor of Selberg's zeta function for Kleinian groups,\/}
	Duke Math. J. \textbf{106}(2001), 321--390.

\bibitem[Pe87]{perry} Peter Perry,
	\emph{The Laplace operator on a hyperbolic manifold I. Spectral and scattering theory,\/}
	J. Funct. Anal. \textbf{75}(1987), 161--187.
	
\bibitem[Pe89]{perry2} Peter Perry,
	\emph{The Laplace operator on a hyperbolic manifold. II. Eisenstein series and the scattering matrix,\/}
	J. Reine Angew. Math. \textbf{398}(1989), 67--91.
	
\bibitem[PeSt]{PetkovStoyanov} Vesselin Petkov and Luchezar Stoyanov,
	\emph{Analytic continuation of the resolvent of the Laplacian and the dynamical zeta function,\/}
	Anal. PDE \textbf{3}(2010), 427--489.
	
\bibitem[Pe]{Petridis} Giorgis Petridis,
	\emph{New proofs of Pl\"unnecke-type estimates for product sets in groups,\/}
	Combinatorica \textbf{32} (2012), 721--733.
	
\bibitem[Ra]{Ratcliffe} John Ratcliffe,
	\emph{Foundations of hyperbolic manifolds,\/}
	second edition, Springer, 2006.
	
\bibitem[Ru]{Ruzsa} Imre Ruzsa,
	\emph{Sums of finite sets,\/}
	in \emph{Number Theory: New York Seminar,\/}
	1996, 281--293.
	
\bibitem[Sa]{Sanders} Tom Sanders, 
	\emph{The structure theory of set addition revisited,\/}
	Bull. Amer. Math. Soc. \textbf{50}(2013), 93--127.
		
\bibitem[SjZw]{sj-zw} Johannes Sj\"ostrand and Maciej Zworski,
	\emph{Fractal upper bounds on the density of semiclassical resonances,\/}
	Duke Math. J. \textbf{137}(2007), 381--459.
	
\bibitem[St11]{Stoyanov1} Luchezar Stoyanov,
	\emph{Spectra of Ruelle transfer operators for axiom A flows,\/}
	Nonlinearity \textbf{24}(2011), 1089--1120.
	
\bibitem[St12]{Stoyanov2} Luchezar Stoyanov,
	\emph{Non-integrability of open billiard flows and Dolgopyat-type estimates,\/}
	Erg. Theory Dyn. Syst. \textbf{32}(2012), 295--313.
	
\bibitem[Su]{Sullivan} Dennis Sullivan,
	\emph{The density at infinity of a discrete group of hyperbolic motions,\/}
	Publ. Math. de l'IHES \textbf{50}(1979), 171--202.

\bibitem[TaVu06]{Tao} Terence Tao and Van Vu,
	\emph{Additive Combinatorics,\/}
	Cambridge Studies in Advanced Mathematics \textbf{105}, Cambridge University Press, 2006.

\bibitem[TaVu08]{Tao2} Terence Tao and Van Vu,
	\emph{John-type theorems for generalized arithmetic progressions and iterated sumsets,\/}
	Adv. Math. \textbf{219}(2008), 428--449.

\bibitem[Va13a]{vasy1} Andr\'as Vasy,
	\emph{Microlocal analysis of asymptotically hyperbolic and Kerr--de Sitter spaces,\/}
	with an appendix by Semyon Dyatlov,
	Invent. Math. \textbf{194}(2013), 381--513.

\bibitem[Va13b]{vasy2} Andr\'as Vasy,
	\emph{Microlocal analysis of asymptotically hyperbolic spaces and high energy resolvent estimates,\/}
	Inverse Problems and Applications. Inside Out II, Gunther Uhlmann (ed.),
	MSRI publications \textbf{60}, Cambridge Univ. Press, 2013.

\bibitem[Zw99]{zworski-inventiones} Maciej Zworski,
	\emph{Dimension of the limit set and the density of resonances for convex co-compact hyperbolic surfaces,\/}
	Invent. Math. \textbf{136}(1999), 353--409.

\bibitem[Zw]{e-z} Maciej Zworski,
    \emph{Semiclassical analysis}, Graduate Studies in Mathematics \textbf{138}, AMS, 2012.

\bibitem[Zw15]{vfd} Maciej Zworski,
	\emph{Resonances for asymptotically hyperbolic manifolds: Vasy's method revisited,\/}
	preprint, \arXiv{1511.03352}.

\end{thebibliography}
\end{document}